\documentclass{amsbook}

\usepackage{amssymb}
\usepackage[mathscr]{eucal}
\usepackage[all]{xy}
\usepackage{amsmidx}
\makeindex{HM-index}
\makeindex{HM-notations}

\newtheorem{theorem}{Theorem}[chapter]
\newtheorem{lemma}[theorem]{Lemma}
\newtheorem{proposition}[theorem]{Proposition}

\newtheorem{corollary}[theorem]{Corollary}

\theoremstyle{definition}
\newtheorem{definition}[theorem]{Definition}
\newtheorem{example}[theorem]{Example}

\theoremstyle{remark}
\newtheorem{remark}[theorem]{Remark}

\DeclareMathOperator{\Int}{Int}
\DeclareMathOperator{\Ad}{Ad}
\DeclareMathOperator{\ind}{ind}
\DeclareMathOperator{\Hom}{Hom}

\def\2by2#1#2#3#4{\left(
\begin{array}{cc}
{#1}&{#2}\\  {#3}&{#4}\end{array}
\right)}
\def\vectr#1#2{\binom{#1}{#2}}
\def\otd{\otimes\cdots\otimes}
\def\bs{\backslash}

\font \boldfrak eufb10

\def\bfr#1{\hbox{\boldfrak #1}}

\def\otd{\otimes\cdots\otimes}

\def\nat{{\natural}}
\def\p{\mathfrak{p}}

\def\v#1{\underline{#1}}

\def\paperbook{paper}

\def\C{\mathbb{C}}

\def\R{\mathbb{R}}
\def\Z{\mathbb{Z}}
\def\F{\mathbb{F}}

\def\g{\mathfrak{g}}

\def\t{\mathfrak{t}}

\def\z{\mathfrak{z}}
\def\gP{\mathfrak{P}}
\def\gO{\mathfrak{O}}

\def\gJ{\mathfrak{J}}
\def\cG{\mathcal{G}}

\def\cZ{\mathcal{Z}}
\def\cA{\mathscr{A}}
\def\cB{\mathscr{B}}

\def\cK{\mathcal{K}}

\def\cI{\mathcal{I}}
\def\cM{\mathcal{M}}

\def\cN{\mathcal{N}}

\def\cP{\mathcal{P}}
\def\cH{\mathcal{H}}
\def\cZ{\mathcal{Z}}
\def\cS{\mathcal{S}}
\def\s{\mathfrak{s}}
\def\tr{\hbox{tr}}

\def\gal{\hbox{Gal}}
\def\bs{\backslash}

\def\bG{\mathbf{G}}
\def\bGL{\mathbf{GL}}
\def\bH{\mathbf{H}}
\def\bM{\mathbf{M}}
\def\bN{\mathbf{N}}
\def\bS{\mathbf{S}}
\def\bT{\mathbf{T}}
\def\bU{\mathbf{U}}
\def\bZ{\mathbf{Z}}

\def\ni{\noindent}

\def\bGU{\mathbf{U}}

\numberwithin{section}{chapter}
\numberwithin{equation}{chapter}

\begin{document}
\frontmatter
\title{Distinguished Tame Supercuspidal Representations}

\author{Jeffrey Hakim\footnote{Research of the first author supported in part by an NSA grant.}}
\address{Department of Mathematics and Statistics, American University, 4400 Massachusetts Avenue NW, Washington, DC 20016}
\email{jhakim@american.edu}

\author{Fiona Murnaghan\footnote{Research of the second author supported in part by an
 NSERC Discovery Grant.}}

\address{Department of Mathematics, University of Toronto, 40 Saint George Street, Toronto,
Canada M5S 2E4}
\email{fiona@math.toronto.edu}

\date{July 1, 2006}

\begin{abstract}
This \paperbook\ studies the behavior
of Jiu-Kang Yu's tame supercuspidal representations
relative to involutions of reductive $p$-adic groups.
Symmetric space methods are used to illuminate
various aspects of Yu's construction. Necessary
conditions for a tame supercuspidal representation
of $G$
to be distinguished by (the fixed points of)
an involution of $G$ are expressed in
terms of properties of the $G$-orbit of the associated $G$-datum.
When these conditions are satisfied, the question
of whether a tame supercuspidal representation
is distinguished reduces to the question of
whether certain cuspidal representations of finite
groups of Lie type are distinguished relative to
particular quadratic characters.
As an application of the main results, we obtain necessary
and sufficient conditions for
equivalence of two of Yu's supercuspidal representations
associated to distinct $G$-data.
\end{abstract}

\maketitle

\setcounter{page}{4}
\tableofcontents
\mainmatter

\chapter{Introduction}

\section{General overview}
\label{sec:overview}

Let $G$ be the group $\bG(F)$ of $F$-rational points
of a connected, reductive $F$-group $\bG$, where $F$ is
a nonarchimedean local field of odd residual characteristic.
In this paper, we analyze the behavior,
relative to $F$-involutions of $\bG$, of those
irreducible supercuspidal representations of
$G$ constructed by Jiu-Kang Yu in \cite{Y}.  The terminology
 ``tame supercuspidal representation'' will always be used to refer 
precisely to Yu's representations.   The reader should consult  \cite{Y}
for more details on the motivation for this terminology.  \label{tameness}

Suppose that $\theta$ is an involution of $G$, that is, an
automorphism of $\bG$ of order two
that is defined over $F$, and let $G^\theta$
be the subgroup of $G$ consisting of those points that are fixed by $\theta$.
Given $\theta$, the theory of distinguished representations
involves the study of the space $\Hom_{G^\theta}(\pi,1)$,
for $\pi$ an irreducible smooth representation of $G$.
The latter space comprises those linear forms
$\lambda:V_\pi\rightarrow \C$ on the representation
space $V_\pi$ of $\pi$ that satisfy $\lambda(\pi(h)v)=\lambda(v)$
for all $h\in G^\theta$ and $v\in V_\pi$. The representations
$\pi$ for which $\Hom_{G^\theta}(\pi,1)$ is nonzero are called
$\theta$-\textit{distinguished representations} or 
$G^\theta$-\textit{distinguished representations} \label{distinguishedness} (or simply
\textit{distinguished representations} when the choice of
$\theta$ is clear). Our main result concerning distinguished 
representations, Theorem \ref{maindimformula}, gives a formula for
the dimension of $\Hom_{G^\theta}(\pi,1)$ when $\pi$
is a tame supercuspidal representation of $G$.  
The other main results of the paper, Theorems~\ref{partialeqprob}
and \ref{equivtheorem},
give necessary and sufficient conditions for equivalence
of two tame supercuspidal representations associated to
distinct $G$-data. These conditions
are expressed in terms of properties of the data that are
used in Yu's construction. As discussed below, these theorems are 
proved by applying our results concerning
distinguished tame supercuspidal representations
in a particular setting.

In Section~\ref{sec:fivetwotwo}, we discuss the statement of Theorem~\ref{maindimformula} 
in more detail and in Section~\ref{sec:distoverview} we describe some of its better known 
applications to the theory of harmonic analysis on symmetric spaces and the theory of 
periods of automorphic forms.
Before doing this, we would like to emphasize that our work has a variety of applications with no apparent connection to the theory of distinguished representations 
(such as Theorems~\ref{partialeqprob} and \ref{equivtheorem})
and these results shed light on the most basic aspects of tame supercuspidal representations.
Some of these applications are necessary in the development of Theorem \ref{maindimformula} and others are a consequence of it.

To describe them, let us develop some background.  Yu's construction, which is recalled in
some detail in Sections~\ref{sec:notations} and \ref{sec:construction}, starts with an object
$\Psi$ that we call a \textit{generic cuspidal $G$-datum}. \label{cuspGdat}
(The term ``cuspidal''  is included to emphasize that we are working with data that yield supercuspidal representations, since the construction applies to a more general class of data to give representations that are not necessarily supercuspidal.)
Yu associates to each generic cuspidal $G$-datum
$\Psi$ a compact-mod-center subgroup $K=K(\Psi)$
of $G$ and an irreducible smooth representation
$\kappa=\kappa(\Psi)$ of $K$, such that the
representation $\pi=\pi(\Psi)$ of $G$ obtained by
smooth, compactly supported induction from $\kappa$
is irreducible and supercuspidal.

Given two generic cuspidal $G$-data $\Psi_1$ and $\Psi_2$, one may ask when 
$\kappa (\Psi_1)$ and $\kappa (\Psi_2)$ or $\pi (\Psi_1)$ and
 $\pi (\Psi_2)$ are equivalent.  In other words, we are asking about the fibers of the maps $\Psi \mapsto [\kappa(\Psi)]$ and $\Psi\mapsto [\pi(\Psi)]$, where $[\tau]$ is used to denote the equivalence class of a representation $\tau$.  Though these questions do not seem to involve distinguished representations, the answers are a consequence of Theorem \ref{maindimformula}.
Understanding why this is so begins with the elementary observation that if
 $\pi_1 = \pi(\Psi_1)$ and $\pi_2= \pi(\Psi_2)$ then $\pi_1$ and $\pi_2$ are equivalent 
exactly when the tensor product representation $\pi_1\times \tilde\pi_2$  of $G\times G$ 
is distinguished with respect to the involution $\theta (g,h) = (h,g)$.  
(In general, $\tilde\pi$ denotes the contragredient of $\pi$.)
One then introduces a natural contragredient operation $\Psi\mapsto\tilde\Psi$ on generic 
cuspidal $G$-data. This operation is defined in Section~\ref{sec:contragredients}
and it has the property that $\pi (\tilde\Psi)$ is equivalent to
the contragredient $\tilde\pi(\Psi)$ of $\pi(\Psi)$.
Next, one defines a product operation $\Psi\times \Psi'$ 
in such a way that the product of 
two generic cuspidal $G$-data is a generic cuspidal $(G\times G)$-datum and 
$\pi (\Psi\times\Psi')$ is equivalent to $\pi(\Psi)\times \pi(\Psi')$.  
(See Section~\ref{sec:Yuproducts}.)
Then we note that $\pi_1$ and $\pi_2$ are equivalent
 exactly when $\pi (\Psi_1\times\tilde\Psi_2)$ is 
distinguished and we apply Theorem \ref{maindimformula} to determine 
when this occurs.
 
Our description of the fibers of the map $\Psi \mapsto [\pi(\Psi)]$
is in terms of two basic operations on generic cuspidal $G$-data: 
(1) conjugation by $G$, and (2) refactorization.
As for conjugation by $G$, there is an obvious way in which an $F$-automorphism of $\bG$ acts on the set of all generic cuspidal $G$-data.  It is easy to see that in the special case of conjugation by an element of $G$, the equivalence class of the resulting tame supercuspidal representation is preserved.

Refactorization generalizes a procedure used in Howe's construction
 of tame supercuspidal representations of general linear groups.  
A given generic cuspidal datum involves certain quasicharacters $\phi_i$ and a 
finite-dimensional representation $\rho$ that is related to a cuspidal
representation of a reductive group over a finite field. In the case of general linear
groups $\rho$ is associated to a quasicharacter of an unramified torus 
(via the construction of Deligne and Lusztig).
The quasicharacters $\phi_i$, together with the quasicharacter associated to $\rho$,
 are rough analogues of the factors in the Howe factorization 
in Howe's construction. 
In Howe's construction, the Howe factorization is not unique and one often needs 
to adjust the factors for convenience in applications.   Our notion of refactorization 
gives a similar adjustment of the $\phi_i$'s and of $\rho$ in such a way that the 
equivalence classes of the corresponding representations $\kappa$ and $\pi$ do not change.  
For more information on the relation between Howe's
construction and Yu's construction as it applies to general linear groups,
the reader may refer to Section~\ref{sec:Howeconstruction}.

We show in Theorem~\ref{partialeqprob} that  whenever two generic cuspidal $G$-data 
yield equivalent tame supercuspidal representations then the $G$-data are 
essentially related by 
refactorization and conjugation.  In Theorem~\ref{equivtheorem}, we give
necesssary and sufficient conditions for equivalence in terms of
conjugacy of certain groups attached to the two generic,
cuspidal $G$-data, and equivalence up to conjugacy of certain twists of
representations occurring in the $G$-data.
Our determination of the fibers of Yu's map $\Psi\mapsto [\pi (\Psi)]$ complements 
the recent work \cite{K} of Ju-Lee Kim, which studies the image of the correspondence.
We also determine the fibers of the map $\Psi\mapsto [ \kappa (\Psi)]$.

Utilizers of Yu's construction might find some value in our exposition of the 
construction.   For example, Yu's construction associates to a generic cuspidal
$G$-datum not just a tame supercuspidal representation of $G$, but also
a finite sequence $\vec \pi$ of representations.  Yu indicates that this 
fact should be of some use in constructing inductive arguments when studying tame supercuspidal representations.  
We explain this idea in more detail and provide applications of it.

We also emphasize the fact that the inducing representation $\kappa = \kappa (\Psi)$ has a natural tensor product decomposition
$\kappa_{-1}\otd \kappa_d$, where $\kappa_i$ is attached to the quasicharacter $\phi_i$, when $i\ne -1$, and $\kappa_{-1}$ is attached to the depth zero representation $\rho$.  The study of $\kappa(\Psi)$ often reduces to a study of the factors $\kappa_i$ in much the same way that certain aspects in the theory of automorphic representations reduce to analogous local issues.  In the case of automorphic representations, the uniqueness of Whittaker models (or some symmetric space analogue of this) is the key.  In the theory of tame supercuspidal representations, this is replaced by certain uniqueness properties of Heisenberg groups over finite fields with prime order.
 
In two previous papers \cite{HM2} and \cite{HM3}, we studied distinguishedness
of tame supercuspidal representations of general linear groups relative
to three different involutions. In Section~\ref{sec:glexamples}, we 
indicate how the results of those papers can be interpreted in relation 
to Theorem~\ref{maindimformula}.
In Section~\ref{sec:examples}, we give two examples of tame supercuspidal
representations $\pi$ and involutions $\theta$ with $\dim{\rm Hom}_{G^\theta}(\pi,1)>1$. 
We also give
an example of an application of Theorem~\ref{maindimformula}
to a family of tame supercuspidal representations (originally
studied in \cite{HM3}) 
for which two distinct $K$-orbits of involutions
could potentially contribute to the space of
$G^\theta$-invariant linear forms on the space of $\pi$.
However, the analysis in \cite{HM3} can be used to see that only one of these
$K$-orbits does contribute a nonzero value.

The authors thank Jeffrey Adler for helpful conversations,
and both Joshua Lansky and the referee for detailed comments on 
the manuscript.

 \section{The main theorem}
\label{sec:fivetwotwo}
We now offer a rough statement of Theorem \ref{maindimformula}.  Section \ref{sec:mainproof} gives an extended outline of the proof and thereby provides a road-map for the structure of much of the \paperbook.  (The reader is encouraged to use Section~\ref{sec:mainproof} as a guide rather than reading the \paperbook\ sequentially.)
For the remainder of this section, we fix both $G$ and the inducing subgroup $K$.  A generic, cuspidal $G$-datum with $K(\Psi) =K$ will be called a ``$(G,K)$-datum.''  If $\Psi_1$ and $\Psi_2$ are $(G,K)$-data that are related by refactorization and conjugation by $K$ then we write $\Psi_1\sim\Psi_2$ and regard the data as being equivalent.  (From Theorem \ref{partialeqprob}, we see that $\Psi_1\sim \Psi_2$ exactly when $\kappa (\Psi_1)\simeq\kappa (\Psi_2)$, though the proof requires Theorem \ref{maindimformula}.)
The set of $(G,K)$-data will be denoted by $\Xi$, while the set of equivalence classes 
of $(G,K)$-data will be denoted by $\Xi^K$.

Now fix a $G$-orbit $\Theta$ of involutions of $G$ relative to the action $g\cdot \theta= \Int (g)\circ \theta \circ \Int (g)^{-1}$ with $\Int (g)(h) = ghg^{-1}$.  Let $\Theta^K$ be the set of $K$-orbits contained in $\Theta$.  If $\Theta'\in \Theta^K$ and $\xi\in \Xi^K$ define
$$\langle \Theta',\xi\rangle_K = \dim \Hom_{K^\theta}(\kappa(\Psi),1),$$ where $\theta$ and $\Psi$ are arbitrary elements of $\Theta'$ and $\xi$, respectively, and $K^\theta = K\cap G^\theta$.  It is straightforward to verify that the choices of $\theta$ and $\Psi$ are of no consequence and thus $\langle \Theta',\xi\rangle_K$ is a well defined (finite) number.
It may happen that $K$ is not $\theta$-stable, but in this case $\langle \Theta',\xi\rangle_K$ must vanish.  

So we have a correspondence between the set $\Theta^K$ of $K$-orbits of involutions in 
$\Theta$ and the set $\Xi^K$ of equivalence classes of $(G,K)$-data such that 
$\Theta'$ and $\xi$ correspond
when $\langle \Theta',\xi\rangle_K$ is nonzero.  Similarly, we have a correspondence between the set of $G$-orbits of involutions of $G$ and the set $\Xi^K$ defined via the pairing
$$\langle \Theta,\xi\rangle_G = \dim\Hom_{G^\theta} (\pi(\Psi),1).$$
Recall that the latter Hom-spaces are the fundamental objects studied in this \paperbook.

Proposition \ref{compatorbsum} gives a preliminary formula 
$$\langle \Theta ,\xi\rangle_G = m_K(\Theta)\sum_{\Theta'\in \Theta^K} \langle \Theta', \xi\rangle_K$$ for $\langle \Theta,\xi\rangle_G$ in terms of the constants $\langle \Theta',\xi\rangle_K$ and   a finite geometric constant 
$m_K(\Theta)$ attached to $\Theta$ (described in Sections \ref{sec:mainproof} and \ref{sec:Mackey}).  

The crucial issue is the analysis of the $\langle \Theta' ,\xi\rangle_K$'s.  Assume $\Theta'$ and $\xi$ have been fixed and $\langle \Theta',\xi\rangle_K$ is nonzero.  The nonvanishing condition is a severe restriction that guarantees that $\theta\in \Theta'$ and $\Psi\in \xi$
may be chosen so that $\Psi$ is $\theta$-symmetric in a natural sense.  In fact, if $\theta\in \Theta'$ and $\Psi\in \xi$ are chosen arbitrarily then there exists a refactorization $\dot\Psi$ of $\Psi$ and an element $k\in K$ such that ${}^k\dot\Psi$ is $\theta$-symmetric or, equivalently, $\dot\Psi$ is $(k^{-1}\cdot\theta)$-symmetric.

Once $\Psi$ and $\theta$ have been chosen so that $\Psi$ is $\theta$-symmetric, the task of studying $\langle \Theta',\xi\rangle_K$ becomes more accessible.  Though we do not fully compute $\langle \Theta', \xi\rangle_K$, we obtain a formula for it in terms of simpler objects that essentially involve the representation theory of finite groups.

To further explain our results,  let us first recall in slightly more detail what it 
means for $\Psi$ to be a generic cuspidal $G$-datum.   The datum $\Psi$ is a 4-tuple 
$(\vec\bG,y,\rho,\vec\phi)$ where, very roughly speaking, the components are as follows:
\begin{itemize}
\item $\vec\bG$ is a tower $\bG^0\subset \cdots \subset \bG^d = \bG$ of $F$-subgroups of $\bG$.  
\item $y$ is a suitable point in the extended Bruhat-Tits building of $G^0=\bG^0(F)$.
\item $\rho$ is an irreducible representation of the normalizer $K^0$ of the parahoric subgroup
$G^0_{y,0}$ in $G^0$, the restriction $\rho\,|\, G_{y,0^+}^0$ of $\rho$ to
the pro-unipotent radical $G_{y,0^+}^0$ of $G_{y,0}^0$ is a multiple of
the trivial representation, and $\rho\,|\, G^0_{y,0}$ contains the inflation
of an irreducible cuspidal representation of the finite group $G_{y,0}^0/G_{y,0^+}^0$.
\item $\vec\phi = (\phi_0,\dots ,\phi_d)$, where $\phi_i$ is a quasicharacter of 
$G^i = \bG^i(F)$, and $\phi_i$ satisfies a certain genericity condition relative 
to $G^{i+1}$ if $i\not=d$.
\end{itemize}

Twisting the representation $\rho$ by the character $\prod_{i=0}^d (\phi_i\, |\,K^0)$ 
gives a representation $\rho'$ of $K^0$ that is invariant under refactorizations of $\Psi$.

Let $p$ be the residual characteristic of $F$.
Using the theory of Heisenberg representations over a field with $p$ elements, we define a certain character $\eta_\theta^\prime$ of exponent two of the group $K^{0,\theta}$ of 
$\theta$-fixed points in $K^0$.  
The formula
$$
\langle \Theta' ,\xi\rangle_K = \dim \Hom_{K^{0,\theta}} (\rho' ,\eta'_\theta)
$$   
is one of the results in Theorem~\ref{maindimformula}.
 In other words, Theorem \ref{maindimformula} essentially reduces the theory of 
distinguished tame supercuspidal representations to the theory of distinguished 
representations of finite groups of Lie type.  If the representations of the latter 
finite groups are Deligne-Lusztig representations, one can appeal to Lusztig's 
results \cite{L}.
Otherwise, little is known about the dimension of 
$ \Hom_{K^{0,\theta}} (\rho' ,\eta'_\theta)$.

Theorem~\ref{maindimformula} also gives precise information about the set of all $K$-orbits 
$\Theta'\in \Theta^K$ such that $\langle \Theta',\xi\rangle_K$ is nonzero.  
(See also Proposition \ref{oldLemmasBandC}.)  Specifically, if $\Theta'$ is one such 
$K$-orbit and $\theta\in \Theta'$ then for any other such orbit $\Theta''$ there exists 
$g\in G$ such that $g\theta (g)^{-1}\in K^0$ and $g\cdot \theta\in \Theta''$.  
This information is potentially very useful in the computation of $\langle \Theta,\xi\rangle_G$.

\section{Distinguished representations}
\label{sec:distoverview}

The terminology ``distinguished representation'' is most frequently used
in the mathematical subculture which centers around Jacquet's theory of relative trace formulas and periods of automorphic representations.  (See \cite{J}.) We wish to stress, however, that this \paperbook\ is addressed at several mathematical audiences, only one of which  is the latter group of mathematicians.  What is standard background material for one of these audiences needs to be explained to the other audiences and we have therefore attempted not to assume too much prerequisite knowledge of the reader. 

At the most basic level, the importance of $\theta$-distinguished representations is derived from Frobenius reciprocity, which implies that $\Hom_{G^\theta}(\pi,1)$ is canonically isomorphic to $\Hom_G (\pi , C^\infty(G^\theta\bs G))$.  Thus the $\theta$-distinguished representations may be viewed as precisely the representations which enter into the harmonic analysis on the $F$-symmetric space $G^\theta\bs G$.  But this hardly begins to describe the broad significance of the theory of distinguished representations to representation theory.  

For a given involution $\theta$, the set of all $\theta$-distinguished representations tends to be the image of an important correspondence, such as a Langlands lifting or a theta-lifting.  In practice, the existence of such correspondences is established through indirect means, such as global trace formula arguments.  It is hoped that the results and techniques in this \paperbook\ might provide a better understanding of distinguished representations which will lead to explicit constructions of correspondences between sets of representations on different groups.  

The linear forms in $\Hom_{G^\theta}(\pi,1)$ may be viewed as local analogues of period integrals associated to an automorphic representation.  Such periods arise in many contexts and the local and global theories often  closely parallel each other.   For example, it is not unusual for the existence of a pole of an automorphic $L$-function to depend on whether  or not a certain period integral vanishes.  Similarly, the image of a lifting map  may be 
described by a period condition.  Given a reductive group and a Levi subgroup, one may consider the problem of when an irreducible representation of the Levi subgroup induces (via parabolic induction) an irreducible representation of the group in which it is embedded.  It turns out that whether or not the induced representation is irreducible often depends on whether or not the inducing representation is distinguished in a suitable sense.
In short, the applications of distinguished representations and periods of automorphic forms are numerous and the literature is vast.  We refer to the survey article \cite{J} for more information and references.

\section{Synopsis of the main proof}
\label{sec:mainproof}

The main result in this \paperbook, Theorem \ref{maindimformula}, gives 
an expression for the space $\Hom_{G^\theta}(\pi,1)$, where $\theta$ is an involution of $G$ and $\pi = \ind_K^G(\kappa)$ is a tame supercuspidal representation associated to some 
generic cuspidal $G$-datum 
$\Psi = (\vec\bG , y, \rho,\vec\phi)$ by Yu's construction.  We now  give a general 
outline of the proof of Theorem \ref{maindimformula}.

The first step is elementary.   We use Mackey's theory of induced representations to obtain a canonical isomorphism
$$\Hom_{G^\theta}(\pi ,1)\cong \bigoplus_{KgG^\theta\in K\bs G/G^\theta} \Hom_{K\cap gG^\theta g^{-1}} (\kappa,1).$$   The latter sum is parametrized by double cosets in the space $K\bs G/G^\theta$, however, it is also possible to translate this parametrizing set to a space of twisted conjugacy classes or to a space of orbits of involutions.  (See Section~\ref{sec:Mackey}.) 

If one is interested in studying the geometry of the parametrizing space, then it is 
sometimes more convenient to work in terms of twisted conjugacy classes.  This is the case, 
for example, if $G$ is a general linear group and $G^\theta$ is a unitary or 
orthogonal group, since 
then one ends up working with the convenient space of all hermitian matrices or 
orthogonal matrices of a given rank.  (See \cite{HMa1}, \cite{HMa2} or \cite{HM2}.)  

For many purposes, the most illuminating approach is to work in terms of orbits of involutions.
Recall that we are letting $G$ act on the involutions of $G$, with
$g\cdot \theta = \Int (g)\circ \theta \circ \Int(g)^{-1}$ for $g\in G$ and $\theta$
an involution of $G$.  Let $\Theta$ be the $G$-orbit of $\theta$.  Then we may 
write $d_\Theta (\pi)$ for the dimension of $\Hom_{G^\theta}(\pi,1)$, since it is easy to see that the latter dimension stays constant as we vary the choice of $\theta$ in a fixed $G$-orbit $\Theta$.  Lemma \ref{orbmult} translates the above decomposition of $\Hom_{G^\theta}(\pi,1)$ into a dimension formula
$$d_\Theta (\pi) = m_K(\Theta) \sum_{\Theta'} d_{\Theta'}(\kappa),$$ where we are summing over the $K$-orbits $\Theta'$ contained in $\Theta$,
and $d_{\Theta'}(\kappa)$ is the dimension of $\Hom_{K\cap G^{\theta'}}(\kappa,1)$,
for $\theta'\in \Theta'$.  The factor 
$m_K(\Theta)$ is geometric in nature and reflects the fact that various double cosets naturally correspond to a single $K$-orbit $\Theta'$ and all of these double cosets which get grouped together for a fixed $\Theta'$ make exactly the same contribution to $d_\Theta(\pi)$.  These multiplicities are not evident when one uses the double coset point of view.

Suppose now we want to study a particular summand $\Hom_{K\cap gG^\theta g^{-1}}(\kappa,1)$.  One of the first things to observe is that $gG^\theta g^{-1} = G^{g\cdot \theta}$.  So, after replacing $g\cdot \theta$ by $\theta$, we may as well assume we are studying $\Hom_{K^\theta}(\kappa,1)$, where $K^\theta = K\cap G^\theta$.

Next, we develop and apply a basic tool called ``refactorization.''  \label{refactorintro} In the datum $\Psi = (\vec\bG,y,\rho, \vec\phi)$, the last component $\vec\phi$ is a sequence $(\phi_0,\dots ,\phi_d)$ of quasicharacters which are analogous to the factors in Howe's factorization in the construction of tame supercuspidal representations of general linear groups.  (See \cite{Ho}, \cite{Moy}.)  In Howe's construction, it is possible to make certain adjustments to the factors without affecting the supercuspidal representation which is constructed.  An analogue of this is developed in 
Section~\ref{sec:refactorization}.  Then, in Section~\ref{sec:symmetrizing}, we show that, after refactoring, we may assume that our given datum is ``weakly $\theta$-symmetric'' \label{weaksymintro} in the sense that each component $\bG^i$ of $\vec\bG$ is $\theta$-stable and $\phi_i \circ\theta = \phi_i^{-1}$.

Having reduced to studying $\Hom_{K^\theta}(\kappa,1)$ when $\Psi$ is weakly $\theta$-symmetric, we now make a further reduction.  We show that the latter Hom-space vanishes unless $\theta$ fixes the point $[y]$ in the reduced building $\cB_{\rm red}(\bG,F)$ which comes from $y$.  The condition $\theta [y] = [y]$ should really be viewed as a condition on the $K$-orbit $\Theta'$ of $\theta$, since $\theta [y]= [y]$ implies $\theta' [y]=[y]$ for all $\theta'\in \Theta'$.  Showing that $\Hom_{K^\theta}(\kappa,1)=0$ unless $\theta[y]=[y]$ requires a descent argument which we now describe.  

Let us say that the cuspidal  $G$-datum $\Psi$ has ``degree $d$''  \label{degreeintro} if $\vec\bG$ and $\vec\phi$ have $d+1$ components.  Roughly speaking, one obtains another cuspidal $G$-datum $\partial\Psi$ of degree $d-1$ by deleting the last entries in $\vec\bG$ and $\vec\phi$.  Let $\partial\kappa$ be the analogue of  $\kappa$ for $\partial\Psi$.  We will say that $\kappa$ is ``quadratically distinguished'' \label{qdistintro} if there is a character $\alpha$ of $K^\theta$ such that $\alpha^2=1$ and $\Hom_{K^\theta}(\kappa,\alpha)$ is nonzero. Proposition \ref{quaddistprop} asserts that if $\kappa$ is quadratically distinguished then so is $\partial\kappa$.  Suppose now that $\Hom_{K^\theta}(\kappa,1)\ne 0$.  Then, in particular, $\kappa$ is quadratically distinguished.  Applying Proposition \ref{quaddistprop} repeatedly, we ultimately deduce that the representation $\rho$ is quadratically distinguished.  Now suppose that $\theta [y]\ne [y]$.  We show that $\rho$ cannot be quadratically distinguished because that would contradict the cuspidality of the representation $\bar\rho$ of $G^0_{y,0}/G^0_{y,0^+}$ which comes from $\rho$.  The latter argument is contained in the proof of Proposition \ref{compatorbsum}.

It should be noted that the proof of Proposition \ref{quaddistprop} is deceptively short.  In fact, the proof uses a difficult geometric fact, developed in Section \ref{sec:maxiso} which says that the fixed points of $\theta$ determine a maximal isotropic subspace in each of the symplectic spaces associated to the Heisenberg $p$-groups involved in Yu's construction (under the assumption of weak $\theta$-symmetry).  This fact is quite simple to establish when $\theta [y] = [y]$, but requires a lengthy argument otherwise.

We are now reduced to studying $\Hom_{K^\theta}(\kappa,1)$ in the case of a weakly $\theta$-symmetric cuspidal  $G$-datum with $\theta [y]=[y]$.  For this, we need a stronger tool than the $\partial$ operator described above, namely, the factorization theory in 
Section~\ref{sec:factorization}.  The inducing representation $\kappa$ has a factorization 
$\kappa =\kappa_{-1}\otd \kappa_d$, where $\kappa_{-1}$ is attached to $\rho$ and, otherwise, 
$\kappa_i$ is attached to $\phi_i$.  (See Section~\ref{sec:construction} for more details.)  This factorization may be viewed as roughly analogous to the factorization of an automorphic representation into local components.  For automorphic representations, in ideal circumstances, one has a Hasse principle which reduces a particular property of an automorphic representation into a collection of analogous local problems.  Similarly, in our case we are able to separate out the contributions of the various $\phi_i$'s to $\Hom_{K^\theta}(\kappa,1)$, under the present symmetry conditions.  In particular, we show in 
Section~\ref{sec:factorization} that $\Hom_{K^\theta}(\kappa,1)$ factors as a tensor product of Hom-spaces, where there is one Hom-space for each $\phi_i$ and another Hom-space for~$\rho$.

The factorization theory requires a ``multiplicity one'' \label{multoneintro} argument analogous to the use of the multiplicity one property of local Whittaker models in the theory of automorphic representations.  Ultimately, we show that the Hom-space attached to each $\phi_i$ has dimension one and thus makes no contribution to the tensor product.  This leaves us with the factor associated to $\rho$, which is exactly what one needs to finish the proof of Theorem \ref{maindimformula}.

Showing that the Hom-space for $\kappa_i$ has dimension one involves several steps.  Our description of Yu's construction in Section~\ref{sec:construction} is different from \cite{Y} in style in that it highlights the fact that $\kappa$ is a tensor product.  Each of the factors $\kappa_i$, other than $\kappa_{-1}$, is obtained by starting with a Heisenberg representation $\tau_i$ on some space $V_i$.  Then, using the theory of the Weil representation, the representation $\tau_i$ is extended to a representation $\phi'_i$ on a larger group, but with the same representation space $V_i$.  The representation $\kappa_i$ of $K$ is obtained by a natural inflation process from the group $K^{i+1}$ on which $\phi'_i$ is defined.  (Note that the symplectic space involved in the Heisenberg construction may have dimension zero, in which case $\tau_i$, $\phi'_i$ and $\kappa_i$ have dimension one.)

The Heisenberg representation $\tau_i$ is defined on a certain quotient $\cH_i= J^{i+1}/N_i$ of subgroups of $K$ and we also regard it as a representation of $J^{i+1}$.  In Section~\ref{sec:Heis}, we study in detail the abstract theory of representations of Heisenberg $p$-groups, such as $\cH_i$.  In particular, we consider the behavior of Heisenberg representations with respect to an involution $\alpha$ of the Heisenberg group.  In the present context, the involution $\theta$ of $G$ determines an involution $\alpha$ of $\cH_i$ and we show that $\Hom_{\cH^\alpha_i}(\tau_i,1)$ has dimension one, where $\cH_i^\alpha$ is the group of fixed points of $\alpha$ in $\cH_i$.

Suppose $\lambda_i$ is a nonzero element of $\Hom_{\cH^\alpha_i}(\tau_i,1)$.  Since $\tau_i$, $\phi'_i$ and $\kappa_i$ all act on the same representation space $V_i$, we may consider the invariance properties of $\lambda_i$ with respect to $\phi'_i$ and $\kappa_i$.  In Proposition \ref{multonekappai}, we show that there is a character $\xi_i$ of $K^\theta$ with $\xi_i^2=1$ such 
$$\Hom_{\cH^\alpha_i}(\tau_i,1) = \Hom_{K^{i+1,\theta}}(\phi'_i,\xi_i) = \Hom_{K^\theta}(\kappa_i, \xi_i).$$  The space $\Hom_{K^\theta}(\kappa_i,\xi_i)$ is the Hom-space attached to $\kappa_i$ which occurs in the factorization theory.

The theory of Heisenberg $p$-groups and their representations is discussed in several
 sections of the \paperbook.  Section~\ref{sec:Heis} discusses the intrinsic theory, that is,
 those aspects of the theory that do not depend on how $\cH_i$ embeds as a subquotient of $K$, 
but instead only depend on $\cH_i$ viewed as an abstract group.  For example, 
Yu's notion of ``special isomorphism'' \label{sisointro} is discussed in Section~\ref{sec:Heis}.
 (A special isomorphism is essentially an isomorphism of an abstract Heisenberg $p$-group 
with some standard Heisenberg $p$-group.)  The extrinsic properties of the Heisenberg 
$p$-groups $\cH_i$ associated to generic quasicharacters $\phi_i$ are first discussed in 
Section~\ref{sec:genHeis}.  It is necessary to choose special isomorphisms that are 
simultaneously well behaved with respect to the ``embedding'' of the Heisenberg group in 
$K$ and compatible with the involution $\theta$.  We show that, in fact, the special 
isomorphisms constructed by Yu are suitably compatible with $\theta$.  Later, in 
Section~\ref{sec:appliedHeis}, we concentrate on the space $\Hom_{K^\theta}(\kappa_i,\xi_i)$.

\chapter{Algebraic background}

\section{Basic facts about distinguished supercuspidal  representations}
\label{sec:Mackey}

Let $G$ be a totally disconnected group with center $Z$ and suppose $K$ is an open subgroup of $G$ such that $K/(K\cap Z)$ is compact.  Let $(\kappa,V_\kappa)$ be a smooth irreducible representation of $K$ and let $\pi = \ind_K^G(\kappa)$ be the representation of $G$ obtained by smooth compactly supported induction from $\kappa$.  Hence, the space of $\pi$ consists of all functions $f: G\to V_\kappa$ that satisfy $$f(kg) = \kappa(k)f(g),$$ for all $k\in K$ and $g\in G$, that have compact support modulo $Z$, and are right $K_f$-invariant for some open subgroup $K_f$ of $G$.

The representation $\pi$ is irreducible exactly when $\Hom_G(\pi,\pi)$ has dimension one.  Mackey theory describes the latter space in terms of the intertwining properties of $\kappa$.  In particular, we have a canonical decomposition
$$\Hom_G(\pi,\pi)\cong \bigoplus_{KgK\in
K\bs G/K}I_g(\kappa),$$
where $$I_g(\kappa)= \Hom_{gKg^{-1}\cap
K}({}^g\kappa,\kappa)$$ and ${}^g\kappa$ is the
representation of $gKg^{-1}$ on the space of
$\kappa$ given by
${}^g\kappa (g')= \kappa( g^{-1}g'g)$.   So $\pi$ is irreducible exactly when there is a unique double coset $KgK$ such that $I_g(\kappa)\ne 0$.  In fact, this double coset must be the double coset of the identity element, since $I_1(\kappa)$ is nonzero.

Now assume $\pi$ is irreducible.  The contragredient of $\pi$ is the representation $\tilde\pi= \ind_K^G(\tilde\kappa)$, where $\tilde\kappa$ is the contragredient of $\kappa$. 
Let $\kappa\times\tilde\kappa$ and $\pi\times\tilde\pi$ be the tensor product representations
of $K\times K$ and $G\times G$, respectively.
(Note that we are using the $\times$ notation, relative to representations, to denote
a representation of a direct product of groups that is obtained as a tensor product
of representations of the factor groups. Elsewhere we use the $\otimes$ notation
in the setting where the factor groups are equal and where we are restricting
the tensor product to the diagonal subgroup of the direct product.)
 Given a $K$-invariant pairing on $\kappa\times\tilde\kappa$, we obtain a $G$-invariant pairing on 
$\pi\times \tilde\pi$ given by
$$\langle f,\tilde f\rangle = \int_{K\bs G} \langle f(g),\tilde f(g)\rangle\ dg.$$ 
The latter formula implies that the matrix coefficients of $\pi$ have
compact-mod-center support. In other words, $\pi$ is supercuspidal.

The previous discussion of intertwining generalizes as follows.  Suppose $H$ is a closed subgroup of $G$.  
The space $\Hom_H(\pi,1)$ decomposes canonically as $$\Hom_H(\pi,1)\cong \bigoplus_{KgH\in
K\bs G/H} I_{gHg^{-1}}(\kappa),$$
where $$I_{gHg^{-1}}(\kappa)=
\Hom_{K\cap gHg^{-1}}(\kappa,1).$$ The isomorphism
is given explicitly by $\Lambda\mapsto
(\lambda_g)$, with
$$\Lambda(f)= \sum_{KgH\in K\bs
G/H}\quad\sum_{h\in (g^{-1}Kg\cap H)\bs
H} \lambda_g(f(gh)).$$  
To see that this truly generalizes the above discussion, one 
replaces $(G,H,\pi)$ by $(G\times G,G,\pi\times\tilde\pi)$, where $G$ 
is embedded as the diagonal of $G\times G$.  One also uses the fact 
that $\Hom_G(\pi,\pi) \cong \Hom_G(\pi\times\tilde\pi,1)$.

\begin{definition}\label{defdist}
The representation $\pi$ is {\it $H$-distinguished} if $\Hom_H(\pi,1)$ is nonzero.  Similarly, we 
say $\kappa$ is {\it $(K\cap H)$-distinguished} if $I_H(\kappa)$ is nonzero.
\end{definition}

Thus $\pi$ is $H$-distinguished exactly when there exists at least one double coset $KgH$ such that $\kappa$ is $(K\cap gHg^{-1})$-distinguished.
In other words, for each double coset
$KgH$, compact induction $\ind_K^G$ defines a
functor that maps
$(K\cap gHg^{-1})$-distinguished representations $\kappa$ of
$K$ to distinguished representations
$\pi$ of $G$.  

\begin{remark}\label{contrmult} 
\begin{itemize}
\item[(i)] If $\pi$ is as above, then, since $\kappa$ is finite-dimensional,
$I_{gHg^{-1}}(\kappa)$ is naturally isomorphic to 
$I_{gHg^{-1}}(\tilde\kappa)$ for each $g\in G$.
Hence it follows from the discussion above that
$\Hom_H(\pi,1)$ is isomorphic to $\Hom_H(\tilde\pi,1)$.
\item[(ii)] More generally, suppose that $\pi$ is an
irreducible admissible representation of $G$ on a Hilbert
space $V$ with inner product $(\cdot,\cdot)$. If 
$\pi$ is unitary (relative to the given inner product),
then $\Hom_H(\pi,1)$ and $\Hom_H(\tilde\pi,1)$ are
naturally isomorphic. Indeed, if $v\in V$ is fixed, then
the linear functional $(\cdot,v)$ that maps a vector
$v^\prime\in V$ to $(v^\prime,v)$ belongs to the space
of $\tilde\pi$. If $\lambda\in\Hom_H(\pi,1)$, then
the element $\lambda^\vee$ of $\Hom_H(\tilde\pi,1)$
that corresponds to $\lambda$ satisfies $\lambda^\vee
((\cdot,v))=\lambda(v)$ for $v\in V$.
\end{itemize}
\end{remark}

The pairs $(G,H)$ that are of interest to us are as follows.
Let $F$ be a nonarchimedean local field of odd residual
characteristic. From now on, we take $G$ to be the
group $\bG(F)$ of $F$-rational points of a connected 
reductive group $\bG$ that is defined over $F$.
 (Throughout this \paperbook, we adopt this convention of using 
boldface letters for $F$-groups and the corresponding non-bold letters 
for the $F$-rational points.)  

\begin{definition}\label{definv}
An {\it involution of $G$} is an automorphism $\theta$ of $\bG$ of 
order two that is defined over $F$.
\end{definition}

The set of such involutions will be denoted by $\cI$.  Given $\theta$, the subgroup of $\bG$ consisting of the fixed points of $\theta$ will be denoted $\bG^\theta$.  We are interested in 
analyzing $\Hom_H(\pi,1)$ when $H = G^\theta$, for some $\theta\in \cI$, 
and $\pi$ is an irreducible supercuspidal representation of $G$ that
is tame in the sense of \cite{Y}.

If $g\in G$, let $\Int(g)$ be the automorphism of $G$ given by conjugation by $g$.
The group $G$ acts on $\cI$ by
$$g\cdot \theta = \Int ( g)\circ \theta\circ \Int (g^{-1})$$ and we have $$\bG^{g\cdot \theta} = g\bG^\theta g^{-1}.$$  
The following elementary result says that the isotropy group of an involution $\theta$ is nearly $G^\theta Z$.

\begin{lemma}\label{invstab}
If $\theta\in \cI$ and $g\in G$ then $g\cdot \theta = \theta$ if and only if $g\theta(g)^{-1}\in Z$.
\end{lemma}

For each $g\in G$, there is a canonical  isomorphism
$$\Hom_{G^{ \theta}}(\pi,1)\cong \Hom_{G^{g\cdot\theta}}(\pi,1)$$ given explicitly by $\Lambda\mapsto \Lambda\circ\pi(g^{-1})$.  Consequently, the property of being $G^\theta$-distinguished only depends on the $G$-orbit $\Theta$ of $\theta$ and we may define
$$\langle \Theta ,\pi\rangle_G = \dim \Hom_{G^\theta}(\pi,1),$$ where $\theta$ is any element of $\Theta$.  Similarly, if $\Theta'$ is a $K$-orbit in $\cI$ then we let
$$\langle \Theta' ,\kappa \rangle_K  = \dim \Hom_{K\cap G^\theta}(\kappa,1),$$ where $\theta$ is an arbitrary element of $\Theta'$.  We observe that in the latter pairings only the equivalence class of the representation ($\pi$ or $\kappa$) is significant.

\begin{definition}\label{defdistorb}
The representation $\pi$ is {\it $\Theta$-distinguished} if  $\langle\Theta ,\pi\rangle_G$ is nonzero. 
The representation $\kappa$ is {\it $\Theta'$-distinguished} if  $\langle \Theta' ,\kappa\rangle_K$ is nonzero.
\end{definition}

When $\pi= \ind_K^G (\kappa)$ then $\pi$ is $\Theta$-distinguished if and only if $\kappa$ is $\Theta'$-distinguished for some $\Theta'\subset \Theta$.   

Now fix a $G$-orbit $\Theta$ in $\cI$ and let $\Theta^K$ denote the set of $K$-orbits in $\Theta$.  

\begin{lemma}\label{sthetaTheta} If $\theta\in \Theta$ and $\Theta'\in \Theta^K$ let $$S(\theta, \Theta') = \{\, KgG^\theta \in K\bs G/G^\theta \ | \ g\cdot \theta\in \Theta'\,\}.$$  Then:
\begin{enumerate}
\item $S(g\cdot \theta , \Theta') = S(\theta,\Theta')g^{-1}$, if $g\in G$.
\item $S(\theta, K\cdot \theta)$ consists of the double cosets in $K\bs G/G^\theta$ that contain an element $g$ such that $g\theta (g)^{-1}\in Z$.
\item If $\theta\in \Theta$ and $g\in G$ then $Kg_1 G^\theta \mapsto Kgg_1g^{-1} G^{g\cdot \theta}$, with $g_1\theta (g_1)^{-1}\in Z$, defines a bijection between the sets $S(\theta, K\cdot \theta)$ and $S(g\cdot \theta,Kg\cdot \theta)$.
\item The cardinality of $S(\theta,\Theta')$ only depends on the $G$-orbit $\Theta$, and not on the choice of $\theta$ and $\Theta'$.
\end{enumerate}
\end{lemma}

\begin{proof}
The set $S(g\cdot \theta,\Theta')$ consists of the double cosets $Kh G^{g\cdot \theta}$ as $h$ ranges over the elements of $G$ such that $hg\cdot \theta\in \Theta'$.  But $KhG^{g\cdot\theta} = KhgG^\theta g^{-1}$. 
 This implies Property~(1).  Lemma~\ref{invstab} implies Property (2).

To prove (3), it suffices to show $Kg_1 G^\theta \mapsto Kgg_1g^{-1} G^{g\cdot \theta}$ gives a well-defined map from $S(\theta, K\cdot \theta)$ to $S(g\cdot \theta,Kg\cdot \theta)$, since then we obtain an inverse map by replacing $\theta$ by $g\cdot \theta$ and then replacing $g$ by $g^{-1}$.  So assume $Kg_1G^\theta\in S(\theta, K\cdot \theta)$ and $g_1\theta (g_1)^{-1}\in Z$.  It is easily verified that $gg_1g^{-1} \theta_1(gg_1g^{-1})^{-1} = g_1\theta (g_1)^{-1}\in Z$.  Property (3) follows.

It remains to prove (4). Suppose we are given $\theta\in \Theta$ and $\Theta'\in \Theta^K$.  We may choose $g\in G$ so that $g\cdot \theta\in \Theta'$.  Applying Property (1), we see that $S(\theta, \Theta') = S(g\cdot \theta , Kg\cdot \theta)g$.  Consequently, it suffices to prove Property (4) in the special case in which $\theta\in \Theta'$.  But this follows immediately from (3).
\end{proof}

Given a $G$-orbit $\Theta$ in $\cI$, we now let $m_K(\Theta)$ denote the cardinality of $S(\theta,\Theta')$ for any (hence all) $\theta\in\Theta$ and $\Theta'\in \Theta^K$.
The fact that $m_K(\Theta)$ does not depend on the choice of $\theta$ follows from Lemma \ref{sthetaTheta}(4).  If $\theta\in \Theta'\in \Theta^K$,
we have a surjective map $K\bs G/G^\theta \to \Theta^K$ that sends the double coset $KgG^\theta$ to the $K$-orbit of $g\cdot \theta$.  Note that  the cardinality of the fiber of $\Theta'$ is  $m_K(\Theta)$.
Then we have the following preliminary multiplicity formula:

\begin{lemma}\label{orbmult}
If $\Theta$ is a $G$-orbit in $\cI$ then $$\langle \Theta ,\pi\rangle_G = 
m_K(\Theta)\, \sum_{\Theta'\in\Theta^K}  \langle \Theta' ,\kappa\rangle_K.$$
\end{lemma}

We have just discussed the relation between the double cosets in the space $K\bs G/G^\theta$ and $K$-orbits of involutions in the $G$-orbit of $\theta$.  One can also describe things in the language of twisted conjugation.  Let $\cS_\theta$ denote the set of elements of the form $g\theta(g)^{-1}$, with $g\in G$.  Then $K$ acts on $\cS_\theta$ by $\theta$-twisted conjugation: $k\cdot x = kx\theta(k)^{-1}$.  Let $\cS^K_\theta$ denote the set of $K$-orbits in $\cS_\theta$.  We have a bijection $K\bs G/G^\theta\to \cS^K_\theta$ that sends the double coset $KgG^\theta$ to the $K$-orbit of $g\theta(g)^{-1}$.  Thus our surjective mapping from $K\bs G/G^\theta$ to $\Theta^K$ can be replaced by a surjection $\cS^K_\theta \to \Theta^K$.  We observe that the latter map is simply the map that sends the $K$-orbit of $x\in \cS_\theta$ to the $K$-orbit of the involution $\Int (x)\circ \theta$.  To summarize, we have a commuting triangle
$$\xymatrix{
K\bs G/G^\theta\ar@{<->}[rr]\ar@{>>}[dr]&
&\cS^K_\theta\ar@{>>}[dl]\\
&\Theta^K}$$ with the maps being given by:
$$\xymatrix{
K gG^\theta\ar@{<->}[rr]\ar@{|->}[dr]&
&K\cdot g\theta(g)^{-1}\ar@{|->}[dl]\\
&Kg\cdot \theta}$$
The above triangles convey that there are three essentially equivalent settings for studying distinguished representations.  
The space $\Theta^K$ is the most canonical, since it does not depend on the choice of a particular involution $\theta$.  It also highlights the fact that the double cosets and twisted $\theta$-orbits which comprise a fiber over an element of $\Theta^K$ are fused in the theory of distinguished representations.  In practice, the $K\bs G/G^\theta$ and $\cS^K_\theta$ 
settings have their advantages too.  For example, it is often most natural to study the geometry of $K\bs G/G^\theta$ and $\Theta^K$ by transferring to $\cS^K_\theta$.  (An example of this occurs when one studies the geometry of $GL(n)/U(n)$ via hermitian matrices.)

We close with some remarks regarding $m_K(\Theta)$.  For simplicity, 
we assume $Z\subset K$ since this will be the case with Yu's construction.  Given a $G$-orbit $\Theta$ in $\cI$ and $\theta\in \Theta$, we define abelian groups
\begin{equation*}
\begin{split}
Z^1_\Theta &=\{\, z\in Z\ | \ \theta (z)= z^{-1}\,\},\cr
B^1_\Theta &=\{\, z\theta (z)^{-1}\ | \ z\in Z\,\},\cr
H^1_\Theta & = Z^1_\Theta/B^1_\Theta.
\end{split}
\end{equation*}
As the notations suggest, the groups $Z^1_\Theta$, $B^1_\Theta$ and $H^1_\Theta$ do not depend on the choice of $\theta$ in $\Theta$.  

\begin{lemma}\label{mThetabound}  If $\Theta$ is a $G$-orbit in $\cI$ and $Z\subset K$ then $m_K(\Theta)\le |H^1_\Theta| < \infty$.
\end{lemma}

\begin{proof}
The number $m_K(\Theta)$ is the number of $K$-orbits in $\cS^K_\theta$ in each 
fiber over $\Theta^K$.  We may as well simply consider the fiber of $K\cdot \theta$. 
 If $x=g\theta(g)^{-1}\in \cS_\theta$ then the $K$-orbit of $x$ lies in the fiber of 
$K\cdot \theta$ 
if and only if $g\cdot\theta\in K\cdot\theta$. By Lemma~\ref{sthetaTheta}(2), this is 
equivalent to $kg\theta(kg)^{-1}\in Z$ for some $k\in K$, that is,
$kx\theta(k)^{-1}\in Z$.
So $m_K(\Theta)$ may be interpreted as the number of $K$-orbits in $\cS^K_\theta$ 
that have a representative in $\cS_\theta\cap Z$.

Define an equivalence relation on $\cS_\theta \cap Z$ by $z_1\sim z_2$ when $K\cdot z_1 = K\cdot z_2$.  In other words, $z_1\sim z_2$ exactly when  $z_1z_2^{-1}\in Z\cap K^{1-\theta}$, where
$$K^{1-\theta} = \{\,  k\theta(k)^{-1}\ | \ k\in K\,\}.$$  
We can now interpret $m_K(\Theta)$ as the number of equivalence classes in $\cS_\theta\cap Z$.

Another equivalence relation may be defined on $\cS_\theta\cap Z$ by letting $z_1\approx z_2$ when $z_1z_2^{-1}\in B^1_\Theta$.    Let $n$ be the number of these equivalence classes.  Then $n\le |H^1_\Theta|$.  Moreover, since $z_1\approx z_2$ implies $z_1\sim z_2$, we have $m_K(\Theta)\le n$.  
  Hence, we have
$m_K(\Theta ) \le |H^1_\Theta|$ as claimed.

It remains to show that $H^1_\Theta$ is finite.
Since $(Z^1_\Theta)^2\subset B_\Theta^1$, it suffices
to show that $Z^1_\Theta / (Z^1_\Theta)^2$ is finite.
Let $\bT$ be the identity component of $\{\, z\in \bZ\ | \ \theta(z)=z^{-1}
\,\}$. As $Z_\Theta^1/T$ is finite, it is enough to show
that $T/T^2$ is finite. 
The group $\bT$ is an $F$-torus and is a product of
an anisotropic $F$-torus $\bT_a$ and a split 
$F$-torus $\bT_s$, with $\bT_a\cap \bT_s$ finite.
Finiteness of $F^\times/(F^\times)^2$ implies that $T_s/(T_s)^2$
is finite. Because the
map $t\mapsto t^2$ is submersive,  $(T_a)^2$ is open in $T_a$.
In addition, $T_a$ is
compact, Thus the quotient $T_a/(T_a)^2$ is finite.
\end{proof}

\section{$\theta$-stable subgroups}
\label{sec:stablesubgroups}

Various aspects of our theory demand that we deal with subgroups that are stable under a given involution $\theta$ of our group $G$.  Assuming $A$, $B$ and $C$ are $\theta$-stable subgroups of $G$ with $C= AB$, we often need to know that we have a decomposition $C^\theta = A^\theta B^\theta$ involving the subgroups of $\theta$-fixed points.  If $A\cap B$ is trivial then the latter decomposition is automatic.  More generally, it is elementary to see that if a cohomology set $H^1_\theta (A\cap B)$, defined below, is trivial then we again get the desired decomposition.  It then remains to determine when the cohomology vanishes in the cases of interest to us.  This turns out to reduce to establishing the existence of suitable square roots.  Our main result in this section, Proposition \ref{twodivprop}, states that every subgroup of a Moy-Prasad group of positive depth must have trivial cohomology with respect to any automorphism $\alpha$ of exponent two.

We now define the appropriate cohomology.  Suppose $C$ is a group and $\alpha$ is an automorphism of $C$ such that $\alpha^2=1$.  The subgroup of fixed points of $\alpha$ is denoted $C^\alpha$ and, more generally, if $D$ is a subgroup of $C$ then we let $D^\alpha = D\cap C^\alpha$.  If $D$ is $\alpha$-stable then we use the notations
\begin{equation*}
\begin{split}
Z^1_\alpha (D)&= \{\, z\in D\ | \ \alpha (z) = z^{-1}\,\} \cr
B^1_\alpha (D)&= \{\, y\alpha (y)^{-1}\ | \ y\in D\,\}.
\end{split}
\end{equation*}
We also let $H^1_\alpha (D)$ denote the space of $D$-orbits in $Z^1_\alpha (D)$ with respect to the action: $x\cdot y= xy\alpha(x)^{-1}$.  Thus $H^1_\alpha (D)=1$ exactly when $Z^1_\alpha (D) = B^1_\alpha(D)$.

\begin{lemma}\label{alphafactor}
Suppose $\alpha$ is an automorphism of a group $C$ such that $\alpha^2=1$.  Assume $A$ and $B$ are $\alpha$-stable subgroups of  $C$ such that $C=AB$ and $H^1_\alpha (A\cap B)=\{1\}$.  Then $C^\alpha = A^\alpha B^\alpha$.
\end{lemma}

\begin{proof}
Suppose $c\in C^\alpha$ and choose $a\in A$ and $b\in B$ such that $c=ab$.  We have $ab= c=\alpha (c) = \alpha (a)\alpha (b)$ and thus $a^{-1}\alpha (a) = b\alpha (b)^{-1} \in Z^1_\alpha (A\cap B)$.  We can therefore choose $y\in A\cap B$ such that $a^{-1}\alpha (a) = y\alpha (y)^{-1}$.  Letting $a' = ay$ and $b'= y^{-1} b$ defines elements $a'\in A^\alpha$ and $b'\in B^\alpha$ such that $c= a' b'$.
\end{proof}

\begin{definition}\label{deftwodiv}
A subset $S$ of group $C$ is {\it 2-divisible} if for
every $s\in S$ there exists $t\in S$ such that
$t^2 =s$.  If $t$ is always unique then we say $C$ is {\it strongly $2$-divisible} and we write $t=\sqrt{s}$.
\end{definition}

\begin{lemma}\label{twodivHone}
Suppose $\alpha$ is an automorphism of  a group $C$ such that $\alpha^2=1$.  If
$Z^1_\alpha (C)$ is
2-divisible then $H^1_\alpha (C) = \{ 1\}$. 
If $C$ is strongly 2-divisible  then $Z^1_\alpha (C)$ is also 
strongly 2-divisible and, consequently, $H^1_\alpha (C)=\{ 1\}$.
\end{lemma}

\begin{proof}
Suppose $Z^1_\alpha (C)$ is 2-divisible and $h\in Z^1_\alpha(C)$.  Then
we may choose $t\in Z^1_\alpha (C)$ such that $t^2
= h$.  We then have $h = t^2 =
t\alpha(t)^{-1}\in B^1_\alpha (C)$, which proves
our assertion.

Assume now that $C$ is strongly 2-divisible.  We observe that if $h\in C$ then 
$\sqrt{\alpha(h)}$ and $\alpha (\sqrt{h})$ are both square roots of $\alpha (h)$ and hence they must be equal.  Similarly, we have the relation $\sqrt{h^{-1}}= \smash{\sqrt h}^{-1}$.
If $h\in Z^1_\alpha (C)$ then $\alpha (\sqrt{h}) = \sqrt{\alpha (h)}= \sqrt{h^{-1}}=
 \smash{\sqrt h}^{-1}$.  The claim follows.
\end{proof}

We now focus on the examples of most importance to us and revert to our usual notations.  
Thus $\bG$ is a connected, reductive $F$-group that splits over a tamely ramified extension 
of $F$, and $G=\bG(F)$. 
Let $\bfr{g}$ be the Lie algebra of $\bG$, and let $\g=\bfr{g}(F)$. (In analogy
with our conventions for $F$-groups, we use boldface fraktur letters for
the Lie algebras of $F$-groups and the corresponding non-boldface fraktur
letters for the $F$-rational points of the Lie algebras.) Let $\cB(\bG,F)$
be the (extended) Bruhat-Tits building of $G$. If $x\in \cB(\bG,F)$ and $t$
is a nonnegative real number, let $G_{x,t}$ and $G_{x,t^+}$ be the
associated filtration subgroups of $G$, which we are referring to as
Moy-Prasad groups. Similarly, let $\g_{x,t}$ and $\g_{x,t^+}$
be the associated filtration lattices of $\g$. For more information on
the Moy-Prasad groups and filtrations, the reader may refer to
Section~\ref{sec:buildings}.
 The main result of this section is the following:

\begin{proposition}\label{twodivprop}
Every subgroup $\cG$ of a Moy-Prasad group $G_{x,0^+}$ with $x\in \cB (\bG,F)$
is strongly 2-divisible and thus $H^1_\alpha (\cG)=\{ 1\}$ for all automorphisms $\alpha$ of $\cG$ such that $\alpha^2=1$.
\end{proposition}

Before proving Proposition~\ref{twodivprop}, we establish several auxiliary lemmas.

As discussed in Section \ref{sec:hypotheses}, when ${\bG}$ splits over
a tamely ramified extension of $F$,
given $x\in \cB(\bG,F)$ and $t>0$, there is an isomorphism 
between the quotient group
$$G_{x,t:t^+} = G_{x,t}/G_{x,t^+}$$ and the corresponding Lie algebra quotient 
$$\g_{x,t:t^+} = \g_{x,t}/\g_{x,t^+}$$ of additive groups.  
Yu refers to the latter isomorphism as the {\it Moy-Prasad isomorphism}. \label{defMPiso}
  It turns out to be a highly useful tool in inductive arguments, such as the one we use 
to prove Proposition~\ref{twodivprop}.

\begin{lemma}\label{MPsubtwodiv}
Every subgroup of a group $G_{x,t:t^+}$, with $t>0$, is strongly 2-divisible.
\end{lemma}

\begin{proof}
In light of the Moy-Prasad isomorphism, we might as well show that every subgroup $H$ of $\g_{x,t:t^+}$ is strongly 2-divisible.  Now $\gamma\mapsto 2\gamma$ defines a group homomorphism from $H$ into itself.  The kernel of this homomorphism is trivial.  Thus, since $H$ is finite, the map must be surjective.  Our claim follows.
\end{proof}

The next step in the proof of Proposition~\ref{twodivprop} is the case of the 
Moy-Prasad groups themselves.  

\begin{lemma}\label{MPtwodiv}
The Moy-Prasad groups $G_{x,t}$ with $x\in \cB (\bG,F)$ and $t>0$ are  strongly 2-divisible.
\end{lemma}

\begin{proof}
Define real numbers $t_0<t_1< \dots$ by the conditions $G_{x,t_0} = G_{x,t}\ne G_{x,t_0^+}$ and $G_{x,t_{i+1}} = G_{x,t_i^+}\ne G_{x,t_{i+1}^+}$, for $i\ge 0$.  Given $a\in G_{x,t_0}$, we construct a convergent sequence $x_0,x_1,\dots$ in $G_{x,t}$ such that $x_i$ is the unique square root of $a$ modulo $G_{x,t_{i+1}}$.  Once we establish that such a sequence exists, then it follows that its limit must be the unique square root of $a$ in $G_{x,t_0}$.

Using the Moy-Prasad isomorphism $G_{x,t_0:t_1} \cong \g_{x,t_0:t_1}$, we see that we can choose $x_0\in G_{x,t_0}$ so that $x_0^{-2}a\in G_{x,t_1}$ and, moreover, $x_0$ is unique modulo $G_{x,t_1}$.

Now suppose $x_i\in G_{x,t}$ has been defined, uniquely modulo $G_{x, t_{i+1}}$,
 so that $x_i^{-2}a\in G_{x,t_{i+1}}$.  Choose $y_{i+1}\in G_{x,t_{i+1}}$ so that $y_{i+1}^{-2} x_i^{-2}a\in G_{x,t_{i+2}}$.  Then, again using a   Moy-Prasad isomorphism, we see that $y_{i+1}$ is unique modulo $G_{x,t_{i+2}}$.  Let $x_{i+1} = x_iy_{i+1}$.  Then $$x_{i+1}^{-2}a = (y^{-1}_{i+1}(x_i^{-1}y_{i+1}^{-1}x_iy_{i+1}) y_{i+1})(y_{i+1}^{-2}x_i^{-2}a).$$  We observe that $x_i^{-1}y_{i+1}^{-1}x_iy_{i+1}$ lies in the commutator subgroup $
[G_{x,t},G_{x,t_{i+1}}]$, which is contained in $G_{x,t_{i+2}}$.  Since $y_{i+1}$ normalizes $G_{x,t_{i+2}}$, we deduce that $x_{i+1}^{-2}a\in G_{x,t_{i+2}}$.  Our assertion follows.
\end{proof}

The construction in the previous proof can be cast in a more general setting to yield:

\begin{lemma}\label{abstractsqroot} Let $\cG$ be a group.  Suppose $\cG= \cG_0, \cG_1,\dots$ is a sequence of subgroups of $\cG$ such that the commutator group $[\cG,\cG_i]$ is contained in $\cG_{i+1}$ for all $i$.  Then for each $i$ the group $\cG_{i+1}$ is a normal subgroup of $\cG_i$ and the quotient group $\cG_i/\cG_{i+1}$ is abelian.  Assume the latter abelian groups are all 2-divisible.  Then for each element $a\in \cG$ there exists a sequence $x_0,x_1,\dots$ in $\cG$ such that $x_i^{-1}x_{i+1}\in \cG_{i+1}$ and $x_i^{-2}a\in \cG_{i+1}$ for all $i$.
\end{lemma}

\begin{proof}
Since $\cG_0/\cG_1$ is 2-divisible, we may choose $x_0\in \cG$ such that $x_0^{-2}a\in \cG_1$.  Once $x_0,\dots, x_i$ have been chosen, as required in the statement of the lemma, we choose $y_{i+1}\in \cG_{i+1}$ so that $y^{-2}_{i+1}x_i^{-2}a\in \cG_{i+2}$, using the fact that $\cG_{i+1}/\cG_{i+2}$ is 2-divisible.  Then we take $x_{i+1}= x_iy_{i+1}$.  A similar argument as in the proof of Lemma \ref{MPtwodiv} shows that this constructs a sequence with the desired properties.
\end{proof}

\begin{proof}[Proof of Proposition \ref{twodivprop}]  Given $\cG$, we let $\cG_i = G_{x,t_i}\cap\cG$, where $t_i$ is defined as in the proof of Lemma \ref{MPtwodiv}, for $i=0,1,\dots$.  Then 
$[\cG,\cG_i]\subset \cG_{i+1}$ for all $i$.  Next, we observe that $\cG_i/\cG_{i+1}$ is isomorphic to the image of $\cG_i$ in $G_{x,t_i:t_i^+}$ and hence, according to Lemma \ref{MPsubtwodiv}, it must be strongly 2-divisible.  We can now use Lemma \ref{abstractsqroot} to construct a square root of any element of $\cG$.  Thus $\cG$ is 2-divisible.  Our assertion now follows from the fact that a 2-divisible subgroup of a strongly 2-divisible group (in this case $G_{x,t}$) must be strongly 2-divisible.
\end{proof}

\section{Heisenberg representations over $\F_p$}
\label{sec:Heis}

Representations of finite groups enter into the construction of tame supercuspidal representations in two important ways.  The first involves cuspidal representations of finite groups of Lie type.  
The second is the focus of this section and it involves Heisenberg and Weil representations associated to symplectic spaces over fields of order $p$, where $p$ is the characteristic of the residue field of $F$.  (Recall that we have assumed that $p\ne 2$.)  In this section, we develop the theory of distinguished representations in the context of Heis\-enberg
$p$-groups, that is, Heisenberg groups associated
to finite symplectic spaces over a field $\F_p$
of odd prime order.  Because of the
extraordinary nature of the Heisenberg group,
we are able to obtain with relatively little
effort reasonably complete results.  

Much of the material in this chapter logically precedes Yu's construction and this explains its placement towards the beginning of this \paperbook.  However, before reading this chapter, some readers may wish to better understand the context in which the theory we describe is useful.  These readers are advised to skip ahead to our description of Yu's construction in the next chapter and the subsequent applications, referring back to this chapter as needed along the way.

The main
results are as follows.  Let $\cH$ be a
Heisenberg $p$-group with center $\cZ$ and
suppose $\alpha$ is an automorphism of $\cH$ of
order two.  Let $\cH_\alpha^+$ be the subgroup
of fixed points of $\alpha$.  Let $\tau$ be the
Heisenberg representation associated to some
nontrivial character $\zeta$ of $\cZ$.  It is shown in Lemmas \ref{hplusfixed} and \ref{HmodZhom} that in order
for $\Hom_{\cH_\alpha^+}(\tau,1)$ to be
nonzero,  it is necessary and
sufficient that $\alpha\,|\,\cZ$ is not the identity
map.  Therefore, we always assume $\alpha\,|\,\cZ$ is
nontrivial.  In this case, Lemma \ref{hplusfixed} shows that 
$\Hom_{\cH_\alpha^+}(\tau,1)$ has dimension one
and we give an explicit generator. We also show
that $\tau$ must satisfy the symmetry relation
$\tau\circ \alpha\simeq
\tilde\tau$, where $\tilde\tau$ is the
contragredient of $\tau$.  Many of the other main results in this section, including the latter fact, are summarized in Theorem \ref{Heisthm}.
We show in Theorem \ref{HeisGelf} that
the pair
$(\cH,\cH_\alpha^+)$ is a Gelfand pair in the
sense that for every irreducible representation
$\rho$ of $\cH$ the space ${\rm
Hom}_{\cH_\alpha^+}(\rho,1)$ has dimension at
most one. Equivalently, the space of
$\cH_\alpha^+$-fixed vectors in the space of
$\rho$ has dimension at most one.
Now let $\widehat{\cH}_\alpha^-$ be the set of elements in
$\cH$ such that $\alpha (h) = h^{-1}$.  Then we
show in Theorem \ref{Heisthm} that the images of $\cH_\alpha^+$ and
$\widehat{\cH}_\alpha^-$ in the symplectic space $\cH/\cZ$
form a polarization.  This fact is
applied to construct polarizations of the
Heisenberg groups occurring in Yu's construction
and these polarizations are ideal for studying
distinguished tame supercuspidal representations.

We also recall in Definition \ref{defspeciso} Yu's notion of a special isomorphism and we establish in Remark \ref{specisomu} and Lemmas \ref{specisoeq}--\ref{polsplitting} some  basic properties of special isomorphisms which result from abstract finite group theory.  Special isomorphisms are needed when extending Heisenberg representations to obtain Weil representations.  The Heisenberg representations of interest to us are associated to generic characters and they will be discussed in the next section.  For Heisenberg representations coming from generic characters, Yu has defined special isomorphisms in a canonical way.  However, for our applications involving distinguished representations, there are other special isomorphisms that are far more convenient to use.  An important fact, which will become evident in the next section, is that the choice of special isomorphisms which one uses in Yu's construction does not affect the isomorphism class of the supercuspidal representation produced by Yu's construction (so long as the special isomorphisms are ``relevant'' in a sense which we will define).

We now warn the reader that for most of this
section we will be using the traditional additive
notation for our symplectic spaces and
Heisenberg groups, rather than the
multiplicative notation which is natural for our
primary applications.

Let us review  some basic facts about finite
Heisenberg
$p$-groups and their Heisenberg representations. 
Let  $C$ be a cyclic group of order $p$ and let $W$ be a group that is isomorphic to a finite direct sum of copies of $C$.  Assume we have a pairing $\langle \; ,\; \rangle : W\times W\to C$ such that:
\begin{itemize}
\item \quad $\langle a+b,c\rangle = \langle a,c\rangle + \langle b,c\rangle$, for all $a,b,c\in W$, and,
\item\quad  $\langle a,b\rangle = -\langle b,a\rangle$, for all $a,b\in W$.
\end{itemize}

\ni In other words, $W$ is essentially a finite symplectic space over $\F_p$.
In this situation, 
we use  the notation $W\boxtimes C$ for the set $W\times C$ viewed as a group with multiplication
$$(w_1,z_1)(w_2,z_2) = (w_1+w_2, z_1+z_2+\frac{1}{2}\langle w_1,w_2\rangle ).$$  We remark that that $W\boxtimes C$ is abelian exactly when $W$ is totally isotropic, that is, $\langle a,b\rangle =0$, for all $a,b\in W$.  

\begin{definition}\label{Hpgroup}
A {\it Heisenberg $p$-group} is any abstract group $\cH$ that is isomorphic to some group $W\boxtimes C$, where  $W$ is nondegenerate.
\end{definition}

We now assume $\cH = W\boxtimes \F_p$ where $W$ is a nondegenerate symplectic space over $\F_p$ of dimension $2\ell$.  Let $\cZ = 0\boxtimes \F_p$ be the center of $\cH$.
Fix a polarization (a.k.a., a complete polarization or a Witt decomposition) $W= W^+ +
W^-$ and embed $W$ in $\cH$ via $w\mapsto (w,0)$. 
We stress that this embedding is not a group homomorphism and its image is not a subgroup of $\cH$.  However, the restriction of this embedding to $W^+$ and $W^-$ gives homomorphisms that we use to identify $W^+$ and $W^-$ with abelian subgroups of $\cH$.

Let $(\tau , V)$ be the Heisenberg
representation
$\hbox{Ind}_{W^-\boxtimes \F_p}^\cH (1\times \zeta)$, where
$\zeta$ is a nontrivial character of $\cZ$. 
The isomorphism class of the Heisenberg
representation depends only on the choice of
$\zeta$, however, our specific model for the
Heisenberg representation also depends on the
choice of polarization.

The Weil representation is defined as follows.  Fix $\F_p$-bases
$e_1,\dots ,e_\ell$ and $e_{\ell+1},\dots
,e_{2\ell}$ of $W^+$ and $W^-$, respectively,
such that, with respect to the basis $e_1,\dots,
e_{2\ell}$ of $W$, the symplectic form on
$W$ is the associated to the matrix $$j = \2by2
0{1_\ell}{-1_\ell}0.$$  Then $\cS = \hbox{Sp}(W)$
consists of block matrices $s=\2by2 abcd$ such that
${}^tsjs= j$.  Let $\cM$ denote the subgroup of
block diagonal matrices in $\cS$.  Such matrices
have the form $$m(y)=\2by2 y00{{}^ty^{-1}},$$ where
$y$ ranges over $GL(\ell,\F_p)$.  We recall that
$\cM$ has a unique character of order two and it
is defined by $$\chi^{\cM}(m(y))= (\det
y)^{(p-1)/2}.$$  We also observe that $\cM$ is
precisely the stabilizer of the polarization  $W=
W^++W^-$.

We will represent $\cS$ and $\cH$ as  subgroups of
$GL(2\ell+2 ,\F_p)$.  Specifically, 
$s=\2by2 abcd\in \cS$ maps to the matrix 
$$\left(
\begin{array}{ccc}
1&0&0\cr 0&s&0\cr 0&0&1
\end{array}
\right)
$$ and $h=(w,z)\in
\cH$ maps to 
$$\left(
\begin{array}{ccc}
1&\frac{1}{2}{}^twj&z\cr
0&1_\ell&w\cr 0&0&1
\end{array}
\right)
.$$  With these
representations, $\cS$ acts by conjugation on $\cH$
and we have a semidirect product $\cS\ltimes \cH$
that consists of elements
$$s\ltimes h = 
\left(
\begin{array}{ccc}
1&0&0\cr 0&s&0\cr
0&0&1
\end{array}
\right)
\left(
\begin{array}{ccc}
1&\frac{1}{2}{}^twj&z\cr
0&1_\ell&w\cr 0&0&1
\end{array}
\right)
.$$  

Except in the case in which $p=3$ and $\ell =1$, there is a unique
extension $\hat\tau : \cS\ltimes \cH\to GL(V)$ of
the Heisenberg representation to $\cS\ltimes
\cH$.  Indeed, any two such extensions would be twists of each other by a character of $\cS$, however, since $\cS$ is its own commutator subgroup it does not admit any nontrivial characters.
On the other hand, when $p=3$ and $\ell =1$ the commutator subgroup of $\cS$ is a proper subgroup of $\cS$ and there are three extensions of $\tau$ to $\cS\ltimes \cH$.  In the appendix at the end of this section, we single out a unique extension of $\tau$ that we will denote by $\hat\tau$.  (This is consistent with the conventions in \cite{Ge}.)

\begin{definition}\label{HeisWeilrep}
Suppose $\tau$ is a Heisenberg representation of a Heisenberg $p$-group 
of the form  $\cH = W\boxtimes \F_p$.  Except in the case when $p=3$ and $\ell=1$, the {\it Heisenberg-Weil lift} of $\tau$ to $\cS\ltimes \cH$ is the unique extension $\hat\tau$ of $\tau$ to a representation of $\cS\ltimes \cH$ that acts on the space of $\tau$.  When $p=3$ and $\ell=1$, the {\it Heisenberg-Weil lift} of $\tau$ to $\cS\ltimes \cH$  is the extension $\hat\tau$ of $\tau$ defined in Section \ref{sec:eselltwo}.
In all cases, the {\it Weil representation} of $\cS$ associated to $\tau$ is the restriction of $\hat\tau$ to
$\cS$.  The notions of Heisenberg-Weil lift and Weil representation for abstract Heisenberg $p$-groups are discussed below in Remark \ref{absHWliftdef}.
\end{definition}

The contragredient of $(\tau , V)$ is the representation $(\tilde\tau, \widetilde V)$  induced from the character $1\times \zeta^{-1}$ of $W^-\boxtimes \F_p$.  
An $\cH$-invariant pairing between $\tau$ and $\tilde\tau$ is given by:
$$\langle f_1,f_2\rangle  = \sum_{w_+\in W^+} f_1(w_+)\
f_2(w_+) = \sum_{h\in W^-\cZ\bs \cH}
f_1(h)\
f_2(h).$$  
The Heisenberg-Weil lift $\hat{\tilde\tau}$ of $\tilde\tau$ has a contragredient $(\hat{\tilde\tau})^\sim$ that restricts to $\tau$.  Since $(\hat{\tilde\tau})^\sim$ must be a Heisenberg-Weil lift of $\tau$, the uniqueness of such lifts implies that $(\hat{\tilde\tau})^\sim = \hat\tau$.
Hence, there must exist a nonzero $\cS\ltimes \cH$-invariant pairing between the Heisenberg-Weil lift $\hat\tau$ of $\tau$ and  $\hat{\tilde\tau}$.  Since such a pairing must also be $\cH$-invariant and since such pairings must be unique up to scalar multiples, it must be the case that the pairing we have defined is automatically $\cS\ltimes \cH$-invariant.  Note that the mapping $\varphi \mapsto \langle -,\varphi \rangle$ defines a $\cS\ltimes \cH$-equivariant isomorphism from $\widetilde V$ to $\Hom(V,\C)$, where $\cS\ltimes \cH$ acts on the latter space by $$(x\cdot \lambda )(v) = \lambda (\hat\tau (x^{-1}) v).$$

When
$\cK$ is a subgroup of $\cH$, let $V^\cK$ be the
space of right $\cK$-invariant functions in $V$ or, equivalently, the space of $\cK$-fixed vectors in $V$.
If $x\in \cH$ let  $V^\cK_x$ denote the
space of functions in $V^\cK$ with support in
$W^-\cZ x\cK$.  Lemma 7.1 in \cite{HM2}
says that $$V^\cK =\bigoplus_{x\in W^-\cZ\bs
\cH/\cK} V^\cK_x$$ and
$V^\cK_x$ is zero unless $W^- x\cK x^{-1}\cap \cZ
= \{ 1\}$, in which case it has dimension one.  When $V^\cK_x$ is nonzero, it must be generated by the function $\varphi_x$ that vanishes outside $W^-\cZ x\cK$ and is defined on $W^-\cZ x\cK$ by
$$
\varphi_x (w_-zxk ) = \zeta (z),
$$ 
with $w_-\in W^-$, $z\in \cZ$ and $k\in \cK$.  Thus, the above decomposition of $V^\cK$ allows us to compute $V^\cK$ exactly, not merely up to isomorphism.  Applying the previous discussion to $\widetilde V$ and using $\varphi\mapsto \langle - , \varphi\rangle$ to identify $\widetilde V$ with $\Hom (V,\C)$, we obtain an explicit description of $\Hom_{\cK}(\tau,1)$ which we record in the following lemma:

\begin{lemma}\label{fixedvecs}
If $\cK$ is a subgroup of $\cH$ then 
$$\Hom_{\cK}(\tau,1) = \bigoplus_{x\in W^-\cZ\bs
\cH/\cK} \widetilde{V}^\cK_x.$$ The space
$\widetilde{V}^\cK_x$ is zero unless $W^- x\cK x^{-1}\cap \cZ
= \{ 1\}$, in which case it has dimension one and is generated by the linear form
$$\lambda_x (\varphi)= \sum_{k\in \cK} \varphi (xk).$$  
\end{lemma}

\begin{example}\label{exzfixed}
If $\cK$ contains $\cZ$ then we see that $\widetilde{V}^\cK_x=0$
for all $x$ and thus $\Hom_{\cK} (\tau,1) =0$.
\end{example}

\begin{example}\label{exwplus}
Suppose $\cK = W^+$.  Then $\cH = W^- \cZ W^+$ implies $\widetilde{V}^\cK = \widetilde{V}^\cK_1$  and, since $W^-W^+\cap \cZ=\{1\}$,
we see that $\widetilde{V}^\cK_1 \cong \C$.  So $\Hom_{\cK}(\tau,1)$ is
one-dimensional  and it is spanned by the linear form
 defined by  $$\lambda_1(\varphi) = \sum_{w_+\in W^+}\varphi (w_+).$$
\end{example}

\begin{example}\label{exwzero}
Suppose $\cK=W_0$ is some
arbitrary maximal totally isotropic subspace in
$W$. If we are only interested in the dimension of $\Hom_{W_0}(\tau ,1)$, we can replace $\tau$ by an equivalent representation so that the present example reduces to the previous example.  An alternate approach using the Weil representation is as follows.  
Choose $s\in\cS$ such that $  s\cdot W_0
=  W^+$.  Thus the subgroups $W_0$ and $W^+$ of $\cS\ltimes \cH$ are
conjugate by an element of $\cS$.  Now if $v\in V$ then
\begin{equation*}
\begin{split}
v\in V^{W_0}&\Leftrightarrow\tau
(w_0)v=v,\forall w_0\in W_0\\
&\Leftrightarrow \tau
(s^{-1}\cdot w_+)v=v,\forall w_+\in W^+\\ 
&\Leftrightarrow
\hat\tau(s)^{-1}\tau(w_+)\hat\tau (s) v=v,\forall
w_+\in W^+ \\
&\Leftrightarrow \tau(w_+)\hat\tau (s)
v=\hat\tau(s)v,\forall w_+\in W^+\\
&\Leftrightarrow \hat\tau(s)v\in V^{W^+}.
\end{split}
\end{equation*}
In other words, $\hat\tau(s)V^{W_0}=
V^{W^+}\cong \C$.
\end{example}

Example \ref{exwzero} generalizes as follows:

\begin{lemma}\label{hplusfixed}
Suppose $\cH^+$ is an abelian subgroup of a Heisenberg $p$-group $\cH$ with center $\cZ$ such that the image $W^+$ of $\cH^+$ in $W$ is a maximal isotropic subspace and $\cH^+\cap \cZ = \{ 1\}$.  Then $\Hom_{\cH^+}(\tau ,1)$ has dimension one.\end{lemma}

\begin{proof}
Examining the statement of the lemma, it is easy to see that it suffices to prove it in the case in which $\cH = W\boxtimes \F_p$.  For each $w_+\in W^+$, there exists a unique element in $\cH^+$ of the form $(w_+,\mu(w_+))$, where $\mu (w_+)$ lies in $\F_p$, and every element of $\cH^+$ can be uniquely expressed in this way.  Since $\cH^+$ is abelian, it is easy to see that $\mu$ defines a homomorphism from $W^+$ to $\F_p$.  Hence, we can choose $w_0\in W$ such that $\mu (w_+) = \langle w_+,w_0\rangle$, for all $w_+\in W^+$.  The identity
$$(w_+,\mu(w_+))= (w_0,0)^{-1} (w_+,0)(w_0,0)$$ shows that $\cH^+$ is conjugate to 
$W^+= W^+\times \{ 0\}$.
We can now apply Example~\ref{exwzero} to deduce that $\Hom_{\cH^+}(\tau,1)$ has dimension one.
\end{proof}

We now turn to the problem of computing
$\Hom_{\cH_\alpha^+}(\tau ,1)$ when $\alpha$ is
an automorphism of $\cH$ of order two with fixed
points
$\cH_\alpha^+$.  
In the cases of interest to us, the group $\cH_\alpha^+$
will be a group of the form $\cH^+$ mentioned in
Lemma~\ref{hplusfixed}.
The following result may be found in \cite{Ho}:

\begin{lemma}\label{HmodZautos}
The group of all
automorphisms of $\cH= W\boxtimes \F_p$ that act trivially on
$\cZ$ is canonically isomorphic to $\cS\ltimes
W$, where $\cS$ acts on $\cH$ via
its action on the first factor and $W=\cH/\cZ$
acts on $\cH$ by inner automorphisms.
\end{lemma}

One can rephrase this as follows. 
Every automorphism $\alpha$ of $\cH$ that is
trivial on $\cZ$ has the form
$\alpha (w,z)= (s_0\cdot w, z+ \langle w_0,
w\rangle)$, for some $s_0\in \cS$ and $w_0\in W$.

\begin{lemma}\label{HmodZhom}
If $\alpha$ is an automorphism
of $\cH$ of order two that acts trivially on
$\cZ$ then
$\Hom_{\cH_\alpha^+}(\tau,1)=0$.
\end{lemma}

\begin{proof}
If $\lambda\in
\Hom_{\cH_\alpha^+}(\tau,1)$ and $z\in \cZ$ then
since $z\in \cH_\alpha^+$ we have
$\zeta(z)\lambda(v)= \lambda
(\tau(z)v)=\lambda(v)$.  Choosing $z$ so that
$\zeta(z)\ne 1$, we see that we must have
$\lambda=0$.
\end{proof}

So the previous lemma says that the automorphisms
of most interest to us will always have
nontrivial restriction to $\cZ$.  Nevertheless, we
still need to have some facts about the
automorphisms that are trivial on $\cZ$. 

We remark that the notion of an $\F_p$-linear automorphism of $W$ coincides with the notion of an automorphism of the additive group of $W$ and thus we can  use the ambiguous terminology ``automorphism of $W$.''  Let $\cS^-$ be the set of all automorphisms
$s$ of $W$ such that $\langle s\cdot w,s\cdot
w'\rangle = -\langle w,w'\rangle$, for all $w,w'\in
W$.   Then the (disjoint) union $\widehat\cS = \cS\sqcup \cS^-$ is a
group that acts on $\cH$ by its natural action on the first factor of $W\boxtimes \F_p$.  

The next
fact follows from an obvious calculation:

\begin{lemma}\label{SWordtwo}
The elements of order two in
$\widehat\cS\ltimes W$ are the nontrivial elements
of the form
$(s,w)$, where $s\in \widehat\cS$ has order two,
$w\in W$  and
$s\cdot w = -w$.
\end{lemma}

When $s\in \widehat\cS$ has order two, we take
\begin{equation*}
\begin{split}
W_s^+&= \{\, w\in W\ | \ s\cdot w=w\,\} \\
W_s^-&= \{\, w\in W\ | \ s\cdot w=-w\,\}.
\end{split}
\end{equation*}

\begin{lemma}\label{seigdec}
If $s\in \widehat{\cS}$ has order two
then the decomposition
$$w= \frac{w+s\cdot w}{2}+ \frac{w-s\cdot w}{2}$$
gives rise to a direct sum decomposition $W=
W_s^+ +W_s^-$.  If $s\in \cS^-$ then this is a
polarization of $W$.  Conversely, if $W=W^+ +W^-$
is a polarization of $W$ then there is a unique
$s\in
\cS^-$ of order two such that $W^+=W_s^+$ and
$W^- =W_s^-$.
\end{lemma}

\begin{proof}
The first statement is obvious and the
second statement follows from an obvious
calculation.  Now suppose we are given a
polarization
$W=W^+ +W^-$.  Given $w=w_++ w_-$, with $w_+\in
W^+$ and $w_-\in W^-$, define $s\cdot w = w_+
-w_-$.  The calculation $\langle w_+ -w_-,
w'_+-w'_-\rangle = -\langle w_-,w'_+\rangle -
\langle w_+,w'_-\rangle = -\langle w_++ w_-, w'_++
w_-\rangle$ shows that $s\in \cS^-$.  Since $s$ has
order two, the proof is complete.
\end{proof}

\begin{lemma}\label{ordtwoauto}
The only automorphism of $\cZ$
of order two is $z\mapsto z^{-1}$ or, equivalently, $(0,z_0)\mapsto (0,-z_0)$.
\end{lemma}

\begin{proof}
Let us identify $\cZ$ with 
$\Z/p\Z$.  Then we have an isomorphism
$(\Z/p\Z)^\times \cong \hbox{Aut}(\Z/p\Z)$ that
sends $n+p\Z$ to the automorphism $\alpha_n(m) =
mn \; (\hbox{mod }p)$.  But $\alpha_n^2=1$ exactly
when $p$ divides $n^2-1$ or, equivalently, when
$n\equiv \pm 1\; (\hbox{mod }p)$.  Our claim
follows.
\end{proof}

We are interested in the automorphisms $\alpha$ of
$\cH$ that have order two and are nontrivial on
$\cZ$.  Such automorphisms $\alpha$ must satisfy
$\alpha(z)=z^{-1}$ for all $z\in \cZ$. On the
other hand, given two such automorphisms,
$\alpha_1$ and $\alpha_2$, their composite must be
trivial on $\cZ$.  So the key  remaining question
is whether any such automorphisms exist.  The
following result is obvious:

\begin{lemma}\label{polarauto}
Suppose $W= W^+ + W^-$ is a
polarization of $W$.  If $w_+\in W^+$, $w_-\in
W^-$ and $z\in \cZ$ let $\alpha(w_++w_-,z) =
(w_+-w_-,-z)$.  Then $\alpha$ is an automorphism
of $\cH$ of order two that is nontrivial on $\cZ$.
\end{lemma}

The discussion above addresses groups of the form $W\boxtimes \F_p$.  
Recall that an abstract group that is isomorphic  to a group $W\boxtimes \F_p$ 
is called a Heisenberg $p$-group  if $W$ is nondegenerate.  
Let $\cH$ be a Heisenberg $p$-group with center $\cZ$ and let $W= \cH/\cZ$.  
The commutator map 
\begin{equation*}
\begin{split}
\cH\times \cH & \to \cZ\cr
(h_1,h_2)&\mapsto [h_1,h_2] = h_1h_2 h_1^{-1} h_2^{-1}
\end{split}
\end{equation*}
yields a bimultiplicative $\cZ$-valued symplectic form on $W$
\begin{equation*}
\begin{split}
W\times W&\to \cZ\cr
(h_1\cZ, h_2\cZ)&\mapsto \langle h_1\cZ, h_2\cZ\rangle = [h_1,h_2].
\end{split}
\end{equation*}
  Let
$W^\sharp = W\boxtimes \cZ$.

\begin{definition}\label{defspeciso}
A {\it special isomorphism on $\cH$} is  a
homomorphism
$\nu :\cH\to W^\sharp$ such
that the following diagram  commutes:
$$
\xymatrix{
1 \ar[r] & {\mathcal Z} \ar[r]\ar@{=}[d]
& {\mathcal H} \ar[r]\ar[d]^\nu
& {\mathcal W} \ar[r]\ar@{=}[d] & 1 \\
1 \ar[r] & {\mathcal Z} \ar[r]
& {\mathcal W}^\sharp \ar[r] & {\mathcal W} \ar[r] & 1.}
$$  
\end{definition}

\begin{remark}\label{specisomu}
To give a special isomorphism $\nu$ on $\cH$ is equivalent to giving a set theoretic function
$\nu : \cH \to W^\sharp$ of the form $\nu (h) = (h\cZ, \mu(h))$, where $\mu : \cH \to \cZ$ is any function that satisfies:
\begin{enumerate}
\item $\mu(z) = z$, for all $z\in \cZ$, and
\item $\mu(h_1h_2) = \mu(h_1)\mu(h_2) [h_1,h_2]^{(p+1)/2}$, for all $h_1,h_2\in \cH$.
\end{enumerate}
Given two special isomorphisms $\nu_1(h) = (h\cZ,\mu_1(h))$ and $\nu_2(h) = (h\cZ, \mu_2 (h))$, we may  define a homomorphism $\chi: W\to \cZ$ by $\chi (h\cZ) = \mu_2(h)\mu_1(h)^{-1}$ and then we have $\nu_2 (h) = \chi(h\cZ) \nu_1(h)$, for all $h\in \cH$.  For each such character $\chi$, there exists a unique $w_0\in W$ such that $\chi (w) = \langle w,w_0\rangle$.  Fixing $\nu$ and varying $w_0$ to obtain other special isomorphisms, we see that the set of special isomorphisms on $\cH$ forms a principal homogeneous space of $W$.
\end{remark}

\begin{lemma}\label{specisoeq}
Suppose $\nu_1$ and $\nu_2$ are special isomorphisms on $\cH$.  Then the following are equivalent:
\begin{enumerate}
\item $\nu_1=\nu_2$,
\item $\nu_1^{-1}(W\times 1) = \nu_2^{-1}(W\times 1)$,
\item There exists $s\in \cS$ such that $\nu_2= s\circ \nu_1$.
\end{enumerate}
\end{lemma}

\begin{proof}
We first show that (3) implies (1).  Suppose $s$ is as in the statement of Condition (3).  If $h\in \cH$ and $w=h\cZ$ then there exist elements $z_1,z_2\in \cZ$ such that $\nu_i(h) = (w,z_i)$, for $i=1,2$.  On the other hand, $\nu_2(h) = s\cdot \nu_1(h) = (sw,z_1)$.  So we have $(w,z_2)= (sw,z_1)$.  It follows that $s=1$ and thus $\nu_1=\nu_2$.

Next, we show (2) implies (1).  Assume Condition (2) holds.   Let $\chi : W\to \cZ$ be defined by $\nu_2 (h) = \chi(h\cZ) \nu_1(h)$, for all $h\in \cH$.  Now fix $w\in W$ and let $h= \nu_1^{-1}(w,1)$.   Then $\nu_2(\nu_1^{-1}(w,1))= \nu_2(h) = \chi(w)\nu_1(h)= \chi(w) (w,1)$.  This shows that $\chi(w)(w,1)$ lies in $W\times 1$.  Hence $\chi(w)=1$.  Since Condition (1) obviously implies the other two conditions, the proof is complete.
\end{proof}

For reference purposes, we record the following obvious fact:

\begin{lemma}\label{specisores}
Suppose $\nu':\cH'\to W'^\sharp$ is a special isomorphism on a Heisenberg $p$-group with center $\cZ$.  Suppose also that $\cH$ is a Heisenberg $p$-subgroup of $\cH'$, that is, $\cH$ is a subgroup of $\cH'$ with center $\cZ$ and the space $W= \cH/\cZ$ is a nondegenerate subspace of $W'= \cH'/\cZ$.  Then restricting $\nu'$ to $\cH$ gives a special isomorphism $\nu: \cH \to W^\sharp$.
\end{lemma}

\begin{remark}\label{absHWliftdef}  Suppose $\cH$ is a Heisenberg $p$-group with center $\cZ$ and let $\zeta$ be a nontrivial  character of $\cZ$.  There is a unique isomorphism of the additive group of $\F_p$ with the group $\mu_p$ of complex $p$-th roots of unity that sends $1\in \F_p$ to $e^{2\pi i/p}$.  We use this to identify $\F_p$ with $\mu_p$.  Then a nontrivial character of $\cZ$, such as $\zeta$, is nothing other than a group isomorphism $\cZ\to \F_p$.  As above, we have a nondegenerate $\cZ$-valued symplectic form on $W= \cH/\cZ$.  Using the isomorphism $\zeta : \cZ\to \F_p$, we may view the values of the symplectic form as elements of $\F_p$.  We have an isomorphism of $W^\sharp= W\boxtimes \cZ$ with $W\boxtimes \F_p$ given by $(w,z)\mapsto (w,\zeta (z))$.  Now fix a special isomorphism $\nu : \cH \to W^\sharp$.  We obtain a semidirect product $\cS\ltimes_\nu\cH$ by pulling back the natural semidirect product $\cS\ltimes W^\sharp$ via $1\times \nu$.  Let $\tau$ be a Heisenberg representation of $\cH$ with central character $\zeta$.  Let $\tau^\sharp$ be the associated representation of $W^\sharp$.  Identifying $W^\sharp$ with $W\boxtimes \F_p$, as above, we obtain a Heisenberg-Weil lift of $\tau^\sharp$ to a representation $\hat\tau^\sharp$ of $\cS\ltimes W^\sharp$, in the sense of Definition \ref{HeisWeilrep}.  The corresponding representation $\hat\tau$ of $\cS\ltimes_\nu\cH$ will be referred to as the ``Heisenberg-Weil lift of $\tau$ to $\cS\ltimes_\nu \cH$.''  Except in the case when $p=3$ and $\ell=1$, this is just the unique extension of $\tau$ to $\cS\ltimes_\nu \cH$.
\end{remark}

\begin{definition}\label{defsppolar}
Suppose $\cH$ is a Heisenberg $p$-group with center $\cZ$ and let $W= \cH/\cZ$.   A {\it 
polarization of $\cH$}
is a pair $(\cH^+,\widehat{\cH}^-)$ of subgroups of $\cH$ with
images
$W^+$ and
$W^-$, respectively, in $W$ such
that $\cH^+\cap \cZ=\{1\}$ and
$\cZ\subset \widehat{\cH}^-$ 
and
$W = W^+ + W^-$ is a
polarization of $W$.  A {\it split
polarization of $\cH$}
is a pair $(\cH^+,\cH^-)$ of subgroups of $\cH$ with
images
$W^+$ and
$W^-$ in $W$ such
that $\cH^+\cap \cZ=\{1\}=\cH^-\cap \cZ$ 
and
$W = W^+ + W^-$ is a
polarization of $W$.  A {\it splitting} of a polarization $(\cH^+,\widehat{\cH}^-)$ of $\cH$ consists of the choice of a subgroup $\cH^-$ of $\widehat{\cH}^-$ such that 
$\widehat{\cH}^- = \cH^-\cZ$ and $(\cH^+,\cH^-)$ is a split polarization.
\end{definition}

The significance of split polarizations is seen in the next result:

\begin{lemma}\label{sppolarspeciso}
If $(\cH^+,\cH^-)$ is a split polarization of a Heisenberg $p$-group $\cH$  then there is an associated special isomorphism ${\nu} : \cH \to W^\sharp$ that is given by
$$\nu (w_+w_-z) = (\bar w_+\bar w_-,  z [w_+,w_-]^{(p+1)/2}),$$ where $w_+\in \cH^+$, $w_-\in \cH^-$, $z\in \cZ$ and $\bar w_+$ and $\bar w_-$ are the images of $w_+$ and $w_-$ in $W$.  If $\nu' : \cH\to W^\sharp$ is any special isomorphism such that $\nu'(\cH^+)$ and $\nu'(\cH^-)$ are contained in $W\times 1$ then $\nu'= \nu$.
\end{lemma}

\begin{proof}
The fact that $\nu$ defines a special isomorphism is essentially Lemma 10.1 of \cite{Y}.  Clearly, $\nu^{-1}(W\times 1)$ consists of the elements of $\cH$ of the form $$w_+w_-[w_+,w_-]^{-(p+1)/2},$$ with $w_+\in \cH^+$ and $w_-\in \cH^-$.  We observe that $\nu'(w_+w_-[w_+,w_-]^{-(p+1)/2}) = (\bar w_+,1)(\bar w_-,1)[w_+,w_-]^{-(p+1)/2} = (\bar w_+\bar w_-,1)\in W\times 1$.  Therefore,
$\nu^{-1}(W\times 1)$ is contained in $\nu'^{-1}(W\times 1)$.  Since these are finite sets of the same cardinality, they must be equal.  Hence, $\nu=\nu'$ by Lemma \ref{specisoeq}. 
\end{proof}

We now fix a
polarization $(\cH^+,\widehat{\cH}^-)$ of $\cH$.  The subgroup $\widehat{\cH}^-$
is necessarily a maximal abelian subgroup of
$\cH$.  The set of splittings of the exact sequence 
$$1\to \cZ \to \widehat{\cH}^-\to W^-\to 1$$ is a principal homogeneous space of $W^+$.  This is seen as follows.  Let $\delta : \widehat{\cH}^-\to W^-$ be the natural projection.  Then a splitting is a homomorphism $\sigma : W^-\to \widehat{\cH}^-$ such that $\delta\circ\sigma = {\rm Id}_{W^-}$.  Associated to the splitting is an isomorphism $\widehat{\cH}^- \cong W^- \times \cZ$ given by $h\mapsto (\delta (h), h\; \sigma(\delta(h))^{-1})$.  This isomorphism restricts to the identity map on $\cZ$ and modding out by the center gives the identity map on $W^-$.  Conversely, if we are given an isomorphism $\mu : \widehat{\cH}^- \cong  W^-\times \cZ$ that restricts to ${\rm Id}_\cZ$ and induces the identity map on $W^-$, we obtain a splitting $\sigma$ by 
$\sigma (w) = \mu^{-1} (w,1)$.  
Now suppose $w_+\in W^+$ and define an automorphism $\gamma$ of $W^- \times \cZ$ by 
$$
\gamma (w,z) = (w, \langle w,w_+\rangle + z).
$$  
Let $\sigma$ be a splitting, and $\mu$ the associated automorphism of
$W^-\times \cZ$.
Then $\gamma\circ \mu$ is an isomorphism $\widehat{\cH}^-\cong W^-\times \cZ$ 
that restricts to ${\rm Id}_\cZ$ and induces the identity map on $W^-$ and 
$w\mapsto \mu^{-1}(\gamma^{-1}(w,1))$ defines a splitting, which we will
denote by $w^+\cdot\sigma$.

Given a polarization of $\cH$ and a nontrivial character $\zeta$ of $\cZ$, a splitting of the polarization is essentially equivalent to the choice of an extension $\hat\zeta$ of $\zeta$ to a character of $\widehat{\cH}^-$ or, in other words, it is the same as the choice of an irreducible component 
of the induced representation ${\rm Ind}_{\cZ}^{\widehat{\cH}^-}(\zeta)$.  Indeed, the image of such a character $\hat\zeta$ must be the group $\mu_p$ of complex $p$-th roots of unity and the composite 
$$\ker\hat\zeta \hookrightarrow \cH \rightarrow\kern-.7em\rightarrow  W^-$$
must be an isomorphism.  The inverse map $W^- \to \ker\hat\zeta$ composed with the inclusion $\ker\hat\zeta\hookrightarrow \widehat{\cH}^-$ defines a splitting of the polarization with $\cH^- = \ker\hat\zeta$.
The choice of $\hat\zeta$ (or, equivalently, the splitting) determines a Heisenberg representation
$$\tau = {\rm Ind}_{\widehat{\cH}^-}^\cH (\hat\zeta).$$  We emphasize that we mean to say ``representation'' rather than ``isomorphism class of representations.''  

In general, if one is given an abstract Heisenberg $p$-group $\cH$  then there is no canonical special isomorphism on $\cH$ and no canonical polarization of $\cH$.  However, in our applications, the Heisenberg $p$-groups will be embedded as subquotient groups of an ambient group $G$.  Following Yu, we will use the extrinsic properties of $\cH$ to obtain canonical special isomorphisms on our Heisenberg $p$-groups.  The group $G$ will also come with an involution $\theta$ that yields canonical polarizations of our $p$-Heisenberg groups.  
The next result says that we may use our canonical 
special isomorphisms to obtain canonical splittings for our canonical polarizations.

\begin{lemma}\label{polsplitting}
Suppose $\nu : \cH\to W^\sharp$ is a special isomorphism on  a Heisenberg $p$-group $\cH$ with center $\cZ$ and suppose $(\cH^+,\widehat{\cH}^-)$ is a polarization of $\cH$.  Then the set $\cH^-= \widehat{\cH}^- \cap \nu^{-1}(W\times 1)$  is a subgroup of $\widehat{\cH}^-$ that defines a splitting of $(\cH^+,\widehat{\cH}^-)$ which is canonically associated to $\nu$.
\end{lemma}

\begin{proof}
Since $\nu(\cZ)=\cZ$ intersects $W\times \{\, 1\,\}$ only trivially,
we have $\cH^-\cap \cZ=\{\, 1\,\}$. If $\hat h^-\in\widehat{\cH}^-$ has
image $(w,z)\in W^\sharp$, then
$\hat h^-z^{-1}\in \cH^-$, so
$\widehat{\cH}^-\subset \cH^-\cZ$. The reverse containment is
obvious.
\end{proof}


We now recall a basic result from \cite{Ge}.  For concreteness, we take $W_0$ to be $\F_p^\ell$ with the standard dot product and let 
$$
W= \F^{2\ell}_p = \left\{\, \vectr x y \ | \ x,y\in W_0\,\right\}
$$
 with the symplectic form 
$$\left\langle \vectr{x_1}{y_1} ,\vectr{x_2}{y_2}\right\rangle =x_1\cdot y_2 - y_1\cdot x_2.$$ We also use the polarization
\begin{equation*}
\begin{split}
W^+&=\left\{\, \vectr x0\ | \ x\in W_0\,\right\}\cr
W^-&=\left\{\, \vectr 0y\ | \ y\in W_0\,\right\}.
\end{split}
\end{equation*}
We identify $GL(W)= GL(2\ell ,\F_p)$ with the group of invertible block matrices $\2by2 abcd$ with $\ell\times\ell$ blocks.  Then
$$
\cS = {\rm Sp}(W)= {\rm Sp}(2\ell,\F_p) = 
\left\{\, g\in GL(W)\ | \ {}^t g\2by2 01{-1}0 g = \2by2 01{-1}0\,\right\}.
$$
We also define subgroups
\begin{equation*}
\begin{split}
\cM&=\left\{\, \2by2 y00{{}^ty^{-1}}\ | \ y\in GL(\ell,\F_p)\,\right\}\cr
&= \{\, s\in \cS\ | \ sW^+\subset W^+, sW^-\subset W^-\,\}\cr
\cN&=\left\{\, \2by2 1x01\ | \ x\in M(\ell, \F_p), {}^tx=x\,\right\}\cr
\cP&=\cM\cN = \left\{\, s\in \cS\ | \ sW^+\subset W^+\,\right\}.
\end{split}
\end{equation*}
Then $\cP$ is a maximal parabolic subgroup in $\cS$ with Levi decomposition $\cP = \cM\cN$.
Such a parabolic subgroup in $\cS$, that is, one whose Levi factor is isomorphic
to $GL(\ell,\F_p)$, is known as a Siegel parabolic.  Define characters of $\cM$ and $\cP$ by
\begin{equation*}
\begin{split}
\chi^{\cM}\left( \2by2 y00{{}^ty^{-1}}\right)&=(\det y)^{(p-1)/2},\quad  y\in GL(\ell,\F_p),\cr
\chi^{\cP}(mn)&=\chi^{\cM}(m),\quad m\in\cM, n\in \cN.
\end{split}
\end{equation*}
The values of these characters lie in $\{ \pm 1\}$.

Consider now the Heisenberg representation $\tau_+ = {\rm Ind}_{W^+\times\F_p}^\cH (1\times \zeta)$ of $\cH = W\boxtimes \F_p$ associated to some nontrivial central character $\zeta$.  (We are using the notation $\tau_+$ rather than $\tau$ to emphasize that this is a different model for the Heisenberg representation than we were using earlier, namely, it is not ${\rm Ind}_{W^-\times\F_p}^\cH (1\times \zeta)$.)  Let $\hat\tau_+$ denote the Heisenberg-Weil lift of $\tau_+$ to 
$\cS\ltimes \cH$.  Then it is known that
$$
(\hat\tau_+ (g)\varphi) (h) = \chi^\cP (g)\ \varphi (g^{-1}\cdot h),
$$ 
for all $g\in \cP$ and $h\in \cH$.  (See Theorem 2.4 (b) in \cite{Ge}.)       
Lemma \ref{hplusfixed} implies that ${\rm Hom}_{W^+}(\tau_+,1)$ has dimension one.  However,  we are using a different model for the Heisenberg representation than the one used in Lemma \ref{hplusfixed}.  The space of $W^+$-invariant linear forms for the Heisenberg representation $\tau_+$ is spanned by the linear form
$$\lambda_+ (\varphi) = \varphi (1).$$  When $g\in \cP$, we have 
$$\lambda_+ (\hat\tau_+(g)\varphi) = (\hat\tau_+ (g)\varphi)(1) = \chi^\cP (g) \varphi (1) = \chi^\cP (g) \lambda_+ (\varphi).$$
This fact is expressed in a more abstract setting in the following result.

\begin{lemma}\label{realtrace}
Let $\tau$ be a Heisenberg representation of a Heisenberg $p$-group $\cH$.  Let $\cZ$ be the center of $\cH$ and let $\zeta$ be the central character of $\tau$.  Suppose $\nu$ is a special isomorphism on $\cH$ and, as in Remark \ref{absHWliftdef}, define  the Heisenberg-Weil lift $\hat\tau$ of  $\tau$ to $\cS \ltimes_\nu \cH$.  Assume we are given a polarization $W= W^+ + W^-$ of $W=\cH/\cZ$ and define $\cP$, $\cM$, $\chi^\cP$ and $\chi^\cM$ as above.  
Then ${\rm Hom}_{\nu^{-1}(W^+\times \{\,1\,\})}(\tau ,1)$ has dimension one and if $\lambda$ lies in this space then $$\lambda(\hat\tau (g)\varphi) = \chi^\cP (g)\ \lambda(\varphi),$$  for all $g\in \cP$ and all $\varphi$ in the space of $\tau$.  The trace of $\hat\tau\, |\,\cM$ is real-valued and its sign is given by $\chi^\cM$.
\end{lemma}

\begin{proof}
The facts regarding the trace of $\hat\tau\, |\,\cM$ follow from Theorem 2.4 (c) and Proposition 1.4 (b) in \cite{Ge}.  Everything else follows from the discussion preceding 
the statement of this lemma.
\end{proof}

As indicated above, the Heisenberg $p$-groups of most interest to us are associated to some ambient group $G$ that is equipped with an involution $\theta$.  The involution $\theta$ induces involutions of the various Heisenberg 
$p$-groups.  The next result tells us that these involutions of the 
Heisenberg $p$-groups induce polarizations.

\begin{theorem}\label{Heisthm}
Let $\cH$ be a Heisenberg
$p$-group and suppose $\alpha$ is an automorphism
of order two of $\cH$ whose restriction to the
center $\cZ$ is not the identity map.  Then the sets 
\begin{equation*}
\begin{split}
\cH_\alpha^+&= \{\, h\in \cH\ | \ \alpha
(h) = h\,\},\cr
\widehat{\cH}_\alpha^-&= \{\, h\in \cH \ | \ \alpha
(h) = h^{-1} \,\},
\end{split}
\end{equation*}
are subgroups of $\cH$ that form a polarization of $\cH$.  Let $W_\alpha^+$ and
$W_\alpha^-$  be the images of $\cH_\alpha^+$
and $\widehat{\cH}_\alpha^-$, respectively, in $W =
\cH/\cZ$. 
If $\bar\alpha$ is the automorphism of $W$
obtained from $\alpha$ by reduction modulo $\cZ$
then $\bar\alpha\in \cS^-$ and  
\begin{equation*}
\begin{split}
W_{\alpha}^+&= \{\, w\in W\ | \
\bar\alpha (w) = w\,\},\cr
W_{\alpha}^-&= \{\, w\in W\ | \ \bar\alpha
(w) = -w \,\}.
\end{split}
\end{equation*}
Let $\tau$ be a Heisenberg representation of $\cH$.  Then the space
$\Hom_{\cH_\alpha^+}(\tau,1)$ has dimension one and $\tau$ satisfies $\tau\circ \alpha\simeq
\tilde\tau$.  
Let
$\nu : \cH \to W^\sharp$ be a special isomorphism such that 
$\nu (\cH_\alpha^+) = W^+_\alpha\times 1$ and let $\hat\tau$ be the Heisenberg-Weil 
lift of $\tau$ to $\cS\ltimes_\nu\cH$.  If $\chi^{\cP}$ is the
unique character  of order two of the parabolic
subgroup 
$$\cP=\{\, s\in \cS \ | \ s\cdot
W_{\alpha}^+\subset W_{\alpha}^+\,\}$$  then
$$\Hom_{\cH^+_\alpha}(\tau,1)\subset \Hom_{\cP}(\hat\tau,\chi^\cP).$$
\end{theorem}

\begin{proof} The fact that $\bar\alpha$ lies in $\cS^-$ follows from
 Lemma \ref{ordtwoauto}
and the computation $[\alpha (h_1),\alpha (h_2)] = \alpha ([h_1,h_2]) = [h_1,h_2]^{-1}$.

Suppose $h\in\cH$ and $w=h\cZ\in W$. 
We claim that
$\bar\alpha (w) = w$ exactly when there exists
$z\in \cZ$ such that $\alpha (hz)= hz$.  Indeed,
$\bar\alpha(w)=w$ is equivalent to the condition
$h^{-1}\alpha(h)\in \cZ=\cZ^2$.  Thus
we may choose $z\in \cZ$ such that $h^{-1}\alpha
(h)=z^2$ or, equivalently, $\alpha(hz) = hz$.  Hence, $W_\alpha^+$ is identical to the set of fixed points of $\bar\alpha$.  

If $w=h\cZ$ then $\bar\alpha( w) = - w$ is equivalent to the existence of $z\in \cZ$ such that $h\alpha (h) =z$.  If the latter condition holds, then $z= h\alpha (h) = h(\alpha (h) h) h^{-1} = h \alpha (z) h^{-1} = \alpha (z) = z^{-1}$.  But $z=z^{-1}$ implies $z=1$ and thus $\alpha (h) = h^{-1}$.  So we have shown that $W_\alpha^-$ consists of those elements $w\in W$ such that $\bar\alpha (w)=-w$.
Lemma \ref{seigdec} now
implies
$W= W_\alpha^++ W_\alpha^-$ is a polarization
of $W$.

The argument in the previous paragraph actually shows that 
$$
\widehat{\cH}_\alpha^- = \{\, h\in \cH \ | \ h\alpha (h) \in \cZ\,\}.
$$  This implies that $\widehat{\cH}_\alpha^-$ is a group.  Moreover, it is a subgroup of $\cH$ that is
 invariant under translations by $\cZ$ and that projects to $W_\alpha^-$.  These properties imply that $\widehat{\cH}_\alpha^-$ must be the unique maximal abelian subgroup of $\cH$ that projects to $W_\alpha^-$.  It now follows that we have a polarization of $\cH$.

Lemma \ref{hplusfixed} implies that 
$\Hom_{\cH_\alpha^+}(\tau,1)$ has dimension one.   We obtain the relation  $\tau\circ \alpha\simeq
\tilde\tau$ upon noting that $\tau\circ\alpha$ and $\tilde\tau$ are both Heisenberg representations with central character inverse to that of $\tau$.  
The final assertion follows from Lemma \ref{realtrace} and the assumption 
$\nu (\cH^+_\alpha) = W_\alpha^+\times \{\,1\,\}$. \end{proof}


\begin{theorem}\label{HeisGelf}
Let $\cH$ and $\alpha$ be as in Theorem~\ref{Heisthm}.
Then
$$
\cH_\alpha^+ \alpha(h)\cH_\alpha^+ =
\cH_\alpha^+ h^{-1}\cH_\alpha^+,'
$$ 
for all
$h\in \cH$ and thus, according to Gelfand's Lemma,
$(\cH ,\cH_\alpha^+)$ is a Gelfand pair.
\end{theorem}

\begin{proof}
The theorem follows from the identity
$$
(-w_+ ,0)(w_+-w_-,-z)(-w_+,0)=
(-w_+-w_-,-z),
$$ 
for $w_+\in W_\alpha^+$, $w_-\in
W_\alpha^-$ and $z\in \cZ$.
\end{proof}

For more details on Gelfand pairs and Gelfand's Lemma, we refer 
the reader to \cite{Gr}.


\section{${\rm SL}(2,3)$ (an appendix to Section 2.3)}
\label{sec:eselltwo}

In this appendix, we consider the special case of the theory from Section \ref{sec:Heis} in which $p=3$ and $\ell=1$.
We use the symbols $-1,0,1$ to ambiguously denote elements of $\F_3$ and real numbers.  With this abuse of notation, we define an additive character of $\F_3$ by $\zeta (t) = \omega^t$, where
$$\omega = e^{2\pi i/3} = -\frac{1}{2} + \frac{\sqrt{3}}{2} i.$$ Then $\widehat{\F_3} = \{\,1,\zeta , \zeta^{-1}\,\}$.  Similarly, we define $\chi\in \widehat{\F^\times_3}$ by $\chi(t)=t$ and we have $\widehat{\F^\times_3} = \{ 1,\chi\}$.

In the case at hand, the symplectic group ${\rm Sp}(2,\F_3)$ coincides with ${\rm SL}(2,\F_3)$, which we also denote by ${\rm SL}(2,3)$.  This is the only example of a symplectic group over a finite field of odd characteristic that is not a perfect group in the sense that it does not equal its commutator subgroup.  Indeed, a nontrivial character $\alpha$ of ${\rm SL}(2,3)$ is defined by the conditions
\begin{equation*}
\begin{split}
\alpha \left(\left(
\begin{array}{cc}
a&0\cr 0&a
\end{array}
\right)\right)
&= 1\cr
\alpha \left(\left(
\begin{array}{cc}
1&b\cr 0&1
\end{array}
\right)\right)
&=\zeta (-b)\cr
\alpha \left(\left(
\begin{array}{cc}
0&1\cr -1&0
\end{array}
\right)\right)
&=1,\cr
\end{split}
\end{equation*} 
where $a\in \F_3^\times$ and $b\in \F_3$.  (Note that the only diagonal matrices in 
${\rm SL}(2,3)$ are the scalars $\pm 1$.)  The complete set of one-dimensional representations
of ${\rm SL}(2,3)$ is the set $\{\, 1,\alpha,\alpha^{-1}\,\}$.

We also define a 2-dimensional representation $\beta$ of ${\rm SL}(2,3)$ by
\begin{equation*}
\begin{split}
\beta \left(\left(
\begin{array}{cc}
a&0\cr 0&a
\end{array}
\right)\right)
&= \left(
\begin{array}{cc}
\chi(a)&0\cr 0&\chi(a)
\end{array}
\right)\cr
\beta \left(\left(
\begin{array}{cc}
1&b\cr 0&1
\end{array}
\right)\right)
&=
\left(
\begin{array}{cc}
1&0\cr 0&\zeta (-b)
\end{array}
\right)\cr
\beta \left(\left(
\begin{array}{cc}
0&1\cr -1&0
\end{array}
\right)\right)
&=
\left(
\begin{array}{cc}
\frac{i}{\sqrt{3}}&\frac{2i}{\sqrt{3}}\cr 
\frac{i}{\sqrt{3}}&\frac{-i}{\sqrt{3}}
\end{array}
\right),\cr
\end{split}
\end{equation*} 
Note that $\det\circ \beta = \alpha$.  

Suppose we are given a symplectic space $W$ of dimension two over $\F_3$ and a polarization $W= W^++W^-$.  Then we can choose $e_1\in W^-$ and $e_2\in W^+$ such that $\langle e_1,e_2\rangle =1$.  For simplicity, we make the identifications 
$$e_1= \left(
\begin{array}{c}
1\cr 0
\end{array}\right),\quad 
e_2= \left(
\begin{array}{c}
0\cr 1
\end{array}
\right)$$ and $W= \F_3^2$.  The symplectic form is given by
$$\langle v,w\rangle = {}^t vjw,\quad
j= \left(
\begin{array}{cc}
0&1\cr 
-1&0
\end{array}
\right).$$
One model for the Heisenberg representation $\tau$ of $\cH = W\boxtimes \F_3$ associated to the character $\zeta$ (defined above) is the induced representation ${\rm Ind}_{W^-\boxtimes \F_3}^{\cH} (1\times\zeta)$.  Restriction of functions from $\cH$ to $W^+$ defines an isomorphism of the space of $\tau$ with the space $\C [W^+]$ of all complex-valued functions on $W^+$.  The map $te_2\mapsto t$ identifies $W^+$ with $\F_3$ and we use this to identify the space of $\tau$ with the space $\C [\F_3]$ of complex-valued functions on $\F_3$.
With this model for $\tau$, the Heisenberg representation acts according to:
\begin{equation*}
\begin{split}
\tau (xe_1)f(t)&=\zeta(-xt) f(t)\cr
\tau (ye_2)f(t)&=f(t+y)\cr
\tau (z)f(t)&=\zeta(z)f(t),
\end{split}
\end{equation*}
with $x,y,t\in \F_3$ and $z\in \F_3^\times$.

We can extend $\tau$ to a representation $\hat\tau$ of ${\rm SL}(2,3)\ltimes \cH$ on 
$\C [\F_3]$ by defining 
\begin{equation*}
\begin{split}
\hat\tau  \left(\left(
\begin{array}{cc}
a&0\cr 0&a
\end{array}
\right)\right) f(t)&=\chi(a) f(at)\cr
\hat\tau  \left(\left(
\begin{array}{cc}
1&b\cr 0&1
\end{array}
\right)\right) f(t)&=\zeta(-bt^2)f(t)\cr
\hat\tau  \left(\left(
\begin{array}{cc}
0&1\cr -1&0
\end{array}
\right)\right) f(t)&=i\hat f(t),
\end{split}
\end{equation*}
where $a\in\F_3^\times$, $b,t\in \F_3$ and the Fourier transform is given by
$$\hat f(t) = \frac{1}{\sqrt{3}} (f(0)+ f(1)\zeta (-t) + f(-1)\zeta (t)).$$
(The latter Weil representation coincides with the representation specified in Theorem 2.4 of \cite{Ge}.)  However, it should be emphasized that there are two other extensions of $\tau$ to 
${\rm SL}(2,3)\ltimes \cH$.  These are obtained by twisting $\hat\tau$ by the inflation of $\alpha$ or $\alpha^{-1}$ to a character of ${\rm SL}(2,3)\ltimes \cH$.

One can explicitly describe the Weil representation $\hat\tau$ on ${\rm SL}(2,3)$ in terms of matrices as follows.  
Define a basis $\xi_1,\xi_2,\xi_3$ of $\C[\F_3]$ by letting $\xi_1 (t) =t$ and letting $\xi_2$ and $\xi_3$ be the characteristic functions of $\{0\}$ and $\{ -1,1\}$, respectively.  The span of $\xi_1$ is just the space of odd functions in $\C[\F_3]$ and ${\rm SL}(2,3)$ acts on this space according to the character $\alpha$.  On the other hand, $\xi_2$ and $\xi_3$ span the space of even functions in $\C[\F_3]$.    We use the map 
$$u\xi_2+v\xi_3\mapsto 
\left(\begin{array}{c}
u\cr v
\end{array}\right)$$
 to identify the space of even functions with $\C^2$.  Then the Weil representation assigns to each $g\in {\rm SL}(2,3)$ a 2-by-2 unitary matrix that operates on $\C^2$.  This matrix is precisely the matrix $\beta (g)$ defined above.  In particular, we see that our Weil representation coincides with $\alpha \oplus\beta$.
 
We close this section by completing Definition \ref{HeisWeilrep} of ``Heisenberg-Weil lift'' in the case of $p=3$ and $\ell=1$.  Fix a Heisenberg representation $\tau$ of a Heisenberg group $\cH = W\boxtimes \F_3$, where $W$ has dimension 2.  Suppose $e_1,e_2\in W$ and $\langle e_1,e_2\rangle=1$.  Given $s\in \cS={\rm Sp}(W)$, define a matrix
$$\gamma (s) =
\left(\begin{array}{cc}
a&b\cr c&d
\end{array}\right),$$
where
\begin{equation*}
\begin{split}
s(e_1)&=ae_1+ce_2\cr
s(e_2)&=be_1+de_2.
\end{split}
\end{equation*}
It is routine to verify that $\gamma$ defines an isomorphism $\cS\cong {\rm SL}(2,3)$.  Moreover, if 
\begin{equation*}
\begin{split}
e'_1&=a_0e_1+c_0e_2\cr
e'_2&=b_0e_1+d_0e_2
\end{split}
\end{equation*}
is another basis of $W$ with $\langle e'_1,e'_2\rangle=1$ then the change of basis matrix
$$\xi =
\left(\begin{array}{cc}
a_0&b_0\cr c_0&d_0
\end{array}\right)$$
lies in ${\rm SL}(2,3)$ and the isomorphism $\gamma' : \cS\to {\rm SL}(2,3)$ associated to $e'_1,e'_2$ is just
$$\gamma'(s) = \xi^{-1}\gamma (s)\xi.$$  Thus the character of the Weil representation $\alpha \oplus \beta$ transfers to a character of a unique isomorphism class of representations of $\cS$.  We call representations in this class ``Weil representations associated to $\zeta$'' and representations in the contragredient class will be called ``Weil representations associated to $\zeta^{-1}$.''  The {\it Heisenberg-Weil lift}  of $\tau$ to $\cS\ltimes \cH$ is the unique lift $\hat\tau$ of $\tau$ to $\cS\ltimes \cH$ such that $\hat\tau\, |\,\cS$ is a Weil representation associated to the central character of $\tau$.  This completes Definition \ref{HeisWeilrep}.

\section{Buildings, Moy-Prasad filtrations, and twisted Levi sequences}
\label{sec:buildings}

Recall that $\bG$ is a connected reductive $F$-group that splits over a 
tamely ramified extension of $F$, $G=\bG(F)$, $\bfr{g}$ is the Lie
algebra of $\bG$, and $\g=\bfr{g}(F)$.

Bruhat-Tits theory plays a key role in aspects of Yu's construction 
of tame supercuspidal representations. In this section, we
recall some basic facts about buildings and Moy-Prasad filtrations.
We also state the definitions of some other subgroups of $G$ and 
lattices in $\g$
that are defined via Bruhat-Tits theory.

Let $\cB(\bG,F)$ be the (extended) Bruhat-Tits building
of $G$. For every extension $E$ of $F$ of finite
ramification degree, we have a building $\cB(\bG,E)$ for
$\bG(E)$. If $E$ is Galois over $F$, then $\mathrm{Gal}(E/F)$ acts
on $\cB(\bG,E)$, and the fixed point set contains $\cB(\bG,F)$,
with equality when $E$ is tamely ramified over $F$ (\cite{P},
\cite{Ro}).

Let $\bG_{\mathrm{der}}$ be the derived group of $\bG$.
The reduced building $\cB_{\mathrm{red}}(\bG,F)$ of $G$ is the
building $\cB(\bG_{\mathrm{der}},F)$ of $\bG_{\mathrm{der}}(F)$.
Let $X_*(\bZ,F)$ be the group of $F$-rational
cocharacters of the center $\bZ$ of $\bG$.
By definition (see Section~4.2.16 of \cite{BT2}),
$$
\cB(\bG,F)=\cB_{\mathrm{red}}(\bG,F)\times (X_*(\bZ,F)\otimes \R).
$$
When convenient, we view $\cB_{\mathrm{red}}(\bG,F)$ as the
space of $(X_*(\bZ,F)\otimes \R)$-orbits in
$\cB(\bG,F)$, under the action of $X_*(\bZ,F)\otimes
\R$ by affine translations. If $x\in \cB(\bG,F)$, let
$[x]$ denote the image of $x$ in $\cB_{\mathrm{red}}(\bG,F)$.

Every maximal $F$-split torus $\bS$ in $\bG$ has an associated
apartment $A(\bG,\bS,F)\subset \cB(\bG,F)$. Let $\bT$ be a
maximal $F$-torus in $\bG$ containing $\bS$. Then 
$\bT$ splits over some finite Galois extension $E$ of $F$, so
$\bT$ has an apartment $A(\bG,\bT,E)$ in $\cB(\bG,E)$.
Furthermore $A(\bG,\bS,F)$ is equal to $A(\bG,\bT,E)^{\mathrm{Gal}(E/F)}$.
More generally, if $\bT$ is a maximal $F$-torus in $\bG$
that splits over a finite tamely ramified Galois extension $E$
of $F$,
we set $A(\bG,\bT,F)=A(\bG,\bT,E)\cap \cB(\bG,F)$.
Note that if $\bT$ does not contain a maximal $F$-split
torus in $\bG$, then $A(\bG,\bT,F)$ is not an apartment
in $\cB(\bG,F)$, although we do have $A(\bG,\bT,F)=
A(\bG,\bT,E)^{\mathrm{Gal}(E/F)}$, and $A(\bG,\bT,F)$ is
independent of the choice of $E$, as shown in \cite{Y}.

\begin{definition}\label{deftwistedLevi}
Suppose that $\bG'$ is an $E$-Levi $F$-subgroup of
$\bG$ for some finite extension $E$ of $F$. Such a group
will be called a \textit{twisted Levi subgroup} of $\bG$.
If we can choose $E$ to be tamely ramified over
$F$, then we call $\bG'$ a \textit{tamely ramified
twisted Levi subgroup}. 
\end{definition}

If $\bG^\prime$ is a tamely ramified twisted Levi
subgroup of $\bG$, then, letting
$G^\prime=\bG^\prime(F)$,
there is a family of $G^\prime$-equivariant
embeddings of the building
$\cB(\bG^\prime,F)$ of $G^\prime$ into
$\cB(\bG,F)$.
All of these embeddings have the same image,
allowing us to identify $\cB(\bG^\prime,F)$
with a subset of $\cB(\bG,F)$.

In \cite{MP1}, Moy and Prasad associated
to any point $x$ in $\cB(\bG,F)$ a parahoric subgroup
$G_{x,0}$ of $G$, a filtration $\,\{\,G_{x,r}\,\}_{r\ge 0}$
of the parahoric, and a filtration $\{\, \g_{x,r}\,\}_{r\in \R}$
of the Lie algebra $\g$.
The indexing of these filtrations depends on a choice
of affine roots, hence on a choice of normalization of valuation on $F$.
In this paper, we take $v_F$ to be the valuation on $F$ 
such that $v_F(F^\times)=\Z$. When working over an
algebraic extension $E$ of $F$, we extend $v_F$
to a valuation (also called $v_F$) on $E$.

If $g\in G$, the notation $\Int(g)$ will
be used to denote the automorphism of $G$ given
by conjugation by $g$. The adjoint representation of $G$ on $\g$
will be denoted by $\Ad$.

We list a few properties of the Moy-Prasad filtrations.
Let $x\in \cB(\bG,F)$ and $r$, $s\in \R$. Then:
\begin{enumerate}
\item If $\theta$ is an automorphism of $\bG$ that is
defined over $F$, then
$\theta(G_{x,r})=G_{\theta(x),r}$, $r\ge 0$,
and $\theta(\g_{x,r})=\g_{\theta(x),r}$.
(Here the notation $\theta$ is also used for the differential
of $\theta$, and, if $x\in \cB(\bG,F)$, $\theta(x)$ denotes the 
image of $x$ under
the automorphism of $\cB(\bG,F)$ induced by $\theta$.)
\item 
$\Int (g)G_{x,r}=G_{gx,r}$, $g\in G$, 
$\Ad (g)\g_{x,r}=\g_{gx,r}$. 
\item
$[G_{x,r},G_{x,s}]\subset G_{x,r+s}$ and $[\g_{x,r},\g_{x,s}]
\subset\g_{x,r+s}$.
\item If $\varpi$ is a prime element in $F$, then
$\varpi\g_{x,r}=\g_{x,r+1}$.
\item If $\bG^\prime$ is a tamely ramified twisted  Levi subgroup
of $\bG$ and $x\in \cB(\bG^\prime,F)$,
then
$\g^\prime_{x,r}=\g_{x,r}\cap \g^\prime$, and if $r>0$,
$G^\prime_{x,r}=G_{x,r}\cap G^\prime$.
\end{enumerate}

\begin{remark}
The parahoric subgroup $G_{x,0}$ is a subgroup
of the stabilizer of $x$ in $G$. Hence it follows
from Property~(2) above that $\g_{x,r}$ is 
$\Ad G_{x,0}$-stable, $r\in \R$, and $G_{x,r}$
is a normal subgroup of $G_{x,0}$, $r\ge 0$.

\end{remark}

Let $\widetilde\R=\R\cup \{\, r^+\ | \ r\in \R\,\}\ \cup
\{\,\infty\,\}$. The ordering on $\R$ can be extended to
an ordering on $\widetilde\R$ by decreeing that for
all $r$ and $s\in \R$,
\begin{equation*}
\begin{split}
 r<s^+ \ \ \ \ &\Longleftrightarrow \ \ \ \ r\le s,\\
 r^+<s^+ \ \ \ &\Longleftrightarrow \ \ \ \ r<s,\\
r^+<s \ \ \ \ \ &\Longleftrightarrow \ \ \ \ r<s.
\end{split}
\end{equation*}

If $x\in \cB(\bG,F)$ and $r\in \R$, set
$\g_{x,r^+}=\bigcup_{s>r} \g_{x,s}$
and if $r\ge 0$,  $G_{x,r^+}=\bigcup_{s>r} G_{x,s}$.
Moy and Prasad defined filtration lattices in the dual 
$\g^*$ of $\g$ as follows:
$$
\g^*_{x,r}=\{\, \lambda\in \g^*\ | \ \lambda(\g_{x,(-r)^+})
\subset \gP_F\,\},\qquad x\in \cB(\bG,F),\ r\in \R,
$$
where $\gP_F$ is the maximal ideal in the ring of
integers of $F$. 
All but the third property in the above list
have obvious analogues for the filtrations
of $\g^*$.
Set $\g_{x,r^+}^*=\bigcup_{s>r}\g_{x,s}^*$.

From now on, we assume that $\bG$ splits over a tamely
ramified extension of $F$.
 Let $E$ be a finite
tamely ramified Galois extension of $F$ over
which $\bG$ splits. Let $\bT$ be a maximal
$E$-split $F$-torus in $\bG$. Then
we have filtration groups
$\bT(E)_r$, $r\in \widetilde\R$, $r\ge 0$,
and lattices ${\bfr t}(E)_r$, $r\in \widetilde\R$,
where ${\bfr t}$ is the Lie algebra of $\bT$. 
Let $\Phi=\Phi(\bG,\bT)$ be the set of roots
of $\bT$ in $\bG$. If $a\in \Phi$, let ${\bfr u}_a(E)$,
and
$\bU_a(E)$ 
be the corresponding root space and root subgroup
of ${\bfr g}(E)$ and $\bG(E)$, respectively.
If $x\in A(\bG,\bT,E)$ and $a\in \Phi$, 
then $x$ (together with the valuation $v_F$)
determines filtration lattices ${\bfr u}_a(E)_{x,r}$
in ${\bfr u}_a(E)$, 
and filtration subgroups $\bU_a(E)_{x,r}$
in $\bU_a(E)$, $r\in\widetilde\R$.
For $x\in A(\bG,\bT,E)$, and $r\in \widetilde \R$,
set ${\bfr g}(E)_{x,r}= {\bfr t}(E)_{r}
\oplus \bigoplus_{a\in \Phi} {\bfr u}_a(E)_{x,r}$
and, for $r\ge 0$, let $\bG(E)_{x,r}$ be the subgroup of $\bG(E)$
generated by $\bT(E)_r$ and the groups $\bU_a(E)_{x,r}$,
$a\in \Phi$. Then $\{\, {\bfr g}(E)_{x,r}\ | \ r\in \widetilde \R\,\}$
and $\{\, \bG(E)_{x,r}\ | \ r\in \widetilde \R,\, r\ge 0\,\}$
are the 
Moy-Prasad filtrations of ${\bfr g}(E)$ and $\bG(E)$ associated to
the point $x$. Furthermore, if 
$x\in A(\bG,\bT,F)$, then
\begin{equation*}\begin{split}
& \g_{x,r}={\bfr g}(E)_{x,r}^{\mathrm{Gal}(E/F)}={\bfr g}_{x,r}(E)\cap \g\\
& G_{x,r}=\bG(E)_{x,r}^{\mathrm{Gal}(E/F)}=\bG(E)_{x,r}\cap G,\ \ \ r>0
\end{split}
\end{equation*}
The parahoric subgroup $G_{x,0}$ might not
equal $\bG(E)_{x,0}\cap G$, though the
index of $G_{x,0}$ in $\bG(E)_{x,0}\cap G$ is
finite.

\begin{definition}\label{defconcave}
Let $\bT$, $E$ and $\Phi$ be as above. 
A function $f:\Phi\cup \{0\}
\rightarrow \widetilde\R$ is \textit{concave} if
for any nonempty finite sequence $\{a_i\}$ in $\Phi\cup \{0\}$
such that $\sum_i a_i\in \Phi\cup \{0\}$, $f\left(\sum_i a_i\right)
\le \sum_i f(a_i)$.
\end{definition}

Suppose that $f$ is a concave function on $\Phi\cup\{0\}$.
If $x\in A(\bG,\bT,E)$, let
$$
{\bfr g}(E)_{x,f}= {\bfr t}(E)_{f(0)} \oplus \bigoplus_{a\in
\Phi} {\bfr u}_a(E)_{x,f(a)},
$$
If $f$ is concave and takes nonnegative values,
let $\bG(E)_{x,f}$ be the subgroup of $\bG(E)$ generated
by $\bT(E)_{f(0)}$ and the groups $\bU_a(E)_{x,f(a)}$,
$a\in \Phi$. If $x\in A(\bG,\bT,F)$ and $f$ is $\mathrm{Gal}(E/F)$-invariant,
set
\begin{equation*}\begin{split}
& \g_{x,f}={\bfr g}(E)_{x,f}^{\mathrm{Gal}(E/F)}={\bfr g}(E)_{x,f}\cap \g, \\
& {\rm for\ }f \ {\rm nonnegative,} \ 
G_{x,f}=\bG(E)_{x,f}^{\mathrm{Gal}(E/F)}=\bG(E)_{x,f}\cap G.\end{split}
\end{equation*}
Various subgroups that appear in Yu's construction of
tame supercuspidal representations are of the form
$G_{x,f}$ for certain choices of $x$ and $f$.

\begin{definition}\label{deftwistedLeviseq}
A sequence $\vec\bG=(\bG^0,\dots,\bG^d)$ of connected
reductive $F$-groups is a \textit{twisted Levi sequence}
in $\bG$ if 
$$
\bG^0\subsetneq\bG^1\subsetneq \cdots\subsetneq 
\bG^d=\bG
$$
and there exists a finite extension $E$ of $F$ such
that $\bG^0\otimes E$ splits over $E$
and $\bG^i\otimes E$ is a Levi subgroup (that is,
an $E$-Levi $E$-subgroup) of $\bG^d\otimes E$,
for all $i\in \{\, 0,\dots,d\,\}$.
In this case, $\vec\bG$ is said to
\textit{split} over $E$. If $E$ can be chosen to be tamely
ramified over $F$, we say that $\vec\bG$ is
\textit{tamely ramified}.
\end{definition}

If $\vec\bG=(\bG^0,\dots,\bG^d)$ is a twisted Levi
sequence, let $\bZ^i$, $\bT^i$, $\bfr{z}^i$, and $\bfr{z}^{i,*}$
be the center of $\bG^i$, the identity component of
the center of $\bG^i$, the center of the Lie algebra 
$\bfr{g}^i$ of $\bG^i$, and the dual of $\bfr{z}^i$, for
$i\in \{\, 0,\dots,d\,\}$.
Set $\z^{i,*}=\bfr{z}^{i,*}(F)$, for $i\in \{\, 0,\dots,d\,\}$.
and $\bZ=\bZ^d$, $\bfr{z}=\bfr{z}^d$, $\bfr{z}^*=\bfr{z}^{d,*}$,
etc. As in Section~8 of \cite{Y}, we identify $\bfr{z}^{i,*}$
with the subspace of $\bfr{g}^{i,*}$ of elements that are
invariant under the co-adjoint action of $\bG^i$
on $\bfr{g}^{i,*}$.

Suppose that $\vec\bG$ is
a tamely ramified twisted Levi sequence that splits
over a tamely ramified finite Galois extension $E$ of
$F$.

\begin{definition}\label{defadmseq}
A sequence $\vec r=\{ r_i\}_{i=0}^d$ in $\widetilde\R$
is \textit{admissible} if there exists $v$ with
$0\le v\le d$ and 
$0\le r_0=r_1=\cdots=r_v$ and $r_v/2\le r_{v+1}
\le\cdots\le r_d$.
\end{definition}

\noindent Fix an admissible sequence
$\vec r=\{r_i\}_{i=0}^d$.
 Choose a maximal $E$-split $F$-torus $\bT\subset
\bG^0$, set $\Phi_i=\Phi(\bG^i,\bT)$,
for $i\in \{\, 0,\dots,d\,\}$, and define 
$$f_{\vec r}(a)=\begin{cases} r_0, & {\rm if}\  a\in\Phi_0\cup \{0\}\\
          r_i, & {\rm if}\  a\in\Phi_i\setminus\Phi_{i-1},\ i\in\{\, 1,\dots,d\,\}.
\end{cases}
$$
Then, as shown in \cite{Y}, $f_{\vec r}$ is a concave function on
$\Phi\cup \{0\}$.
For $x\in A(\bG,\bT,F)$, set $\vec G_{x,\vec r}=G_{x,f_{\vec r}}$
and $\vec\g_{x,\vec r}=\g_{x,f_{\vec r}}$. 
Both $\vec\bG_{x,\vec r}$ and $\vec\g_{x,\vec r}$ are independent
of the choice of $E$-split maximal $F$-torus $\bT\subset \bG^0$
such that $x\in A(\bG,\bT,F)$. 

If $\vec\bG=(\bG^0)$, $\bG^0=\bG$, and $\vec r=(r_0)$,
then
$\vec\g_{x,\vec r}=\g_{x,r_0}$, and if $r_0>0$,
$\vec G_{x,\vec r}=G_{x,r_0}$.
More generally, if $\vec\bG=(\bG^0,\dots,\bG^d)$
and $0<r_0\le r_1\cdots\le r_d$, then, as in \cite{Y},
$$
\vec G_{x,\vec r}=G_{x,r_0}^0G_{x,r_1}^1
\cdots G_{x,r_d}^d, \qquad x\in A(\bG,\bT,F),
$$
where Remark~2.11 of \cite{Y} explains how to
identify $x$ with an element of $\cB(\bG^i,F)$,
and $G_{x,r_i}^i$ is the associated
Moy-Prasad subgroup of $G^i$, for $i\in \{\, 1,\dots,d\,\}$.

\section{Quasicharacters}
\label{sec:hypotheses}

If $H$ is a totally disconnected group,
we refer to a smooth one-dimensional representation
of $H$ as a quasicharacter of $H$.
In this section, we prove some basic results
about quasicharacters that will be used
later 
in the paper
when proving our main results.

If $r\in \widetilde\R$, set 
$$\g_r=\bigcup_{x\in 
\cB(\bG,F)}\, \g_{x,r}, \qquad \g_r^*=\bigcup_{x\in 
\cB(\bG,F)}\, \g_{x,r}^*$$ and, if $r\ge 0$, $$G_r=\bigcup_{x\in \cB(\bG,F)}
G_{x,r}.$$ Then $\g_r$, $\g^*_r$, and $G_r$
are open and closed, and $\Ad \, G$, $\Ad^* G$
and $G$-invariant,
respectively. (Here, 
$\Ad^*$ is the 
representation dual to the adjoint representation
 $\Ad$.)
If $\bG^\prime$ is a twisted
Levi subgroup of $G$, the notations
$\g_r^\prime$, $\g^{\prime, *}_r$,
and $G^\prime_r$ will be used for the
analogous subsets of $\g^\prime$,
$\g^{\prime, *}$, and $G^\prime$, respectively.

\begin{lemma}\label{depthlemma}\label{qc1} Let $\phi$
be a quasicharacter of $G$ and suppose $x$, $y\in \cB(\bG,F)$ and $r\in \widetilde \R$  with $r>0$. Then
$\phi\,|\, G_{x,r}\cap \, G_{y,r}= 1$ if and only if $\phi\, |\, G_r=1$.
\end{lemma}

\begin{proof} Clearly, $\phi\, |\, G_r=1$ implies $\phi\,|\, G_{x,r}\cap \, G_{y,r}=1$.  
So we will assume $\phi\,|\, G_{x,r}\cap \, G_{y,r}=1$ and address the proof of the 
converse assertion.

We begin by recalling some facts about unipotent elements from \cite{D}.  Let $\mathcal{U}$ denote the set of unipotent elements in $G$.  Then it is a standard fact that $\mathcal{U}$ is contained in the derived group of $G$.  (This fact is obvious if one formulates the definition of ``unipotent'' as in \cite{D}.)  Consequently, all quasicharacters of $G$ must be trivial on $\mathcal{U}$.  We also observe that Theorem 4.1.4 of \cite{D} implies that $G_{x,r}\subset {\mathcal U}G_{y,r}$
and $G_{y,r}\subset {\mathcal U}G_{x,r}$.  Thus the conditions $\phi\, |\, G_{x,r}=1$ and $\phi\, |\, G_{y,r}=1$ are equivalent and, moreover, they are both equivalent to the condition $\phi\, |\, G_r=1$.

Now choose a maximal
$F$-split torus ${\bf S}\subset \bG$
such that $x$, $y\in A(\bG,{\bf S},F)$.
Let ${\bf M}$ be the centralizer
of ${\bf S}$ in $\bG$. Let $Z_M$ be
the center of $M=\bM (F)$.
Then $M/Z_M$ is compact. Therefore,
since
$x$, $y\in A(\bG,{\bf S},F)=\cB(\bM ,F)$,
we have $M_{x,t}=M_{y,t}$ for all $t\ge 0$.

Let ${\bf P}$ be a (minimal) $F$-parabolic subgroup
of $\bG$ with Levi decomposition ${\bf P}=\bM \bN$.
Let $\bar\bN$ be the unipotent radical
of the parabolic subgroup of $\bG$ that is
opposite to ${\bf P}$. Applying Theorem~4.2 of
\cite{MP2}, we have an Iwahori
decomposition of $G_{x,t}$ for $t>0$,
$$
G_{x,t}=(\bar N\cap G_{x,t})M_{x,t}(N\cap G_{x,t}),
$$
with an analogous decomposition for $G_{y,t}$.
Now suppose
that $\phi$ is as in the statement of the
lemma. Then $\phi\,|\, M_{x,r}=
\phi\,|\, M_{y,r}=1$. Because
$\bar N$ and $N$ consist of unipotent
elements, $\phi$ must be trivial on both of these subgroups. In view
of the Iwahori decompositions of
$G_{x,r}$ and $G_{y,r}$, we now conclude
that $\phi$ is trivial on $G_{x,r}$ and $G_{y,r}$.
Hence, $\phi\, |\, G_r=1$.
\end{proof}

In \cite{MP1} and \cite{MP2}, Moy and Prasad
introduced the notion of depth of an
irreducible admissible representation of $G$.
For quasicharacters of $G$, ``depth'' may be defined as follows:

\begin{definition}\label{qchardepth}
If $\phi$ is a quasicharacter of $G$ then the {\it depth of $\phi$}, denoted $r(\phi)$, is the smallest nonnegative real number $r$ that satisfies any of the following equivalent conditions:
\begin{itemize}
\item $\phi \,| \, G_{x,r^+}=1$, for some $x\in \cB (\bG,F)$,
\item $\phi \,|\, G_{x,r^+}=1$, for all $x\in \cB (\bG,F)$,
\item $\phi \,| \, G_{x,r^+}\cap G_{y,r^+}=1$, for some $x,y\in \cB(\bG,F)$,
\item $\phi \, |\, G_{r^+}=1$.
\end{itemize}
\end{definition}

Suppose that $t$, $u\in \widetilde\R$
and $t<u$. If $x\in \cB(\bG,F)$,
set $\g_{x,t:u}=\g_{x,t}/\g_{x,u}$
and, if $t\ge 0$, set $G_{x,t:u}=
G_{x,t}/G_{x,u}$.
Now suppose that $r\in \R$, $r>0$ and
$s=r/2$.  We have an isomorphism 
$$e=e_{x,r}:\g_{x,s^+:r^+}\rightarrow G_{x,s^+:r^+}$$
of abelian groups that is a special case of the inverse of the isomorphism in 
 Corollary 2.4 of \cite{Y}.    (Though Yu does not explicitly state the definition of his isomorphism, the definition is implicit in the proof of Lemma 1.3 \cite{Y} and the remarks following the proof.  An explicit definition in the case at hand is given in Section 1.5 of \cite{A} using Adler's mock exponential maps.  To extend the definition to Yu's more general setting, one uses 6.4.48 of \cite{BT1}.)

\begin{lemma}\label{expequivariance}
Let $\theta$ be an involution of $G$, $x\in \cB(\bG,F)$, $r>0$
and $s=r/2$.
Suppose that $X\in \g_{x,s^+}\cap\theta(\g_{x,s^+})$.
Then there exists $k\in G_{x,s^+}\cap\theta(G_{x,s^+})$
such that $e(X+\g_{x,r^+})=k\,G_{x,r^+}$ and $e(\theta(X)+\g_{x,r^+})
=\theta(k)\,G_{x,r^+}$.
\end{lemma}

\begin{proof} 
A straightforward generalization of the proof 
of Proposition 1.6.7 \cite{A} with the 
automorphism $\Int(g)$ in Adler's proof replaced by
 the involution $\theta$ gives an analogue of
Adler's result relative to $\theta$. 
The lemma then follows from the fact that Adler's map
induces the canonical map $e_{x,r}$ 
at the level of cosets in $G_{x,s^+:r^+}$.
\end{proof}

The isomorphism $e$ restricts to
an isomorphism (that we will
also denote by $e=e_{x,r}$)
between $\g_{x,r:r^+}$ and $G_{x,r:r^+}$.
Another useful restriction of $e$ is given as follows.
Suppose that $r$ and $s$ are as above,
$\bG^\prime$ is a tamely ramified
twisted Levi subgroup of $\bG$, and
$x\in \cB(\bG^\prime,F)$.
Then $e$ induces an isomorphism
between $\g^\prime_{x,s^+:r^+}$
and $G^\prime_{x,s^+:r^+}$.
Let $(G^\prime,G)_{x,(r,s^+)}$
and $(\g^\prime,\g)_{x,(r,s^+)}$
be the subgroup of $G$
and the lattice in $\g$ (respectively)
associated to
the tamely ramified twisted Levi
sequence $\vec\bG=(\bG^\prime,\bG)$
and the admissible sequence $(r,s^+)$
(as defined in Section~\ref{sec:buildings}).
Then $e$ restricts to an 
isomorphism between $(\g^\prime,\g)_{x,(r,s^+)}
/\g_{x,r^+}$ and $(G^\prime,G)_{x,(r,s^+)}/G_{x,r^+}$,
and this restriction in turn
induces the same isomorphism between 
$\g^\prime_{x,r:r^+}$ and $G^\prime_{x,r:r^+}$
that is obtained as the restriction of
the above isomorphism
between $\g^\prime_{x,s^+:r^+}$
and $G^\prime_{x,s^+:r^+}$.

Fix a character $\psi$ of $F$ that
is nontrivial on the ring of 
integers $\gO_F$ of $F$
and trivial on the maximal
ideal $\gP_F$ of $\gO_F$.

\begin{definition}\label{defrealized}
If $r>0$, $s=r/2$ and $x\in \cB(\bG,F)$
and $S$ is a subgroup of $G_{x,s^+}$
that contains $G_{x,r}$, 
let ${\mathfrak s}$ be the lattice
in $\g_{x,s^+}$ such that
${\mathfrak s}\supset \g_{x,r}$
and
$e({\mathfrak s}/\g_{x,r^+})= S/G_{x,r^+}$.
An element $X^*\in \g_{x,-r}^*$
defines a character of $S$ that is
trivial on $G_{x,r^+}$ as follows:
$$
e(Y+\g_{x,r^+})\mapsto \psi(X^*(Y)),
\qquad Y\subset {\mathfrak s}.
$$
This character of $S$ is said to be
\textit{realized}
 by the element
$X^*$ of $\g_{x,-r}^*$, or
by the coset $X^*+ {\mathfrak s}^\bullet$,
where $${\mathfrak s}^\bullet=\{\, Y^*\in
\g_{x,-r}^*\ | \ Y^*({\mathfrak s})\subset \gP_F\,\}.$$
\end{definition}
\smallskip

\newtheorem*{hypC}{Hypothesis C($\bG$)}\label{defHypC}

\begin{hypC}
Let $\phi$ be a quasicharacter of $G$ of positive
depth. 
If $r=r(\phi)$
and $x\in \cB(\bG,F)$,
then  $\phi\,|\, G_{x,(r/2)^+}$
is realized by an element of $\z^{*}_{-r}$,
where $\z^{*}$ is the dual of the
center of $\g$.
\end{hypC}

\begin{remark}\label{genhyp}  We will
often need to assume that the hypothesis is satisfied
by all of the subgroups $\bG^i$ that occur
in a twisted Levi sequence  $\vec\bG=
(\bG^0,\dots,\bG^d)$ in $\bG$.
We will say that
Hypothesis~C($\vec\bG$) is satisfied
whenever Hypothesis~C($\bG^i$) is
satisfied for each group $\bG^i$ in the
sequence.
It is clear that Hypothesis~C($\vec\bG$)
holds if and only if Hypothesis~C($\mathrm{Int}\, g(\vec\bG)$)
holds for all $g\in G$.
\end{remark}

\begin{lemma}\label{hypgl} Let $n$ be an integer
such that $n\ge 2$. Then Hypothesis~\rm{C}($\bGL_n$)
is satisfied.
\end{lemma}

\begin{proof} Let $\bG=\bGL_n$.
Let $\phi$ be a quasicharacter of $G$ of 
positive depth. Let $r$ be the depth
of $\phi$. Then $G_{x,r}\not=G_{x,r^+}$
for all $x\in \cB(\bG,F)$. If $x\in \cB(\bG,F)$
is a hyperspecial point, then $G_{x,t}= G_{x,t^+}$
whenever $t>0$ is not an integer. Hence
$r$ is a positive integer.
Let $\det:G\rightarrow F^\times$ be the determinant
map. There exists a quasicharacter $\chi$ of $F^\times$
such that $\phi=\chi\circ\det$. 
It is easy to
verify that
if $m$ is a nonnegative integer,
then $\det(G_{m^+})=\det(G_{t})=\det(G_{m+1})
=1+\gP_F^{m+1}$ for all $t\in \widetilde{\R}$
such that $m^+\le t\le m+1$.
As $\phi\,|\, G_r\not=1$ and
$\phi\,|\, G_{r^+}=1$, we must
have $\chi\,|\, 1+\gP_F^r\not=1$
and $\chi\,|\, 1+\gP_F^{r+1}=1$.
That is, $\chi$ has depth $r$.

Let $\varpi$ be a prime element in $F$.
Since we are assuming that $p$ is odd,
the extension $L=F(\sqrt\varpi)$ of
$F$ is tamely ramified. Let $m$ be an
integer. Because $\sqrt\varpi$ is a prime
element in $L$, $\bfr{g}(L)_{x,(m+(1/2))^+}
=\varpi^m\sqrt\varpi\,{\bfr g}(L)_{x,0^+}$
for  $x\in \cB(\bG,L)$. 
The product of two matrices in $\bfr{g}(L)_{x,0^+}$ 
also lies in $\bfr{g}(L)_{x,0^+}$.
It follows that if $X$, $Y\in \bfr{g}(L)_{x,(m+(1/2))^+}$,
then $XY\in \bfr{g}(L)_{x,(2m+1)^+}$. 
If $x\in \cB(\bG,F)$ and $X$, $Y\in \g_{x,(m+(1/2))^+}$,
then, since $L$ is
tamely ramified over $F$, we have
$XY\in \bfr{g}_{x,(2m+1)^+}(L)\cap \g
=\g_{x,(2m+1)^+}$. This fact will be used
when
the depth $r$ of
$\phi$ is an odd integer.
When $r$ is even, we will use the
fact that if $m$ is an integer,
 $x\in \cB(\bG,F)$ and $X$, $Y\in \g_{x,m^+}=\varpi^m\g_{x,0^+}$,
then $XY\in \varpi^{2m}\g_{x,0^+}= \g_{x,(2m)^+}$.

Let $x\in \cB(\bG,F)$. As shown above, the matrix product
of two elements of $\g_{x,(r/2)^+}$ lies in $\g_{x,r^+}$.
This can be used to show that the isomorphism
$e=e_{x,r}:\g_{x,(r/2)^+:r^+}\rightarrow G_{x,(r/2)^+:r^+}$
satisfies $e(X+\g_{x,r^+})=(1+X)\,G_{x,r^+}$,
for all $X\in \g_{x,(r/2)^+}$.
Hence there exists $X^*\in \g_{x,-r}^\ast$ such that
$\chi(\det(1+X))=\psi(X^*(X))$, $X\in \g_{x,(r/2)^+}$.
To prove the lemma, we must show that $X^*$ can be chosen to lie
in $\z_{-r}^*$. Note that $\z^*$ is the
set of elements in $\g^*$ that are defined
by $X\mapsto \alpha\, \tr(X)$ for some
$\alpha\in F$, and the elements of $\z^*_{-r}$
are those for which $\alpha\in \gP_F^{-r}$.

Next, we show that if $x\in \cB(\bG,F)$,
then
$\det(1+X)\in (1+\tr(X))(1+\gP_F^{r+1})$ for 
$X\in \g_{x,(r/2)^+}$.
Let $X\in \g_{x,(r/2)^+}$.
Then $X^2
\in \g_{x,r^+}=\varpi^r\g_{x,0^+}$.
Let $E$ be an extension of $F$ that
contains the (not necessarily distinct)
 eigenvalues
$\lambda_1,\dots,\lambda_n$ of $X$.
Because $\varpi^{-r}X^2\in \g_{x,0^+}$,
each eigenvalue $\varpi^{-r}\lambda_j^2$
of $\varpi^{-r}X^2$ must lie in $\gP_E$.
It follows that, if $e$ is the ramification
degree of $E$ over $F$, then $\lambda_j\in \gP_E^{1+(er)/2}$
if $er$ is even, and $\lambda_j\in \gP_E^{(er+1)/2}$
if $er$ is odd, $j\in \{\,0,\cdots,n\,\}$.
Thus the product of two or more eigenvalues
of $X$ lies in $\varpi^r\gP_E$. Hence
\begin{equation*}
\det(1+X)=\prod_{j=1}^n (1+\lambda_j)
\in 1+\tr(X) + \varpi^{r}\gP_E=(1+\tr(X))(1+\varpi^r\gP_E).
\end{equation*}
Since $\det(1+X)$, $\tr(X)\in F$ and $\gP_E\cap F=\gP_F$,
we have $\det(1+X)\in (1+\tr(X))(1+\gP_F^{r+1})$.

Set $\ell=(r/2) + 1$ if $r$ is even and set $\ell=(r+1)/2$
if $r$ is odd. Because $\alpha+\gP_F^{r+1}
\mapsto (1+\alpha)(1+\gP_F^{r+1})$ defines an isomorphism
from $\gP_F^\ell/\gP_F^{r+1}$ to $(1+\gP_F^\ell)/(1+\gP_F^{r+1})$,
a character of the latter group has the form
$(1+\alpha)(1+\gP_F^{r+1})\mapsto \psi(\beta\alpha)$ for 
some fixed $\beta \in \gP_F^{-r}$. 

Recall that $\chi\,|\, 1+\gP_F^{r+1}=1$.
Let $\beta\in \gP_F^{-r}$ be an element that
realizes the restriction $\chi\,|\, 1+\gP_F^\ell$.
If $x\in \cB(\bG,F)$ and $X\in \g_{x,(r/2)^+}$, then
\begin{equation*}
\phi(1+X)=\chi(\det(1+X))=\chi(1+\tr(X))=\psi(\beta\,\tr(X)).
\end{equation*}
That is, the element $X\mapsto \beta\,\tr(X)$ of $\z^*_{-r}$
realizes $\phi\,|\, G_{x,(r/2)^+}$. 
\end{proof}

\begin{lemma} \label{centralrep}
Let $\phi$ be a quasicharacter of
$G$ of depth $r>0$ such that
there exist $x_j\in \cB(\bG,F)$
and $\Gamma_j\in \z_{-r}^*$
such that $\phi\,|\, G_{x_j,r}$
is realized by the coset 
$\Gamma_j+\g_{x_j,(-r)^+}^*$,
$j=1$, $2$. Then
$\Gamma_1-\Gamma_2\in \z_{(-r)^+}^*$.
\end{lemma}

\begin{proof}
Because $\phi\,|\, G_{x_j,r}$ is an
unrefined minimal $K$-type of $\phi$,
$j=1$, $2$, associativity properties of
unrefined minimal $K$-types (\cite{MP1,MP2})
imply that
$$
\Ad^* G\,(\Gamma_1+\g_{x_1,(-r)^+}^*)
\cap (\Gamma_2 + \g_{x_2,(-r)^+}^*)
\not=\emptyset.
$$
Since $\Ad^* g(\Gamma_1)=\Gamma_1$ for
all $g\in G$, we have
$$
\g_{(-r)^+}^*\cap (\Gamma_2-\Gamma_1 + \g_{x_2,(-r)^+}^*)
\not=\emptyset.
$$
Note that
$\Gamma_2-\Gamma_1\in \z_{-r}^*\subset \g_{x_2,-r}^*$.
If $x\in \cB(\bG,F)$ and $t\in \R$,
then any coset in $\g_{x,t}^*/\g_{x,t^+}^*$
that intersects $\g_{t^+}^*$ must lie
inside $\g_{t^+}^*$. 
Hence $\Gamma_2-\Gamma_1\in \z^*\cap \g_{(-r)^+}^*
=\z_{(-r)^+}^*$.
\end{proof}

\begin{lemma} \label{depth} 
Let $\bG'$
be a tamely ramified twisted Levi subgroup
of $\bG$ and let $\phi$ be a quasicharacter
of $G$ of positive depth $r>0$. Assume that
there exists $x\in \cB(\bG^\prime,F)$ such that $\phi\,|\, G_{x,r}$
is realized by a coset $\Gamma + \g_{x,(-r)^+}^*$
with $\Gamma\in \z_{-r}^*$. Then
$r(\phi\, |\, G') = r(\phi)$.
\end{lemma}

\begin{proof} Note that $\Gamma\notin \z_{(-r)^+}^*$, because
$\phi\,|\, G_{x,r}$ is nontrivial. Now $\phi\,|\, G_{x,r}^\prime$
is realized by the coset $\Gamma+ \g_{x,(-r)^+}^{\prime, *}$.
Since $G_{x,r^+}^\prime=G_{x,r^+}\cap G^\prime$, we have
$\phi\,|\, G_{x,r^+}^\prime=1$. If $\phi\,|\, G_{x,r}^\prime
=1$, then $\Gamma\in \g_{x,(-r)^+}^{\prime, *}\subset
\g_{(-r)^+}^{\prime, *}$. But, as noted above,
$\Gamma\in \z_{-r}^*$ and $\Gamma\notin\z_{(-r)^+}^*$. Thus we have
$\Gamma\notin \g_{(-r)^+}^{\prime, *}$, which implies that
$\phi\,|\, G_{x,r}^\prime\not=1$.
\end{proof}

\chapter{Yu's construction of tame supercuspidal representations}

\section{Cuspidal $G$-data}
\label{sec:notations}

This section is largely a compendium of notations and definitions
from \cite{Y},
with some modifications. In particular, we state the
definitions of ``cuspidal $G$-datum,''
``generic element'' and ``generic quasicharacter,'' and we
also describe certain open subgroups of $G$ that occur in
the construction of tame supercuspidal representations.
The reader may find it necessary to refer back to Section~\ref{sec:buildings} 
for some of the notation
used here.
It is important to note the slight differences between our definitions
and notations and those in \cite{Y}. (See, for example,
Remarks~\ref{reminddata} and \ref{misstatement}.)
Otherwise, the reader is 
encouraged to advance to the next section as quickly as possible.
Yu's construction will be discussed in detail in 
Section~\ref{sec:construction}.

Yu's construction begins with triples of the form $(\vec\bG,\pi_{-1},\vec\phi)$, 
where $\vec\bG=(\bG^0,\dots,\bG^d)$ 
is a tamely
ramified twisted Levi sequence in $\bG$ (see Section~\ref{sec:buildings}
for the definition), $\pi_{-1}$ is an irreducible supercuspidal representation
of $G^0=\bG^0(F)$ of depth zero, and $\vec\phi=(\phi_0,\dots,\phi_d)$
is a sequence of quasicharacters (of $G^0,\dots,G^d$, respectively). 
The main result of \cite{Y} shows that
if the triple $(\vec\bG,\pi_{-1},\vec\phi)$ satisfies certain conditions,
then a sequence $\vec\pi=(\pi_0,\dots,\pi_d)$
of irreducible supercuspidal representations of $G^0,\dots,G^d$, respectively,
can be constructed from the triple.
(The reason we use the notation $\pi_{-1}$, rather than
Yu's notation $\pi_0$, is explained in Remark \ref{misstatement}
below.)

In fact, the actual construction starts 
with  5-tuples
$(\vec\bG, y,\vec r,\rho,\vec\phi)$ (as defined below).
As explained in \cite{Y},
any such $5$-tuple determines a triple, although a triple
can be determined by different $5$-tuples, due to the
fact that there are usually several possible 
choices for $y$ and $\rho$.
If two $5$-tuples determine the same triple,
then the $i$th representations in the two sequences
of supercuspidal representations are equivalent
representations of $G^i$,
for all $i\in \{\, 0,\dots,d\,\}$.

A $5$-tuple $(\vec\bG, y, \vec r,\rho,\vec\phi)$ satisfying
the following conditions will be called a \textit{cuspidal
$G$-datum}:\label{Dconditions}

\begin{itemize}
\item[\textbf{D1.}] $\vec{\bG}$ is a tamely ramified twisted Levi sequence $\vec{\bG}
= (\bG^0, \ldots, \bG^d)$ in $\bG$  and $\bZ^0/\bZ$
is $F$-anisotropic, where $\bZ^0$ and $\bZ$ are the centers of 
$\bG^0$ and $\bG=\bG^d$, respectively.

\item[\textbf{D2.}] $y$ is a point in $A(\bG,\bT,F)$, where $\bT$ is
a tame maximal $F$-torus of $\bG^0$ and $E$ is a Galois tamely
ramified extension of $F$ over which $\bT$ (hence $\vec\bG$)
splits.
(Recall from Section~\ref{sec:buildings}
that $A(\bG,\bT,E)$ denotes the apartment in $\cB(\bG,E)$
corresponding to $\bT$ and $A(\bG,\bT,F)=A(\bG,\bT,E)\cap \cB(\bG,F)$.)

\item[\textbf{D3.}] $\vec{r} = (r_0, \ldots, r_d)$ is a sequence
of real numbers satisfying
$0 < r_0 < r_1 < \ldots < r_{d-1} \leq r_d$, if $d > 0$, and $0 \leq r_0$
if $d = 0$.

\item[\textbf{D4.}] $\rho$ is an irreducible representation of the stabilizer
$K^0 = G^0_{[y]}$ of $[y]$ in $G^0$ such that
$\rho\,|\, {G^0_{y,0^+}}$ is $1$-isotypic and
the compactly induced representation $\pi_{-1} = \ind_{K^0}^{G^0} \rho$
is irreducible (hence supercuspidal). Here, [y] denotes the
image of y in the reduced building of $G$.

\item[\textbf{D5.}] $\vec\phi = (\phi_0,\dots ,\phi_d)$ is a 
sequence of quasicharacters, where $\phi_i$ is a quasicharacter of $G^i$.  
We assume that $\phi_d =1$ if $r_d=r_{d-1}$ (with $r_{-1}$ defined to be 0),
 and in all other cases if $i\in \{\, 0,\dots,d\,\}$ 
then $\phi_i$ is trivial on $G^i_{y,r_i^+}$ 
but nontrivial on $G^i_{y,r_i}$.
\end{itemize}

\begin{remark}\label{reminddata}
Conditions \textbf{D1}--\textbf{D4} are identical to the corresponding
 conditions in \cite{Y}, 
except that what we call $\pi_{-1}$ is called $\pi_0$ in \cite{Y}.  
It is unclear from the statement of Yu's Condition \textbf{D5} how the condition 
should be interpreted when $d=0$. This is why we have modified his statement.
Note that Yu often suppresses the subscript $y$, and, for example, writes
$G_{r}$ in place of $G_{y,r}$. Since it has lately become the
convention to reserve the notation $G_r$ for $\bigcup_{x\in \cB(\bG,F)} G_{x,r}$,
we do not suppress the subscript $y$ in our notation.
\end{remark}

\begin{remark}\label{misstatement}
The reason for our use of the notation $\pi_{-1}$ involves a misstatement in \cite{Y}
 which is easily fixed.  Looking in \cite{Y} at the statements of Conditions 
\textbf{D3} and \textbf{D4} as well as Remark 3.6, one finds that:
\begin{enumerate}
\item $\pi_0 = \ind_{K^0}^{G^0}(\rho)$,\item for general $i$, the number $r_i$
 is the depth, in the sense of
Moy and G.\ Prasad (\cite{MP1}), of the representation $\pi_i$,
\item when $d=0$, the constant $r_0$ is restricted by the condition $0\le r_0$.
\end{enumerate}

\noindent Conditions (1) and (2) imply that $r_0=0$.  Of course, this is 
consistent with Condition (3), but there is a good reason why Yu requires 
$r_0\ge 0$ in his Condition \textbf{D3}.  It seems that Yu intended to define 
$$\pi_0 = \ind_{K^0}^{G^0}(\rho \otimes (\phi_0\,|\,K^0)),$$ since then if $r_0$ 
is the depth of $\pi_0$ we have the possibility of positive values of $r_0$.  
More importantly, the latter definition of $\pi_0$ is convenient in inductive arguments.  
By letting $\pi_{-1} = \ind_{K^0}^{G^0}(\rho)$ and $r_{-1}=0$, the number $r_i$ coincides 
with the depth of $\pi_i$, even when $i=-1$. 
\end{remark}

\begin{remark}\label{reduced} 
As in Section~\ref{sec:buildings}, $\bZ^i$ denotes 
the center of $\bG^i$, for $i\in \{\, 0,\dots,d\,\}$,
and $\bZ=\bZ^d$. Note that the Condition \textbf{D1}
guarantees that $X_*(\bZ^i,F)\otimes \R=X_*(\bZ,F)\otimes \R$,
for $i\in \{\, 0,\dots,d\,\}$
and $\bZ=\bZ^d$. As mentioned in Section~\ref{sec:buildings},
we may regard $\cB(\bG^i,F)$ as a subset of $\cB(\bG,F)$.
Hence, viewing the reduced building $\cB_{\text{red}}(\bG,F)$ as the 
set of $(X_*(\bZ,F)\otimes \R)$-orbits
in $\cB(\bG,F)$, we see that we can embed  the
reduced building $\cB_{\text{red}}(\bG^i,F)$ in $\cB_{\text{red}}(\bG,F)$,
thus identifying the images of $y$ in $\cB_{\text{red}}(\bG^i,F)$
and  $\cB_{\text{red}}(\bG,F)$. (Note that we do not have 
analogous embeddings for the reduced buildings over $E$.)
\end{remark}

\begin{remark}\label{vertex} 
The point $[y]$, viewed as a point in $\cB_{\text{red}}(\bG^0,F)$ is a vertex, 
according to Proposition~6.8 of \cite{MP2}, and thus $G^0_{y,0}$ is a maximal parahoric 
subgroup of $G^0$.  The group $G^0_{[y]}$ is the normalizer of  $G^0_{y,0}$ 
in $G^0$, according to Lemma 3.3 (i) in \cite{Y}. Note also that $G^0_{y,0}$ has 
finite index in the isotropy group $G^0_y$ of $y$.
\end{remark}

As indicated in Section~\ref{sec:hypotheses},
it follows from results of DeBacker \cite{D} that
if $i\in \{\,0,\dots,d-1\,\}$, or if $i=d$ and
$\phi_d$ is nontrivial, then
$r_i$ is the depth
of $\phi_i$, in the sense of Moy and Prasad.
When $\phi_d$ is trivial, $r_d=r_{d-1}$
is the depth of $\phi_{d-1}$.
Consequently the vector $\vec r$ which appears in the
$5$-tuple $(\vec\bG,y,\vec r,\rho, \vec\phi)$ is redundant, 
since it is completely
 determined by $\vec\phi$.
For this reason, henceforth we will suppress
the notation $\vec r$, and work with $4$-tuples
$(\vec\bG,y,\rho,\vec\phi)$.

\begin{definition}\label{definddat}
A 4-tuple $\Psi=(\vec\bG,y,\rho,\vec{\phi})$ is an \textit{extended cuspidal $G$-datum}
 if it satisfies Conditions \textbf{D1}--\textbf{D5}.
Given an extended cuspidal $G$-datum $(\vec\bG,y,\rho,\vec\phi)$, if $\pi_{-1}$ is as in
Condition \textbf{D4}, the triple 
$(\vec\bG,\pi_{-1},\vec\phi)$ is called a \textit{reduced cuspidal $G$-datum}.  
The terminology \textit{cuspidal $G$-datum} is used to ambiguously refer to either 
an extended or reduced cuspidal $G$-datum.  The number $d$ is called the \textit{degree} 
of the $G$-datum.
\end{definition}

In \cite{Y}, a cuspidal datum is simply referred to as a datum.
However, if we drop the condition that $\bZ^0/\bZ$ be $F$-anisotropic
and make some modifications to Condition \textbf{D5}, we can
define more general data that can be used to construct
certain irreducible representations of open compact modulo
center subgroups of $G$. These representations 
are expected to play a role in parametrizing
nonsupercuspidal admissible representations of $G$.
The cuspidal $G$-data are precisely those
those $G$-data which
give rise to supercuspidal representations.

Yu gives some additional conditions on cuspidal
$G$-data that are sufficient for the data to
yield supercuspidal representations.
Before stating these conditions,
we define certain subgroups of $G$ that are mentioned in
those conditions and that are used in the actual construction.

Fix an extended cuspidal $G$-datum $\Psi=(\vec\bG,y,\rho,\vec\phi)$.
Set $s_i=r_i/2$, for $i\in \{\, 0,\dots,d-1\,\}$.
Let $\bT$ and $E$ be as in \textbf{D2}.
Define
\begin{equation*}
\begin{split}
K^0 &= G_{[y]}^0,\ \ \ \ K^0_+=G_{y,0^+}^0,\\
K^{i+1} &= K^0G^1_{y,s_0}\cdots G^{i+1}_{y,s_i},\ \ \ \ i\in \{\, 0,\dots,d-1\,\},\\
K^{i+1}_+ &= K^0_+G^1_{y,s_0^+}\cdots G^{i+1}_{y,s_i^+},
\ \ \ \ i\in \{\, 0,\dots,d-1\,\}.
\end{split}
\end{equation*}
Note that if $i\in \{\,0,\dots,d-1\,\}$, then $\vec\bG^{(i+1)}=(\bG^0,\dots,
\bG^{i+1})$ is a tamely ramified twisted Levi sequence
that splits over $E$, $(0^+,s_0,\dots,s_{i})$
and $(0^+,s_0^+,\dots,s_{i}^+)$ are admissible sequences
(as defined in Section~\ref{sec:buildings}),
and
\begin{equation*}
K^{i+1}=K^0\vec G^{(i+1)}_{y,(0^+,s_0,\dots,s_{i})}
\ \ \ \ \text{and}\ \ \ \ K^{i+1}_+=\vec G^{(i+1)}_{y,(0^+,s_0^+,
\dots,s_{i}^+)}.
\end{equation*}
Let $K=K^d$ and $K_+=K^d_+$. If we wish to emphasize
the dependence on $\Psi$, we write $K=K(\Psi)$
and $K_+=K_+(\Psi)$. 

Let $\Phi_i=\Phi(\bG^i,\bT)$ be the roots of $\bT$ in
$\bG^i$, for $i\in \{\, 0,\dots,d\,\}$. If $i\in \{\, 0,\dots,d-1\,\}$,
the group $J^{i+1}(E)$ is defined to be the
compact open subgroup of $\bG^{i+1}(E)$ generated by 
$\bT(E)_{r_i}$ and the subgroups $\bGU_a(E)_{y,r_i}$, with 
$a\in \Phi_i$,
 and
$\bGU_a(E)_{y,s_i}$, with $a\in
\Phi_{i+1}\setminus\Phi_i$.  
The
group $J^{i+1}_+(E)$ is defined similarly, with
$s_i$ replaced by $s_i^+$.  Let $J^{i+1}=J^{i+1}(E)\cap G^{i+1}$
and $J_+^{i+1}=J^{i+1}_+(E)\cap G^{i+1}$.
The pair $(\bG^i,\bG^{i+1})$ is a tamely ramified Levi
sequence that splits over $E$ and $(r_i,s_i)$ and
$(r_i,s_i^+)$ are admissible sequences. We have
\begin{equation*}
J^{i+1}=(G^i,G^{i+1})_{y,(r_i,s_i)}\ \ \ \ \text{and}
\ \ \ \ J^{i+1}_+=(G^i,G^{i+1})_{y,(r_i,s_i^+)}.
\end{equation*}
Because $G_{y,s_i}^iJ^{i+1}=G^{i+1}_{y,s_i}$
and $G_{y,s_i^+}^iJ_+^{i+1}=G^{i+1}_{y,s_i^+}$,
we have
\begin{equation*}
K^{i+1}= K^iJ^{i+1}= K^iG^{i+1}_{y,s_i}\ \ \ \ 
\text{and}\ \ \ \ 
K^{i+1}_+=K^i_+ J^{i+1}_+ = K^i_+ G^{i+1}_{y,s_i^+},
\ \ \ i\in \{\, 0,\dots,d-1\,\}.
\end{equation*}

In Proposition~4.6 of \cite{Y}, Yu shows that if $\Psi$ is such that a
set of three conditions \textbf{SC1}${}_i$--\textbf{SC3}${}_i$
is satisfied for all $i$ in $\{\, 0,\dots,d-1\,\}$, then $\Psi$ gives
rise to a sequence $\vec\pi=(\pi_0,\dots,\pi_d)$,
where $\pi_i$ is an irreducible supercuspidal representation
of $G^i$ that is compactly induced from a smooth representation
of $K^i$, for $i\in \{\, 0,\dots,d\,\}$.
Fix $i\in \{0,\dots ,d-1\}$. Let $G_{y,s_i^+:r_i^+}^i=G_{y,s_i^+}^i/G_{y,r_i^+}$
and $G_{y,s_i^+:r_i^+}=G_{y,s_i^+}/G_{y,r_i^+}$.
Because $\phi_i$ is trivial on $G_{y,r_i^+}^i$,
the restriction $\phi_i\,|\, G_{y,s_i^+}^i$ factors to
a character of $G_{y,s_i^+:r_i^+}^i$.
 In the beginning of Section~ 4 of \cite{Y},
Yu describes  a natural inflation process\label{inflationlabel}
which we use to define
$$\inf\nolimits^{G_{y,s_i^+}}_{G^i_{y,s_i^+}}(
\phi_i )=
\inf\nolimits_{G^i_{y,s_i^+:r_i^+}}^{
G_{y,s_i^+:r_i^+}}(\phi_i\,|\, G^i_{y,s_i^+}).$$  
Next, take $\hat\phi_i$ to be the quasicharacter
of $K^0G^i_{y,0}G_{y,s_i^+}$ that agrees with
$\phi_i$ on $K^0G^i_{y,0}$ and agrees with
$\inf\nolimits^{G_{y,s_i^+}}_{G^i_{y,s_i^+}}(
\phi_i )$ on $G_{y,s_i^+}$. The above inflation process
is defined in such a way that the restriction
$\hat\phi_i\,|\, (G^i,G)_{y,(r_i^+,s_i^+)}$ is
trivial. Since $G_{y,s_i^+}=G_{y,s_i^+}^i(G^i,G)_{y,(r_i^+,
s_i^+)}$, we see that $\hat\phi_i$ may also be
described as the
quasicharacter of $K^0G_{y,0}^iG_{y,s_i^+}$
that agrees with $\phi_i$ on $K^0 G_{y,0}^i$
and is trivial on $(G^i,G)_{y,(r_i^+,s_i^+)}$.

The following conditions on $\phi_i$  are equivalent to the corresponding 
conditions in \cite{Y} though they are stated slightly differently:

\begin{itemize}\label{defSCcond}
\item[\textbf{SC1}${}_i$.] If $g\in G^{i+1}$ and $\hat\phi_i (g^{-1} jg) = 
\hat\phi_i (j)$  for all $j\in gJ^{i+1}_+ g^{-1} \cap J^{i+1}_+$ (in other words,
 $g$ intertwines $\hat\phi_i\, |\,J^{i+1}_+$) then $g\in J^{i+1}G^i J^{i+1}$.

\item[\textbf{SC2}${}_i$.] There is an irreducible representation $\tilde\phi_i$ of 
$K^i\ltimes J^{i+1}$ such that (i) the restriction of $\tilde\phi_i$ to $J^{i+1}_+ =
 1\ltimes J^{i+1}_+$ is $(\hat\phi_i\, |\, J^{i+1}_+)$-isotypic; and (ii) the restriction 
of $\tilde\phi_i$ to $K^i_+\ltimes 1$ is 1-isotypic.

\item[\textbf{SC3}${}_i$.] Given $\tilde\phi_i$ as in \textbf{SC2}${}_i$, we define a 
representation $\phi'_i$ of $K^{i+1}$ on the same space as $\tilde\phi_i$ by 
$$
\phi'_i (kj)= \phi_i (k) \tilde\phi_i (k,j),
 \quad k\in K^i,\quad j\in J^{i+1}.
$$
Let $V_i$ denote the space of $\phi'_i$ and let $\tau_i$ denote the restriction 
of $\phi'_i$ to $J^{i+1}$.    
 For all $g\in G^0$, there is a unique up to scalar multiples nonzero linear 
endomorphism $\Lambda  : V_i\to V_i$ such that 
 $$
\Lambda (\tau_i (g^{-1} jg) \varphi)
 = \tau_i (j) \Lambda (\varphi ),
$$ 
for all $j\in gJ^{i+1} g^{-1} \cap J^{i+1}$.  In addition, 
$\Lambda$ has the property that
$$
\Lambda (\tilde\phi_i (g^{-1} kg,1 ) \varphi) = \tilde\phi_i (k,1 ) 
\Lambda (\varphi ),
$$ 
for all $k\in gK^i g^{-1} \cap K^i$.
\end{itemize}

Yu defined a notion of genericity for quasicharacters
(see Definition~\ref{defgenchar})
and proved that if $\Psi$ is a cuspidal $G$-datum
having the property that for all $i\in \{\,0,\dots,d-1\,\}$,
 $\phi_i\,|\, G_{y,r_i}^i$
 is $G^{i+1}$-generic, then 
Conditions \textbf{SC1}${}_i$--\textbf{SC3}${}_i$ 
are satisfied for all $i\in \{\,0,\dots,d-1\,\}$.
The genericity conditions, as well as some parts of the
construction of the inducing data for supercuspidal
representations, are relative to Levi sequences
of the form $(\bG^i,\bG^{i+1})$. This amounts
to specializing to the case $d=1$. As in \cite{Y},
we use the notation $(\bG^\prime,\bG)$ in the case $d=1$.
We refer to this as the $(\bG^\prime,\bG)$
case. 
When working in this setting $\bT$ denotes a tame maximal
$F$-torus of $\bG^\prime$,  $E$ is a Galois tamely ramified
finite extension of $F$ over which $\bT$ (hence $(\bG^\prime,\bG)$)
splits, and $y\in A(\bG,\bT,F)$.
We use the notation $\bZ^\prime$, $\bfr{z}^\prime$,
$\bfr{z}^{\prime, *}$, $\z^{\prime}$ and $\z^{\prime, *}$
for the center of $\bG^\prime$,
the center of the Lie algebra $\bfr{g}^\prime$,
the dual of $\bfr{z}^\prime$, $\bfr{z}^\prime(F)$,
and $\bfr{z}^{\prime, *}(F)$, respectively.
Since we are working with cuspidal $G$-data, $\bZ^\prime/\bZ$ is $F$-anisotropic.

Recall from Section~\ref{sec:buildings} that we have fixed a valuation $v_F$ on $F$,
and we also denote its extension to any algebraic field extension of $F$ by
$v_F$.
The following conditions apply to an element $X^*\in \z^{\prime,*}_{-r}$, $r\in \R$,
 and  are used in the definition of genericity as it is stated in the 
$(\bG^\prime ,\bG)$ case: 

\begin{itemize}\label{GEconditions}
\item[\textbf{GE1.}] ${v_F}(X^*(H_a))=-r$, for all $a\in 
\Phi (\bG,\bT)\setminus\Phi (\bG^\prime,\bT)$, where $H_a = d\check a(1)$,
$\check a$ is the coroot associated to $a$.
\item[\textbf{GE2.}] Suppose $\varpi_r$ is an element of the algebraic
closure $\overline{F}$ of $F$ of valuation $r$ 
and $\widetilde{X}^*$ is the residue class of $\varpi_r X^*$ in the residue 
field of $\overline{F}$.  Then the isotropy subgroup of $\widetilde{X}^*$ 
in the Weyl group of $\Phi (\bG,\bT)$ coincides with the Weyl group 
of $\Phi (\bG^\prime,\bT)$.
\end{itemize}

\begin{remark}\label{gemotivation} For information on the action of
the Weyl group of $\Phi(\bG,\bT)$ on elements of the form
$\widetilde{X}^*$, the reader may refer to Section~8 of \cite{Y}.
Condition \textbf{GE2} is especially technical.  Fortunately, it can usually 
be ignored since, according to Lemma 8.1 in \cite{Y}, it is almost always 
implied by \textbf{GE1}.
  To appreciate its application, the reader should examine the proof of Lemma 8.3 
in \cite{Y}.  We also note that the condition does not depend on the choice of $\varpi_r$.
\end{remark}

\begin{definition}\label{defgenelt}
An element $X^*\in \z^{\prime,*}_{-r}$ is \textit{$G$-generic of depth $-r$} if it satisfies 
Conditions  \textbf{GE1} and \textbf{GE2}.
\end{definition}

\begin{remark}\label{geneltrem}
In the previous definition, our notion of depth on the Lie algebra dual is the 
opposite of Yu's.  In other words, what is depth $-r$ for us corresponds to depth 
$r$ in \cite{Y}.  Our convention appears to be more standard in the literature.
\end{remark}

As in Section~\ref{sec:hypotheses}, fix a character $\psi$ of $F$ 
that is nontrivial on the ring of integers $\gO_F$ of
$F$
and trivial on the maximal ideal $\gP_F$ of $\gO_F$.
If $y\in A(\bG,\bT,F)$ and $r>0$, the restriction
to $G^\prime_{y,r}$ of a quasicharacter $\phi$ of $G^\prime$ of depth $r$
is realized by an element $X^*\in \z^{\prime,*}_{-r}$ if
$$\phi(e(Y +\g^\prime_{y,r^+}))=\psi (X^*(Y)),
\qquad Y\in \g^\prime_{y,r}.
$$
Here $e$ is the isomorphism between
$\g_{y,r:r^+}^\prime$ and $G_{y,r:r^+}^\prime$
discussed in Section \ref{sec:hypotheses}.
We remind the reader that we are using colons, as in \cite{Y}, 
to abbreviate quotients. For example,
$G^\prime_{y,r:r^+}$ and $\g^\prime_{y,r:r^+}$ are shorthands for $G^\prime_{y,r}
/G^\prime_{y,r^+}$ and $\g^\prime_{y,r}/\g^\prime_{y,r^+}$, respectively.
Also, wherever it is convenient, we view $\z^{\prime, *}$ as the set of
$\Ad G^\prime$-fixed elements of $\g^{\prime, *}$.

\begin{definition}\label{defgenchar} Let $r\in \R$, $r>0$.
A quasicharacter $\phi$ of $G^\prime$ is said to be \textit{ $G$-generic (relative to $y$) 
of depth $r$} if $\phi$ is trivial on $G^\prime_{y,r^+}$, and nontrivial on 
$G^\prime_{y,r}$,
and there exists a $G$-generic element $X^*\in \z^{\prime,*}_{-r}$ of depth $r$ 
that realizes the restriction of $\phi$ to $G^\prime_{y,r}$.
\end{definition}

\begin{remark}\label{gencharrem}
In Remark 9.1 of \cite{Y}, it is observed that the notion of $G$-genericity for 
a quasicharacter $\phi$ of $G^\prime$ often does not depend on the choice of the point $y$.  
For example, let $\bG^\prime_{\text{der}}$ denote the derived group of $\bG^\prime$ and suppose 
$G^\prime_{y,r}= (Z^\prime)^\circ_r \bG^\prime_{\text{der}}(F)_{y,r}$.  
In this case, if $\phi$ is trivial on 
$\bG^\prime_{\text{der}}(F)$ (which is not necessarily the same as the derived group of 
$G^\prime$) 
then the notion of $G^\prime$-genericity for $\phi$ (with depth $r$) is independent of $y$.
\end{remark}

\begin{definition}\label{defgendatum}
If $\Psi= (\vec\bG, y,\rho, \vec\phi)$ satisfies Conditions \textbf{D1}--\textbf{D5}
(that is, $\Psi$ is a cuspidal $G$-datum), and if $\phi_i$ is $G^{i+1}$-generic of depth 
$r_i$ relative to $y$ for all $i\in \{ 0,\dots ,d-1\}$, then the $G$-datum $\Psi$ is 
said to be \textit{generic}.  In this case, the reduced $G$-datum $(\vec\bG, \pi_{-1},\vec\phi)$
is called \textit{generic}.
\end{definition}

It is shown in \cite{Y} that if $\Psi$ is a generic cuspidal $G$-datum,
then Conditions \textbf{SC1}${}_i$--\textbf{SC3}${}_i$ are satisfied for
all $i\in \{\, 0,\dots,d-1\,\}$. (See the beginning of Section~15 in \cite{Y}.)
The details of the construction of the
inducing data for the representations $\pi_0,\dots,\pi_d$
will be discussed in Section~\ref{sec:construction}.

\section{Compatibility with involutions}
\label{sec:compatinv}

In our applications to distinguished representations, we are usually 
provided with some involution $\theta$ of $G$, in the sense of \ref{definv}.  To 
have a working theory, the first step is to show that we can assume we are dealing 
with objects, such as cuspidal $G$-data, that are compatible with $\theta$ 
in a suitable sense.  This section gives some indication of what 
``compatibility with $\theta$'' means for certain objects involved in Yu's construction.
 
If $x\in \cB(\bG,F)$ and $\theta$ is an involution of $G$,
let $\theta([x])$ be the image of $\theta(x)$
in the reduced building $\cB_{\text{red}}(\bG,F)$. Note that this is well defined
because $\theta(\bZ)=\bZ$.
If $\vec\bG=(\bG^0,\dots,\bG^d)$ is a twisted Levi sequence in $\bG$,
then $\theta(\vec\bG)=(\theta(\bG^0),\dots,\theta(\bG^d))$
is also a twisted Levi sequence in $\bG$, and is tamely ramified
if and only if $\vec\bG$ is tamely ramified.

The following elementary fact is essentially Remark 3.5 in \cite{Y}.

\begin{lemma}\label{threefive}
Suppose $\vec\bG = (\bG^0,\dots ,\bG^d)$ is a tamely ramified twisted Levi 
sequence in $\bG$ satisfying Condition {\bf D1} and suppose 
$\vec t = (t_0,\dots ,t_d)$ is an admissible sequence.  If $y,y' \in \cB (\bG^0,F)$ 
and $[y]=[y']$ then $\vec G_{y,\vec t}= \vec G_{y',\vec t}$.  
In particular, if $\theta$ is an involution of $G$ such that $\theta (\bG^i)= \bG^i$,
 for all $i$, and $\theta ([y]) = [y]$ then the group $\vec G_{y,\vec t}$ is $\theta$-stable. 
\end{lemma}

\begin{proof} Our claim follows directly from Remark 3.5 in \cite{Y} and the fact 
that $\theta (\vec G_{y,\vec t}) = \theta(\vec G)_{\theta(y),\vec t}$.
\end{proof}

\begin{definition}\label{defweaksymdatum}
If $\theta$ is an involution of $G$ and   $\Psi= (\vec\bG,y,\rho ,\vec\phi)$ is a cuspidal $G$-datum then $\Psi$ is 
\textit{weakly $\theta$-symmetric} if $\theta (\vec\bG)= \vec\bG$,
 and $\phi_i\circ\theta = \phi_i^{-1}$, for all $i\in \{ 0,\dots ,d\}$,
that is, each $\phi_i$ is $\theta$-symmetric.  
If $\Psi$ is weakly $\theta$-symmetric and $\theta([y])=[y]$ then we say $\Psi$ is 
\textit{$\theta$-symmetric}.
\end{definition}

Note that in the above definition there are no conditions imposed on
the representation $\rho$.

\begin{proposition}\label{KJfactor} If $\Psi = (\vec\bG,y,\rho,\vec\phi)$ is a 
$\theta$-symmetric cuspidal $G$-datum then all of the subgroups 
of the forms $K^i$, $J^i$, $K^i_+$, $J^i_+$ and $G^i_{y,t}$ are 
$\theta$-stable,  and we have the relations 
\begin{equation*}
\begin{split}
K^{i+1,\theta} &= K^{i,\theta}J^{i+1,\theta}= 
K^{i,\theta} G^{i+1,\theta}_{y,s_i}\cr
&= K^{0,\theta}J^{1,\theta}\cdots J^{i+1,\theta}\cr
&= G^{0,\theta}_{[y]}G^{1,\theta}_{y,s_0}\cdots G^{i+1,\theta}_{y,s_i}\cr
K^{i+1,\theta}_+&= K^{i,\theta}_+ J^{i+1,\theta}_+ = K^{i,\theta}_+ G^{i+1,\theta}_{y,s_i^+}\cr
&= K^{0,\theta}_+ J^{1,\theta}_+\cdots J^{i+1,\theta}_+\cr
&= G^{0,\theta}_{y,0^+}G^{1,\theta}_{y,s_0^+}\cdots G^{i+1,\theta}_{y,s_i^+}.
\end{split}
\end{equation*}
\end{proposition}

\begin{proof} This follows from Lemmas \ref{alphafactor} and \ref{threefive} 
and Proposition \ref{twodivprop}.
\end{proof}

\section{Heisenberg $p$-groups associated to generic characters}
\label{sec:genHeis}

Our objective in this section is to apply the theory from
Section~\ref{sec:Heis} to the context that is relevant
for Yu's construction.
In this context, the Heisenberg groups are realized as quotients 
of compact open subgroups of $G$.  As in \cite{Y},  we have certain canonical special 
isomorphisms.  Involutions that stabilize the relevant compact
open subgroups of $G$ also give rise to canonical special isomorphisms 
which are especially convenient for our purposes.  
One of the main results of this section, Proposition \ref{stJpol}, says 
that the two types of special isomorphisms turn out to be identical.  
(Earlier results concerning distinguishedness that appeared in the papers 
\cite{HM2} and \cite{HM3} were obtained
without this fact, which could now be used to simplify
some parts of the proofs.)

We will work in the $(\bG^\prime,\bG)$ setting described in
Section~\ref{sec:notations}, and we will continue to
use the notation defined there. In particular, 
$\bT$ is a tame maximal $F$-torus in $\bG^\prime$, $E$ is a finite
Galois tamely ramified extension of $F$ over which $\bG^\prime$
splits, and $y$ is a fixed element of $A(\bG^\prime,\bT,F)$.
Let $\phi$ be a character of $G^\prime_{y,r:r^+}$ that is
$G$-generic (relative to $y$) of depth $r>0$. Choose a
$G$-generic element $X^*\in \z^{\prime, *}_{-r}$ that
realizes $\phi\,|\, G_{y,r}^\prime$ in the sense of 
Section~\ref{sec:hypotheses}.
If $e:\bfr{g}^\prime(E)_{y,r:r^+}\rightarrow \bG^\prime(E)_{y,r:r^+}$
is an isomorphism as in Section~\ref{sec:hypotheses} (except
here we are working over the $E$-rational points), then,
if $\psi^E$ is a character of $E$ that coincides with
$\psi$ on $F$, 
$$
\phi^E(e(Y+\bfr{g}^\prime(E)_{y,r^+}))=\psi^E(X^*(Y)),
\qquad Y\in \bfr{g}^\prime(E)_{y,r},
$$
defines a $\bG(E)$-generic character of $\bG^\prime(E)_{y,r:r^+}$
that agrees with $\phi$ on $G^\prime_{y,r:r^+}$, and which
we also view as a character of $\bG^\prime(E)_{y,r}$.

Let
$s = r/2$. Let $J(E)=(\bG^\prime,\bG)(E)_{y,(r,s)}$
and $J_+(E)=(\bG^\prime,\bG)(E)_{y,(r,s^+)}$.
These are the subgroups of $\bG(E)$ associated
to the tamely ramified twisted Levi sequence
$(\bG^\prime,\bG)$ and to the admissible
sequences $(r,s)$ and $(r,s^+)$ respectively.
The group $J(E)$ is the compact open
subgroup of $\bG(E)$ generated by $\bG^\prime(E)_{y,r}$ and 
the subgroups 
$\bGU_a(E)_{y,s}$, with $a\in \Phi\setminus\Phi^\prime$,
where $\Phi=\Phi(\bG,\bT)$ and $\Phi' = \Phi (\bG^\prime,\bT)$.
The group $J_+(E)$ is the normal subgroup of
$J(E)$ generated by $\bG^\prime_{y,r}(E)$
and the subgroups $\bGU_a(E)_{y,s^+}$, $a\in \Phi\setminus
\Phi^\prime$.

Let $\zeta^E$ be the character of $J_+(E)$ that agrees with 
$\phi^E$ on $\bG^\prime(E)_{y,r}$ and is trivial on all of the groups 
$\bGU_a(E)_{y,s^+}$ with $a\in \Phi\setminus \Phi'$.  
Then, since $J_+(E)=\bG^\prime(E)_{y,r}(\bG^\prime,\bG)(E)_{y,(r^+,s^+)}$
we see that $\zeta^E$ is determined by the property that $\zeta^E$
agrees with $\phi^E$ on $\bG^\prime(E)_{y,r}$ and is trivial
on $(\bG^\prime,\bG)(E)_{y,(r^+,s^+)}$.
Let $N(E)$ be the kernel of $\zeta^E$.
  Then $\cH(E) = J(E)/N(E)$ is a Heisenberg $p$-group with center $\cZ(E) = J_+(E)/N(E)$. 
We let $W(E) = \cH(E)/\cZ(E) = J(E)/J_+(E)$.  We also regard $\zeta^E$ as a character 
of $\cZ(E)$.  The pairing defined by 
$$
\langle u,v\rangle = \zeta^E ([u,v]) = \zeta^E(uvu^{-1}v^{-1}),\qquad u,v\in J(E),
$$
factors to a symplectic form on $W(E)$.

Next, we turn to Yu's Proposition 11.4 and its proof to construct a canonical 
split polarization on $\cH(E)$.  Choose an ordering on $\Phi$ and define $J(E)(+)$ 
to be the subgroup of $J(E)$ generated by $\bGU_a(E)_{y,s}$ for all positive roots 
$a\in \Phi\setminus\Phi'$.  
Define $J(E)(-)$ similarly using negative roots instead of positive roots.  
The group $\cH (E)(+)= J(E)(+)N(E) /N(E)$ maps injectively into $W(E)=J(E)/J_+(E)$ 
and we let $W(E)(+)$ denote the image of $\cH(E)(+)$ in $W(E)$.  
The spaces $\cH (E)(-)$ and $W (E)(-)$ are defined similarly.   Then Yu shows 
that $(\cH (E)(+), \cH (E)(-))$ is a split polarization of $\cH (E)$.   
Let $\nu_E^\bullet : \cH(E)\to W(E)^\sharp$ be the special isomorphism 
associated to this split polarization and the character 
$\zeta^E$ by Lemma \ref{sppolarspeciso}.

Now let
\begin{equation*}
\begin{split}
J&= J(E)\cap G\cr 
J_+&= J_+(E)\cap G\cr 
\zeta&=\zeta^E\,|\, J_+\cr
N&=N(E)\cap G= \ker \zeta\cr
\cH&= J/N\cr
\cZ&= J_+/N= \cZ(E)\cr
W&=J/J_+= \cH/\cZ.
\end{split}
\end{equation*}
We remark that it follows immediately from
properties of $\zeta^E$ that the character $\zeta$ of $J_+$
is determined by the two properties
$\zeta\,|\, G^\prime_{y,r}=\phi$ and $\zeta\,|\,
(G^\prime,G)_{y,(r^+,s^+)}=1$.
The symplectic form on $W(E)$ restricts to a symplectic form on $W$.  
We observe that $\cH$ is a Heisenberg $p$-subgroup of $\cH(E)$ and,
according to Lemma \ref{specisores}, the special isomorphism $\nu_E^\bullet$ 
restricts to a special isomorphism $\nu^\bullet : \cH \to W^\sharp$.
\
\begin{definition}\label{defYuspecialiso} The special isomorphism $\nu^\bullet$ will be 
referred to as \textit{Yu's special isomorphism}.
\end{definition}

\begin{remark}\label{rsiYurem}
Yu's special isomorphism $\nu^\bullet$  does not depend on the choice of 
tamely ramified maximal torus $\bT$ in $\bG^\prime$,
the splitting field $E$, or the ordering of the 
root system $\Phi (\bG,\bT)$, as long as $\bT$ is chosen so that 
$y\in A(\bG,\bT,F)$.  (See Proposition 11.4 \cite{Y}.)
\end{remark}

We are only interested in special isomorphisms that have an additional
  property which we now explain.  Let $[y]$ be the point in the reduced building of 
$G^\prime$ corresponding to $y$ and let $K^\prime$ denote the stabilizer $G^\prime_{[y]}$ 
of $[y]$ in $G^\prime$.  
Note that the fact that $Z^\prime/Z$ is compact guarantees that $K^\prime$
normalizes $J$ and $J_+$.
Let ${\rm Sp}(\cH)$ be the group of automorphisms of $\cH$ that restrict to the 
identity map on $\cZ$.
Let $f: K^\prime\to {\rm Sp}(\cH)$ be the map which comes from the action of $K^\prime$ on 
$J$ by conjugation.  

Recall from Section~\ref{sec:Heis} that $W^\sharp=W\boxtimes \cZ$.
Let $\alpha : {\rm Sp}(W) \to {\rm Sp}(W^\sharp)$ be the usual map 
$\alpha (\gamma) (w,z) = (\gamma w,z)$.

\begin{definition}\label{defrel}
A  special isomorphism $\nu$ on $\cH$ is \textit{relevant} if the mapping 
$K^\prime\to {\rm Sp}(W^\sharp): k\mapsto \nu\circ f(k)\circ \nu^{-1}$ has image in ${\rm Sp}(W)$.
\end{definition}

Now let $f' :K^\prime\to {\rm Sp}(W)$ be the map which comes from the action of $K^\prime$
 on $J$ by conjugation.  In other words, $f'(k) (h\cZ) = (f(k)(h))\cZ$, for all $k\in 
K^\prime$ and $h\in\cH$.

\begin{lemma}\label{relspeciso}
If $\nu$ is a relevant special isomorphism on $\cH$ then 
$$\nu\circ f(k)\circ \nu^{-1} (w,z) = (f'(k)w,z),$$ for all $k\in K^\prime$, 
$w\in W$ and $z\in \cZ$.  In particular, $\nu\circ f(k)\circ \nu^{-1}$ is independent 
of the choice of $\nu$.
\end{lemma}

\begin{proof}
For $k\in K^\prime$, let $\xi (k)$ be the map from $W^\sharp$ onto itself given by $\xi (k) = \nu\circ f(k)\circ \nu^{-1}$.  Since $\nu$ is relevant, $\xi(k)$ lies in ${\rm Sp}(W)$.  Fix $(w,z)\in W^\sharp$ and let $h= \nu^{-1} (w,z)$.  Note that $h\cZ = w$.  Now $\xi(k) (w,z) = \nu (f(k)h) = (f(k)h \cZ, z')$ for some $z'\in \cZ$.  However, since $\xi (k)$ lies in ${\rm Sp}(W)$, it must be the case that $z' = z$.  Our claim follows.
\end{proof}

\begin{lemma}\label{rsihom}
Suppose $\nu_1$ and $\nu_2$ are relevant special isomorphisms on $\cH$ and let $\chi : W\to \cZ$ be the homomorphism defined by $\nu_2(h) = \chi(h\cZ) \nu_1(h)$ for all $h\in \cH$.  Then $$\chi(f'(k)w ) = \chi(w),$$ for all $k\in K^\prime$ and $w\in W$.
\end{lemma}

\begin{proof}Fix $(w,z)\in W^\sharp$ and let $h_i= \nu_i^{-1} (w,z)$.  Note that $h_i\cZ = w$ and  $f'(k) (w,z) = \nu_i (f(k)h_i)$.
We have 
\begin{equation*}
\begin{split}
f'(k) (w,z)&= \nu_2(f(k)h_2) = \chi (f'(k)w)  \nu_1(f(k) h_2)\cr  &= \chi (f'(k) w)\ \nu_1(f(k) h_1)(0,z')\cr &= f'(k)(w,z) \chi (f'(k)w) (0,z'),
\end{split}
\end{equation*}
where $z'= h_1^{-1}h_2$.  Hence, $\chi(f'(k)w) (0,z')=1$.  Since $z'$ is independent of $k$, our assertion follows.
\end{proof}

\begin{remark}\label{rsisymp}
  For each special isomorphism $\nu$ on $\cH$, there is an isomorphism $\beta_\nu : {\rm Sp}(W^\sharp) \to {\rm Sp}(\cH)$ given by $\beta_\nu (\gamma) = 
\nu^{-1}\circ \gamma \circ \nu$.  If we take $f_\nu = \beta_\nu \circ \alpha 
\circ f'$, then $(f_\nu,\nu)$ is a symplectic action in the sense of \cite{Y}.
  To say that $\nu$ is relevant is equivalent to saying $f= f_\nu$.  
It is also the same as saying that $(f,\nu)$ is a symplectic action.
\end{remark}

Now fix a Heisenberg representation $\tau^\sharp$ of $W^\sharp$ with central character $\zeta$. 
 (Of course, the equivalence class of $\tau^\sharp$ is determined by $\zeta$).  
Let $\hat\tau^\sharp$ be the Heisenberg-Weil lift of $\tau^\sharp$ to a representation 
of ${\rm Sp}(W)\ltimes W^\sharp$ with the same representation space as that of $\tau^\sharp$.
Assume $\nu$ is  a relevant special isomorphism.  Then the following map is well defined 
and it is a homomorphism:
\begin{equation*}
\begin{split}
K^\prime\ltimes \cH &\to  {\rm Sp}(W)\ltimes W^\sharp\cr
(k,h)&\mapsto (f'(k) ,\nu(h)).
\end{split}
\end{equation*}
Pulling back $\hat\tau^\sharp$ via this homomorphism gives the representation $\omega^\nu$ of 
$K^\prime\ltimes \cH$ which Yu refers to as the \textit{Weil representation of
 $K^\prime\ltimes \cH$}.  \label{YuWeil}
Note that
$$\omega^\nu (k,h) = \hat\tau^\sharp (f'(k))\ \tau^\nu (h),$$ for all $k\in K^\prime$
 and $h\in \cH$, where $\tau^\nu (h) = \tau^\sharp (\nu(h))$.
Now define a representation $\phi'^\nu$ of $K^\prime J$ by 
$$\phi'^\nu(kj) =  \phi (k)\  \omega^\nu (k,jN),$$ for all $k\in K^\prime$ and $j\in J$.

\begin{lemma}\label{rsiequiv}
If $\nu_1$ and $\nu_2$ are relevant special isomorphisms on $\cH$ then  $\omega^{\nu_1} \simeq\omega^{\nu_2}$ and thus $\phi'^{\nu_1}\simeq\phi'^{\nu_2}$.
\end{lemma}

\begin{proof}
It suffices to show that $\omega^{\nu_1}$ and $\omega^{\nu_2}$ have the same character.
  Let $\chi :W\to \cZ$ be such that $\nu_2(h)= \chi(h\cZ) \nu_1(h)$, for all $h\in \cH$.  
Then we have 
\begin{equation*}
\begin{split}
\omega^{\nu_2}(k,h)&= \hat\tau^\sharp (f'(k)) \tau^{\nu_2}(h)= \zeta (\chi(h\cZ)) \ \hat\tau^\sharp (f'(k))\  \tau^{\nu_1}(h)\cr
&= \zeta (\chi(h\cZ))\ \omega^{\nu_1}(k,h).
\end{split}
\end{equation*}
By properties of characters of Weil representations, 
if the element\break $(f'(k) ,\nu_1 (h))$ is not conjugate
in ${\rm Sp} (W)\ltimes W^\sharp$ to an element of ${\rm Sp}(W)\times \cZ$,
then the characters of $\omega^{\nu_1}$ and $\omega^{\nu_2}$ both vanish on the element
$(k,h)$.

We suppose, without loss of generality, that there exist $s_1\in 
{\rm Sp}(W)$ and $w_1\in W$, such that 
$$
(s_1,w_1)(f'(k),\nu_1 (h))(s_1,w_1)^{-1}\in {\rm Sp}(W)\times \cZ.
$$
Write $w_1=\nu_1 (h_1)$, $h_1\in \cH$. Then
$$
(s_1f^\prime(k)s_1^{-1}, s_1\cdot (f^\prime(k)^{-1}(\nu_1 (h_1))\nu_1 (h)
\nu_1 (h_1)^{-1})\in {\rm Sp}(W)\times \cZ.
$$
Equivalently, as ${\rm Sp}(W)$ acts via the identity on $\cZ$,
$$
f'(k)^{-1}(\nu_1 (h_1))\nu_1(h)\nu_1 (h_1)^{-1}=\nu_1 (
f'( k)^{-1}(h_1)hh_1^{-1})\in \cZ.
$$
Hence $\chi ( f'( k)^{-1}(h_1)hh_1^{-1}\cZ)=1$,
or, using Lemma~\ref{rsihom}  ,
$$
1=\chi(f'(k)^{-1}(h_1)\cZ)\chi(h\cZ)\chi(h_1\cZ)^{-1}
=\chi(h\cZ).
$$
Therefore $\omega^{\nu_2} (k,h)=\zeta(\chi (h\cZ))\omega^{\nu_1} (k,h)
=\omega^{\nu_1} (k,h)$
for all pairs $(k,h)$ such that the conjugacy class of
$(f'(k),\nu_1 (h))$ in ${\rm Sp}(W)\ltimes W^\sharp$ intersects
${\rm Sp}(W)\times \cZ$.
\end{proof}

The following result is contained in Yu's Proposition 11.4.

\begin{lemma}\label{rsiYu}
Yu's special isomorphism $\nu^\bullet$ is relevant.
\end{lemma}

In the special case in which $\nu$ is Yu's special isomorphism $\nu^\bullet$, 
we use the notations $\tau^\bullet$, $\omega^\bullet$ and $\phi'^\bullet$ 
for $\tau^\nu$, $\omega^\nu$ and $\phi'^\nu$, respectively.  When $\nu$ 
is an arbitrary relevant special isomorphism, we let $\chi^\nu : W\to \cZ$ be 
the homomorphism such that $\nu(h)= \chi^\nu(h\cZ)\nu^\bullet (h)$, for all 
$h\in \cH$.  Then we have the relations
\begin{equation*}
\begin{split}
\tau^\nu (h) &= \zeta (\chi^\nu (h\cZ))\  \tau^\bullet (h) \cr
\omega^\nu(k,h)&=\zeta (\chi^\nu(h\cZ))\ \omega^\bullet (k,h),
\end{split}
\end{equation*}
for all $k\in K^\prime$ and $h\in \cH$.
The first relation says that $\tau^\nu$ is a twist of $\tau^\bullet$ 
by the character $\zeta \circ\chi^\nu$.  Note that twisting a Heisenberg 
representation of $\cH$ by a character of $W$ gives another Heisenberg 
representation with the same central character; hence it gives an 
equivalent representation.

Our next objective is to show that, in a certain sense, Yu's special 
isomorphism $\nu^\bullet$ is compatible with certain involutions of $G$.
Fix an involution $\theta$ of $G$ (in the sense of Definition \ref{definv}).
In order to apply the results of the previous
section, we assume, first of all,
that $J$ and $N$ are $\theta$-stable. 
These conditions ensure that $\theta$ reduces to
an automorphism $\alpha$ of the Heisenberg
group $\cH$. Since $\alpha$ is an
automorphism of $\cH$ it must preserve the
center $\cZ$.  In other words, $J_+$ must
be $\theta$-stable. Additionally, we assume
that
$\alpha$ is nontrivial on $\cZ$ in
order to get nonzero $\cH^\alpha$-invariant 
linear forms.  The
next lemma reformulates and clarifies these
conditions.  

\begin{lemma}\label{stableJ}
If $J$ is
$\theta$-stable then the following are equivalent:
\begin{enumerate}
\item $N$ is $\theta$-stable and
$\alpha$ is nontrivial on $\cZ$,
\item $J_+$ is $\theta$-stable and
$\zeta\circ\theta = \zeta^{-1}$.
\end{enumerate}
 If $J$, $N$ and $G^\prime_{y,r}$ are $\theta$-stable then the condition  
$\zeta\circ\theta =\zeta^{-1}$ is equivalent to  $\phi\circ\theta =
\phi^{-1}$ on $G^\prime_{y,r}$.
\end{lemma}

\begin{proof}
 Assume $J$ is
$\theta$-stable.   Suppose first
that condition
(1) holds.  Then $\alpha (z)= z^{-1}$, for all
$z\in \cZ$ according to Lemma \ref{ordtwoauto}.  We
therefore obtain $\zeta\circ \alpha =
\zeta^{-1}$ and  hence $\zeta\circ\theta =
\zeta^{-1}$.

Now suppose condition (2) holds.  Then clearly
$N$ must be $\theta$-stable.  Suppose that
$\alpha\,|\,\cZ$ is the identity map.  Then
$\zeta = \zeta^{-1}$.  But since $\cZ$ has
odd order, it does not admit any characters of
order two.  This implies $\zeta$ must be
trivial, which is absurd.  Therefore, it must be
the case that $\alpha\, |\,\cZ$ is nontrivial.

Assume that $J$, $N$ and $G_{y,r}^\prime$ are
$\theta$-stable.
The asserted equivalence follows from the properties
$\zeta\,|\, G_{y,r}^\prime=\phi\,|\, G_{y,r}^\prime$,
$J_+=G_{y,r}^\prime G_{y,(r^+,s^+)}$, $G_{y,(r^+,s^+)}\subset N$
and $\theta(N)=N$. 
\end{proof}
 
 The next result summarizes various results we have obtained so far.  It shows that the special isomorphisms associated to those involutions that stabilize $\bG^\prime$ 
and $[y]$  and map $\zeta$ to $\zeta^{-1}$ are, in fact, identical to Yu's special isomorphism.
 This provides us with concrete realizations of $\nu^\bullet$ whose compatibility with such involutions is useful in determining certain signs that arise
in our computations involving distinguished representations.
 
\begin{proposition}\label{stJpol}
 Assume $\theta (\bG^\prime) = \bG^\prime$, $\theta ([y]) = [y]$.
 Then $\theta$ stabilizes $J$ and $J_+$. Assume also that $\zeta\circ\theta=\zeta^{-1}$.
Then $\theta$ gives rise to an automorphism of $\cH$ that we also denote by $\theta$.  The sets
\begin{equation*}
 \begin{split}
 \cH_\theta^+&= \{\, h\in \cH\ |\  \theta
(h) = h\,\},\cr
\widehat{\cH}_\theta^- &= \{\, h\in \cH\ | \ \theta
(h) = h^{-1} \,\},
 \end{split}
\end{equation*}
 form a polarization of $\cH$ that splits canonically to give a split polarization $(\cH_\theta^+,\cH_\theta^-)$ via 
 Lemma \ref{polsplitting} using Yu's special isomorphism $\nu^\bullet$.  
 Let $\nu^\theta$ denote the special isomorphism associated to this split polarization by Lemma \ref{sppolarspeciso}.  Then $\nu^\theta=\nu^\bullet$.
\end{proposition}

\begin{proof}
The first several assertions follow directly from Theorem \ref{Heisthm}, as well as Lemmas \ref{sppolarspeciso}, \ref{polsplitting} and \ref{stableJ}.  
It remains to prove that $\nu^\theta = \nu^\bullet$.  Let  $$W^\bullet = (\nu^{\bullet})^{-1}(W\times 1).$$  We claim that
 $W^\bullet$ must be $\theta$-stable.  Before proving this, we will show how it implies what we need.
Abbreviate the groups $\cH^+_\theta$ and $\widehat{\cH}^-_\theta$ in the statement above as $\cH^+$ and $\widehat{\cH}^-$ and let
\begin{equation*}
 \begin{split}
W^+&= \{\, w\in W\ | \theta (w)= w\,\}\cr
W^{-}&= \{\, w\in W\ | \ \theta (w) = w^{-1}\,\}\cr
\cH^-&= \widehat{\cH}^-\cap W^\bullet.
 \end{split}
\end{equation*}
We observe that 
\begin{equation*}
 \begin{split}
\cH^+\cZ &= \{\, h\in \cH\ | \ \theta (h)\in h\cZ\,\}\cr
\widehat{\cH}^- &= \cH^-\cZ= \{\, h\in \cH \ | \ \theta (h) \in h^{-1}\cZ\,\}
 \end{split}
\end{equation*}
and $\nu^{-1} (W^+)\subset \cH^+\cZ$ for every special isomorphism $\nu$.

We now show that $(\nu^{\bullet})^{-1}(W^+)= \cH^+$.  First of all, note that the sizes of the two sets are the same.  Now suppose $h\in (\nu^{\bullet})^{-1}(W^+)$.   Since, by assumption, $W^\bullet$ is $\theta$-stable $\theta (h)$ must  lie in $W^\bullet$.  On the other hand, if two elements, such as $h$
 and $\theta (h)$ lie in $W^\bullet$ and are congruent mod $\cZ$ then they 
must be equal.  
So $(\nu^{\bullet})^{-1}(W^+)\subset \cH^+$ and since both sets have the same size they must be equal.

It follows now that if $w_+\in \cH^+$ then $\nu^\bullet (w_+) = (w_+\cZ,1)=\nu^\theta (w_+)$.
Similarly, one shows that $(\nu^{\bullet})^{-1} (W^-) = \cH^-$ and thus 
$\nu^\bullet$ and $\nu^\theta$ agree on $\cH^-$.
Consequently,  we have established that $\nu^\bullet$ and $\nu^\theta$ are identical if it is indeed true that $W^\bullet$ is $\theta$-stable.

Recall that $\nu^\bullet$ was defined by restricting a special isomorphism $\nu^\bullet_E$ over the splitting field $E$.  To show that $W^\bullet$ is $\theta$-stable, it suffices to show that $(\nu^{\bullet}_E)^{-1} (W(E)\times 1)$ is $\theta$-stable.  Equivalently, it suffices to show $W^\bullet$ is $\theta$-stable when $E=F$.  So let us now assume $E=F$.  Fix an ordering of $\Phi = \Phi (\bG,\bT,F)$ and let
 $(\cH(+),\cH(-))$ be the split polarization of $\cH$ associated to the given ordering of $\Phi$ in the construction of $\nu^\bullet$.  Thus $\cH(+)$ (respectively, $\cH(-)$) is generated by certain subgroups of the root subgroups $U_a$ associated to positive (respectively, negative) roots $a\in \Phi$.  Suppose $h\in \cH$.  Then $h = h_+h_-z$, for unique $h_+\in \cH(+)$, $h_-\in \cH(-)$ and $z\in \cZ$, and we have
 $$\nu^\bullet (h) = (h_+h_-\cZ, [h_+,h_-]^{(p+1)/2}z).$$
On the other hand, Remark \ref{rsiYurem} says that we may replace $\bT$ by any other 
maximal $F$-torus $\bT^\prime$ in $\bG^\prime$ such that $y\in A(\bG,\bT^\prime,F)$ 
and we may use any ordering of $\Phi (\bG,\bT^\prime,F)$ without affecting the resulting 
special isomorphism.  In particular, we may take $\bT^\prime=\theta (\bT)$ and 
transfer the ordering on $\Phi$ to an ordering on $\Phi_\theta = \Phi (\bG,\theta (\bT),F)$ by declaring that $a\circ\theta\in \Phi_\theta$ is positive exactly when $a\in \Phi$ is positive.  The split polarization in this case is $(\theta (\cH(+)),\theta(\cH(-)))$.  Using $\theta(\bT)$ and the latter ordering of roots, we see that
\begin{equation*}
 \begin{split}
\nu^\bullet (\theta(h)) &= (\theta(h_+)\theta(h_-)\cZ, [\theta(h_+),\theta(h_-)]^{(p+1)/2}\theta(z))\cr
&= (\theta(h_+)\theta(h_-)\cZ, [h_+,h_-]^{-(p+1)/2}z^{-1}).
\end{split}
\end{equation*}
It follows that $\nu^\bullet (h)\in W\times 1$ exactly when $\nu^\bullet (\theta(h))\in W\times 1$.  This completes the proof.
\end{proof}

\section{Yu's construction via tensor products and inflation}
\label{sec:construction}

Yu's construction starts with an extended generic cuspidal $G$-datum $\Psi$.  
Then an inducing group $K= K(\Psi)$ and a representation $\kappa= \kappa (\Psi)$ of $K$ 
are constructed such that the representation $\pi$ of $G$ induced from $\kappa$ (via 
compactly supported smooth induction) is irreducible and supercuspidal.  To realize 
the full power of the construction, it is helpful to make explicit various aspects 
and properties of the construction which were not needed in \cite{Y}.

For example, Yu says that his construction ``has a nice inductive structure''  resulting 
from the fact that he actually constructs from a cuspidal $G$-datum a sequence 
$\vec\pi = (\pi_0,\dots ,\pi_d)$ of supercuspidal representations of $G^0,\dots ,G^d$, 
respectively.  To be more precise, this means we can attack problems regarding 
tame supercuspidal representations by induction on $d$.  This idea will be 
made explicit in the next section and elsewhere in this \paperbook.
We will also show later in this \paperbook\ that, in fact, the number $d$ is an 
invariant of the equivalence class of $\pi=\pi_d$.  

In this section, we recapitulate Yu's construction, but our presentation is 
different  in several ways.  First of all, we focus on 
the fact that the inducing representation $\kappa$ is naturally expressed as a 
tensor product $\kappa_{-1}\otd\kappa_d$, where $\kappa_{-1}$ depends 
only on $\rho$ and, otherwise, $\kappa_i$ only depends on $\phi_i$.  This tensor 
product structure is the starting point for a strategy for studying tame supercuspidal 
representations.  As we will see later in the \paperbook, certain facts about the 
supercuspidal representation $\pi$ induced from $\kappa$ can be effectively reduced 
to a collection of facts regarding the individual factors $\kappa_i$.  In other 
words, we frequently can study the separate contributions of the various quasicharacters 
$\phi_i$ without having to study any interplay between different quasicharacters.  
This is analogous to studying an automorphic representation by analyzing its local 
factors.  A key issue in the latter context, and for us as well, is certain 
so-called ``multiplicity one'' properties of the representations being studied.  

We also highlight the fact that, for $i\in \{\, 0,\dots,d-1\,\}$, each of the factors $\kappa_i$
is the inflation of a certain Weil-Heisenberg representation of the type discussed earlier.
The inflation operation and its basic properties are made much more explicit than in \cite{Y}.
The philosophy here is that the Weil-Heisenberg representations, together with the finite 
field representation attached to $\rho$, can be regarded as elementary particles in Yu's 
construction.  These elementary  objects are defined over finite fields and, at least 
in the case of the Weil-Heisenberg representations, their structure is remarkably 
tractable in many ways.  (See Section~\ref{sec:Heis}.)
From our point of view, Yu's construction takes these elementary 
objects and creates supercuspidal representations by applying the very simple functorial 
operations of inflation, tensor multiplication, and induction.

The most substantive result in this section is Proposition \ref{rsiisoclass}, which 
is an easy corollary of the theory of ``relevant special isomorphisms''  already developed. 
 To explain this, we recall first that Yu introduced the notions of ``special isomorphism'' 
and ``symplectic action'' to precisely describe how his symplectic groups were acting on 
his Heisenberg groups.  As Yu notes, the failure to address the issue of how the symplectic 
groups are acting has been the source of some confusion in the literature of supercuspidal 
representation theory.  With some effort, Yu is able to give canonical constructions 
of special isomorphisms and symplectic actions and then he uses these in his construction.  
If one varies the choices of the special isomorphisms and symplectic actions, one still 
obtains an irreducible supercuspidal representation, but there is no obvious reason 
to expect that one obtains an equivalent representation.
We consolidate the notions of special isomorphism and symplectic action into the notion 
of a relevant special isomorphism and
Proposition \ref{rsiisoclass} then says that the equivalence class of 
the supercuspidal representation constructed does not depend on the choices of the 
relevant special isomorphisms.  

Let us now begin our description of Yu's construction. Assume we are given an
extended generic cuspidal $G$-datum $\Psi=(\vec\bG,y,\rho,\vec\phi)$.
For the rest of this section,  fix $i\in \{\, 0,\dots ,d-1\,\}$.   

\begin{remark}\label{hatprop}
Recall from Section~\ref{sec:notations} that
$\hat\phi_i$ is the quasicharacter of $K^0G_{y,0^+}^iG_{y,s_i^+}$
that agrees with $\phi_i$ on $K^0G_{y,0^+}^i$ and
is trivial on $(G^i,G)_{y,(r_i^+,s_i^+)}$.
Because $J_+^{i+1}
=G_{y,r_i}^i (G^i,G^{i+1})_{y,(r_i^+,s_i^+)}$, we see,
referring back to Section~\ref{sec:genHeis}, that
$\hat\phi_i\,|\, J_+^{i+1}$ is the character of 
$J_+^{i+1}$ that in the $d=1$ setting would have been
denoted by $\zeta$. 
\end{remark}

Let 
\begin{equation*}
\begin{split}
W_i&=J^{i+1}/J^{i+1}_+\cr
N_i &= \ker (\hat\phi_i\,|\,
J^{i+1}_+)\cr
\cH_i&= J^{i+1}/N_i\cr
\cZ_i&= J^{i+1}_+/N_i\cr
\cS_i&=\hbox{Sp}(W_i)\cr
\zeta_i&=\hat\phi_i \,|\, J_+^{i+1}.
\end{split}
\end{equation*}

Whenever convenient, we will treat $\zeta_i$ as a character of $\cZ_i$.

In cases where $J^{i+1}=J_+^{i+1}$, $W_i$ is trivial and a
Weil-Heisenberg construction is not needed to define $\kappa_i$.
When $J^{i+1}=J_+^{i+1}$, we have $K^{i+1}=K^iJ_+^{i+1}$,
and we define a quasicharacter $\phi_i^\prime$ of $K^{i+1}$
by setting $\phi_i^\prime(kj)=\phi_i(k)\hat\phi_i(j)$, for
$k\in K^i$ and $j\in J_+^{i+1}$.

The next three paragraphs describe how the Weil-Heisenberg
theory is applied to define a representation $\phi_i^\prime$
of $K^{i+1}$ in cases where $J^{i+1}\not=J_+^{i+1}$.

At this point, we need to fix (arbitrarily) a relevant special isomorphism 
$\nu_i : \cH_i \to W_i^\sharp$ and a Heisenberg
representation
$(\tau_i , V_i) $ of $\cH_i$ with central character $\zeta_i$.  
Though the equivalence class of $\tau_i$ is determined by $\zeta_i$, in 
applications some models for the representation are more convenient to work 
with than others.  (It is a common abuse of terminology for one to use the 
term ``representation'' when one really is referring to an equivalence class 
of representations.  We use the term ``model'' for emphasis when we are 
speaking of a specific representation rather than its equivalence class.)

The Heisenberg representation $\tau^\sharp_i= \tau_i \circ (\nu_i)^{-1}$ 
of $W^\sharp_i = W_i\boxtimes \cZ_i$ extends uniquely to a representation 
$\hat\tau^\sharp_i$ of $\cS_i\ltimes W^\sharp_i$ on the space $V_i$.  
Letting $f'_i : K^i \to\cS_i$ be the map given by conjugation, we obtain a homomorphism
\begin{equation*}
\begin{split}
K^i\ltimes \cH_i&\to \cS_i \ltimes W^\sharp_i\cr
(k,h)&\mapsto (f'_i(k),\nu_i(h)).
\end{split}
\end{equation*}
Pulling back $\hat\tau^\sharp_i$ via this homomorphism yields the Weil 
representation $$\omega_i(k,h) = \hat\tau^\sharp_i (f'_i(k),1)\ \tau_i(h)$$ of 
$K^i\ltimes \cH_i$.  Note that condition \textbf{SC2}${}_i$ is satisfied with
$\tilde\phi_i = \omega_i$.

Now define a representation   $$\phi'_i : K^{i+1}\to
{\bGL}(V_i)$$  by
$$\phi'_i (kj) = \phi_i (k) \ \omega_i (k,j)= 
\phi_i(k)\ \hat\tau^\sharp_i( f'_i(k))\
\tau_i(j),$$ with $k\in K^i$ and $j\in
J^{i+1}$.

We have defined a representation $\phi_i^\prime$ of
$K^{i+1}$ that is attached to the quasicharacter $\phi_i$
and now we indicate how to obtain the
representation $\kappa_i$ of $K$ from $\phi_i^\prime$. If $\mu$
is a representation of $K^i$ which is 1-isotypic on
$K^i\cap J^{i+1}= G^i_{y,r_i}$ then there is a unique extension
of $\mu$ to a representation, denoted
$\inf\nolimits_{K^i}^{K^{i+1}}(\mu)$, of $K^{i+1}$ which is
1-isotypic on $J^{i+1}$.  This inflated representation is 1-isotypic on $K^{i+1}\cap J^{i+2}$, since
$$K^{i+1}\cap J^{i+2}= G^{i+1}_{y,r_{i+1}} \subset G^{i+1}_{y,r_i}\subset J^{i+1},$$
and, consequently, we may repeatedly inflate $\mu$.
More precisely, if $0\le
i\le j\le d$ then we may define \label{anotherinflation}
$$\inf\nolimits_{K^i}^{K^j}(\mu) =
\inf\nolimits^{K^j}_{K^{j-1}}\circ \cdots \circ
\inf\nolimits^{K^{i+1}}_{K^i}(\mu).$$
The factor $\kappa_i$  is given by $\kappa_d =\phi_d \,|\,K$ and, otherwise,
$$\kappa_i =
\inf\nolimits_{K^{i+1}}^{K}(\phi'_i),$$
where $\phi'_{-1}= \rho$.    In order to inflate from $K^{i+1}$ in the definition of $\kappa_i$  when $i<d-1$, we must use the fact that $\phi'_i$ is 1-isotypic on $K^{i+1}\cap J^{i+2}$.

\begin{proposition}\label{rsiisoclass}
For each $i\in \{\, 0,\dots ,d\,\}$, the equivalence class of the representation $\kappa_i$ 
only depends on the quasicharacter $\phi_i$ (except that the groups $K$, $K^i$ and $J^{i+1}$
depend on the full sequence $\vec r = (r_0,\dots , r_d)$ which is derived from 
the depths of the components of $\vec\phi$).  In particular, when $i<d$ this equivalence 
class does not depend on the choice of model for the Heisenberg representation $\tau_i$ 
or the choice of relevant special isomorphism $\nu_i$.
\end{proposition}

\begin{proof}
Fix a relevant special isomorphism $\nu_i: \cH_i\to W^\#_i = W_i\boxtimes \cZ_i$.  
Let $(\tau_i,V_i)$ and $(\dot\tau_i, \dot V_i)$ be two (equivalent) Heisenberg representations with central character $\zeta_i$ and let $I: V_i\to \dot V_i$ be a nonzero intertwining operator.  Thus $\dot\tau_i (h) = I\tau_i(h)I^{-1}$ for all $h\in \cH_i$.  The Heisenberg representations $\tau^\#_i= \tau_i\circ \nu_i^{-1}$ and $\dot\tau{}^\#_i = \dot\tau_i\circ \nu_i^{-1}$ of $W^\#_i$ are also intertwined by $I$.  The representations $\tau^\#_i$ and $\dot\tau{}^\sharp_i$ lift uniquely to representations of $\cS_i\ltimes W^\#_i$ which we denote by $\hat\tau{}^\sharp_i$ and $\hat{\dot\tau}{}^\sharp_i$, respectively.  The uniqueness property implies that $\hat{\dot\tau}{}^\sharp_i$ must be equivalent to the representation $(s,w)\mapsto I\hat\tau{}^\#_i(s,w) I^{-1}$.  It now follows that $I$ intertwines the representations $\kappa_i$ and $\dot\kappa_i$ of $K$ associated to $\tau_i$ and $\dot\tau_i$, respectively. 

The invariance of the equivalence class of $\kappa_i$ under the choice of $\nu_i$  follows directly from Lemma \ref{rsiequiv}.
\end{proof}

We define representations
\begin{equation*}
\begin{split}
\kappa &=
\kappa_{-1}\otimes \kappa_0\otd
\kappa_{d-1}\otimes \kappa_d\cr
\pi&=\pi_d=
\hbox{ind}_{K}^G(\kappa ).
\end{split}
\end{equation*}
From now on, the notation $V_i$ will be used for the
space of $\kappa_i$, for $i\in \{\, -1,\dots,d-1\,\}$.

To cleanly state our next result, we extend the definition of $\hat\phi_i$ by letting $\hat\phi_{-1}= 1$ and $\hat\phi_d = \phi_d$.  
We also set $K_+^{d+1}=K_+$.

\begin{lemma}\label{hatphikplus}
If $i\in \{\, -1,\dots,d\,\}$ then 
$$
\hat\phi_i \,|\, K_+ = 
\inf\nolimits_{K^{i+1}_+}^{K_+}(\hat\phi_i \,|\, K^{i+1}_+)
$$ 
and 
$\kappa_i\,|\, K_+$ is
$\hat\phi_i\,|\,K_+$-isotypic.
\end{lemma}

\begin{proof}
If $i=-1$ or $d$ then our
assertion is trivial, so we assume
$i\in \{\, 0,\dots,d-1\,\}$.  We start by considering the
restriction of
$\kappa_i$ to
$J^{i+2}_+\cdots J^d_+$.  Here, since
$\kappa_i$ is an inflation from
$K^{i+1}$ to $K$, this restriction is 1-isotypic.  So we
need to show that the restriction of
$\hat\phi_i$ to this subgroup is trivial. 
It is immediate from the definitions 
and the fact that $r_i<r_{i+1}<\cdots<r_{d-1}$
that 
$J_+^{i+2}\cdots J^d_+\subset
 (G^{i+1},G)_{y,(r_i^+,s_i^+)}$.
As noted in Remark \ref{hatprop}
(see also Section~4 of \cite{Y}),
$\hat\phi_i\,|\, (G^{i+1},G)_{y,(r_i^+,s_i^+)}$
is trivial.

On $J^{i+1}_+$, we have $\kappa_i=
\phi'_i =\tau_i$ and
$\tau_i\, |\, J^{i+1}_+$ is a multiple of
$\zeta_i = \hat\phi_i\,|\,J^{i+1}_+$.  On $K^i_+$,
we have $\kappa_i = \phi'_i =
\phi_i \otimes \hat\tau_i^\sharp$
and $\hat\phi_i = \phi_i$.  Since $\hat\tau_i^\sharp$ 
factors through the map 
$f_i^\prime :  K^i\to {\rm Sp}(\cH_i)$,  it suffices to
show that $f_i^\prime\, |\, K^i_+ =1$.  This is
equivalent to showing that $[K^i_+, J^{i+1}]
\subset J^{i+1}_+$.  In fact, we have
$[K^i_+,J^{i+1}] \subset [G^{i+1}_{y,0^+},
J^{i+1}]\subset (G^i,G^{i+1})_{y,
(r_i^+,s_i^+)}\subset J^{i+1}_+$.  
\end{proof}

Now let $\vartheta=\vartheta(\Psi)$ be the character 
of $K_+$ defined by 
$\vartheta = \prod_{i=0}^d (\hat\phi_i\, |\, K_+)$.
As we will see in Section~\ref{sec:symmetrizing}, 
certain properties of $\kappa$ are determined
by properties of the character $\vartheta$.

\begin{corollary}\label{varthetaisotypy}
The restriction $\kappa\,|\, K_+$ is $\vartheta$-isotypic.
\end{corollary}

\section{The connection with Howe's construction}
\label{sec:Howeconstruction}

Throughout this section, we assume that $G=\bGL_n(F)$.
We discuss relations between
Howe's method of constructing tame supercuspidal
representations of general linear groups
and Yu's construction. 
Although Bushnell and Kutzko \cite{BK} have
a general construction of the admissible dual
of $\bGL_n(F)$ that includes all supercuspidal representations,
we focus on the tame supercuspidal representations
constructed by Howe because these are the ones
that can also be obtained from Yu's construction.
There are no new results
in this section. It is intended as a guide for
the reader, especially in connection with
later sections where we comment on results
for the case $G=\bGL_n(F)$. 

The Howe construction of tame supercuspidal
representations attaches supercuspidal representations
of $\bGL_n(F)$ to $F$-admissible quasicharacters
of the multiplicative groups of tamely ramified
degree $n$ extensions of $F$ \cite{Ho}. 
If $L$ is an extension of $F$, we denote the
ring of integers of $L$ and the maximal ideal in the
ring of integers by $\gO_L$ and $\gP_L$, respectively.

\begin{definition}\label{admissible}
Suppose that $E$ is a tamely ramified extension of $F$ of
degree $n$ and $\varphi$ is a quasicharacter of $E^\times$.
As defined in \cite{Ho}, $\varphi$ is $F$-\textit{admissible}
 (or \textit
{admissible} over $F$) if 
\begin{itemize}
\item there does not exist a proper subfield $L$ of
$E$ containing $F$ such that $\varphi$ factors through
the norm map $N_{E/L}:E^\times\rightarrow L^\times$;
\item if $L$ is a subfield of $E$ containing
$F$ and $\varphi\,|\, (1+\gP_E)$ factors through
$N_{E/L}$, then $E$ is unramified over $L$.
\end{itemize}
\end{definition}

\begin{definition}\label{conjugate}
Suppose that $E$ and $E^\prime$ are tamely ramified
extensions of $F$ of degree $n$, and
$\varphi$ and $\varphi^\prime$ are $F$-admissible quasicharacters
of $E^\times$ and $E^{\prime\times}$, respectively.
Then $\varphi$ and
$\varphi^\prime$ are said to be $F$-\textit{conjugate}
if there exists an $F$-isomorphism of $E$ with $E^\prime$
that takes $\varphi$ to $\varphi^\prime$.
\end{definition}

In \cite{Ho}, given an $F$-admissible
quasicharacter of the multiplicative group of
a degree $n$ tamely ramified extension of $F$,
Howe constructed an equivalence class of
irreducible supercuspidal
representations of $G$. 
Howe also proved that two $F$-admissible quasicharacters
$\varphi$ and $\varphi^\prime$ are $F$-conjugate
if and only if they give rise to the same equivalence class
of supercuspidal representations.
In \cite{Moy}, Moy proved that if $p$ is odd
and does not
divide $n$, then every irreducible supercuspidal 
representation of $G$ arises via Howe's construction.

Given an $F$-admissible quasicharacter $\varphi$, the first step
in producing
an open compact modulo center subgroup and a representation
of the subgroup that induces a representation of $G$ 
that belongs to the 
equivalence class associated to $\varphi$
involves factoring the quasicharacter $\varphi$
in a nice way. Such factorizations are called Howe
factorizations (see Lemma~2.2.4 \cite{Moy}). Any two Howe 
factorizations of a given
$F$-admissible quasicharacter give 
rise to the same equivalence class of representations of $G$.
We will outline the connections between Howe factorizations
of $F$-admissible quasicharacters and
generic cuspidal $G$-data.

\begin{definition}\label{condexp}
If $F^\prime$ is a finite tamely ramified extension
of $F$
and $\varphi$ is a quasicharacter of $F^{\prime\times}$,
the \textit{conductoral exponent} $f(\varphi)$ of
$\varphi$ is the smallest positive integer such that
$\varphi\,|\, 1+\gP_{F^\prime}^{f(\varphi)}=1$.
\end{definition}

Let $F^\prime$ be a finite tamely ramified extension of
$F$.
Choose a prime element $\varpi_{F^\prime}$ in $F^\prime$
having the property that $\varpi_{F^\prime}^e$ belongs
to $F$,
where $e$ is the ramification index of $F^\prime$
over $F$.
Let $C_{F^{\prime}}$ be the subgroup of
$F^{\prime\times}$ generated by  $\varpi_{F^\prime}$
and the roots of unity in $\gO_{F^\prime}^\times$
that have order relatively prime to $p$. Let $\psi^\prime$ be
 a character of $F^\prime$ that is trivial on $\gP_{F^\prime}$ 
and
nontrivial on $\gO_{F^\prime}$.
 If $f(\varphi)>1$, then there exists a unique
$$
\gamma_\varphi\in C_{F^\prime}\cap
(\gP_{F^\prime}^{1-f(\varphi)}-\gP_{F^\prime}^{2-
f(\varphi)})
$$
such that $\varphi(1+t)=\psi^\prime(\gamma_\varphi t)$, 
$t\in \gP_{F^\prime}^{f(\varphi)-1}$. 

\begin{definition}\label{generic}
Let $F^\prime$ be a tamely ramified extension of $F$ and
let $\varphi$ be a quasicharacter of $F^{\prime\times}$.
If $f(\varphi)>1$, we say that
$\varphi$ is \textit{generic} over $F$ if $F[\gamma_\varphi]=F^\prime$.
If
 $f(\varphi)=1$, then we say that $\varphi$ is \textit{generic} over $F$ 
if $\varphi$ is $F$-admissible.
\end{definition}

It is easy to see from the definition of
$F$-admissible that if $f(\varphi)=1$, then $\varphi$
is generic over $F$ if and only if 
$F^\prime$ is unramified over $F$ and $\varphi$ is not fixed by
any nontrivial element of the Galois group $\gal(F^\prime/F)$.
Note that genericity implies admissibility in all cases.

Let $E$ be a tamely ramified extension of $F$ of degree
$n$, and let $\varphi$ be an $F$-admissible quasicharacter
of $E^\times$.

\begin{definition}\label{howefact}
 A \textit{Howe factorization} of $\varphi$
may be defined as follows.
It consists of a tower of fields
$F=E_d\subsetneq E_{d-1}\subsetneq\cdots\subsetneq E_0\subset E$,
$d\ge 0$, 
together with a collection of quasicharacters $\varphi_i$, $i=-1,
\dots, d$, having certain properties. 
Let $N_{E/E_i}$ denote the norm map from $E^\times$ to
$E_i^\times$, for $i\in \{\, 0,\dots,d\,\}$.
For each $i\in \{\, 0,\dots,d\,\}$,
 $\varphi_i$
is a quasicharacter of $E_i^\times$ such that the conductoral
exponent $f_i=f(\varphi_i\circ N_{E/E_i})$ of $\varphi_i\circ N_{E/E_i}$
 is greater than $1$, and such that
$\varphi_i$ is generic over $E_{i+1}$
if $i\not=d$. The condition $f_0<f_1<\cdots <f_{d-1}$ must also
be satisfied. In addition, if $\varphi_d$ is nontrivial,
then $f_d>f_{d-1}$.
If $E_0=E$, then $\varphi_{-1}$ is the trivial
character of $E^\times$.
If $E_0\subsetneq E$, then 
$\varphi_{-1}$ is a quasicharacter
of $E^\times$ such that $f(\varphi_{-1})=1$
 and $\varphi_{-1}$ is generic over $E_0$.
The final requirement is that
 $\varphi=\varphi_{-1}\prod_{i=0}^d\varphi_i\circ N_{E/E_i}$.
\end{definition}

In order to attach a generic cuspidal $G$-datum
to a Howe factorization of $\varphi$, we begin
with a choice of basis of $E$ over $F$. This
gives an injective homomorphism from $E^\times$
to $G$. That is, 
an element of $E^\times$
is mapped to the matrix (relative to the given basis)
of the invertible operator on the $F$-vector space
 $E$ defined by left multiplication by the element
of $E^\times$.
We identify $E^\times$ with
its image in $G$. 
For each $i\in \{\, 0,\dots,d-1\,\}$, let $G^i$ be the 
centralizer
in $G$ of (the image of) $E_i^\times$. Let $G^d=G$.
Then $G^i\cong (R_{E_i/F}\bGL_{n_i})(F)
\cong \bGL_{n_i}(E_i)$, where $n_i=n[E_i:F]^{-1}$
and $R_{E_i/F}$ denotes restriction of scalars,
and $\vec\bG
=(\bG^0,\dots,\bG^d)$ is a tamely ramified
twisted Levi sequence. For each $i\in \{\, 0,\dots,d\,\}$, let
$\phi_i=\varphi_i\circ\det_i$, where $\det_i:G^i\rightarrow E_i^\times$
is the usual determinant homomorphism.
It can easily be checked that if $i\in \{\, 0,\dots,d\,\}$,
Howe's genericity condition on $\varphi_i$ implies
that $\phi_i$ is $G^{i+1}$-generic.
Now $E^\times\cong (R_{E/F}\bGL_1)(F)$ is
a tamely ramified twisted Levi subgroup of
$G$, and the Bruhat-Tits building of this group 
embeds in $\cB(\bG,F)$ and the image has the form $[y]=
y+ X_*(\bZ,F)$ for some $y\in \cB(\bG,F)$.
If $E=E_0$, let $\rho$ be the trivial representation
of $G^0=E^\times$. 

Suppose that $E\not=E_0$. Then, because $f(\varphi_{-1})=1$
and $\varphi_{-1}$ is $E_0$-admissible, 
$\varphi_{-1}$ is
not fixed by any nontrivial element of $\gal(E/E_0)$.
Let $q_0$ be the cardinality of the residue class field
of $E_0$. Then $G^0_{y,0}$ is conjugate to 
$\bGL_{n_0}(\gO_{E_0})$ and $G_{y,0:0^+}^0
\cong\bGL_{n_0}(\F_{q_0})$. 
It is well known that 
there is a bijection (induced by the construction
of Deligne and Lusztig) between the set of equivalence
classes of irreducible cuspidal representations
of $\bGL_{n_0}(\F_{q_0})$ and
the $\gal(\F_{{q_0}^{n_0}}/\F_{q_0})$-orbits
of characters of $\F_{{q_0}^{n_0}}^\times$ that are in
general position. (Recall that a character of
$\F_{{q_0}^{n_0}}^\times$ is in general position
whenever it is not fixed by any nontrivial element
of $\gal(\F_{{q_0}^{n_0}}/\F_{q_0})$.)
The restriction $\varphi_{-1}\,|\, \gO_E^\times$
factors to a character of $\F_{{q_0}^{n_0}}^\times$
that is in general position, and hence determines
an equivalence class of irreducible cuspidal
representations of $\bGL_{n_0}(\F_{q_0})$.
Let $\rho^\circ$ be an irreducible smooth representation
of $G_{y,0}^0$ whose restriction to $G_{y,0^+}^0$
is a multiple of the trivial representation
and that factors to an irreducible cuspidal representation of
$G_{y,0:0^+}^0$ belonging to the above
equivalence class of cuspidal representations.
Note that $G_{[y]}^0=E_0^\times G_{y,0}^0
\cong \langle \varpi_{E_0}\rangle \times
G_{y,0}^0$, for any choice of prime element
$\varpi_{E_0}$ in $E_0$. Let $\rho$ be the representation
of $G_{[y]}^0$ that restricts to $\rho^\circ$ on 
$G_{y,0}^0$, and such that $\rho(\varpi_{E_0})$
is equal to $\varphi_{-1}(\varpi_{E_0})$ times the
identity operator on the space of $\rho^\circ$.

With the above definitions, $\Psi=(\vec\bG,y,\rho,\vec\phi)$ is
an extended generic cuspidal $G$-datum. Note that in the
case $E\not=E_0$, we have the freedom to vary
the choice of $\rho$ somewhat (subject to the condition
that $\rho\,|\, G_{y,0}^0$ factors to
an element of the appropriate equivalence class
of cuspidal representations). We remark that
if we choose a different basis of $E$ over $F$,
we will get a $G$-datum that can be obtained
from $\Psi$ by conjugating by the appropriate
change of basis matrix.

Now suppose that $(\vec\bG,y,\rho,\vec\phi)$ is
a extended generic cuspidal $G$-datum. 
A  tamely ramified twisted Levi subgroup
of $G$ is isomorphic to a direct product
of general linear groups over tamely ramified
extensions of $F$. The center of such a group
is isomorphic to the direct product
of the multiplicative groups of those extensions
and hence is compact modulo the center $F^\times$
of $G$ if and only if only one extension occurs, that
is if and only if the twisted Levi subgroup is isomorphic
to $\bGL_m(F^\prime)$, where $F^\prime$ is a tamely
ramified extension of $F$, and $m=n[F^\prime:F]^{-1}$.
It follows that there exist tamely ramified field extensions
$F=E_d\subsetneq E_{d-1}\subsetneq \cdots\subsetneq E_0$
such that $[E_0:F]$ divides $n$ and $\bG^i\cong
R_{E_i/F}\bGL_{n_i}$, $n_i=n[E_i:F]^{-1}$, for 
$i\in \{\,0,\dots,d\,\}$.

If $i\in \{\,0,\dots,d\,\}$, there exists a unique
quasicharacter $\varphi_i$ of $E_i^\times$ such
that $\phi_i=\varphi_i\circ\det_i$, where
$\det_i:G^i\rightarrow E_i^\times$ is the determinant map.
If $i\not=d$, Yu's condition that $\phi_i$ be
$G^{i+1}$-generic translates into Howe's 
condition that  $\varphi_i$ be generic over $E_{i+1}$.
Also,  Yu's condition on the depths of the quasicharacters
$\phi_i$ translates into the above conditions
on the conductoral exponents of the quasicharacters
$\varphi_i\circ N_{E/E_i}$.

Because $G_{y,0}^0$ is a maximal parahoric subgroup
of $G^0\cong \bGL_{n_0}(E_0)$, $G_{y,0}^0$
is conjugate to $\bGL_{n_0}(\gO_{E_0})$.
The restriction $\rho\,|\, G_{y,0}^0$ factors to
an irreducible cuspidal representation of
$G_{y,0:0^+}^0\cong \bGL_{n_0}(\F_{q_0})$,
where $q_0$ is the cardinality of the residue
class field of $E_0$. As remarked above, the
equivalence class of the given cuspidal
representation corresponds to a 
$\gal(\F_{{q_0}^{n_0}}/\F_{q_0})$-orbit of
characters of $\F_{{q_0}^{n_0}}^\times$
that are in general position.
Let $\tau^\circ$ be any one of these characters.
Let $E$ be an unramified extension of $E_0$
of degree $n_0$, and set $\tau$ equal to the
character of $\gO_E^\times$ that
is trivial on $1+\gP_E$ and
factors to the character $\tau^\circ$
of $\F_{{q_0}^{n_0}}^\times$. 
By Schur's Lemma, the operator $\rho(\varpi_{E_0})$ is
scalar. Extend $\tau$ to a character of
$E^\times$ by setting $\tau(\varpi_{E_0})$ equal
to that scalar.

If $E=E_0$, then $G^0\cong E_0^\times=E^\times$
is an elliptic maximal torus and
$\rho$ is simply the
character $\tau$ 
of  $E^\times$.
Let $\varphi_{-1}$ be the trivial character
of $E^\times$. Note that, because $\rho
\,|\, 1+\gP_{E}=1$, $\rho\varphi_0$
is generic over $E_1$, and has the same conductoral
exponent as $\varphi_0$.
Set $\varphi=\prod_{i=0}^d\varphi_i\circ N_{E/E_i}$.
Then $\varphi$ is $F$-admissible and the
 collection of extensions $E_0,\dots,E_d$,
together with the quasicharacters $\varphi_{-1}$, 
$\rho\varphi_0$, $\varphi_1,\dots,\varphi_{d}$,
is a Howe factorization of the quasicharacter
$\varphi$ of $E^\times=E_0^\times$. Note that
in this case, $\varphi=\rho\prod_{i=0}^d \phi_i\,|\, G^0$.

If $E\not=E_0$, let $\varphi_{-1}=\tau$ be the quasicharacter
of $E^\times$ defined above. It is clear that
the condition that the character of $\F_{{q_0}^{n_0}}^\times$
corresponding to $\tau$ as above
be in general position translates into the condition
that $\varphi_{-1}$ be $E_0$-admissible.
Set $\varphi=\varphi_{-1}\prod_{i=0}^d \varphi_i\circ N_{E/E_i}$.
Then $\varphi$ is $F$-admissible, and the extensions
$E_0,\dots,E_d$, together with the quasicharacters
$\varphi_{-1},\varphi_0,\dots,\varphi_d$, forms a Howe factorization
of $\varphi$. Note that we have some freedom in choosing
the quasicharacter $\varphi_{-1}$. We can replace $\varphi_{-1}$ by
$\varphi_{-1}\circ\sigma$ for some $\sigma\in \gal(E/E_0)$
and produce another $F$-admissible quasicharacter
attached to $\Psi$. This quasicharacter is equal
to $\varphi\circ\sigma$, which is $F$-conjugate to
$\varphi$ and so must give rise to the same equivalence
class of supercuspidal representations.

We remark that it might seem more natural in
the case $[E_0:F]=n$ to set $\varphi_{-1}=\rho$.
However, this would not be consistent with
the definition of Howe factorization.
If we were to modify the original
cuspidal datum (in the case $[E_0:F]=n$) by
replacing $\phi_0$ by $\rho\phi_0$, and replacing
$\rho$ by the trivial representation of $G_{[y]}^0
=E_0^\times$, we would get the same $F$-admissible
quasicharacter $\varphi$, as well as the same Howe
factorization of $\varphi$. Furthermore, 
the supercuspidal
representation obtained via Yu's construction
from the modified $G$-datum is equivalent to
the one obtained from the original $G$-datum.

Let $\Psi=(\vec\bG,y,\rho,\vec\phi)$ be a 
generic cuspidal $G$-datum, with associated 
tamely ramified degree $n$ extension $E$ of $F$,
$F$-admissible quasicharacter $\varphi$, and
Howe factorization of $\varphi$, as discussed above.
Then, as discussed below, Howe's construction and Yu's construction
give rise to equivalent representations of
the group $K(\Psi)$, hence to equivalent
supercuspidal representations of $G$.

Without loss of generality (after conjugating $\Psi$
by some element of $G$), we may assume that
$G^i=(R_{E_i/F}\bGL_{n_i})(F)=\bGL_{n_i}(E_i)$,
$n_i=n[E_i:F]^{-1}$, for each $i\in \{\,0,\dots,d\,\}$,
and $G_{y,0}^0=\bGL_{n_0}(\gO_{E_0})$.
 Recall from Section~\ref{sec:buildings} that we have normalized the
valuation $v_F$ on $F$ so that  $v_F(F^\times)=\mathbb{Z}$,
and we have extended $v_F$ to tame extension fields of $F$.
Consequently
$\g^i_{y,t}=\g_{y,t}\cap \g^i$ for all
$t$, $G^i_{y,t}=G_{y,t}\cap G^i$ for $t>0$.
and $\varpi \g_{y,t}^i=\g_{y,t+1}^i$,
where $\varpi$ is any prime element in $F$.
Let $e$ be the ramification degree of $E$ over
$F$. As in earlier sections, let $r_i$ be the depth of $\phi_i$, 
for $i\in \{\,0,\dots,d-1\,\}$. With the above conventions, we
have $r_i=(f_i-1)/e$, where $f_i=f(\varphi_i\circ N_{E/E_i})$,
$i\in \{\, 0,\dots,d-1\,\}$.

The inducing datum in the Howe construction is defined
in terms of the Howe factorization, and it is a simple matter 
to check that the inducing
subgroup attached to the Howe factorization
is the same as $K=K(\Psi)$.
For $i\in \{\, 1,\dots,d\,\}$, the Howe construction associates
an irreducible representation $\kappa_i^H$ of $K$
to the  quasicharacter $\varphi_i$ of $E_i^\times$,
and if $E\not=E_0$, an irreducible representation
$\kappa_{-1}^H$ of $K$. When $E=E_0$, there is
a representation $\kappa_0^H$ attached to
$\rho\varphi_0$, and when $E\not=E_0$, $\kappa_0^H$ is
attached to $\varphi_0$. The associated supercuspidal
representation is induced from the representation
$\kappa^H$ obtained as the (internal) tensor product
of $\kappa_i^H$, $i\in \{\,0,\dots,d\,\}$, when $E=E_0$ and
of $\kappa_i^H$, $i\in \{\, -1,\dots,d\,\}$, when $E\not=E_0$.
Let $\kappa_{i}^Y$, $i\in\{\, -1,\dots,d-1\,\}$ be
the representations attached by Yu's construction
to $\rho$, $\phi_0,\dots,\phi_d$ (as described
in Section~\ref{sec:construction}), and let $\kappa^Y=\kappa(\Psi)$.
Assuming that a relevant special isomorphism is
used in the construction of $\kappa_i^H$ for
$i\in \{\, 0,\dots,d\,\}$,
whenever there is a (nontrivial) Heisenberg construction,
then $\kappa_i^H\simeq \kappa_i^Y$ for $i\in \{\, 1,\dots,d\,\}$.
If $E\not=E_0$, then $\kappa_0^H\simeq\kappa_0^Y$.

In the case $E=E_0$, $\kappa_0^H$, being associated
to $\rho\varphi_0$, rather than to $\varphi_0$, is equivalent
to $\rho^\sharp\kappa_0^Y$, where $\rho^\sharp$
is the inflation of
the quasicharacter $\rho$ of $E^\times$ to the group
$K$. Note that $\kappa_{-1}^Y=\rho^\sharp$ when
$E=E_0$. Hence, when $E=E_0$,
$$\kappa^Y=\rho^\sharp\otimes \bigotimes_{i=0}^d \kappa_i^Y
\simeq \bigotimes_{i=0}^d\kappa_i^H = \kappa^H.
$$

Now suppose that $E\not=E_0$. Then $\kappa_i^H
\simeq \kappa_i^Y$ for $i\in \{\, 0,\dots, d\,\}$.
Hence to show that $\kappa^H\simeq\kappa^Y$, we
need only show that $\kappa^H_{-1}
\simeq \kappa^Y_{-1}$.
Referring to comments above, the quasicharacter
$\varphi_{-1}$ of $E^\times$ was chosen so that the
irreducible representation of $E_0^\times
\bGL_{n_0}(\gO_{E_0})$ is equivalent to
$\rho$. Hence $\kappa_{-1}^H\simeq \kappa_{-1}^Y$.

\chapter{Further properties of cuspidal $G$-data}

\section{Polarizations associated to involutions}
\label{sec:maxiso}

Many of the notations in this section are the same as in Section~\ref{sec:genHeis}, 
except that we only make the following assumptions regarding the involution 
$\theta$ of $G$ and the quasicharacter $\phi$ of $G^\prime$:
\begin{itemize}
\item $\theta (G^\prime) = G^\prime$ \ \ \ \ (Equivalently, $\theta(\bG^\prime)
=\bG^\prime$.)
\item $\phi$ is $G$-generic of depth $r>0$ and $\phi \circ \theta = \phi^{-1}$.
\end{itemize}
Under these assumptions, there is no involution of the symplectic space $W= J/J_+$ 
that is obviously associated to $\theta$, since  $J$ and $J_+$ might not be $\theta$-stable.  
Nevertheless, one of the two main results in this section, Proposition \ref{canpolar}, 
asserts that there is indeed a canonical involution of $W$ associated to $\theta$ and an 
associated polarization of $W$.  

The existence of this polarization is  used to show  that the space $\Hom_{J^\theta}(\tau,1)$ must have dimension one, where $\tau$ is the Heisenberg representation of $\cH = J/N$ associated to $\zeta$.    This fact and the behavior of the latter Hom-space with respect to the Weil representation are treated in Proposition \ref{corcanpolar}.  These things are also used later in the proof of Proposition \ref{quaddistprop}.

Stating the main results requires some additional notations, beyond the notations from 
Section~\ref{sec:genHeis}.   Recall that  $J(E)=(\bG^\prime,\bG)(E)_{y,(r.s)}$ 
is the group generated by the subgroups 
$\bU_a(E)_{y,r}$, with $a\in \Phi^\prime$, the group $\bT(E)_r$, and the groups $\bU_a(E)_{y,s}$,
with $a\in \Phi\setminus\Phi^\prime$. (Here, $\Phi = \Phi (\bG,\bT,E)$ and 
$\Phi^\prime = \Phi (\bG^\prime,\bT,E)$
are the $E$-roots of $\bT$ in $\bG$ and $\bG^\prime$, respectively.)
Because $y$ and $\theta(y)$ belong to $\cB(\bG^\prime ,E)$
and we can always find an
apartment in $\cB(\bG^\prime ,E)$ containing
any two points of $\cB(\bG^\prime,E)$, we may (and do) assume that the apartment
$A(\bG^\prime,\bT,E)$ contains both $y$ and $\theta(y)$.
If we modify the definition of $J(E)$ and we only allow 
the groups $\bU_a(E)_{y,r}$, with $a\in \Phi^\prime$ that satisfy $a(y-\theta(y))<0$,
 and the groups $\bU_a(E)_{y,s}$, with $a\in \Phi\setminus\Phi^\prime$ that
satisfy $a(y-\theta(y))<0$, 
then we obtain a 
subgroup $J_1(E)$ of $J(E)$.  Similarly, the 
condition $a(y-\theta (y))>0$ yields a subgroup $J_3(E)$.
Next, let $J_2(E)$ denote the group generated by $\bT(E)_r$ and
the groups $U_a(E)_{y,r}$ with $a\in \Phi^\prime$ such that
$a(y-\theta(y))=0$, and
$U_a(E)_{y,s}$ with $a\in \Phi\setminus\Phi^\prime$ such that $a(y-\theta(y))=0$.
Now let
\begin{equation*}
 \begin{split}
J_4(E)&= J(E)\cap \theta (J(E))\cr
J_5(E)&= J(E)\cap \theta (J_+(E))\cr
 \end{split}
\end{equation*}
and let $J_i= J_i(E)\cap G$, for $i=1,\dots,5$ and $W_i = J_iJ_+/J_+$, for $i=1,2,3$.  
Let $J_{4+}= J_5\theta(J_5)$, $W_4= J_4/J_{4+}$ and $W^+ = J^\theta J_+/J_+$.  

Recall from Section~\ref{sec:genHeis} that we use the 
notation $K^\prime$ for the stabilizer 
$G^\prime_{[y]}$ of $[y]$ in $G^\prime$.

\begin{proposition}\label{canpolar}
The spaces $W_1$, $W_2$ and $W_3$ satisfy:
\begin{itemize}
\item $W= W_1\oplus W_2\oplus W_3$,
\item $W_1$ and $W_3$ are totally isotropic,
\item $W_2$ is nondegenerate.
\end{itemize}
The inclusion $J_2\hookrightarrow J_4$ induces an isomorphism $W_2\cong W_4$ of symplectic spaces.  The natural involution of $W_4$ associated to $\theta$ transfers to an involution $\vartheta$ of $W_2$ and there is an associated polarization $W_2= W_2^++W_2^-$, with
\begin{equation*}
 \begin{split}
W_2^+&= \{\, w\in W_2 \ | \ \vartheta
(w) = w\,\},\cr
W_2^- &= \{\, w\in W_2 \ | \ \vartheta
(w) = w^{-1} \,\}.
 \end{split}
\end{equation*}
The space $W^+ = J^\theta J_+/J_+$ is identical to $W_1\oplus W_2^+$ and letting $W^-= W^-_2 \oplus W_3$ gives  a polarization $W= W^++ W^-$, where $W^+$ and $W^-$ are both stable 
under $f'(K^{\prime,\theta})$. Here, $f' : K^\prime\to {\rm Sp}(W)$ is the map arising from the 
action of $K^\prime$ on $J$ by conjugation.
 The involution $\vartheta$ of $W_2$ extends to an involution $\vartheta$ of $W$ which is defined by $\vartheta (w_+ w_-) = w_+w_-^{-1}$, with $w_+\in W^+$ and $w_-\in W^-$.
Similarly, if $$Z^1_\theta (J) = \{\,k\in J\ | \ \theta (k)= k^{-1}\,\}$$ and  $W^-_* = Z^1_\theta (J) J_+/J_+$ then $W^-_* = W_1\oplus W_2^-$ and letting $W^+_*= W^+_2 \oplus W_3$ gives  a polarization $W= W^+_*+ W^-_*$, where $W^+_*$ and $W^-_*$ both stable under $f'(K^{\prime,\theta})$.  The involution $\vartheta$ of $W_2$ extends to an involution $\vartheta_*$ of $W$ which is defined by $\vartheta (w_+ w_-) = w_+w_-^{-1}$, with $w_+\in W^+_*$ and $w_-\in W^-_*$.
\end{proposition}

The previous result is a key ingredient in the proof of the following:

\begin{proposition}\label{corcanpolar} The space $\Hom_{J^\theta}(\tau,1)$ has dimension one.  Let $\nu^\bullet : \cH\to W^\sharp$ be Yu's special isomorphism.  Then $\nu^\bullet (\cH^+) = W^+\times 1$, where $\cH^+ = J^\theta N/N$ and $W^+ = J^\theta J_+/J_+$.  Let $\hat\tau$ be the Heisenberg-Weil lift of $\tau$ to $\cS\ltimes_{\nu^\bullet}\cH$ and let $\chi^\cP$ be the unique character of order two of the group
$$\cP = \{\, s\in {\rm Sp}(W)\ | \ sW^+ \subset W^+\,\}.$$  Then $f'(K^{\prime,\theta})\subset \cP$ and $$\Hom_{J^\theta} (\tau ,1) \subset \Hom_{f'(K^{\prime,\theta})} (\hat\tau ,\chi^\cP).$$ 
\end{proposition}

The proofs of Proposition \ref{canpolar} and Proposition \ref{corcanpolar} are lengthy technical exercises.  Briefly stated, our approach is to closely follow the techniques of Sections 12 and 13 in \cite{Y} with the automorphism $\Int(g)$ replaced by $\theta$.  
We modify Yu's definitions of the sets $\Phi_1$, $\Phi_2$ and $\Phi_3$ as follows:
\begin{equation*}
\begin{split}
\Phi_1&= \{\, a\in \Phi  \ | \ a(y-\theta(y))<0\,\}\cr
\Phi_2&= \{\, a\in \Phi  \ | \ a(y-\theta(y))=0\,\}\cup \{0\} \cr
\Phi_3&= \{\, a\in \Phi  \ | \ a(y-\theta(y))>0\,\} .
\end{split}
\end{equation*}
Then we define, for $i=1,2,3$,
\begin{equation*}
\begin{split}
\Phi^{\prime}_i&= \Phi_i \cap (\Phi^\prime \cup\{0\})\cr
\Phi^{\prime\prime}_i&= \Phi_i\setminus\Phi^\prime_i,
\end{split}
\end{equation*}
and concave functions
$$f_i(a)=
\begin{cases}
r,&\text{if $a\in \Phi^\prime_i$,}\cr
s,&\text{if $a\in \Phi^{\prime\prime}_i$,}\cr
\infty,&\text{if $a\not\in \Phi_i$,}\cr
\end{cases}$$
on $\Phi  \cup\{0\}$.  Though the meanings of these objects has changed, it turns out that Yu's proofs carry over with very few modifications.  Note that the fact that $f_1$, $f_2$ and $f_3$ are concave follows from Lemma 13.1 (iv) of \cite{Y}.  Also, the groups $J_i(E)$, with $i=1,2,3$, coincide with the groups $\bG (E)_{y,f_i}$ associated to the 
latter concave functions. (See Section~\ref{sec:buildings} for the definitions
of groups associated to concave functions.)

Recall that in the previous section, we defined the symplectic form on $W(E)$  by
$$\langle u,v\rangle = \zeta_E ([u,v]),$$ with $u,v\in J(E)$, for a 
certain character $\zeta_E$ of $J_+(E)$.

\begin{lemma}\label{zetasym}
The characters $\zeta^{-1}_E$ and $\zeta_E\circ\theta$ agree on $J_+(E)\cap \theta (J_+(E))$. 
Consequently, if $u,v\in J_4(E)$ then $\langle u,v\rangle = \langle \theta (v),\theta (u)\rangle$.
\end{lemma}

\begin{proof} Let $\phi^E$ be the character of $\bG^\prime(E)_{y,r}$
defined in Section~\ref{sec:genHeis}.
The proof of the lemma
uses $\theta(\bG^\prime)=\bG^\prime$, $\phi^E\circ\theta=(\phi^E)^{-1}$
on $\bG^\prime(E)_{y,r}\cap\bG^\prime(E)_{\theta(y),r}$
and $\zeta_E\,|\, \bG^\prime(E)_{y,r}=\phi^E$, and
is a straightforward generalization of the proof of Lemma 9.3 of \cite{Y}.
\end{proof}

\begin{lemma}\label{WonetwothreeA}
The spaces $W_1$ and $W_3$ are totally isotropic and each of these spaces is orthogonal to $W_2$.
\end{lemma}

\begin{proof}
It suffices to work over the splitting field $E$ and to show that the following commutator subgroups are contained in $N(E)$:
$$[J_1(E),J_1(E)], [J_1(E), J_2(E)], [J_2(E), J_3(E)], [J_3(E),J_3(E)].$$  The latter assertion resembles the statement of Yu's Lemma 13.5 and it can be proved analogously.
\end{proof}

\begin{lemma}\label{WonetwothreeB}
The spaces $W_1$, $W_2$ and $W_3$ satisfy the following relations:
\begin{enumerate}
\item $W_1= J_5J_+/J_+ = J_{4+}J_+/J_+$.
\item $W_1\oplus W_2= J_4 J_+/J_+$.
\item $W_1\oplus W_2\oplus W_3 =W$.
\end{enumerate}
\end{lemma}

\begin{proof}  Our claim follows from the same approach used in the proof of Conditions (a) and (d) in Lemma 13.6 of \cite{Y}.
\end{proof}

Let $W_4 = J_4/J_{4+}$. (Note that $W_4$ is not the same as $J_4J_+ /J_+ = W_1\oplus W_2$.)

\begin{lemma}\label{WtwoWfour}
The inclusion of $J_2$ in $J_4$ determines a symplectic isomorphism of $W_2$ with $W_4$.
\end{lemma}

\begin{proof}  According to Parts (1) and (2) of Lemma \ref{WonetwothreeB}, we have natural isomorphisms $W_2\cong (W_1\oplus W_2)/W_1 \cong (J_4J_+)/(J_5J_+) \cong J_4/(J_5(J_4\cap J_+)) = J_4/J_{4+} = W_4$.  Clearly, the resulting isomorphism $W_2\cong W_4$ comes from the inclusion of $J_2$ in $J_4$.
\end{proof}

The next result establishes that the space $W_4$  is nondegenerate.

\begin{lemma}\label{radlem}
$J_{4+}= \{\, k\in J_4 \ | \ \langle k,J_4\rangle =1\,\}$.
\end{lemma}

\begin{proof} Notice that the statement of the result is 
equivalent to nondegeneracy of $W_4$. By Lemmas~\ref{WonetwothreeA} and 
\ref{WonetwothreeB},
and the nondegeneracy of $W$, we have that $W_2$ is
nondegenerate. By Lemma~\ref{WtwoWfour}, $W_4$ is nondegenerate.
\end{proof}

Since $J_4$ and $J_{4+}$ are $\theta$-stable, the (nondegenerate)
 symplectic space $W_4 = J_4/J_{4+}$ 
inherits an involution from $\theta$ and we use the notation $\theta$ for this involution of $W_4$.  There is an associated polarization $W_4 = W_4^+ + W_4^-$ defined by
\begin{equation*}
\begin{split}
W_4^+&= \{\, w\in W_4\ | \ \theta (w) = w\,\}\cr
W_4^-&=\{\, w\in W_4\ | \ \theta (w)= w^{-1}\,\}.
\end{split}
\end{equation*}

\begin{lemma}\label{lemnine} $W_4^+= J^\theta J_{4+}/J_{4+}$.
\end{lemma}

\begin{proof}  Let $W_4^{++} = J^\theta J_{4+}/J_{4+}$.  Clearly, $W_4^{++}$ is contained in  $W_4^+$.
Now suppose $w\in W_4^+$ and choose $h\in J_4$ such that $w=hJ_{4+}$.
  The condition $\theta (w) = w$ means that $h^{-1}\theta (h)\in J_{4+}$.  Thus $h^{-1}\theta (h)\in Z^1_{\theta} (J_{4+})$, in the notation of Section~\ref{sec:stablesubgroups}.  According to Proposition \ref{twodivprop}, there exists $k\in J_{4+}$ such that $h^{-1}\theta (h) = k^{-1}\theta (k)$.  The element $h^\prime = hk^{-1}$ lies in $J^\theta$ and is such that $h^\prime J_{4+} = hJ_{4+}$.  It follows that $W_4^{++} = W_4^+$.
  \end{proof}

Now define $W^+_2$, $W^-_2$, $W^+$ and $W^-$ as in the statement of Proposition \ref{canpolar}.  In particular, $W^+ = J^\theta J_+/J_+$ and $W^- = W^-_2\oplus W_3$.  Let $W^+_0 = W_1\oplus W_2^+$.

\begin{lemma}\label{polaryes}
$W= W^+_0+W^-$ is a polarization of $W$.
\end{lemma}

\begin{proof}
We first note that $W_2 = W_2^++W_2^-$ is a polarization, according to Lemma \ref{WtwoWfour} and the fact that $W_4=W_4^++ W_4^-$ is a polarization of $W_4$.  
Our claim now follows from Lemma \ref{WonetwothreeA} and Lemma \ref{WonetwothreeB} (3).
\end{proof}

Let $J_{4++}=J_5\cap\theta(J_5)$.

\begin{lemma}\label{wtwopluswplus}
$W^+_2\subset W^+$.
\end{lemma}

\begin{proof}
We first show, in the notation of Section~\ref{sec:stablesubgroups}, that
$$Z^1_\theta (J_{4+}) = B^1_\theta (J_5) = B^1_\theta (\theta (J_5)).$$  
Suppose $u\in J_5$, $v\in \theta (J_5)$ and $uv \in Z^1_\theta (J_{4+})$.  
Then $v\theta (u) = u^{-1}\theta (v)^{-1} \in  J_{4++}$.  
Let $w= v\theta (u)$.  Then $\theta (w) = w^{-1}$ and so, by Proposition \ref{twodivprop}, there exists $w_1\in J_{4++}$ such that $w = w_1 \theta (w_1)^{-1}$.  We now observe that $uv = uw_1\theta (w_1)^{-1} \theta (u)^{-1} = uw_1 \theta (uw_1)^{-1}$ and that $u\in J_5$, $w_1\in J_{4++}$ implies $uw_1\in J_5$.  We deduce that $Z^1_\theta (J_{4+}) = B^1_\theta (J_5)$.  One similarly shows that $Z^1_\theta (J_{4+}) = B^1_\theta (\theta (J_5))$.

We now apply the identities just obtained to show $W^+_2\subset W^+$.  
Suppose $k\in J_2$ and $kJ_+\in W^+_2$.  Then $kJ_{4+}\in W^+_4$ and so 
$k^{-1}\theta (k) \in Z^1(J_{4+}) = B^1_\theta (\theta (J_5))$.  
There exists $k_1\in \theta (J_5)$ such that $k^{-1}\theta (k) = k_1 \theta (k_1)^{-1}$.
  Since $kk_1\in J^\theta$ and $k_1\in \theta (J_5)\subset J_+$, 
we have $kJ_+ = kk_1J_+ \subset J^\theta J_+$. Hence $Z_\theta^1(J_{4+})\subset
B_\theta^1(J_5)$.
\end{proof}

\begin{lemma}\label{WoneWthree}
$|W_1| = |W_3|$.
\end{lemma}

\begin{proof}
Our assertion is equivalent to the assertion that
$$[J_4J_+ : J_+] [J_5J_+:J_+] = [J:J_+],$$ since Lemma \ref{WonetwothreeB}(1)
 implies $[J_5J_+:J_+] = |W_1|$ and parts (2) and (3) of
 Lemma \ref{WonetwothreeB} imply that
$|W_3|  = [J:J_+]/[J_4J_+:J_+]$.  The same argument used to prove 
Lemma 12.8 of \cite{Y} now finishes the proof.
\end{proof}

\begin{lemma}\label{lemten} $W^+ = W^+_0$.
\end{lemma}

\begin{proof}
 We have a homomorphism of abelian groups 
$$f: J_5/J_{4++}\to J_{4+}/J_{4++}$$ given by $$f(kJ_{4++}) = k\theta (k) J_{4++}.$$

We claim, first of all,  that the image of $f$ is $J^\theta_{4+}J_{4++}/J_{4++}$.  
Suppose $uv\in J_{4+}^\theta$, with $u\in J_5$ and $v\in \theta(J_5)$.  We observe  that $\theta (v) (v^{-1}u^{-1}v)= \theta (u)^{-1}v$ and, moreover, $\theta (v) (v^{-1}u^{-1}v)\in J_5$ and $\theta (u)^{-1} v\in \theta (J_5)$.  It follows that $\theta (u)^{-1} v\in J_5\cap \theta (J_5) = J_{4++}$.   Hence $uv = u\theta(u)\theta(u)^{-1} v\in u\theta (u)J_{4++}$.  
It follows that 
$$
J^\theta_{4+}\subset \{\, k\theta (k) J_{4++} \ | \ k\in J_5\,\}J_{4++}.
$$

Next, we observe that if $k\in J_5$ and $\ell= k\theta (k)$ then $\ell^{-1}\theta (\ell)\in Z^1_\theta (J_{4++})$, in the notation of Section \ref{sec:stablesubgroups}.  Thus, according to Proposition \ref{twodivprop}, there exists $m\in J_{4++}$ such that $\ell^{-1}\theta (\ell)= m\theta (m)^{-1}$.  Since $\ell m\in J^\theta_{4+}$, we have shown that the image of $f$ is contained in $J^\theta_{4+}J_{4++}/J_{4++}$ and thus  the image of $f$ is indeed $J^\theta_{4+}J_{4++}/J_{4++}$.

We now observe that $f$ is injective.  
Indeed, if $k\in J_5$ and $k\theta (k)\in J_{4++}$ then both of these elements lie in $\theta (J_+)$.  It follows that $k$ lies in $J_+\cap \theta (J_+) = J_{4++}$.

Therefore, $f$ defines an isomorphism
$$J_5/J_{4++}\cong J^\theta_{4+}J_{4++}/J_{4++}.$$ Hence, we have an isomorphism 
$$J_5/J_{4++}\cong J^\theta_{4+} /J^\theta_+.$$

We now compute the size of $W^+$:
\begin{equation*}
\begin{split}
|W^+| &= [J^\theta J_+:J_+] = [J^\theta : J^\theta_+]\cr
&=[J^\theta: J^\theta_{4+}] [J^\theta_{4+}: J^\theta_+].
\end{split}
\end{equation*}
Now Lemma \ref{lemnine} implies that $$[J^\theta: J^\theta_{4+}] = |W_2|^{1/2}.$$  
On the other hand, as shown above, $J_5/J_{4++}\cong J_{4+}^\theta/J_+^\theta$.
Putting this together with Lemma~\ref{WonetwothreeB}(1), 
Lemma~\ref{WoneWthree}, and the isomorphisms 
$J_{4+}/J_5\cong \theta (J_5)/J_{4++}\cong J_5/ J_{4++}$, we obtain
\begin{equation*}
\begin{split}
[J^\theta_{4+}:J^\theta_+]  &= [J_5:J_{4++}] = [J_{4+}:J_5]\cr
&=|W_1| = |W_3|.
\end{split}
\end{equation*}
Hence $$|W^+| = |W_1|^{1/2} |W_2|^{1/2} |W_3|^{1/2} = | W|^{1/2}= |W^+_0|.$$  

To show that $W^+_0 = W^+$, it now suffices to show that $W^+_0\subset W^+$ since we have finite sets of the same size.
But $W^+_0 = W_1 \oplus W^+_2$ and we have already shown, in Lemma \ref{wtwopluswplus}, that $W^+_2\subset W^+$.  So it is enough to show that $W_1\subset W^+$.
Suppose $w\in W_1$.  Then, according to 
Lemma \ref{WonetwothreeB}(1), we can find $k\in J_5$ such that $w=kJ_+$.  Our determination of the image of $f$ implies that there exists $\ell\in J^\theta_{4+}$ such that $k\theta (k)J_{4++} = \ell J_{4++}$.  It follows that $w=kJ_+= k\theta (k) J_+ =\ell J_+\in W^+$.
\end{proof}

Now define $W^-_*$ as in the statement of Proposition \ref{canpolar} and let $W^-_0 = W_1\oplus W^-_2$.  The next result is analogous to Lemma \ref{lemten}.  The proof is also formally similar in some ways, but there are some key differences.

\begin{lemma}\label{lemtenminus} $W^-_0 = W^-_*$.
\end{lemma}

\begin{proof}
To show $W^-_*\subset W^-_0$, it suffices to show that $W^-_* $ is orthogonal to $W^-_0$, since $W^{-\perp}_0 = W^-_0$.  Equivalently, we need to show that $W^-_*$ is orthogonal to both $W_1$ and $W^-_2$.  But $W_1$ is orthogonal to $W^-_*$, since, according to Lemma \ref{WonetwothreeB} and Lemma \ref{radlem}, $W_1 = J_{4+}J_+/J_+$ is the radical of $W_1\oplus W_2$.  Moreover, $W^-_*$ is orthogonal to $W^-_2$ since if $u\in Z^1_\theta (J)$ and $v\in J_2$ is such that $vJ_+\in W^-_2$ then 
$\langle u,v\rangle = \langle \theta (v),\theta (u)\rangle  = 
\langle v^{-1},u^{-1}\rangle = \langle v,u\rangle$, according to Lemma \ref{zetasym},
and the fact that $Z_\theta^1(J)\subset J_4$.

It remains to show $W^-_0\subset W^-_*$.  We first show that $W_1\subset W^-_*$.  Imitating the proof of Lemma \ref{lemten}, we define a homomorphism of abelian groups $$f: J_5/J_{4++}\to J_{4+}/J_{4++}$$ by
$$f(kJ_{4++}) = k\theta (k)^{-1}J_{4++}.$$  In the proof of Lemma \ref{wtwopluswplus}, we showed that $B^1_\theta (J_5) = Z^1_\theta (J_{4+})$.  It follows that the image of $f$ is $Z^1_\theta (J_{4+})J_{4++}/J_{4++}$.  Just as in the proof of Lemma \ref{lemten}, we can show that $f$ is injective and thus we deduce that $f$ yields an isomorphism
$$J_5/J_{4++}\cong Z^1_\theta (J_{4+})J_{4++}/J_{4++}.$$
Turning to the proof that $W_1\subset W^-_*$, we assume $w\in W_1$ and choose $k\in J_5$ such that $w=kJ_+$.  Then $f(kJ_{4++}) = k\theta (k)^{-1}J_{4++}$.
 So $w= kJ_+ = k\theta (k)^{-1}J_+$.  We have therefore shown that $W_1\subset W^-_*$.

To complete the proof, we show that $W^-_2\subset W^-_*$.  Let $w\in W^-_2$ and choose $h\in J_2$ such that $w=hJ_+$.  Using Proposition~\ref{twodivprop}, we see that there exists a unique element $\ell\in J_2$ such that $\ell^2 =h$.  Now let $w_0 = hJ_{4+}$, $w_- = \ell\theta (\ell)^{-1}J_{4+}$ and $w_+ = \theta (\ell )\ell J_{4+}$.  We then observe that $w_0 = w_-w_+$, with $w_0,w_-\in W^-_4$ and $w_+\in W^+_4$.  It follows that $w_+=1$ and $w_0 = w_-$.  Hence $w= \ell\theta (\ell)^{-1}J_+ \in W^-_*$.
\end{proof}

\begin{proof}[Proof of Proposition \ref{canpolar}] It only remains to show that $W^+$, $W^-$, $W^+_*$ and $W^-_*$ are stable under $f^\prime(K^{\prime,\theta})$.  For $W^+$ and $W^-_*$, this follows immediately from their definitions.  We will show that $f^\prime(K^{\prime,\theta})$ stabilizes each of the spaces $W_1$, $W_2$ and $W_3$.  Then it will follow that $W_2^+ = W^+\cap W_2$ and $W_2^-= W^-_*\cap W_2$ are $f^\prime(K^{\prime,\theta})$-stable.  In addition, the $f^\prime(K^{\prime,\theta})$-stability of $W^- = W_2^-\oplus W_3$  and $W^+_*= W_2^+ \oplus W_3$ will also follow.

Let $k\in K^{\prime\theta}$. We want to show that $\Int(k)(J_iJ_+)\subset J_iJ_+$,
$i=1$, $2$, $3$. It is enough to prove that $\Int k$ takes
a set of generators of $J_i$ into $J_iJ_+$.
Observe that $\Int(k)\bT(E)_r\subset \Int (k)\bG^\prime(E)_{y,r}
=\bG^\prime(E)_{y,r}\subset J_+$, and, if $a\in \Phi^\prime$,
$\Int (k)\bU_a(E)_{y,r}\subset \bG^\prime(E)_{y,r}\subset J_+$.
Hence, to prove that $W_i$ is $f^\prime(K^{\prime\theta})$-stable,
it suffices to show that if $k\in K^{\prime\theta}$
and $a\in \Phi_i^{\prime\prime}$, then
$\Int (k)\bU_a(E)_{y,s}\subset J_iJ_+$, $i=1$, $2$, $3$.
This is a consequence of the following.
Since $k\in K^{\prime\theta}$, we have $k\cdot y=y+ z$
for some $z\in X_*(\bZ,F)\otimes \R$. Thus
$a(k\cdot y - k\cdot\theta(y))=a(k\cdot y -\theta(k\cdot y))
=a(y-\theta(y))$, as $z$, $\theta(z)\in X_*(\bZ,F)\otimes \R$.
\end{proof}

Let $f_0=f_{r,s}:\Phi\cup\{0\}\rightarrow \R$ be the concave function on
$\Phi\cup \{ 0\}$ defined by
$f_0(a)=r$ for $a\in \Phi^\prime\cup \{0\}$
and $f_0(a)=s$ for $a\in \Phi\setminus\Phi^\prime$. Then $J(E)=\bG (E)_{y,f_0}$.
Recall that the construction of $\bG(E)_{y,f_0}$ involves
a choice of maximal $F$-torus $\bT\subset\bG^\prime$ such
that $y\in A(\bG^\prime,\bT,E)$. 
However, as shown in \cite{Y}, the group $\bG(E)_{y,f_0}$
is independent of the choice of $\bT$.
This fact will be used below.

Now let 
\begin{equation*}
\begin{split}
\Phi_\theta &= \Phi (\bG ,\theta (\bT))= \{\, a\circ\theta \ | \ a\in \Phi\,\}\cr
\Phi^\prime_\theta &= \Phi (\bG^\prime,\theta (\bT))= \{\, a\circ\theta \ | \ 
a\in \Phi^\prime\,\}
\end{split}
\end{equation*}
Recall that we have
chosen $\bT$ so that $y$ belongs to
$A(\bG^\prime,\bT,E)\cap A(\bG^\prime ,\theta(\bT),E)$.  
This allows us to realize $J(E)$ via a concave function
on $\Phi_\theta\cup \{0\}$.
Define $f_0^\theta$ on $\Phi_\theta
\cup \{0\}$
taking the value $r$ on $\Phi^\prime_\theta
\cup \{0\}$ and $s$ on $\Phi_\theta \setminus
\Phi^\prime_\theta$. Then
$\bG(E)_{y,f_0^\theta}=\bG(E)_{y,f_0}=J(E)$.

Choose a system $\Phi^+$ of positive roots in $\Phi$
and put $\Phi^-=\Phi\setminus\Phi^+$.
Then the set $\Phi_\theta^+ =\{\, a\circ\theta\ | \ a\in \Phi^+\,\}$
is a system of positive roots in $\Phi_\theta$.
Let $J(E)(+)$ and $J(E)(-)$
be the subgroups of $J(E)$ generated by 
the groups $\bU_a(E)_{y,s}$, for $a\in \Phi^+$, $a\notin
\Phi^\prime$, and for $a\in \Phi^-$, $a\notin \Phi^\prime$,
respectively.
Similarly, let $J(E)_\theta(+)$ and $J(E)_\theta(-)$
be generated by the groups $\bU_{a\circ\theta}(E)_{y,s}$
where $a\circ\theta\in\Phi_\theta^+$, 
$a\circ\theta\notin
\Phi^\prime_\theta$,  and
$a\circ\theta\in \Phi_\theta\setminus\Phi_\theta^+$,
$a\circ\theta\notin \Phi^\prime_\theta$, respectively.

Recall 
 that we are denoting the special isomorphism
associated to the split polarization of $\cH(E)$
defined in Section~\ref{sec:genHeis} by $\nu^\bullet_E$.
Let $\mu:J(E)/N(E)\rightarrow J_+(E)/N(E)$ be the function
such that $$\nu^\bullet_E (kN(E))=(kJ_+(E),\mu(kN(E))),$$
for all $k\in J(E)$. 

\begin{lemma}\label{lemonegap}   If $k$ lies in $J(E)(+)\cup J(E)(-)\cup J(E)_\theta(+)\cup
J(E)_\theta(-)$ then $\mu(kN(E))=1$.
\end{lemma}

\begin{proof}
Suppose that $k\in J(E)(+)\cup J(E)(-)$. Then,
defining $\nu^\bullet_E$ relative to $\bT$
and $\Phi^+$, we have $\nu^\bullet_E (kN(E))=(kJ_+(E),1)$.
For $k\in J(E)_\theta(+)\cup J(E)_\theta(-)$,
defining $\nu^\bullet_E$ relative to $\theta(\bT)$ and
$\Phi_\theta^+$, we have $\nu^\bullet_E (kN(E))
=(kJ_+(E),1)$.
\end{proof}

From the description of $J(E)$ as $\bG(E)_{y,f_0^\theta}$,
for the concave function $f_0^\theta$ on $\Phi_\theta$,
it follows that 
the group $\theta(J(E))$ is generated by
$\theta(\theta(\bT)(E)_r)=\bT(E)_r$ and by 
the groups $\theta(\bU_{a\circ\theta}(E)_{y,r})$, with
$a\in \Phi^\prime$, and $\theta(\bU_{a\circ\theta}(E)_{y,s})$, with 
$a\in \Phi\setminus\Phi^\prime$. 
Note that
$$
\theta(\bU_{a\circ\theta}(E)_{y,t})=\bU_a(E)_{\theta(y),t}
=\bU_a(E)_{y,t+a(y-\theta(y))},
\qquad t\in \R,\ a\in \Phi.
$$
Hence, defining a concave function $f_\theta$ on
$\Phi\cup \{0\}$ by $f_\theta(a)=f_0(a)+a(y-\theta(y))$,
$a\in \Phi$, and $f_\theta(0)=f_0(0)=r$, we see
that $\theta(J(E))=\bG(E)_{y,f_\theta}$.

It now follows from part (ii) of Lemma~13.2 \cite{Y} that
$J(E)\cap \theta(J(E))=\bG(E)_{y,h}$,
where $h(0)=r$ and 
$$h(a)={\rm max}(f_0(a),f_\theta(a))=
\begin{cases}
r, &\text{if $a\in \Phi_1^\prime\cup \Phi_2^\prime$,}\cr
r+a(y-\theta(y)), &\text{if $a\in \Phi_3^\prime$,}\cr
s, &\text{if $a\in \Phi_1^{\prime\prime}\cup 
           \Phi_2^{\prime\prime}$,}\cr
s+ a(y-\theta(y)), &\text{if $a\in \Phi_3^{\prime\prime}$.}
\end{cases}$$

\begin{lemma}\label{lemtwogap}
If $k\in J(E)\cap\theta(J(E))$ then $\mu(kN(E))=\mu(\theta(k)N(E))^{-1}$.
\end{lemma}

\begin{proof}
Since $h(0)=r>0$ and $h$ is concave,
Proposition  6.4.48 \cite{BT1} provides a bijective map
$$
\prod_{a\in \Phi^+} \bU_a(E)_{y,h(a)} \times \bT(E)_r
\times \prod_{a\in \Phi^-} \bU_a(E)_{y,h(a)}
\rightarrow \bG(E)_{y,h}=J(E)\cap\theta(J(E))
$$
once one specifies an ordering of the factors in the products. Let $k\in J(E)\cap\theta(J(E))$.
With a suitable ordering of the products, we obtain an expression of $k$ in the form
$k=k_+k^\prime k_-$, where
\begin{equation*}
\begin{split}
& k_+\in \prod_{a\in \Phi^+,\,a\notin \Phi^\prime} \bU_a(E)_{y,h(a)}\cr
& k_-\in \prod_{a\in \Phi^-,\,a\notin \Phi^\prime} \bU_a(E)_{y,h(a)}\cr
& k^\prime\in \prod_{a\in \Phi^+\cap \Phi^\prime} \bU_a(E)_{y,h(a)}
\times \bT(E)_r\times \prod_{a\in \Phi^-\cap \Phi^\prime} \bU_a(E)_{y,h(a)}.\cr
\end{split}
\end{equation*}
Since $h(a)\ge s$ for all $a\in \Phi\setminus \Phi^\prime$,
it follows that $k_+\in J(E)(+)$ and $k_-\in J(E)(-)$. 
Because $h(a)\ge r$ for all $a\in \Phi^\prime$ and
$\bT(E)_r\subset \bG^\prime (E)_{y,r}\subset J_+(E)$, we have
$k^\prime \in J_+(E)$. Applying Lemma \ref{lemonegap}, we obtain
\begin{equation*}
\begin{split}
\nu^\bullet_E (kN(E))&=(k_+ J_+(E),1)(1, k^\prime N(E))(k_-J_+(E),1)\cr
&
=(k_+ J_+(E) + k_-J_+(E),
[k_+,k_-]^{(p+1)/2} k^\prime N(E)).
\end{split}
\end{equation*}
Because $k^\prime\in J(E)\cap \theta(J(E))\cap \bG^\prime (E)_{y,r}
=\bG^\prime (E)_{y,r}\cap\theta(\bG^\prime (E)_{y,r})$,
we have $\theta(k^\prime)\in \bG^\prime (E)_{y,r}\subset J_+(E)$.
Hence,
$\nu^\bullet_E (\theta(k^\prime))=(1,\theta(k^\prime)N(E))$.

Suppose that $a\in \Phi\setminus\Phi^\prime$. Then
\begin{equation*}
\begin{split}
\theta(\bU_a(E)_{y,h(a)})&=\bU_{a\circ\theta}(E)_{\theta(y),h(a)}
=\bU_{a\circ\theta}(E)_{y, h(a)+ (a\circ\theta)(y-\theta(y))}\cr
&=\bU_{a\circ\theta}(E)_{y, h(a)- a(y-\theta(y))}.
\end{split}
\end{equation*}
By definition of $h(a)$, $h(a)- a(y-\theta(y))\ge s$
for all $a\in \Phi\setminus\Phi^\prime$.
Therefore $\theta(\bU_a(E)_{y,h(a)})\subset
\bU_{a\circ\theta}(E)_{y,s}$ whenever
$a\in \Phi\setminus\Phi^\prime$.
Hence $\theta(k_+)\in J(E)_\theta(+)$ and
$\theta(k^-)\in J(E)_\theta(-)$. 

Applying Lemma~\ref{lemonegap}, we have
\begin{equation*}
\begin{split}
\nu^\bullet_E (\theta(k)N(E))&=(\theta(k_+)J_+(E),1)(1, \theta(k^\prime)N(E))
(\theta(k_-)J_+(E),1)\cr
&=(\theta(k_+)J_+(E) +\theta(k_-)J_+(E),
\theta([k_+,k_-])^{(p+1)/2}\theta(k^\prime)N(E)).
\end{split}
\end{equation*}
Hence $$\mu(\theta(k)N(E))= \theta ([k_+,k_-])^{(p+1)/2}
\theta(k^\prime)N(E).$$
Because $k_+$, $k_-\in J(E)\cap\theta(J(E))$ and $k^\prime
\in \bG^\prime (E)_{y,r}\cap\theta(\bG^\prime (E)_{y,r})$, both
$[k_+,k_-]$ and $k^\prime$ lie in $J_+(E)\cap\theta(J_+(E))$.
As is shown in Lemma~\ref{zetasym}, the characters $\zeta_E\circ\theta$
and $\zeta_E^{-1}$ agree on $J_+(E)\cap\theta(J_+(E))$.
Thus $\mu(\theta(k)N(E))=\mu(kN(E))^{-1}$.
\end{proof}

\begin{corollary}\label{corlemtwogap}
$\nu^\bullet (J^\theta N/N) = 
(J^\theta J_+/J_+ )\times \{1\}$.
\end{corollary}

\begin{proof} Let $k\in J(E)^\theta$.
Applying Lemma~\ref{lemtwogap}, we
see that $\mu(kN(E))=\mu(kN(E))^{-1}$.
Because $p$ is odd, this forces
$\mu(kN(E))=1$.  This implies that $\nu^\bullet (J^\theta N/N)
=(J^\theta J_+/J_+ )\times \{1\}$.
\end{proof}

\begin{proof}[Proof of Proposition \ref{corcanpolar}] 
Let $\cH^+=J^\theta N/N$ and $W^+=J^\theta J_+/J_+$.
Using Lemma~\ref{polarauto} and the polarization $W= W^+ +W^-$ from 
Proposition~\ref{canpolar}, we may define an automorphism 
$\alpha$ of $\cH$ by 
$$
\alpha ((\nu^{\bullet})^{-1}(w_++w_-,z))= (\nu^\bullet)^{-1}(w_+-w_-,-z),
\qquad w_+\in W^+,\ w_-\in W^-,\ z\in \cZ.
$$
Corollary~\ref{corlemtwogap} tells us that $\nu^\bullet(\cH^+)=W^+\times 1$.
Hence we have $\cH^+=\{\, h\in \cH\ | \ \alpha(h)=h\,\}$
and the hypotheses
of Theorem~\ref{Heisthm} are satisfied (relative to the involution
$\alpha$ and Yu's special isomorphism $\nu^\bullet$).
The fact that $f^\prime(K^{\prime \theta})$ is contained in $\cP$ is a consequence of 
Proposition~\ref{canpolar}.
Our claims now follow directly from  Theorem~\ref{Heisthm}.
\end{proof}

\section{The inductive structure}
\label{sec:inductive}

Fix an extended generic cuspidal $G$-datum $\Psi = (\vec\bG ,y,\rho, \vec\phi)$.  We will say that $\Psi$ has \textit{degree} $d$ if $\vec\bG$ and $\vec\phi$ each have $d+1$ components.  Yu's construction associates to $\Psi$ a tame supercuspidal representation $\pi = \pi (\Psi)$ of $G =\bG(F)= \bG^d (F)$.  When $d>0$, we can define a cuspidal $G^{d-1}$-datum $\partial \Psi = (\partial\vec\bG, y,\rho, \partial\vec\phi)$ of degree $d-1$ by letting
\begin{equation*}
\begin{split}
\partial\vec\bG &= (\bG^0,\dots , \bG^{d-1})\cr
\partial\vec\phi&= (\phi_0,\dots , \phi_{d-1}).
\end{split}
\end{equation*}
Similarly, given a reduced cuspidal  $G$-datum $(\vec\bG ,\pi_{-1},\vec\phi)$, there is an analogous notion of degree and if this degree is positive, we let $\partial (\vec\bG,\pi_{-1}, \vec\phi) = (\partial\vec\bG, \pi_{-1}, \partial\phi)$.

Associated to $\partial\Psi$ is a supercuspidal representation $\partial\pi = \pi_{d-1}$ of the group $G^{d-1}$.  Continuing in the manner, we obtain a sequence of representations $\vec\pi = (\pi_0,\dots , \pi_d)$ such that $\pi_{d-i} = \partial^i \pi$.

Consider now how the $\partial$ operation affects the inducing representations.  As discussed in the Section \ref{sec:construction}, the inducing representation $\kappa$ for $\pi$ has a tensor product decomposition $\kappa = \kappa_{-1}\otd \kappa_d$.  The inducing representation $\partial\kappa$ for $\partial\pi$ has a similar decomposition.  If $d>0$ then
$$\partial\kappa = \partial\kappa_{-1}\otd \partial\kappa_{d-1},$$
where $\partial \kappa_{d-1} = \phi_{d-1}\,|\,K^{d-1}$ and, otherwise,
$$\partial\kappa_i = \inf\nolimits_{K^{i+1}}^{K^{d-1}}(\phi^\prime_i) = \kappa_i\,|\,K^{d-1},$$  where $\phi^\prime_{-1}=\rho$.

To complete the definition of the $\partial$ operation, we now treat the case of degree zero.  If $\Psi$ has degree zero then we define $\partial\Psi$ to be identical to $\Psi$ except that $\vec\phi$ is replaced by $\partial\vec\phi = (1)$, the sequence consisting of the trivial character of $G=G^0$.  Therefore, when $r_0=0$ we have $\partial\Psi = \Psi$ and $\partial\pi=\pi$.

Note that with our definitions if one repeatedly applies $\partial$ to an inducing representation $\kappa$, one eventually obtains $\rho$.  Moreover, $\partial \rho = \rho$.  Regarding the image of the $\partial$ operation, we remark that if $\Psi$ has positive degree and $r_d= r_{d-1}$ then $\Psi$ cannot have the form $\partial \Xi$ for some other cuspidal  $G$-datum $\Xi$.

The purpose of the $\partial$ formalism is to provide a tool for proofs involving induction on the degree of a cuspidal $G$-datum.  In the remainder of this section, we illustrate this principle with a specific example that is relevant to the main theme of this \paperbook.

\begin{definition}\label{defquaddist}
The representation $\kappa$ is \textit{quadratically distinguished (with respect to
$\theta$)} if there exists a character $\xi$ of $K^\theta = K\cap G^\theta$ such that 
$\xi^2=1$ and $\Hom_{K^\theta}(\kappa,\xi)\ne 0$.
\end{definition} 

Recall from Definition \ref{defweaksymdatum} that a cuspidal $G$-datum $\Psi$  is said to be ``weakly $\theta$-symmetric'' if $\theta(\bG^i)=\bG^i$ and $\phi_i\circ\theta = \phi_i^{-1}$, for all $i$.

\begin{proposition}\label{quaddistprop} Assume $\Psi$ is a weakly $\theta$-symmetric cuspidal $G$-datum and $\kappa = \kappa (\Psi)$.
If  $\kappa$ is quadratically distinguished then $\partial\kappa$ is quadratically distinguished.  Consequently, $\rho$ is quadratically distinguished.
\end{proposition}

\begin{proof}
Fix a character $\xi$ of $K^\theta$ such that $\xi^2=1$ and a nonzero linear form $\lambda\in \Hom_{K^{\theta}} (\kappa,\xi)$.  
Assume first that $d=0$.  Then since $\phi_0^{-1} = \phi_0\circ \theta$, it follows that 
$(\phi_0\,|\,K^\theta)^2=1$.  Letting $\xi^\prime = \xi(\phi_0\, |\, K^\theta)$, we obtain a character $\xi^\prime$ of $K^\theta$ such that $(\xi^\prime)^2=1$ and $\lambda\in \Hom_{K^\theta}(\rho, \xi')$.  In other words, $\rho = \partial\kappa$ is quadratically distinguished.

Now assume $d>0$. 
Recall that we are denoting the space of $\kappa_i$ by $V_i$ for $i<d$.
 Given $v_{-1}\in V_{-1},\dots , v_{d-2}\in V_{d-2}$, we may define $\Lambda\in \Hom (V_{d-1},\C)$ by
$$\Lambda(v_{d-1})= \lambda (v_{-1}\otd v_{d-1}).$$  Note that $\xi\,|\, J^{d,\theta}=1$, since if $\alpha$ is a character of a pro-$p$-group with $p$ odd and $\alpha^2=1$, it must be the case that $\alpha=1$.  
We claim that $\Lambda$ must lie in $\Hom_{J^{d,\theta}} (\tau_{d-1} ,1)$.  
Indeed, using the fact that  $\kappa_i\, |\,J^d=1$ when $i<d-1$ 
(since $\kappa_i = \inf\nolimits_{K^{i+1}}^K(\phi'_i)$) and $\kappa_{d-1}
\,|\, J^d=\tau_{d-1}$, we have for all $h\in J^{d,\theta}$
\begin{equation*}
\begin{split}
\Lambda(v_{d-1})&=\lambda (v_{-1}\otd v_{d-1})\cr
&= \lambda (\kappa(h)(v_{-1}\otd v_{d-1})) \cr
&= \lambda( \kappa_{-1}(h)v_{-1}\otd \kappa_{d-1}(h) v_{d-1})\, \phi_d(h)\cr
&= \lambda (v_{-1}\otd v_{d-2}\otimes \kappa_{d-1}(h) v_{d-1})\, \phi_d(h)\cr
&= \Lambda (\tau_{d-1}(h)v_{d-1}).
\end{split}
\end{equation*}

Let $W_{d-1}^+=J^{d,\theta}J_+^d/J_+^d$ and let
 $$
\cP_{d-1} = \{\, s\in \cS_{d-1}\ | \ s\cdot W^+_{d-1}\subset W^+_{d-1}\,\}.
$$
Note that Proposition~\ref{corcanpolar} tells us that $f^\prime_{d-1}(K^{\prime,
\theta})\subset \cP_{d-1}$.
Define a character $\alpha$ of $K^{d-1,\theta}$ by
$$
\alpha (k) = \phi_{d-1}(k) \chi^{\cP_{d-1}}(f^\prime_{d-1}(k)), \qquad k\in K^{d-1,\theta},
$$ 
where $\chi^{\cP_{d-1}}$ is the unique character of $\cP_{d-1}$ of order
$2$. We observe that $\alpha^2=1$.

According to Proposition~\ref{corcanpolar}, the space 
$\Hom_{J^{d,\theta}}(\tau_{d-1},1)$ has dimension one and lies
inside $\Hom_{f^\prime_{d-1}(K^{\prime,\theta})}(\hat\tau_{d-1},\chi^{\cP_{d-1}})$.
Hence, fixing a nonzero linear form $\lambda_{d-1}$ in $\Hom_{J^{d,\theta}}(\tau_{d-1},1)$,
we have 
 $$
\lambda_{d-1}(\phi^\prime_{d-1}(k)v_{d-1}) = \alpha (k)\lambda_{d-1}(v_{d-1}),
$$ 
for all $k\in K^{d-1,\theta}$ and $v_{d-1}\in V_{d-1}$.  
 
Since $\Hom_{J^{d,\theta}}(\tau_{d-1},1)$ has dimension one, there exists a complex number\break $\partial\lambda (v_{-1}\otd v_{d-2})$ such that $$\lambda( v_{-1}\otd v_{d-1}) = \partial \lambda (v_{-1}\otd v_{d-2}) \  \lambda_{d-1}(v_{d-1}),$$ for all $v_{d-1}\in V_{d-1}$.  Clearly, $\partial \lambda$ defines a nonzero linear form on $\partial V = V_{-1}\otd V_{d-2}$. 
 Hence, for $k\in K^{d-1,\theta}$, we have
 \begin{equation*}
 \begin{split}
 \partial \lambda&(\partial\kappa (k)(v_{-1}\otd v_{d-2})) \ \alpha(k) \lambda_{d-1}(v_{d-1}) \cr
&=\partial \lambda (\kappa_{-1}(k)v_{-1}\otd \kappa_{d-2}(k) v_{d-2}) 
       \lambda_{d-1}(\phi^\prime_{d-1}(k)v_{d-1}) \cr
&=\lambda (\kappa_{-1}(k)v_{-1}\otd \kappa_{d-1}(k) v_{d-1})  \phi_{d-1}(k)\cr
&=\lambda(\kappa(k)(v_{-1}\otd v_{d-1}))\cr
&=\xi(k) \lambda(v_{-1}\otd v_{d-1})\cr
&=\partial\lambda(v_{-1}\otd v_{d-2}) \xi(k)\lambda_{d-1}(v_{d-1}).
\end{split}
\end{equation*}
 Hence, $\partial\lambda\in \Hom_{K^{d-1,\theta}}(\partial\kappa, \alpha\xi)$.  Since 
$(\alpha \xi\,|\,K^{d-1,\theta})^2=1$, we have shown that $\partial\kappa$ is quadratically distinguished.  Applying the $\partial$ operator repeatedly, we deduce that $\rho$ must also be quadratically distinguished.
\end{proof}

\section{Refactorization of cuspidal $G$-data}
\label{sec:refactorization}

As discussed in Section~\ref{sec:Howeconstruction},
 Howe's construction  associates a tame
supercuspidal representation of ${\bf GL}_n(F)$ to each
$F$-admissible quasicharacter of the multiplicative
group of a tamely ramified degree $n$ extension
of $F$. An essential technical element of Howe's
theory is that $F$-admissible quasicharacters
have certain useful factorizations, known as
Howe factorizations (Definition \ref{howefact}).
A fixed $F$-admissible quasicharacter has
different Howe factorizations, all of which give rise to equivalent supercuspidal
representations. There are standard
procedures for adusting a Howe factorization
to obtain another Howe factorization that is more 
convenient for use in a given application.
In this section, we explain how to analogously alter
a generic cuspidal $G$-datum without changing
the equivalence class of the associated tame supercuspidal
representation.

\subsection*{Definitions} Assume $\Psi= (\vec\bG, y,\rho, \vec\phi)$ is a generic cuspidal $G$-datum.   Throughout this section, we assume that $\vec\bG$ satisfies Hypothesis C($\vec\bG$).  Define a quasicharacter $\phi = \phi (\Psi)$ of $G^0$ by
$$\phi (g) = \prod_{i=0}^d \phi_i (g),\quad g\in G^0.$$
As usual, we take $\pi_{-1} = \ind_{K^0}^{G^0} (\rho)$ and, in addition, we let 
\begin{eqnarray*}
\rho^\prime &=& \rho \otimes (\phi\, |\, K^0)\cr
\pi^\prime_{-1} &=& \pi_{-1}\otimes \phi.
\end{eqnarray*}
Note that $\pi^\prime_{-1}\simeq \ind_{K^0}^{G^0} (\rho^\prime)$.  We also use the notation $\vec r$ for the depth sequence
in Condition \textbf{D3} of Definition \ref{definddat}. (We remark that
the notation $\rho_i^\prime$ is used in Section~4 of \cite{Y},
and should not be confused with our $\rho^\prime$ notation.)

Suppose now that we have another sequence $\vec{\dot\phi} = (\dot\phi_0,\dots ,\dot\phi_d)$  and another representation $\dot\rho$ of $K^0 = K^0 (\Psi)$ associated to the same $\vec\bG$ and $y$.   Define $\dot\phi$, $\dot\pi_{-1}$, $\dot\rho^\prime$ and $\dot\pi^\prime_{-1}$  by analogy with $\phi$, $\pi_{-1}$, $\rho^\prime$ and $\pi^\prime_{-1}$.  We do not explicitly assume that the 4-tuple $\dot\Psi = (\vec\bG , y,\dot\rho, \vec{\dot\phi})$ is a generic cuspidal $G$-datum, however, this will be a consequence of the conditions we impose on $\dot\Psi$ below.

For each $i\in \{\, 0,\dots , d\,\}$, we define a quasicharacter $\chi_i = \chi_i (\Psi, \dot\Psi)$ of $G^i$ by
$$\chi_i (g) = \prod_{j=i}^d\phi_j(g)\dot\phi_j(g)^{-1},\quad g\in G^i.$$
The notion of refactorization is now defined as follows:

\begin{definition}\label{defrefactor}
If $(\dot\rho,\vec{\dot\phi})$ satisfies the conditions
\begin{itemize}
\item[\textbf{F0.}]  if $\phi_d =1$ then $\dot\phi_d =1$,

\item[\textbf{F1.}]  
$\dot\phi_i\,|\, G_{y,r_{i-1}^+}^i=\phi_i\chi_{i+1}\,|\,
G_{y,r_{i-1}^+}^i$ for all $i$, where $r_{-1}=0$ and $\chi_{d+1}=1$,\hfill\break(in other words, $\chi_i\, |\, G^i_{y,r^+_{i-1}} =1$)

\item[\textbf{F2.}] $\dot\rho = \rho \otimes (\chi_0\, |\, K^0)$,\quad
(in other words, $\dot\rho^\prime = \rho^\prime$)
\end{itemize}
 then we say it is a \textit{refactorization of} $(\rho,\vec\phi)$.  We also say $\dot\Psi= (\vec\bG,y,\dot\rho,\vec{\dot\phi})$ is a 
\textit{refactorization of} $\Psi= (\vec\bG, y,\rho, \vec\phi)$.  For reduced data, a similar 
definition of ``refactorization'' applies with \textbf{F2} replaced by the condition $\dot\pi_{-1} = \pi_{-1} \otimes \chi_0 $ or, equivalently, the condition
 $\dot\pi^\prime_{-1} = \pi^\prime_{-1}$.
 \end{definition}

\begin{remark}\label{oldeffthree} Recall that 
$$
J_+^{i+1}=(G^i,G^{i+1})_{y,(r_i,s_i^+)}=
              G_{y,r_i}^i(G^i,G^{i+1})_{y,(r_i^+,s_i^+)}.
$$
If $\eta$ is a character of
$G_{y,r_i}^i$ that is trivial on $G_{y,r_i^+}^i$, let 
$\inf\nolimits_{G^i_{y,r_i}}^{J^{i+1}_+}\eta$ denote the character 
of $J_+^{i+1}$ that agrees with $\eta$ on $G_{y,r_i}^i$
and is trivial on $(G^i,G^{i+1})_{y,(r_i^+,s_i^+)}$.
Because $\chi_{i+1}\,|\, G_{y,r_i^+}^{i+1}=1$ 
(see Condition~\textbf{F1}), the depth of $\chi_{i+1}$ is at most
$r_i$.
Hypothesis C($\vec\bG$) implies that $\chi_{i+1}\,|\,J^{i+1}_+$ 
is realized by an element of $\z^{i+1,*}_{-r_i}$, for all 
$i\in \{\, 0,\dots ,d-1\,\}$, and this implies that
$$\chi_{i+1}\, |\, J^{i+1}_+ = \inf\nolimits^{J^{i+1}_+}_{G^i_{y,r_i}} (\chi_{i+1}\, |\, G^i_{y,r_i}),$$ for all $i\in \{\, 0,\dots , d-1\,\}$.  This fact is used in our proofs.
\end{remark}

\medskip
\subsection*{Genericity}

In this subsection, we show that a refactorization of a generic cuspidal $G$-datum is also a generic cuspidal $G$-datum.  The first step is the following:

 \begin{lemma}\label{generictrivia}
 Let $(\bG^\prime,\bG)$ be a tamely ramified twisted Levi sequence.  Let $\z$ and $\z^\prime$
 be the centers of the Lie algebras of $G$ and $G^\prime$, respectively.  Suppose 
$\Gamma\in \z^{\prime,*}_{-r}$ is $G$-generic of depth $-r$ and $\gamma\in \z^*_{-r}$.  
Then $\dot\Gamma = \Gamma+\gamma$ is $G$-generic of depth $-r$.
 \end{lemma}

 \begin{proof}
 Let $\Phi= \Phi (\bG,\bT)$ and $\Phi^\prime = 
\Phi (\bG^\prime,\bT)$.  Assume $a\in \Phi\setminus\Phi^\prime$.  Since the element $H_a = d\check a (1)$ lies in $[{\bfr g},{\bfr g}]$ and $\gamma$ lies in $\z^*$, we have $\gamma (H_a)=0$.  Thus $\dot\Gamma (H_a) = \Gamma(H_a)$ and we deduce that Condition \textbf{GE1} of \cite{Y} is satisfied.  Next, we consider Condition \textbf{GE2} of \cite{Y}.  The statement of \textbf{GE2} is somewhat technical and involves a certain element denoted $\widetilde X^*$ in \cite{Y}.  Let $\widetilde\Gamma$ and $\widetilde{\dot\Gamma}$ denote the  analogues of $\widetilde X^*$ associated to $\Gamma$ and $\dot\Gamma$, respectively.  Condition \textbf{GE2} involves the isotropy subgroups of the Weyl group  $W(\Phi)$ associated to the elements $\widetilde\Gamma$ and $\widetilde{\dot\Gamma}$.  But for all $w\in W(\Phi)$, we have $w\widetilde\Gamma = \widetilde\Gamma$ if and only if $w\widetilde{\dot\Gamma} = \widetilde{\dot\Gamma}$.  Therefore, the isotropy subgroups associated to $\widetilde\Gamma$ and $\widetilde{\dot\Gamma}$ are the same and hence $\dot\Gamma$ must satisfy Condition \textbf{GE2}.
 \end{proof}

\begin{lemma}\label{effones}
If $\dot\Psi$ is a refactorization of the generic cuspidal $G$-datum $\Psi$ then
 $\dot\Psi$ must also be a generic cuspidal $G$-datum. Furthermore, $\vec{\dot r}
=\vec r$. \end{lemma}

\begin{proof} Conditions \textbf{D1} and \textbf{D2} in the definition of ``cuspidal $G$-datum'' 
are automatic for $\dot\Psi$.  

Suppose that $i\in \{\, 0,\dots,d-1\,\}$.
To prove that $\dot\phi_i$ is $G^{i+1}$-generic, we observe that, 
according to Condition \textbf{F1}, we have $\dot\phi_i\, |\,G^i_{y,r_i} = 
\phi_i \chi_{i+1}\,|\,G^{i}_{y,r_i}$ and, according to 
Remark~\ref{oldeffthree}, 
$\chi_{i+1}\,|\, G^i_{y,r_i}$ is realized by an element of $\z^{i+1,*}_{-r_i}$.
Genericity now follows from Lemma~\ref{generictrivia}. Note that it also
follows that $\dot r_i=r_i$.

Conditions~\textbf{F0} and \textbf{F1} in the definition of ``refactorization''
 imply that $\dot r_d=r_d$.
Thus $\vec{\dot r}=\dot r$. Therefore
Conditions~\textbf{D3} and \textbf{D5} hold.

To verify Condition~\textbf{D4}, we first note that Condition~\textbf{F1}
implies that $\chi_0\, |\, G^0_{y,0^+} =1$ and Condition~\textbf{F2} 
implies that $\dot\pi_{-1}\simeq \pi_{-1}\otimes \chi_0$.
Condition~\textbf{D4} for $\dot\Psi$ now follows from Condition~\textbf{D4} for $\Psi$.  
Hence $\dot\Psi$ is a generic cuspidal $G$-datum.
 \end{proof}

\medskip
\subsection*{Variants of the definition of ``refactorization''}

Assume we are given  $\Psi = (\vec\bG,y,\rho,\vec\phi)$ and $\dot\Psi = (\vec\bG , y,\dot\rho , \vec{\dot\phi})$, as before, except that we explicitly assume $\dot\Psi$ is a generic cuspidal $G$-datum.  Let $\vec{\dot r}$ denote the depth sequence associated to $\dot\Psi$. 

As in Section \ref{sec:construction}, we define a character $\vartheta = \prod_{i=0}^d 
(\hat\phi_i \,|\, K_+)$ of $K_+= K_+(\Psi)$. The analogous character of $\dot K_+ = K_+(\dot\Psi)$  relative to $\dot\Psi$ is denoted $\dot\vartheta$.

Consider the following variants of Condition \textbf{F1}.\label{alteffones}
\begin{itemize}
\item[]
\item[\textbf{F1}$'$.] $\vec{\dot r} = \vec r$ and $\dot\vartheta = \vartheta$.
\item[]
\item[\textbf{F1}$''$.]
\begin{enumerate}
\item $\vec{\dot r} = \vec r$.
\item  If $i\in \{\, 0,\dots , d-1\,\}$ then
$\hat{\dot\phi}_i \, |\, J_+^{i+1} = \hat\phi_i \chi_{i+1}\, |\, J^{i+1}_+$.
\item  $\dot\vartheta = \vartheta$ on $K^0_+ = G^0_{y,0^+}$ (or, equivalently, \textbf{F1} holds for $i=0$).
\end{enumerate}
\item[]
\item[\textbf{F1}$'''$.]
\begin{enumerate}
\item $\vec{\dot r} = \vec r$.
\item  If $i\in \{\, 0,\dots , d-1\,\}$ then
${\dot\phi}_i \, |\, G^i_{y,r_i} =\phi_i \chi_{i+1}\, |\,G^i_{y,r_i}$.  
\item  $\dot\vartheta = \vartheta$ on $K^0_+ = G^0_{y,0^+}$ (or, equivalently, \textbf{F1} holds for $i=0$).
\end{enumerate}
\item[]
\end{itemize}

\begin{lemma}\label{effequiv}
Conditions ${\bf F1}$, ${\bf F1}'$, ${\bf F1}''$ and 
${\bf F1}'''$ are equivalent (assuming $\Psi$ and $\dot\Psi$ are 
generic cuspidal $G$-data).
\end{lemma}

\begin{proof}
Examining the proof of genericity in the proof of Lemma \ref{effones}, one sees that Condition 
\textbf{F1} implies $\vec{\dot r} = \vec r$.  It follows that $K_+ = \dot K_+$ in all cases.

The equivalence of \textbf{F1}$'$ and \textbf{F1}$''$ follows directly from the identity
$$\vartheta (g)= \hat\phi_i (g) \prod_{j=i+1}^d \phi_j(g),$$ for $g\in J^{i+1}_+$, and the analogous identity for $\dot\vartheta$.
(The identity for $\vartheta$ follows from Lemma \ref{hatphikplus}.  The proof of Lemma \ref{hatphikplus} also applies to $\dot\vartheta$.)

Now assume \textbf{F1}$'$.  We have $\vartheta = \prod_{j=i}^d \phi_j$ on $G^i_{y,r^+_{i-1}}$ and similarly for $\dot\vartheta$.  Since we assume $\dot r_{i-1} = r_{i-1}$, we can equate these expressions to obtain \textbf{F1}.  
So \textbf{F1}$'$ implies \textbf{F1}.

Next, we assume \textbf{F1}  holds and  we use the fact that $\chi_{i+1}$ is represented on $G^i_{y,r_i}$ by an element $\gamma_i\in \z^{i+1,*}_{-r_i}$.  Parts (2) and (3) of \textbf{F1}$'''$
 follow directly from \textbf{F1}.  So we now have shown 
$\bf{F1}''\Leftrightarrow {\bf F1}' \Rightarrow {\bf F1} \Rightarrow {\bf F1}'''$.
  We now note $\hat{\dot\phi}_i\,|\, J_+^{i+1}=
\inf\nolimits^{J_+^{i+1}}_{G_{y,r_i}^i}\phi_i$ and
$\hat{\phi}_i\,|\, J_+^{i+1}=
\inf\nolimits^{J_+^{i+1}}_{G_{y,r_i}^i}\phi_i$. Hence
Remark~\ref{oldeffthree} implies that Conditions \textbf{F1}$''$ and 
\textbf{F1}$'''$ are equivalent and thus Conditions \textbf{F1}, 
\textbf{F1}$'$, \textbf{F1}$''$ and \textbf{F1}$'''$ are equivalent.  
\end{proof}

\subsection*{The main result} Suppose $\dot\Psi$ is a refactorization of a generic cuspidal $G$-datum $\Psi$.  The main result of this section says that the representations $\kappa (\Psi)$ and $\kappa (\dot\Psi)$ must be equivalent.

\begin{proposition}\label{refactorequiv}
Assume $\dot\Psi = (\vec\bG, y,\dot\rho,\vec{\dot\phi})$ is a  refactorization of a generic cuspidal $G$-datum $\Psi = (\vec\bG, y,\rho,\vec\phi)$.
Then $\dot\Psi$ is a generic cuspidal $G$-datum such that $\kappa( \dot\Psi) \simeq \kappa (\Psi)$.  In addition, $\vec{\dot r} = \vec r$.
\end{proposition}

\begin{proof}
The fact that $\dot\Psi$ is a generic cuspidal $G$-datum with $\vec{\dot r} = \vec r$ was established in
 Lemma \ref{effones}.
Now let $\kappa = \kappa (\Psi)$ and $\dot\kappa = \kappa (\dot\Psi)$.  We have
$$\kappa = \kappa_{-1}\otimes\cdots \otimes \kappa_d,$$ where $\kappa_{-1} = {\rm inf}_{K^0}^K(\rho)$, $\kappa_d = \phi_d\, | \,K$ and otherwise
 $\kappa_{i}= {\rm inf}_{K^{i+1}}^K (\phi^\prime_i)$.  Similarly, for $\dot\kappa$.

Fix $i\in \{\, 0,\dots ,d-1\,\}$.  The models for our representations $\kappa_i$ and $\dot\kappa_i$ are  as follows.   Fix  a Heisenberg representation $\tau^\natural_i$ of the Heisenberg group $\cH_i^\natural = W_i\boxtimes \mu_p$ whose central character is the identity map on $\mu_p$.  The latter Heisenberg group is formed with respect to the $\mu_p$-valued symplectic form 
$$\langle aJ^{i+1}_+,bJ^{i+1}_+\rangle = \hat\phi_i ([a,b]) = \hat{\dot\phi}_i([a,b])$$ on $W_i= J^{i+1}/J^{i+1}_+$.  
Here, we are using the fact that Condition \textbf{F1}$''(2)$  implies that $\hat{\dot\phi}_i \hat\phi_i^{-1}\, |\, J^{i+1}_+$ extends to a quasicharacter of $G^{i+1}$ and hence is 
trivial on commutators such as $[a,b]$.
 
Let $V_i$ denote the representation space of $\tau^\natural_i$.  
 Let $\hat\tau^\natural_i$ be the Weil representation of $\cS_i = {\rm Sp}(W_i)$ 
associated to the standard action of $\cS_i$ on $\cH^\natural_i$.
We derive from $\tau^\natural_i$ a Heisenberg representation $\tau^\sharp_i$ of $\cH^\sharp_i = W_i\boxtimes \cZ_i$ by pulling back via the map $1\times \hat\phi_i : \cH^\sharp_i \to \cH^\natural_i$.  Similarly, we have a Heisenberg representation $\dot\tau^\sharp_i$ of $\dot\cH^\sharp_i = W_i\boxtimes \dot\cZ_i$.  All three Heisenberg representations $\tau^\natural_i$, $\tau^\sharp_i$ and $\dot\tau^\sharp_i$ act on the same space $V_i$ and yield the identical Weil representation $\hat\tau^\natural_i$ of $\cS_i$ with respect to the standard action of $\cS_i$ on the Heisenberg group.  

Now pull back $\tau^\sharp_i$ via the special isomorphism $\nu_i : \cH_i \to \cH^\sharp_i$ to get a Heisenberg representation $\tau_i$ of $\cH_i = J^{i+1}/N_i$.  Let $\tau_i$ denote the corresponding representation of $J^{i+1}$.  Similarly, one defines another Heisenberg representation $\dot\tau_i$ of $J^{i+1}$ or $\dot\cH_i = J^{i+1}/\dot N_i$.

 Define $\mu_i : J^{i+1}\to \cZ_i$ by $\nu_i (hN_i) = (hJ^{i+1}_+,\mu_i(h))$ 
and let $\mu^\prime_i :J^{i+1}\to \mu_p$ be given by 
$\mu^\prime_i (h) = \hat\phi_i (\mu_i(h))$. 
Define $\dot\mu_i$ and $\dot\mu_i^\prime$ similarly, and let
$\eta_i$ be the character of $J^{i+1}$ 
given by by $\eta_i (h) = \dot\mu^\prime_i(h) \mu^\prime_i (h)^{-1}$.
 Let $\nu^\prime_i : J^{i+1}\to \cH^\natural_i$ be defined by 
$\nu^\prime_i (h) = (1\times\hat\phi_i)(\nu_i(hN_i))$ and define 
$\dot\nu^\prime_i$ similarly.  Then one has
\begin{eqnarray*}
\dot\nu^\prime_i (h) &=& \eta_i (h)\ \nu^\prime_i(h)\cr
\eta_i \,|\,J^{i+1}_+&=& \hat\phi_i^{-1}\ \hat{\dot\phi}_i\  |\, J^{i+1}_+\cr
\dot\tau_i(h)&=&\eta_i(h)\ \tau_i (h),
\end{eqnarray*}
 for $h\in J^{i+1}$.  These  facts follow directly from the definitions. 

Define $\omega_i : K^i\ltimes J^{i+1}\to GL(V_i)$ by $$\omega_i (k,h) = \hat\tau^\natural_i (f'_i(k))\ \tau_i (h),$$ where $f'_i :K^i\to \cS_i$ comes from conjugation.  Defining $\dot\omega_i$ similarly, we have
$$\dot\omega_i = \omega_i \otimes \eta^\flat_i,$$ where $\eta^\flat_i (k,h) =\eta_i (h)$.

Define $\phi^\prime_i : K^{i+1} \to GL(V_i)$ by
$$\phi^\prime_i (kh) = \phi_i (k) \ \omega_i (k,h) = \phi_i(k)\ \hat\tau^\natural_i (f'_i(k))\ \tau_i (h),$$ with $k\in K^i$ and $h\in J^{i+1}$.  Defining $\dot\phi^\prime_i$ similarly, we have
$$\dot\phi^\prime_i = \phi^\prime_i \otimes \eta^\sharp_i,$$ where
$\eta^\sharp_i (kh) = (\dot\phi_i \phi^{-1}_i)(k)\ \eta_i (h)$, when $k\in K^i$ and $h\in J^{i+1}$.

We have 
$$\dot\kappa_i = {\rm inf}_{K^{i+1}}^K (\phi^\prime_i) = 
\kappa_i \otimes {\rm inf}_{K^{i+1}}^K (\eta^\sharp_i).$$
We also have
$$\dot\kappa_{-1} = \kappa_{-1} \otimes {\rm inf}_{K^0}^K \left(\chi_0\right)$$  and $$\dot\kappa_d = \kappa_d\otimes (\dot\phi_d\phi_d^{-1}\,|\,K).$$

It follows that
$$\dot\kappa = \kappa\otimes\xi,$$ where
$$\xi =  {\rm inf}_{K^0}^K \left( \chi_0 \right) \otimes \left(\prod_{i=0}^{d-1} {\rm inf}_{K^{i+1}}^K (\eta^\sharp_i)\right)\otimes  (\dot\phi_d\phi_d^{-1}\,|\,K).$$

Recall that
$$\eta^\sharp_i (kh) = (\dot\phi_i\phi_i^{-1})(k)\ \eta_i (h),$$ with $k\in K^i$ and $h\in J^{i+1}$.  It follows that $\xi$ is trivial on $K^0$.
Now suppose $h\in J^{j+1}$, where $j\in \{\, 0,\dots ,d-1\,\}$.  Then $h$ is annihilated by inflations from $K^0,\dots, K^j$. So
$$\xi (h) = \eta_j(h) \prod_{i=j+1}^d (\dot\phi_i\phi_i^{-1})(h).$$
Recall
$$\eta_j\,|\,J_+^{j+1} = \hat{\dot\phi}_j \hat\phi^{-1}_j  \,|\,J_+^{j+1}$$ and from 
\textbf{F1}$''$ we have
$$\hat{\dot\phi}_j \hat\phi_j^{-1} \,|\, J^{j+1}_+ =
\dot\phi_j\phi_j^{-1} \,|\, J^{j+1}_+= \prod_{i=j+1}^d \phi_i\dot\phi_i^{-1}\, |\,J^{j+1}_+.$$ 
So $\xi$ is trivial on the set $K^0 J^1_+\cdots J^d_+ = K^0 K_+$.

The character $\chi_\kappa$ of $\kappa$ is a product $\prod_{i=-1}^d \chi_{\kappa_i}$ of the 
characters of the $\kappa_i$'s.  
If $d>0$, let $C^d$ be the set of elements in $K$ whose conjugacy class in $K=K^{d-1}J^d$ 
intersects
the subgroup $K^{d-1}J^d_+$.
We observe that if $d>0$ then $\chi_{\kappa_{d-1}}$ has support in the set $C^d$.

Now suppose that $d>1$ and consider $\chi_{\kappa_{d-2}}\,|\, K^{d-1}J^d_+$.  
This is an inflation from $K^{d-1}$, so it is right $J^d_+$-invariant.  On the other hand, 
denoting the set of elements in $K^{d-1}$ whose $K^{d-1}$-conjugacy class intersects
$K^{d-2}J_+^{d-1}$ by $C^{d-1}$, the restriction $\chi_{\kappa_{d-2}}\,|\, K^{d-1}$ has 
support in $C^{d-1}$.

Let $k\in K$ be in the support of $\chi_{\kappa_{d-2}}\otimes \chi_{\kappa_{d-1}}$.
Then there exist $k_1\in K$ and $k_2\in K^{d-1}J^d_+$ such that $k=k_1k_2k^{-1}$.
Since $\chi_{\kappa_{d-2}}(k_1k_2k^{-1})=\chi_{\kappa_{d-2}}(k_2)\not=0$, we
have $k_2=k_3k_4$, with $k_3\in C^{d-1}$ and $k_4\in J_+^d$.
Writing $k_3=k_6k_5k_6^{-1}$ with $k_5\in K^{d-2}J_+^{d-1}$ and $k_6\in K^{d-1}$,
we see that $k=k_1k_6(k_5(k_6^{-1}k_4k_6))(k_1k_6)^{-1}$. Note
that $k_6^{-1}k_4k_6\in J_+^d$, since $K^{d-1}$ normalizes $J_+^d$.
Hence $k\in k_1k_6(K^{d-2}J_+^{d-1}J_+^d)(k_1k_6)^{-1}$.
Thus $\chi_{\kappa_{d-2}}\otimes\chi_{\kappa_{d-1}}$ has support inside the set of
elements in $K$ whose $K$-conjugacy class
intersects the set $K^{d-2}J^{d-1}_+J^d_+$.

Continuing in this manner, one may show that the support of $\chi_\kappa$ is contained 
in the set of elements of $K$ whose $K$-conjugacy class intersects $K^0J^1_+\cdots J^d_+$.
We then have $\chi_{\kappa\otimes\xi} = \chi_\kappa \xi = \chi_\kappa$, 
since this holds on the support of $\chi_\kappa$.  This implies that
 $\kappa \otimes\xi$ and $\kappa$ are equivalent. Hence $\dot\kappa \simeq\kappa$.
\end{proof}

\section{Contragredients}
\label{sec:contragredients}

Fix an involution $\theta$ of $G$ (in the sense of Definition~\ref{definv}).
Given a sequence $\vec\phi = (\phi_0,\dots ,\phi_d)$, it will be convenient to use the notations $\vec\phi^{-1} = (\phi_0^{-1},\dots
,\phi_d^{-1})$ and $\vec\phi^{\theta} = (\phi_0^{\theta},\dots
,\phi_d^{\theta})$, where $\phi^\theta_i$ is the quasicharacter $\phi_i\circ \theta$ of $\theta (G^i)$.  
The notation $\rho^\theta$ denotes the representation $\rho\circ\theta$ of $\theta(K^0)$,
and we will write $\tilde\rho$ for the contragredient of $\rho$.
If $\Psi = (\vec\bG , y,\rho, \vec\phi)$ is an extended generic cuspidal $G$-datum then we use 
the notations 
\begin{equation*}
\begin{split}
\widetilde\Psi&= (\vec\bG ,y, \tilde\rho, \vec\phi^{-1})\cr
\Psi^\theta&=(\theta(\vec\bG),\theta (y),\rho^\theta, \vec\phi^\theta).
\end{split}
\end{equation*}

\label{contralabel} The purpose of this section is to prove:

\begin{theorem}\label{contrathm}
Let $\vec\pi = (\pi_0,\dots,\pi_d)$ be the sequence of tame supercuspidal representations associated to the cuspidal $G$-datum 
$\Psi = (\vec\bG , y,\rho, \vec\phi)$ and let
 $\vec{\check\pi} = (\check\pi_0,\dots, \check\pi_d)$  be the sequence of representations associated to $\widetilde\Psi$.  Then $\pi_i$ and $\check\pi_i$ are contragredient, for all $i$.  Moreover, $\kappa (\Psi)$ and $\kappa (\widetilde\Psi)$ are contragredient.
\end{theorem}

This result should be useful in the general theory of tame supercuspidal representations.  An application of it to the theory of distinguished representations is:

\begin{corollary}\label{distcontra}
Suppose $\vec\pi = (\pi_0,\dots ,\pi_d)$ is the sequence of tame supercuspidal representation associated to the cuspidal $G$-datum $\Psi = (\vec\bG , y, \rho, \vec\phi )$ and $\theta$ is an involution of $G$.   Then the condition $\widetilde\Psi = \Psi^\theta$ implies that each representation
$\pi_i$  must satisfy $\tilde\pi_i\simeq \pi_i\circ \theta$.
\end{corollary}

Let us explain how the proof of Theorem \ref{contrathm} reduces to establishing a simpler proposition.  Fix an extended generic cuspidal $G$-datum 
$\Psi = (\vec\bG , y,\rho, \vec\phi)$ and let $\pi = \ind_K^G(\kappa)$ be the associated tame supercuspidal representation of $G$.
We have a
factorization
$$\kappa =
\kappa_{-1}\otd \kappa_d$$ and a
corresponding
 factorization 
$$V= V_{-1}\otd V_{d-1}$$ of the representation space of
$\kappa$.    Recall that the central character of the
contragredient of a Heisenberg representation is
the inverse of the central character of the
original Heisenberg representation.
This implies that the
 space of the
 tame supercuspidal representation  associated to
$\widetilde\Psi$ is
$$\widetilde V= \widetilde V_{-1}\otd
\widetilde V_{d-1},$$  where $\widetilde V_{-1}$ is the space of $\tilde\rho$ and
$\widetilde V_i$ is the
space of  $\tilde\tau_i$, when $i\in \{\, 0,\dots,d-1\,\}$.  Let
$\lambda_{-1} :
V_{-1}\otimes \widetilde V_{-1}\to \C$ be the natural 
$K^0$-invariant pairing between $\rho$ and $\tilde\rho$.
Similarly, for
each $i$, we take the natural $J^{i+1}$-invariant pairing
$\lambda_i : V_i\otimes \widetilde
V_i\to \C$, when $i\in \{\, 0,\dots,d-1\,\}$.  
Up to scalar multiples, these are the unique nonzero pairings
having the indicated invariance properties.  Now define $\lambda : V\otimes
\widetilde V\to \C$ by $\lambda = \lambda_{-1}
\otd \lambda_{d-1}$.  Theorem \ref{contrathm} now reduces to:

\begin{proposition}\label{linforminvariance}
The linear forms
$\lambda_{-1},\dots,\lambda_{d-1}$ and $\lambda$
are $K$-invariant.
\end{proposition}

The first ingredient in the proof is:

\begin{lemma}\label{contramu}
If $i\in \{\, 0,\dots,d-1\,\}$ and $\mu$ is a
representation of
$K^i$ which is trivial on $K^i\cap J^{i+1}$ then
$\inf\nolimits_{K^i}^{K}(\tilde\mu)$ is the
contragredient of $\inf\nolimits_{K^i}^{K}(\mu)$.
\end{lemma}

\begin{proof}
It is elementary to see that
a $K^i$-invariant pairing $\lambda_\mu$ between
$V_\mu$ and $V_{\tilde\mu}$ is automatically also a
$K$-invariant pairing between
$\inf\nolimits_{K^i}^{K}(\mu)$ and
$\inf\nolimits_{K^i}^{K}(\tilde\mu)$.
\end{proof}

The second ingredient is:

\begin{lemma}\label{contraphiprime}
If $i\in \{0,\dots,d-1\,\}$ then $(\phi_i^{-1})^\prime$ is the contragredient of
$\phi^\prime_i$.
\end{lemma}

\begin{proof}
Throughout the proof, we assume we have fixed our choice of the relevant special isomorphism $\nu_i :\cH_i\to W_i^\sharp$ used to define $\phi^\prime_i$.  Recall that $$\phi^\prime_i(kj) =
\phi_i(k)\,\omega_i
(k, j) =\phi_i (k)\,
\hat\tau_i^\sharp (f_i( k))\  \tau_i (j),$$
for
$k\in K^i$ and $j\in J^{i+1}$.  If $v\in V_i$, $\tilde v\in
\widetilde V_i$ then we need to show
$$\lambda_i (\phi^\prime_i(kj)v\otimes
(\phi_i^{-1})^\prime(kj)\tilde v) = \lambda_i
(v\otimes \tilde v),$$ for all $k\in K^{i}$ and
$j\in J^{i+1}$.  Clearly, this reduces to showing
$$\lambda_i (\hat\tau_i^\sharp (f_i( k))v\otimes
\hat{\tilde\tau}_i^\sharp (f_i( k))\tilde v) = \lambda_i
(v\otimes \tilde v),$$ for all $k\in K^i$.

The representation $\tau_i^\sharp$ has a unique extension to a representation $\hat\tau_i^\sharp$ of $\cS_i\ltimes W_i^\sharp$.  Similarly, $\tilde\tau_i^\sharp$ extends uniquely to some representation $\hat{\tilde\tau}_i^\sharp$.  But the contragredient $\tilde{\hat\tau}_i^\sharp$ of $\hat\tau_i^\sharp$ is another extension of $\tilde\tau_i^\sharp$ and thus $\tilde{\hat\tau}_i^\sharp = \hat{\tilde\tau}_i^\sharp$.
So there must exist a nonzero $(\cS_i\ltimes W_i^\sharp)$-invariant pairing $\hat\lambda_i$ on $V_i\times \widetilde V_i$, viewed as the space of $\hat\tau_i^\sharp \times \hat{\tilde\tau}_i^\sharp$.  Since $\hat\lambda_i$ happens to also be an invariant linear form for $\tau_i^\sharp\times \tilde\tau_i^\sharp$, it must be a nonzero multiple of $\lambda_i$.  In other words, we have shown that $\lambda_i$ is automatically invariant with respect to $\hat\tau_i^\sharp \times \hat{\tilde\tau}_i^\sharp$.   Our assertion now follows.
\end{proof}

Proposition \ref{linforminvariance} and Theorem \ref{contrathm} now follow immediately from Lemmas \ref{contramu} and \ref{contraphiprime}.  

\section{Products of cuspidal $G$-data}
\label{sec:Yuproducts}

Throughout this section, we assume we are given an $F$-group $\bG$ that is a product $\bG = 
\bG^{(1)}\times\bG^{(2)}$ of connected, reductive $F$-groups.  
We will describe how to pass from generic cuspidal $G$-data to pairs consisting
of generic cuspidal $G^{(j)}$-data for $j=1$, $2$, and vice versa. 
More precisely, given a generic cuspidal $G$-datum $\Psi$, we
will produce generic cuspidal $G^{(j)}$-data $\Psi^{(j)}$ for $j=1$, $2$.
 In addition, we define a product operation on cuspidal
data that attaches a generic cuspidal $G$-datum $\Psi^{(1)}\times \Psi^{(2)}$
to generic cuspidal $G^{(j)}$-data $\Psi^{(j)}$ for $j=1$, $2$.
The product operation has the property that if $\Psi^{(1)}$ and $\Psi^{(2)}$ 
arise from a generic cuspidal $G$-datum $\Psi$, then the product $G$-datum
$\Psi^{(1)}\times \Psi^{(2)}$ is a refactorization of $\Psi$.
Hence, according to Proposition~\ref{refactorequiv}, $K(\Psi)=K(\Psi^{(1)}
\times \Psi^{(2)})$ and $\kappa(\Psi)\simeq \kappa(\Psi^{(1)}\times \Psi^{(2)})$.
Furthermore, the product operation is defined in such a way that
$K(\Psi^{(1)}\times \Psi^{(2)})=K(\Psi^{(1)})\times K(\Psi^{(2)})$
and $\kappa(\Psi^{(1)}\times\Psi^{(2)})\simeq \kappa(\Psi^{(1)})\times\kappa(\Psi^{(2)})$.
Thus $\kappa(\Psi)\simeq \kappa(\Psi^{(1)})\times\kappa(\Psi^{(2)})$.
Here we are using the notation $\kappa(\Psi^{(1)})\times\kappa(\Psi^{(2)})$ 
for the tensor product of $\kappa(\Psi^{(1)})$ and $\kappa(\Psi^{(2)})$,
as a representation of $K(\Psi^{(1)})\times K(\Psi^{(2)})$.
(Our conventions regarding notation for tensor products of representations
are as indicated in Section~\ref{sec:Mackey}.)

The above facts are summarized in Theorem~\ref{prodthm}.
The proof is fairly routine, but it is lengthy mainly because it requires a 
certain amount of case-by-case analysis.  A typical reader should find it adequate to 
simply read the statement of Theorem~\ref{prodthm} as well as the definition of the 
factors $\Psi^{(1)}$ and $\Psi^{(2)}$ and the definition of the product operation.

In Section~15 of \cite{Y}, Yu broadens the notion of cuspidal $G$-datum and
defines a notion of ``generalized datum'',
in which twisted Levi sequences are allowed to have repetitions.
In Proposition~15.8, he sketches the theory of generic cuspidal data on product 
groups. However, neither the product operation
nor the factorization process is developed in detail. It is suggested in
\cite{Y} that it most convenient to study generic cuspidal data for
product groups via generalized data. By contrast, we find it more convenient
to regard generic cuspidal data as the primary players in our discussion,
rather than their generalized counterparts.

For this paragraph only, we are not assuming that $\bG$ is the direct
product of groups $\bG^{(1)}$ and $\bG^{(2)}$.
In Chapter~6, we use the product and contragredient operations on cuspidal $G$-data
(together with results from Chapter~5)
in determining when two cuspidal $G$-data $\Psi^{(1)}$ and $\Psi^{(2)}$ 
(for the same group $G$)
 determine equivalent tame supercuspidal representations of $G$.  
The first step in the proof is the observation that $\pi(\Psi^{(1)})\simeq
\pi(\Psi^{(2)})$ exactly when the representation $\pi(\Psi^{(1)})
\times\tilde\pi(\Psi^{(2)})$ of $G\times G$ is distinguished
with respect to the involution $(u,v)\mapsto (v,u)$ of $G\times G$.
The results of Section~\ref{sec:contragredients} 
and of this section tell us that $\pi(\Psi^{(1)})
\times \tilde{\pi}(\Psi^{(2)})$ is equivalent to $\pi(\Psi^{(1)}
\times \widetilde\Psi^{(2)})$, thereby  allowing us to relate the
$G$-data $\Psi^{(1)}$ and $\Psi^{(2)}$ to $G\times G$-data that determine
the tensor product $\pi(\Psi^{(1)})\times\tilde\pi(\Psi^{(2)})$.

\subsection*{Constructing factors} 
We begin by defining the factors $\Psi^{(j)}$ associated to a given extended 
generic cuspidal $G\times G$-datum $\Psi = (\vec\bG, y,\rho,\vec\phi)$.  
Recall that $y\in A(\bG,\bT,F) = A(\bG,\bT,E)\cap \cB (\bG,F)$, where $\bT$ is a 
maximal $F$-torus in $\bG^0$ and $E$ is a tamely ramified finite extension of
$F$ over which $\bT$ splits.  The torus $\bT$ must decompose 
as a direct product $\bT=\bT^{(1)}\times \bT^{(2)}$ of maximal tori 
of $\bG^{(1)}$ and $\bG^{(2)}$.  

Each $\bG^i$ in $\vec\bG$ is a direct product $\bG^i=\bH^{(1),i}\times \bH^{(2),i}$,
where $\bH^{(j)}$ is a tamely ramified twisted Levi subgroup of $\bG^{(j)}$ for
$j=1$, $2$.
The tamely ramified twisted Levi sequence 
$$
\vec\bG^{(j)}= (\bG^{(j),0},\dots, \bG^{(j),d_j})
$$ 
is simply the sequence obtained from $(\bH^{(j),0},\dots, \bH^{(j),d})$ by 
eliminating all repetitions of terms.

We next observe that $A(\bG,\bT,F) = A(\bG^{(1)},\bT^{(1)},F)\times A(\bG^{(2)} , \bT^{(2)},F)$
and thus $y= (y^{(1)},y^{(2)})$, where $y^{(j)}\in A(\bG^{(j)},\bT^{(j)},F)$.  We also 
note that the representation $\rho$ of $K^0= G^0_{[y]}$ is a tensor product 
$\dot\rho^{(1)}\times \dot\rho^{(2)}$ of representations $\dot\rho^{(j)}$ of the groups
$K^{(j),0}= G^{(j),0}_{[y^{(j)}]}$ with $j=1,2$.  
We have $\vec\phi = (\phi_0,\dots, \phi_d)$, where $\phi_i = \varphi_i^{(1)}\times 
\varphi_i^{(2)}$ and $\varphi_i^{(j)}$ is a quasicharacter of $H^{(j),i}$.  
As indicated in Section~\ref{sec:notations}, when $\phi_i$ is nontrivial,
$r_i$ is equal to the depth (in the sense of Moy and Prasad) of the quasicharacter
$\phi_i$ of $G^i$. If $\phi_d$ is trivial, then $r_d=r_{d-1}$. (Recall our convention
that $r_{-1}=0$.) Let $t_i^{(j)}$ be the depth of the quasicharacter
$\varphi_i^{(j)}$ of $H^{(j),i}$. Then, when $r_i>r_{i-1}$ we have
$$
r_i = \max (t_i^{(1)}, t_i^{(2)}).
$$

When $i\in \{\, 0,\dots,d-1\,\}$, the assumption that $\phi_i$ is $G^{i+1}$-generic of depth 
$r_i$ means that there exists a $G^{i+1}$-generic element 
$$
\Gamma_i = (\Gamma^{(1)}_i,\Gamma^{(2)}_i)\in \z^{i,*}_{-r_i}= 
\z^{(1),i,*}_{-r_i}\times \z^{(2),i,*}_{-r_i}
$$ 
of depth $-r_i$ such that $\Gamma_i$ realizes $\phi_i\, |\, G^i_{y,r_i}$.
(Here, $\z^{(j),i,*}$ denotes the dual of the center $\z^{(j),i}$ of
the Lie algebra of $H^{(j),i}$, for $j=1$, $2$.)
Thus $\Gamma^{(j)}_i$ realizes $\varphi_i^{(j)}\,| \,H^{(j),i}_{y,r_i}$, for $j=1,2$.
Now, using the fact that $\Phi (\bG^{i+1},\bT)$ is a disjoint union 
of $\Phi (\bH^{(1),i+1},\bT)$ and 
$\Phi (\bH^{(2),i+1},\bT)$, we see that the genericity condition
\begin{itemize}
\item[\textbf{GE1.}] $v_F(\Gamma_i (H_a))=-r_i$, for all $a\in \Phi (\bG^{i+1},\bT)
\setminus\Phi (\bG^i,\bT)$.
\end{itemize}
is equivalent to the following pair of conditions
\begin{itemize}
\item[\textbf{GE1}${}^1$.] $v_F(\Gamma_i^{(1)} (H_a))=-r_i$, for all 
$a\in \Phi (\bH^{(1),i+1},\bT)\setminus \Phi (\bH^{(1),i},\bT)$.
\item[\textbf{GE1}${}^{2}$.] $v_F(\Gamma_i^{(2)} (H_a))=-r_i$, for all 
$a\in \Phi (\bH^{(2),i+1},\bT)\setminus\Phi (\bH^{(2),i},\bT)$.
\end{itemize}
It should be noted that Condition~\textbf{GE1}${}^{j}$ is vacuous when 
$\bH^{(j),i+1} = \bH^{(j),i}$.

We now consider the genericity  Condition~\textbf{GE2}.  Let $\overline F$ be
an algebraic closure of $F$ and let $\varpi_r$ be an 
element of $\overline{F}^\times$ of valuation $r$.  Let $\t$ be the Lie algebra
of $T$. Then $\varpi_r\Gamma_i$ lies in 
$$
\z^{i ,*}_0 = \z^{(1),i ,*}_0\times \z^{(2),i ,*}_0\subset \t^*\otimes\overline{F}
\cong X^*(\bT)\otimes_{\mathbb{Z}} \overline{F}.
$$
In fact, $\varpi_r\Gamma_i$ lies in $\t^*\otimes \gO_{\overline{F}}
\cong X^*(\bT)\otimes_{\mathbb{Z}} \gO_{\overline{F}}$.
Let $\widetilde\Gamma_i$ be the image of $\varpi_r\Gamma_i$ in 
$X^*(\bT)\otimes_{\mathbb{Z}} \bar\kappa$, where $\bar\kappa$ is the residue field 
of $\overline{F}$.
Let $W= N_{\bG^{i +1}}(\bT)/\bT$ be the Weyl group of $\Phi (\bG^{i +1},\bT)$.
  Then $W$ acts  on $X^*(\bT)\otimes_{\mathbb{Z}}\bar\kappa$.  Condition \textbf{GE2} says that 
the isotropy group of $\widetilde\Gamma_i$ in $W$ is the Weyl group 
$N_{\bG^i }(\bT)/\bT$ of $\Phi (\bG^{i},\bT)$.  It is evident  now that 
Condition \textbf{GE2} is satisfied for both $\Gamma^{(1)}_i$ and $\Gamma^{(2)}_i$.
 
The previous discussion extends to yield:

\begin{lemma}\label{geeee} 
Suppose $i\in \{\, 0,\dots,d-1\,\}$ and $X_i = (X^{(1)}_i, X^{(2)}_i)\in \z^{i,*}_{-r_i}= \z^{(1),i,*}_{-r_i}\times \z^{(2),i,*}_{-r_i}$.  Then $X_i$ is $G^{i+1}$-generic of depth $-r_i$ if and only if the following conditions are satisfied for $j=1,2$:
\begin{enumerate}
\item if $\bH^{(j),i+1}\ne \bH^{(j),i}$ then $X^{(j)}_i$ is $H^{(j),i+1}$-generic 
of depth $-r_i$;
\item if $\bH^{(j),i +1}= \bH^{(j),i }$ then $X^{(j)}_i\in \z^{(j),i ,*}_{-r_i }$.
\end{enumerate}
\end{lemma}

\begin{corollary}\label{corgeeee}
Suppose $i\in \{\, 0,\dots ,d-1\,\}$, $j\in \{ 1,2\}$ and define $j^\prime\in \{1,2\}$ by 
$1' =2$ and $2'=1$.  Then the following conditions hold:
\begin{enumerate}
\item If $\bH^{(j),i+1}\ne \bH^{(j),i}$ then $\varphi^{(j)}_i$ is $H^{(j),i+1}$-generic of 
depth $r_i$ and $r_i = t_i^{(j)}$.
\item If $\bH^{(j),i +1}= \bH^{(j),i }$ then $t^{(j)}_i\le r_i$ and 
$\bH^{(j'),i+1}\ne \bH^{(j'),i}$.
\end{enumerate}
\end{corollary}
Define 
\begin{equation*}
\begin{split}
m(i,j)&=\min \{\, \ell \ | \ \bH^{(j),\ell} = \bG^{(j),i}\,\}\cr
M(i,j)&=\max \{\, \ell \ | \ \bH^{(j),\ell} = \bG^{(j),i}\,\},
\end{split}
\end{equation*}
when $j\in \{1,2\}$ and $i\in \{\,0,\dots, d_j\,\}$.
If $M(i,j)<d$ or, equivalently, $\bG^{(j),i}\ne \bG^{(j)}$, then 
Corollary \ref{corgeeee} implies that $\varphi^{(j)}_{M(i,j)}$ is $H^{(j),M(i,j)+1}$-generic 
of depth $r_{M(i,j)}= t_{M(i,j)}^{(j)}$ with respect to $y^{(j)}$.

We provisionally define 
$$
\phi_i^{(j)} = \prod_{\ell = m(i,j)}^{M(i,j)} \varphi^{(j)}_\ell,
\qquad \rho^{(j)}=\dot\rho^{(j)}.
$$ 
Let $r_i^{(j)}$ be the depth of $\phi_i^{(j)}$.
However, if this  definition results in the inequality $r^{(j)}_{d_j}\le r^{(j)}_{d_j-1}$, 
where $r^{(j)}_{-1}=0$, then:
\begin{enumerate}
\item if $d_j>0$ we let $\phi^{(j)}_{d_j}$ be the trivial character of
 $\bG^{(j)}$ and
$$
\phi_{d_j-1}^{(j)} = \left( \prod_{\ell = m(d_j-1,j)}^{M(d_j-1,j)} \varphi^{(j)}_\ell\right)
 \left( \prod_{\ell= m(d_j,j)}^{M(d_j,j)} \varphi^{(j)}_\ell \, |\, G^{(j),d_j-1}\right),
$$ 
and in this case we put $r^{(j)}_{d_j}=r^{(j)}_{d_j-1}$, where $r^{(j)}_{d_j-1}$ is
 the depth of $\phi_{d_j-1}^{(j)}$;
 \item if $d_j=0$ we let $\phi_0^{(j)}$ be the trivial character of $G^{(j)}$ and we let 
$$
\rho^{(j)}= \dot\rho^{(j)}\otimes \left( \prod_{\ell=0}^d\varphi_\ell^{(j)} \,|\, 
G^{(j)}_{[y^{(j)}]}\right).
$$
\end{enumerate}
 The latter adjustments are similar to applying a refactorization and they are 
needed in order for the sequences $\vec r^{(j)}= (r^{(j)}_0,\dots ,r^{(j)}_{r_{d_j}})$
 and $\vec\phi^{(j)}= (\phi^{(j)}_0,\dots ,\phi^{(j)}_{r_{d_j}})$ and the 
representation $\rho^{(j)}$ to satisfy 
Conditions~\textbf{D3}--\textbf{D5} (from Section \ref{sec:notations}).
 
\begin{lemma}\label{factorgenericity}
$\Psi^{(j)}= (\vec\bG^{(j)}, y^{(j)},\rho^{(j)}, \vec\phi^{(j)})$ is a 
generic cuspidal $G^{(j)}$-datum for $j=1,2$.
\end{lemma}

\begin{proof} Fix $j\in \{ 1,2\}$.  It is clear that Conditions~\textbf{D1} and 
\textbf{D2} are satisfied by $\Psi^{(j)}$.
To see that Conditions \textbf{D3}--\textbf{D5} are satisfied, we first consider 
the case in which $d_j=0$.  If $r^{(j)}_0>0$ then $\phi^{(j)}_0= \prod_{\ell=0}^d \
\varphi_\ell^{(j)}$ has depth $r^{(j)}_0$ with respect to $y^{(j)}$, which is 
consistent with the requirements of Conditions~\textbf{D3}--\textbf{D5} in the 
case of a datum of degree 0.  If $d_j=0$ and $r^{(j)}_0=0$ then the adjustment (2) 
in the definitions of $\phi^{(j)}_0$ and $\rho^{(j)}$ forces 
Conditions~\textbf{D3}--\textbf{D5} to hold.

Next, we assume $d_j>0$.  If $i\in \{\, 0,\dots, d_j-1\,\}$ then $M(i,j)<d$ and 
the depth $r^{(j)}_i$ of $\phi_i^{(j)}$ is equal to $r_{M(i,j)}$.  The 
inequalities $0<r_0^{(j)}<\dots < r^{(j)}_{d_j-1}$ follow from the corresponding 
inequalities of the $r_{M(i,j)}$'s.  Adjustment (1) in the definition of 
$\phi^{(j)}_0$ now forces Conditions~\textbf{D3}--\textbf{D5} to hold.

Finally, we must demonstrate that $\Psi^{(j)}$ is generic.  In general, 
if $i\in \{\, 0,\dots , d_j\,\}$ then we let
$$
\dot\phi^{(j)}_i = \prod_{\ell=m(i,j)}^{M(i,j)} \varphi^{(j)}_\ell,
$$ and we let $\dot r^{(j)}_i$ be the depth of $\dot\phi^{(j)}_i$.
Now assume $i< d_j$.  Then Corollary \ref{corgeeee} implies that 
$\varphi^{(j)}_{M(i,j)}$ is $H^{(j), M(i,j)+1}$-generic of depth $r_{M(i,j)}$.
On the other hand, if $m(i,j)\le \ell < M(i,j)$ then the depth $t^{(j)}_\ell$ 
of $\varphi^{(j)}_\ell$ satisfies
$$
t^{(j)}_\ell \le \max (t^{(1)}_\ell , t^{(2)}_\ell ) = r_\ell < r_{M(i,j)}.
$$ 
 It follows that
$\dot\phi^{(j)}_i$ is $G^{(j),i+1}$-generic of depth $r_{M(i,j)}$ and
$$
\dot\phi^{(j)}_i\, |\, G^{(j),i}_{y^{(j)}, r_{M(i,j)}} = \varphi^{(j)}_{M(i,j)} 
\,| \, G^{(j),i}_{y^{(j)}, r_{M(i,j)}}.
$$ 
 Therefore, we have shown that $\phi^{(j)}_i$ has the necessary genericity 
property in those cases in which $\phi^{(j)}_i = \dot\phi^{(j)}_i$.

Assume that $\phi^{(j)}_i \ne \dot\phi^{(j)}_i$.
Then $d_j>0$, $i=d_j-1$,
$\dot r^{(j)}_{i+1}\le \dot r^{(j)}_{i}$ and
$$
\phi^{(j)}_{i}= \dot\phi^{(j)}_{i} (\dot\phi^{(j)}_{i+1}\, |\, G^{(j),i}).
$$ 
If $\dot r^{(j)}_{i+1}< \dot r^{(j)}_{i}$ then the genericity of $\phi^{(j)}_i$ 
follows from the fact that $\phi^{(j)}_i$ agrees with $\dot\phi^{(j)}_i$ on 
$G^{(j),i}_{y^{(j)}, r_{M(i,j)}}$.  In the case $\dot r^{(j)}_{i+1}= \dot r^{(j)}_{i}$,
 genericity follows from Lemma \ref{generictrivia}. 
\end{proof}

\subsection*{Constructing products}\label{productlabel}
Assume we are given generic cuspidal $G^{(j)}$-data 
$$
\Psi^{(j)}= (\vec\bG^{(j)} , y^{(j)}, \rho^{(j)}, \vec\phi^{(j)}),
\qquad j=1,\, 2.
$$   
Let $d_j$ be the degree of $\Psi^{(j)}$.  (In other words, $\vec\bG^{(j)}$, $\vec r^{(j)}$
and $\vec\phi^{(j)}$ are $(d_j+1)$-tuples.)  Our present objective is to construct a
 ``product datum'', that is, a generic cuspidal $G$-datum
$$
\Psi^{(1)}\times \Psi^{(2)}= (\vec\bG , y,\rho,\vec\phi)
$$ 
such that
$$
\kappa (\Psi^{(1)}\times \Psi^{(2)}) \simeq \kappa (\Psi^{(1)}) \times \kappa (\Psi^{(2)}),
$$ 
as described at the beginning of this section.

There is one particularly simple class of examples which we treat first.  If $\Psi^{(2)}$ 
parametrizes a depth zero representation then  we take $d=d_1$ and for all 
$i\in \{\, 0,\dots, d\,\}$ we take  $\bG^i= \bG^{(1),i}\times \bG^{(2)}$, $y = (y^{(1)},y^{(2)})$, 
$r_i = r^{(1)}_i$, $\rho = \rho^{(1)}\times \rho^{(2)}$ and 
 $\phi_i = \phi^{(1)}_i\times 1$. Of course, the case in which $\pi(\Psi^{(1)})$ has 
depth zero is similar.

The general definition of the product datum is as follows.
The first step is to construct the sequence $\vec r = (r_0,\dots, r_d)$.  
If  $d_1 = d_2 =0$ then $\vec r = (r_0)$, where $r_0 = \max(r^{(1)}_0,r^{(2)}_0)$.  
Otherwise, the numbers $r_0,\dots , r_{d-1}$ are just the nonzero numbers in 
$\{\, r^{(j)}_i \ | \ 1\le j\le 2,\ 0\le i\le d_j-1\,\}$ listed without 
repetitions in increasing order and 
$r_d = \max(r^{(1)}_{d_1},r^{(2)}_{d_2})$.
Note that $d=0$ exactly when $d_1=d_2=0$.

The Levi sequence 
$$
\vec\bG = (\bG^0,\dots,\bG^d) = (\bH^{(1),0}\times \bH^{(2),0},\dots , 
                  \bH^{(1),d}\times \bH^{(2),d})
$$ 
is defined in the following manner.  First, we set $\bH^{(j),0} = \bG^{(j),0}$ 
and $\bH^{(j),d} = \bG^{(j)}$, for $j=1,2$.  Next, suppose $i\in \{\, 0,\dots,d-2\,\}$
 and $j\in \{\, 1,2\,\}$ 
and $\bH^{(j),i}$ has been defined and equals $\bG^{(j),\ell}$ for some $\ell$.  Define
$$\bH^{(j),i+1} = 
\begin{cases}
\bG^{(j),\ell+1},&\text{if $\phi^{(j)}_\ell$ is of depth $r_i$,}\\
\bH^{(j),i},&\text{otherwise.}
\end{cases}
$$

By assumption, $y^{(j)}\in A(\bG^{(j)},\bT^{(j)},F) = A(\bG^{(j)},\bT^{(j)},E_j)
\cap \cB (\bG^{(j)},F)$, where $\bT^{(j)}$ is a tamely ramified maximal torus in $\bG^{(j),0}$
and $E_j$ is a tamely ramified splitting field of $\bT^{(j)}$.
  Let $E=E_1E_2$, $\bT = \bT^{(1)}\times \bT^{(2)}$ and $y= (y^{(1)},y^{(2)})$.  
Then $y\in A(\bG,\bT,F) = A(\bG,\bT,E)\cap \cB(\bG, F)$.  
Let $\rho$ be the representation $\rho^{(1)}\times \rho^{(2)}$ of $K^0 = 
K^{(1),0}\times K^{(2),0}$.
 
Define 
$$
\vec\phi = (\phi_0,\dots ,\phi_d) = (\varphi^{(1)}_0\times 
\varphi^{(2)}_0,\dots , \varphi^{(1)}_d\times \varphi^{(2)}_d)
$$ 
as follows.  
If $d=0$ then we put $\varphi_0^{(j)}=\phi_0^{(j)}$ for $j=1$, $2$.
Otherwise, if $i\in \{\, 0,\dots,d-1\,\}$ and $r_{d-1} <r_d$ then define
$$
\varphi^{(j)}_i =
\begin{cases}
1,&\text{if $\bH^{(j),i}= \bH^{(j),i+1}$,}\\
\phi^{(j)}_\ell,&\text{if $\bG^{(j),\ell}= \bH^{(j),i}\ne \bH^{(j),i+1}$}
\end{cases}
$$
and $$\varphi^{(j)}_d = \phi_{d_j}^{(j)}.$$  
If $d>0$ and $r_{d-1}=r_d$, we use the same definitions except that we take 
$\varphi_d^{(j)}$ to be trivial and
$$\varphi^{(j)}_{d-1} =
\begin{cases}
\phi_{d_j}^{(j)},&\text{if $\bH^{(j),d-1}= \bG^{(j)}$,}\\
\phi^{(j)}_{d_j-1},&\text{if $\bG^{(j),d_j-1}= \bH^{(j),d-1}\ne \bG^{(j)}$.}
\end{cases}
$$
Note that in the latter case, when $\bH^{(j),d-1}\ne \bG^{(j)}$, it must be the 
case that $\phi^{(j)}_{d_j}$ is trivial.  Indeed, $r^{(j)}_{d_j-1} = r^{(j)}_{d_j}$ 
since 
$$
r_{d-1} = r^{(j)}_{d_j-1} \le r^{(j)}_{d_j} \le \max (r^{(1)}_{d_1},r^{(2)}_{d_2})
 \le r_{d-1}.
$$

This completes the definition of the product datum
$$\Psi^{(1)}\times \Psi^{(2)} = (\vec\bG, y,\rho,\vec\phi).$$
Next, we verify that the product datum is indeed a generic cuspidal $G$-datum,
and that it has the other properties mentioned previously in this section.

\begin{lemma}\label{productgenericity}
If $\Psi^{(1)}$ and $\Psi^{(2)}$ are extended generic cuspidal $G^{(j)}$-data 
then $\Psi^{(1)}\times \Psi^{(2)}$ is a extended generic cuspidal $G$-datum,
$K(\Psi^{(1)}\times \Psi^{(2)})=K(\Psi^{(1)})\times K(\Psi^{(2)})$
and 
$$
\kappa (\Psi^{(1)}\times\Psi^{(2)}) \simeq \kappa (\Psi^{(1)})\times \kappa (\Psi^{(2)}).
$$
\end{lemma}

\begin{proof} Let $\Psi = \Psi^{(1)}\times \Psi^{(2)}$.  To show that $\Psi$ is a 
cuspidal $G$-datum, we first need to verify that Conditions~\textbf{D1}--\textbf{D5} hold.  
Condition~\textbf{D1} easily follows once we show that we have 
$\bG^0\subsetneq \dots \subsetneq\bG^d$.  Suppose $i\in \{\, 0,\dots,d-1\,\}$.  Then it is clear 
that $\bG^i\subset \bG^{i+1}$.  We observe that the quasicharacter $\phi_i$ has a 
factorization $\varphi_i^{(1)}\times \varphi_i^{(2)}$ in which at least one of the factors 
is nontrivial.  Moreover, if $\varphi^{(j)}_i$ is nontrivial then 
$\bH^{(j),i}\ne \bH^{(j),i+1}$ and hence, in general, $\bG^i\subsetneq \bG^{i+1}$.  
This implies Condition~\textbf{D1} holds.  Clearly Conditions~\textbf{D2}--\textbf{D4}
 also hold.  

To verify Condition~\textbf{D5}, one uses a straightforward induction to show that 
$\phi_i$ has depth $r_i$, except in the case where $i=d$ and $r_{d-1} = r_d$.  
In the latter case, $\phi_i$ is trivial, as required by Condition~\textbf{D5}.
Conditions~\textbf{D1}--\textbf{D5} are satisfied, the fact that $\Psi$ is a generic 
cuspidal $G$-datum follows from Lemma~\ref{geeee}.

We now sketch the proof that 
$\kappa (\Psi^{(1)}\times\Psi^{(2)}) \simeq \kappa (\Psi^{(1)})\times \kappa (\Psi^{(2)})$.
  If $d=0$ then
\begin{equation*}
\begin{split}
\kappa (\Psi^{(1)}\times \Psi^{(2)}) &= \rho \otimes (\phi_0\,|\, G_{[y]})\cr
&= (\rho^{(1)}\otimes (\phi^{(1)}_0\,|\, G^{(1)}_{[y^{(1)}]}) )
\times(\rho^{(2)}\otimes (\phi^{(2)}_0\,|\, G^{(2)}_{[y^{(2)}]}) )\cr
&= \kappa(\Psi^{(1)})\times \kappa (\Psi^{(2)}).
\end{split}
\end{equation*}
Now assume $d>0$.  Each quasicharacter $\phi^{(j)}_i$ occurs exactly once as a factor 
of some quasicharacter $\phi_k$.  All other factors of the $\phi_k$'s are trivial 
characters.  When $\phi^{(j)}_i$ is a factor of $\phi_k$ it is routine to check that 
the group $K^{(j),i+1}$ on which $\phi^{(j),\prime}_i$ is defined is a factor of the 
group $K^{k+1}$ on which $\phi^\prime_k$ is defined and, moreover, $\phi^{(j),\prime}_i$ is 
a factor of $\phi^\prime_k$. The representation $\phi^{(j),\prime}_i$ inflates to a 
representation $\kappa^{(j)}_i$ of the group $K^{(j)}$ on which $\kappa (\Psi^{(j)})$ 
is defined.  One can check that $K^{(j)}$ is a factor of the inducing group $K$ of 
$\kappa (\Psi^{(1)}\times \Psi^{(2)})$ and $\kappa^{(j)}_i$ is a factor of $\kappa_k$.  
The various trivial characters inflate to give trivial characters of the groups $K^{(j)}$ 
and they make no net contribution to the $\kappa (\Psi^{(j)})$'s.  
 In this way, one shows that each side of the desired isomorphism 
has a factorization with factors equivalent to those on the other side.  
The details are left to the reader.
\end{proof}
     
 \begin{theorem}\label{prodthm}
 Suppose $\Psi$ is a extended generic cuspidal $G$-datum with $\bG= \bG^{(1)}\times\bG^{(2)}$.
 Let $\Psi^{(1)}$ and $\Psi^{(2)}$ be the factors associated to $\Psi$, as above. Then 
$\Psi^{(1)}\times \Psi^{(2)}$ is a refactorization of $\Psi$ and, consequently, 
$$
\kappa (\Psi)\simeq \kappa (\Psi^{(1)}\times\Psi^{(2)}) \simeq 
\kappa (\Psi^{(1)})\times \kappa (\Psi^{(2)}).
$$
 \end{theorem}
 
\begin{proof}
Fix a generic cuspidal $G$-datum $\Psi = (\vec\bG, y,\rho,\vec\phi)$.  
Let 
$$
\Psi^{(j)} = (\vec\bG^{(j)}, y^{(j)},\rho^{(j)}, \vec\phi^{(j)}),\qquad j=1,2,
$$ 
be the factors of $\Psi$ and let $\dot\Psi$ be the product datum $\Psi^{(1)}\times \Psi^{(2)}$.  

We begin by sketching why the analogues of $\vec\bG$ and $\vec r$ for $\dot\Psi$ are 
precisely $\vec\bG$ and $\vec r$.  
We assume $d>0$, since there is nothing to prove when $d=0$.  Given $\Psi$, let 
$u^{(j)}$ be the depth of the quasicharacter
$$
\prod_{\ell = m(d_j,j)}^{M(d_j , j)} \varphi^{(j)}_\ell.
$$  
Then 
$$
r^{(j)}_{d_j} = \max (u^{(j)}, r^{(j)}_{d_j-1}).
$$  
Recall that if $\bG^{(j),i}\ne \bG^{(j)}$ then $r^{(j)}_i = r_{M(i,j)}$.  
It follows that $(\vec\bG, \vec r)$ determines $(\vec\bG^{(1)},\vec r^{(1)})$ and 
$(\vec\bG^{(2)}, \vec r^{(2)})$, except that we also need $u^{(1)}$ and $u^{(2)}$ 
to determine $r^{(1)}_{d_1}$ and $r^{(2)}_{d_2}$.  Let $\vec{\dot r} = 
(\dot r_0,\dots, \dot r_{\dot d})$ be the analogue of $\vec r$ for $\dot\Psi$.  
Then $\dot r_0,\dots ,\dot r_{\dot d -1}$ are just the nonzero numbers of the form 
$r^{(j)}_i$, with $i<d_j$, listed in ascending order.  But this is the same as the 
sequence $r_0,\dots , r_{d-1}$.  In particular, $\dot d = d$.
We also have
$$
\dot r_d = \max (r^{(1)}_{d_1}, r^{(2)}_{d_2}) = \max (u^{(1)},u^{(2)}, r_{d-1}) = 
r_d
$$ 
and hence $\vec{\dot r} = \vec r$.

Given $(\vec\bG^{(1)},\vec r^{(1)})$ and $(\vec\bG^{(2)},\vec r^{(2)})$ one 
can reconstruct $\vec r$ in the manner just described and one can then reconstruct 
$\vec\bG$ as follows.  First, one takes 
$$
\bG^0 =\bG^{(1),0}\times \bG^{(2),0}.
$$ 
Then, if $i\in \{\,0,\dots,d-1\,\}$ and we have indices $a_{i1}$ and $a_{i2}$ such that 
$$
\bG^i = \bG^{(1),a_{i1}}\times \bG^{(2),a_{i2}},
$$ 
we use $r_i$ to construct $\bG^{i+1}$ according to the formula
$$a_{i+1,j} =
\begin{cases}
a_{ij}+1,&\text{if  $r_i\in\vec r^{(j)}$,}\\
a_{ij},&\text{otherwise.}
\end{cases}
$$
But the latter recursion is the same recursion that constructs the Levi sequence 
for $\dot\Psi$.  Since it is also clear that the point $y$ in $\Psi$ coincides 
with the analogous point for $\dot\Psi$, we now deduce that 
$$
\dot\Psi = (\vec\bG ,y, \dot\rho, \vec{\dot\phi}),
$$ 
where it remains to consider the relation between $(\dot\rho,\vec{\dot\phi})$ and 
$(\rho,\vec\phi)$.

To show that $\dot\Psi$ is a refactorization of $\Psi$, it suffices to verify that 
Conditions~\textbf{F0}--\textbf{F2} in the definition of ``refactorization'' hold. 
We will only consider the case in which $d>0$, since the case of $d=0$ is elementary.  
Condition~\textbf{F0} holds since $\Psi$ and $\dot\Psi$ are cuspidal $G$-data that 
share the same objects $\vec\bG$, $y$ and $\vec r$. More precisely, 
$\phi_d= 1$ exactly when $r_{d-1}=r_d$ and the latter condition holds exactly 
when $\dot\phi_d=1$.

We next consider Condition \textbf{F1}.  Assume $d>0$.  Given
$\Psi = (\vec\bG,y,\rho,\vec\phi)$ with factors $\Psi^{(j)} = (\vec 
\bG^{(j)},y^{(j)}, \rho^{(j)}, \vec\phi^{(j)})$,
we now know that the product datum $\Psi^{(1)}\times \Psi^{(2)}$ has  
the same $\vec\bG$, $y$ and $\vec r$ as $\Psi$ and thus
$$\Psi^{(1)}\times \Psi^{(2)} = (\vec\bG , y,\dot\rho, \vec{\dot 
\phi}),$$ for suitable $\dot\rho$ and $\vec{\dot\phi}$.  We associate  
to $\vec{\dot\phi}$ some auxiliary notations:
$$\vec{\dot\phi} = (\dot\phi_0,\dots , \dot\phi_d) =
(\dot\varphi^{(1)}_0\times \dot\varphi^{(2)}_0,\dots ,
\dot\varphi^{(1)}_d\times \dot\varphi^{(2)}_d).$$
We need to show
$$\prod_{\ell =i}^d \phi_\ell | G^i_{y,r^+_{i-1}} =
\prod_{\ell =i}^d \dot\phi_\ell | G^i_{y,r^+_{i-1}} .$$

Let $$\Phi^{(j)}_k = \prod_{\ell = m(k,j)}^{M(k,j)} \varphi^{(j)}_\ell 
\qquad\mbox{ and }\qquad \dot\Phi^{(j)}_k = \prod_{\ell = m(k,j)}^{M 
(k,j)} \dot\varphi^{(j)}_\ell.$$
Fix $i\in \{ 0,\dots , d\}$ and $j\in \{ 1,2\}$.  Then there exists a  
(unique) number $k(i,j)\in \{ 0,\dots d_j\}$ such that $$m(k(i,j),j) 
\le i\le M(k(i,j),j).$$
It is straightforward to verify the following identities:
\begin{eqnarray*}
\prod_{\ell =i}^{M(k(i,j),j)} \varphi_\ell^{(j)}|H^{(j),i}_{y,r^+_ 
{i-1}} &=&
\prod_{\ell =m(k,j)}^{M(k(i,j),j)} \varphi_\ell^{(j)}|H^{(j),i}_{y,r^ 
+_{i-1}},\\
\prod_{\ell =i}^{M(k(i,j),j)} \dot\varphi_\ell^{(j)}|H^{(j),i}_{y,r^+_ 
{i-1}} &=&
\prod_{\ell =m(k,j)}^{M(k(i,j),j)} \dot\varphi_\ell^{(j)}|H^{(j),i}_ 
{y,r^+_{i-1}}.
\end{eqnarray*}
It follows that
\begin{eqnarray*}
\prod_{\ell =i}^d \phi_\ell | G^i_{y, r^+_{i-1}} &=& \prod_{j=1}^2  
\quad\prod_{k= k(i,j)}^{d_j} \Phi_k^{(j)} | H^{(j),i}_{y,r^+_{i-1}},\\
\prod_{\ell =i}^d \dot\phi_\ell | G^i_{y, r^+_{i-1}} &=& \prod_{j=1} 
^2 \quad\prod_{k= k(i,j)}^{d_j} \dot\Phi_k^{(j)} | H^{(j),i}_{y,r^+_ 
{i-1}}.
\end{eqnarray*}

Now it is routine to verify that, in fact,  $\dot\Phi_k^{(j)} =  
\phi_k^{(j)}$, for all $k\in \{ 0,\dots , d_j\}$.
Therefore, we are reduced to showing that for fixed $i\in \{ 0, 
\dots ,d\}$ and $j\in \{ 1,2\}$, we have
$$\prod_{k= k(i,j)}^{d_j} \Phi_k^{(j)} | H^{(j),i}_{y,r^+_{i-1}} =
\prod_{k= k(i,j)}^{d_j} \phi_k^{(j)} | H^{(j),i}_{y,r^+_{i-1}}.
$$

Let $\delta^{(j)}_k$ be the depth of $\Phi^{(j)}_k$ and let
$\dot\delta^{(j)}_k$ be the depth of $\dot\Phi^{(j)}_k$.  We have $ 
\Phi_k^{(j)} = \phi_k^{(j)}$, unless $\delta^{(j)}_{d_j} \le \delta^ 
{(j)}_{d_j -1}$ and $k= d_j-1$ or $k=d_j$.  In the latter cases, $ 
\phi_{d_j}^{(j)} =1$ and, if $d_j>0$, then $$\phi^{(j)}_{d_j-1} =  
\Phi^{(j)}_{d_j-1} (\Phi^{(j)}_{d_j}| G^{(j),d_j-1}).$$
Recall that $$\Phi^{(j)}_k = \prod_{\ell = m(k,j)}^{M(k,j)} \varphi^ 
{(j)}_\ell$$ and that, by definition, $\phi_k^{(j)} = \Phi^{(j)}_k$,  
with the following exceptions.  If $d_j>0$ and $\delta^{(j)}_{d_j}\le  
\delta^{(j)}_{d_j-1}$ then $\phi^{(j)}_{d_j}=1$ and $\phi^{(j)}_ 
{d_j-1} = \Phi^{(j)}_{d_j-1} (\Phi^{(j)}_{d_j}| G^{(j),d_j-1})$.  If  
$d_j=0$ and $\delta^{(j)}_0=1$ then $\phi^{(j)}_0 =1$.

Assume that $d_j>0$ and $\delta_{d_j}^{(j)}\le \delta_{d_j-1}^{(j)} 
$.   If $k(i,j)<d_j$ then both sides of the desired inequality are  
the same.  Now suppose $k(i,j) = d_j$.  So $m(d_j,j) \le i\le d$.  We  
need to show
$$\Phi^{(j)}_{d_j} | H^{(j),i}_{y,r^+_{i-1}} = \phi^{(j)}_{d_j} | H^ 
{(j),i}_{y,r^+_{i-1}}=1.$$
Now $\Phi^{(j)}_{d_j}| H^{(j),i}_{y,(\delta^{(j)}_{d_j})^+}=1$.  But  
$H^{(j),i}_{y,(\delta^{(j)}_{d_j})^+}\supset
H^{(j),i}_{y,(\delta^{(j)}_{d_j-1})^+}$.  So $\Phi^{(j)}_{d_j}| H^ 
{(j),i}_{y,(\delta^{(j)}_{d_j-1})^+}=1$.
   But $(\delta^{(j)}_{d_j -1})^+ \le r_{M(d_j -1,j)}^+\le r_{i-1}^+$.
Therefore, $\Phi^{(j)}_{d_j} | H^{(j),i}_{y,r^+_{i-1}}  =1$.

Finally, if $d_j=0$ and $\delta^{(j)}_0=1$ then
$\Phi^{(j)}_{0} | H^{(j),i}_{y,r^+_{i-1}} =1$ since $\Phi^{(j)}_0 | H^ 
{(j),i}_{y,0^+} =1$.  This completes the verification of Condition  
\textbf{F1}.

It remains to consider Condition~\textbf{F2}.  
If $d_1$ and $d_2$ are positive then $\rho = \dot\rho$ and, arguing as above, we see that
$$
\prod_{\ell=0}^d \phi_\ell \,|\, K^0 = \prod_{\ell=0}^d \dot\phi_\ell \,|\, K^0.
$$ 
Hence, in this case, Condition~\textbf{F2} holds.  The remaining case, in which 
either $d_1=0$ or $d_2=0$ is left as an exercise.
\end{proof}

\chapter{Distinguished tame supercuspidal representations}

\section{Weak and moderate compatibility}
\label{sec:symmetrizing}

\subsection*{Equivalence of $(G,K)$-data}
Given an extended generic cuspidal $G$-datum 
$\Psi = (\vec\bG,y,\rho,\vec\phi)$, we have defined a subgroup 
$K= K(\Psi)$ from which the tame supercuspidal representation 
$\pi(\Psi)$ is induced.  Throughout this chapter, we assume that 
$G$ and this subgroup $K$ have been fixed and we consider those 
extended generic 
cuspidal $G$-data associated to the fixed pair $(G,K)$.

\begin{definition}\label{GKdatumdef}
A {\it $(G,K)$-datum} is an extended generic cuspidal $G$-datum $\Psi$ whose inducing subgroup $K(\Psi)$ is $K$.
\end{definition}

In this chapter we assume that Hypothesis~C($\vec\bG$) holds for the
twisted Levi sequences 
$\vec\bG$ that occur in the $(G,K)$-data we are considering.

We now discuss transformations of a $(G,K)$-datum $\Psi$  that do not affect the equivalence class of the representation $\kappa(\Psi)$.  The most basic such transformations are given in the following definition.

\begin{definition}\label{eltransf}
If $\Psi = (\vec\bG,y,\rho,\vec\phi)$ is a  $(G,K)$-datum and $(y,\rho)$ is replaced by $(\dot y,\dot\rho)$, with $[y] = [\dot y]$ and $\rho\simeq \dot\rho$, then we say $\Psi$ has undergone an {\it elementary transformation}.
\end{definition}

Let $\Psi = (\vec\bG,y,\rho, \vec\phi)$ be a $(G,K)$-datum. If $g\in G$, let
${}^g\vec\bG$ be the tamely ramified twisted Levi sequence
such that the $i$th group in the sequence is ${}^g \bG^i=\Int(g)\bG^i$, for 
$i\in \{\, 0,\dots,d\,\}$.
As in Section~\ref{sec:Mackey}, the notation ${}^g\rho$ is used for the 
representation $\rho\circ\Int (g^{-1})$ of $\Int(g) K^0$. 
Let ${}^g \vec\phi=({}^g\phi_0,\dots,{}^g\phi_d)$.
Then the generic cuspidal $G$-datum 
$$
{}^g \Psi = 
({}^g\vec\bG,\, g\cdot y,\, {}^g \rho, \,  {}^g\vec\phi)
$$  
has the property that $\pi ({}^g\Psi) \simeq \pi (\Psi)$.  
Note that if $g\notin K$, then ${}^g \Psi$ might not be a $(G,K)$-datum,
since $\Int(g)K$ might not be equal to $K$.
This action of $G$ on the set of generic cuspidal $G$-data
will be called $G$-\textit{conjugation}. \label{Gconjlabel} Although we will only
be considering conjugation by elements of $K$ in this chapter,
$G$-conjugation will be used in Chapter~6.

In addition to the elementary transformations, another basic operation on the 
set of $(G,K)$-data is described as follows.
If $g\in K$ then $K({}^g\Psi)=K$, so ${}^g\Psi$ is a $(G,K)$-datum,
Also, $\kappa ({}^g \Psi)\simeq \kappa (\Psi)$ (in fact, $\kappa({}^g\Psi)
={}^g\kappa(\Psi)$).
This action of $K$ on the set of $(G,K)$-data will be called the action of $K$ 
by  \textit{conjugation}.\label{Kconjlabel}

The third basic operation on the set of $(G,K)$-data is refactorization.  If $\dot\Psi$ is a refactorization of  a $(G,K)$-datum $\Psi$  then $\dot\Psi$ is also a $(G,K)$-datum and $\kappa (\dot\Psi)\simeq\kappa (\Psi)$, according to Proposition \ref{refactorequiv}.

\begin{definition}\label{GKdatumeqdef}
Two $(G,K)$-data $\Psi$ and $\dot\Psi$ are said to be \textit{$K$-equivalent} if $\dot\Psi$ can be obtained from $\Psi$ by a finite sequence of refactorizations, $K$-conjugations and elementary transformations.
\end{definition}

The latter definition defines an equivalence relation on the set of $(G,K)$-data.
The discussion above implies that $K$-equivalent $(G,K)$-data $\Psi$ and $\dot\Psi$ give rise to equivalent representations $\kappa (\Psi)$ and $\kappa (\dot\Psi)$.  Theorem \ref{partialeqprob} provides a converse result, as well as a result that describes when generic cuspidal $G$-data yield equivalent representations of $G$. Note that the twisted Levi sequences occurring in
$K$-equivalent $(G,K)$-data are $K$-conjugate.
Therefore, Hypothesis~C($\vec\bG$) holds for the twisted Levi
sequence $\vec\bG$ occuring in one element of
an equivalence class of $(G,K)$-data if and only if it holds for
the twisted Levi sequences that occur in all $(G,K)$-data
in the given class.

The next result establishes that, in a weak sense, elementary transformations, $K$-conjugations and refactorizations commute with each other.

\begin{lemma}\label{GKeqlemone}
Let $\Psi$ and $\dot\Psi$
be $(G,K)$-data. Then the following are equivalent:
\begin{enumerate}
\item $\Psi$ and $\dot\Psi$ are $K$-equivalent.
\item $\dot\Psi$ is an elementary transformation of a $K$-conjugate of a refactorization of $\Psi$.
\item The previous statement remains valid when the terms ``elementary transformation,'' ``$K$-con\-ju\-gate'' and ``refactorization''  are permuted arbitrarily.
\end{enumerate}
\end{lemma}

\begin{proof}  Each of the 
Conditions (2) and (3) clearly implies Condition (1). 
Suppose now that $\dot\Psi = {}^k \Psi^\prime$, where $k\in K$ and $\Psi^\prime$ is a refactorization of $\Psi$.  Then $\dot\Psi = {}^k \Psi^\prime$ is a refactorization of ${}^k \Psi$.  It follows that a $K$-conjugate of a refactorization of $\Psi$ is the same as a refactorization of a $K$-conjugate of $\Psi$.
It is obvious that a $K$-conjugate (respectively, refactorization) of an elementary transformation  of $\Psi$ is the same as an elementary transformation of a $K$-conjugate (respectively, refactorization) of $\Psi$.
Our claim follows.
\end{proof}

\subsection*{Compatibility}
Fix a $K$-equivalence class $\xi$ of $(G,K)$-data.
If $\Psi,\dot\Psi\in \xi$ then $K_+(\Psi) = K_+(\dot\Psi)$ and therefore it makes sense to denote the latter groups by $K_+(\xi)$.  (This is a straightforward consequence of the definition of ``$K$-equivalence''  and the fact that the subgroups $K_+(\Psi)$ and $K_+(\dot\Psi)$ are normal subgroups of $K$.)    For simplicity, however, we will abbreviate $K_+(\xi)$ as $K_+$ in the following discussion.

We have defined a character
$\vartheta(\Psi)$ of $K_+$ and shown 
 that the restriction of $\kappa(\Psi)$
to $K_+$ is a multiple of $\vartheta(\Psi)$.  (See Corollary \ref{varthetaisotypy}.)   
The character of the restriction
of $\kappa(\Psi)$ to $K_+$ is equal to the degree of
$\kappa(\Psi)$ times the character $\vartheta(\Psi)$.
A similar statement applies  to the character of the restriction
of $\kappa(\dot\Psi)$ to $K_+$. As the two representations
are equivalent, they have the same character (and degree).
Hence $\vartheta(\Psi)=\vartheta(\dot\Psi)$ and we therefore are justified in denoting the latter characters by $\vartheta (\xi)$.

If $\theta$ is an involution of $G$ then a necessary condition for $\kappa (\Psi)$ to be $\theta$-distinguished is that $\vartheta (\xi)$ must be trivial on $K^\theta_+$.  According to the next result, the condition $\vartheta (\xi)\, |\, K^\theta_+=1$ only depends on the $K$-orbit of $\theta$.

\begin{lemma}\label{varthetaorbitinv}
If $\theta$ and $\theta^\prime$ are involutions of $G$ in the same $K$-orbit then 
$\vartheta (\xi)\, |\, K^\theta_+=1$ if and only if $\vartheta (\xi)\, |\, K^{\theta^\prime}_+=1$.
\end{lemma}

\begin{proof}
Let $V$ denote the space of $\kappa$ and let $\vartheta = \vartheta(\xi)$.  
Suppose $\theta^\prime= h\cdot \theta$, where $h\in K$. Since $K$ normalizes $K_+$, it is easy to check that $K^{\theta^\prime}_+ = K_+\cap G^{\theta^\prime} = K_+\cap hG^\theta h^{-1} 
= h(K_+\cap G^\theta) h^{-1} = hK^\theta_+ h^{-1}$.  The condition $\vartheta \,|\, K^\theta_+=1$ is equivalent to the condition that $\kappa (k) v=v$ when $v\in V$ and $k\in K^\theta_+$.   Assume $\vartheta\, |\, K^\theta_+ =1$.  If $v\in V$ and $k^\prime = hkh^{-1}\in K^{\theta^\prime}_+$
 then $\kappa (k^\prime) v = \kappa (hkh^{-1})v = \kappa (h)(\vartheta(k) \kappa(h)^{-1}v)= \vartheta (k)v$.  Thus $\vartheta\, |\,K^{\theta^\prime}_+=1$.  Similarly, the latter condition implies $\vartheta\, |\, K^\theta_+=1$.
\end{proof}

Let $\Xi$ denote the set of all $(G,K)$-data and let $\Xi^K$ be the set of all $K$-equivalence classes in $\Xi$.  Fix a $G$-orbit $\Theta$ of involutions of $G$ and recall that $\Theta^K$ denotes the set of $K$-orbits in $\Theta$.
Define two pairings between $\Theta^K$ and $\Xi^K$ by
\begin{equation*}
\begin{split}
\langle \Theta^\prime , \xi\rangle_K &= {\rm dim}\ {\rm Hom}_{K^\theta}(\kappa (\Psi),1),\cr
\langle \Theta^\prime , \xi\rangle_{K_+} &= {\rm dim}\ {\rm Hom}_{K^\theta_+}
(\vartheta (\xi),1) = 
\begin{cases}
1&\text{if $\vartheta(\xi)\, |\, K^\theta_+=1$,}\\
0&\text{otherwise,}
\end{cases}
\end{split}
\end{equation*}
 where $\theta\in \Theta^\prime\in \Theta^K$ and $\Psi\in \xi\in \Xi^K$. 
 
\begin{definition}\label{compatibilitydefs}
 If $\langle \Theta^\prime ,\xi\rangle_K$ is nonzero, we say $\Theta^\prime$ and $\xi$ are \textit{strongly compatible}.  If $\langle \Theta^\prime ,\xi\rangle_{K_+}$ is nonzero, we say
 $\Theta^\prime$ and $\xi$ are \textit{weakly compatible}.  If $\Theta^\prime$ and 
$\xi$ are weakly compatible and $\theta [y] = [y]$ for any (hence all) $\theta\in \Theta^\prime$
 and for any (hence all) $\Psi = (\vec\bG,y,\rho,\vec\phi)\in \xi$ then we say  $\Theta^\prime$
 and $\xi$ are \textit{moderately compatible}.
 \end{definition}

Each type of compatibility defines a correspondence 
 between $\Theta^K$ and $\Xi^K$.   The notion of strong compatibility is most pertinent to our main problem of computing the dimensions of the spaces
 $\Hom_{G^\theta}(\pi(\Psi),1)$,
 where $\theta\in \Theta$ and $\Psi\in\xi\in \Xi^K$.   Since the latter dimension only depends on $\Theta$ and $\xi$, we denote it by $\langle \Theta,\xi\rangle_G$.   Lemma \ref{orbmult} yields the formula
\begin{equation}\label{basicmulteq}
\langle \Theta,\xi\rangle_G = m_K(\Theta)\sum_{\Theta^\prime\in\Theta^K} \langle \Theta^\prime,
\xi\rangle_K,
\end{equation}
  which exhibits the connection between $\langle \Theta,\xi\rangle_G$ and the notion of strong compatibility.
 
 In the remainder of this section, we obtain information about strong compatibility by studying the auxiliary notions of weak and moderate compatibility.
 By definition, moderate compatibility implies weak compatibility.  One of the main technical results of this \paperbook, Proposition \ref{compatorbsum}, says that strong compatibility implies moderate compatibility.
 
\subsection*{Main results on compatibility}
We now state some results about weak, moderate and strong compatibility.  Most of the proofs are deferred until after all of the statements.  We will use many of our standard notations without recapitulating them.  In general, it is safe to assume that the notation 
$\Psi$ designates $(\vec\bG ,y,\rho,\vec\phi)$.  If $\Psi$ is a $(G,K)$-datum then we define
$$\rho^\prime (\Psi) = \rho (\Psi)\otimes \prod_{i=0}^d (\phi_i\, |\, K^0(\Psi)).$$
This is a representation of $K^0(\Psi)$.  A key property of $\rho^\prime(\Psi)$ is that 
$\rho^\prime(\Psi) = \rho^\prime(\dot\Psi)$, for all refactorizations $\dot\Psi$ of $\Psi$.

The next result connects the notions of weak and moderate compatibility with the notions of weak $\theta$-symmetry and $\theta$-symmetry introduced in Section \ref{sec:compatinv}.

\begin{proposition}\label{TFAErefactor} Let $\Theta^\prime\in 
\Theta^K$ and $\xi\in \Xi^K$. 
\begin{enumerate}
\item If $\Theta^\prime$ and $\xi$ are weakly compatible then:  for every
 $\theta \in\Theta^\prime$ there exists a weakly $\theta$-symmetric element $\Psi$ in $\xi$.
\item $\Theta^\prime$ and $\xi$ are moderately compatible precisely when:  for every
 $\theta \in\Theta^\prime$ there exists a $\theta$-symmetric $\Psi$ in $\xi$.
\end{enumerate}
\end{proposition}

\begin{corollary}\label{quaddistnecessary}
If $\Theta^\prime\in \Theta^K$ and $\xi\in \Xi^K$ are strongly compatible,
then for all $\theta\in \Theta^\prime$ there exists $\Psi\in \xi$ such that 
$\Psi$ is weakly $\theta$-symmetric.  
For such $\theta$ and $\Psi$, there exists a quadratic character $\chi$ of $K^0(\Psi)^\theta$ such that\break $\Hom_{K^{0}(\Psi)^\theta}(\rho^\prime(\Psi),\chi)$ is nonzero.  Furthermore, the latter space is nonzero whenever $\Psi$ is replaced by a refactorization.
\end{corollary}

\begin{proof}[Proof of Corollary \ref{quaddistnecessary}]
Fix $\theta\in \Theta^\prime$.  Since $\Theta^\prime$ and $\xi$ are strongly compatible they must be weakly compatible and we may use Proposition \ref{TFAErefactor}(1) to choose a weakly $\theta$-symmetric element $\Psi\in \xi$.  Strong compatibility and Proposition \ref{quaddistprop} imply that there exists a quadratic character $\chi$ of $K^0(\Psi)^\theta$ such that $\Hom_{K^0(\Psi)^\theta} (\rho (\Psi),\chi)$ is nonzero.  Since $\Psi$ is weakly $\theta$-symmetric, the 
character $\chi^\prime = \chi\otimes (\prod_{i=0}^d (\phi_i\,|\,K^0(\Psi)^\theta ))$ is also quadratic and 
$\Hom_{K^0(\Psi)^\theta}( \rho^\prime(\Psi),\chi^\prime)$ is nonzero.  
Hence our claim is true for $\Psi$.  It is also true for all refactorizations of $\Psi$ 
since $\rho^\prime(\Psi)$ is invariant under refactorizations.  
\end{proof}

Recall that  $\Theta^\prime$ and $\xi$ are weakly compatible when $\vartheta (\xi)\,|\, K^\theta_+=1$
 for some (hence all) $\theta\in \Theta^\prime$.   If $\Psi\in \xi$ and  $g\in K^0_+(\Psi)$, we have the simplified formula
$$
\vartheta (\xi) (g)= \prod_{i=0}^d \phi_i (g)
$$
 for the value of $\vartheta (\xi)(g)$.  We now give an alternate 
characterization of moderate compatibility. 

\begin{proposition}\label{oldLemmaA} 
Let $\Theta^\prime\in \Theta^K$ and  $\Psi =(\vec\bG , y, \rho,\vec\phi) \in \xi\in \Xi^K$. 
The following are equivalent.
\begin{enumerate}
\item $\Theta^\prime$ and $\xi$ are moderately compatible.
\item There exists $\theta\in \Theta^\prime$ such that 
$\theta (\bG^0) = \bG^0$, $\theta [y] = [y]$, and $\vartheta (\xi ) 
\,|\,K_+^0(\Psi)^\theta =1$.
\item There exists $\theta\in \Theta^\prime$ such that 
$\theta (\vec\bG) = \vec\bG$, $\theta [y] = [y]$, and $\vartheta (\xi ) 
\,|\,K_+^0(\Psi)^\theta =1$.
\end{enumerate}
\end{proposition}

We remark that 
the equivalence of (1) and
(2) is saying that moderate
compatibility may be detected at the level of
$G^0$: there must be a $\theta\in \Theta^\prime$
that stabilizes both $G^0$ and $[y]$, and
furthermore the restrictions of the quasicharacters
$\prod_{i=0}^d \phi_i\circ\theta\,|\, G^0$ and
$\prod_{i=0}^d \phi_i^{-1}\,|\, G^0$
to the group $G_{y,0^+}^0$ must agree. (To see this,
we use Proposition~\ref{twodivprop}.)

The first two conditions in Proposition~\ref{oldLemmaA}(3) imply 
that all of the groups  in Yu's construction (for example, the $J^i$'s and $K^i$'s) 
are $\theta$-stable.
We also remark that the conditions in (3) are invariant under 
refactorizations and elementary transformations.  In other words,  
Proposition~\ref{oldLemmaA} defines a relation between involutions 
$\theta$ in $\Theta^\prime$ and the equivalence classes $[\Psi]$ in $\xi$ 
under the equivalence relation defined by refactorizations and elementary 
transformations (but not $K$-conjugation).  If $\theta$ and $[\Psi]$ are 
related and $k\in K$ then $k\cdot\theta$ and $[{}^k\Psi]$ are also related.

\begin{proposition}\label{oldLemmasBandC} 
Assume  $\Theta^\prime,\Theta^{\prime\prime}\in \Theta^K$.  
Suppose $\theta\in\Theta^\prime$, $\Psi = (\vec\bG,y,\rho,\vec\phi)\in \xi\in \Xi^K$ 
and some refactorization of $\Psi$ is weakly $\theta$-symmetric.  
\begin{enumerate}
\item If $\Theta^\prime$ and $\xi$ are weakly compatible and  
$\Theta^{\prime\prime}$ and $\xi$ are also weakly compatible then 
there exists $g\in G$ such that $g\theta(g)^{-1}\in G^0$ and $g\cdot \theta \in 
\Theta^{\prime\prime}$.
\item  If $\Theta^\prime$ and $\xi$ are moderately compatible then  
$\Theta^{\prime\prime}$ and $\xi$ are also moderately compatible precisely when there exists $g\in G$ such that $g\theta(g)^{-1}\in K^0$ and $g\cdot \theta \in \Theta^{\prime\prime}$.
\end{enumerate}
\end{proposition}

\begin{remark}\label{remoldLemBandC}
 If $H^1_\theta(K^0)$ is trivial then Proposition \ref{oldLemmasBandC} implies that if
  $\Theta^\prime$ and $\xi$ are moderately compatible then $\Theta^\prime$ is the only 
element of $\Theta^K$ that is moderately compatible with $\xi$.
\end{remark}

\subsection*{Auxiliary lemmas}
In general, if $\cG$ is a topological group, we will let $\cG_{{\rm der}}$ denote
its derived group, that is, the closed subgroup of $\cG$ generated by all commutators.

\begin{definition}\label{defalphasymmetry}
If $\phi$ is a quasicharacter of a topological group $\cG$ and $\theta$ is an automorphism of $\cG$ of exponent two, then we say $\phi$ is  $\theta$-\textit{symmetric} if $\phi  (\theta (g)) = \phi (g)^{-1}$, for all $g\in \cG$.
\end{definition}

To prove Proposition \ref{TFAErefactor}, we assume we are given an involution $\theta$ of $G$ and a weakly compatible $(G,K)$-datum $\Psi$ and we show that we can replace the given $(\rho,\vec\phi)$ with a weakly $\theta^\prime$-symmetric refactorization $(\dot\rho,\vec{\dot\phi})$ for some $\theta^\prime$ in the $K$-orbit of $\theta$.  We construct $\dot\phi_d,\dots ,\dot\phi_0,\dot\rho$ recursively  by repeatedly applying the next two lemmas in sequence.

\begin{lemma}\label{firsthalf}
Fix a $(G,K)$-datum $\Psi$ and an integer $i\in \{\,0,\dots ,d\,\}$, where $d$ is the 
degree of $\Psi$.  
 Suppose that $\theta$ is an involution of $G$ such that $G^i,\dots , G^d$ are 
$\theta$-stable and $\vartheta (\Psi)\, |\, G_{y,r_{i-1}^+}^{i,\theta}=1$.  
Assume that a quasicharacter  $\dot\phi_j$ of $G^j$ has already been defined, satisfies 
Conditions {\bf F0} and {\bf F1} and is $\theta$-symmetric  for all 
$j\in \{\, i+1,\dots , d\,\}$.
Then there exists a quasicharacter $\dot\phi_i$ of $G^i$ that satisfies 
Conditions {\bf{F0}} and {\bf{F1}} and is $\theta$-symmetric.
\end{lemma}

\begin{proof}
Our claim is trivial if $i=d$ and $\phi_d=1$, so we assume we are not in this case. 
  Equivalently, we assume $r_{i-1}<r_i$.

If $i<d$, define a quasicharacter $\chi_{i+1}$ of $G^{i+1}$ by
$\chi_{i+1}=\prod_{j=i+1}^d \phi_j\dot\phi_j^{-1}\,|\, G^{i+1}$.
If $i=d$, let  $\chi_{d+1}$ be the trivial character of $G$.
Define a quasicharacter $\dot\phi_i^\circ$  of $G^i$ by 
$\dot\phi_i^\circ(g)=\phi_i(g)\chi_{i+1}(g)$, $g\in G^i$.

Note that $\prod_{j=i+1}^d \dot\phi_j\,|\, 
G^{i,\theta}_{y,r_{i-1}^+}=1$, because
of $\theta$-symmetry of each $\dot\phi_j$, together with the fact that
$p$ is odd. Hence $\dot\phi_i^\circ\,|\, G_{y,r_{i-1}^+}^{i,\theta}=
\vartheta(\Psi)\,|\, G_{y,r_{i-1}^+}^{i,\theta}=1$.
Using Proposition~\ref{twodivprop}, this implies that
$\dot\phi_i^\circ\circ\theta$ and $(\dot\phi_i^{\circ})^{-1}$
agree on $G_{y,r_{i-1}^+}^i\cap G_{\theta(y),r_{i-1}^+}^i$.
That is, $\dot\phi_i^\circ(\dot\phi_i^\circ\circ\theta)$
is trivial on $G_{y,r_{i-1}^+}^i\cap G_{\theta(y),r_{i-1}^+}^i$.
Applying Lemma~\ref{depthlemma}, we see that $\dot\phi_i^\circ\circ
\theta$ and $(\dot\phi_i^{\circ})^{-1}$ agree on $G^i_{r_{i-1}^+}$.
In particular, $\dot\phi_i^\circ(\theta(k))=\dot\phi_i^\circ(k)^{-1}$
for all $k\in G_{y,r_{i-1}^+}$.

Let $\cG=G^{i,ab}=G^i/G_{{\rm der}}^i$. 
Note that $\theta$ determines an involution of $\cG$.
Let $\cA$ be the image of $\{\, g\theta(g)\ | \ g\in G^i\,\}$
in $\cG$. This is a closed  subgroup of $\cG$ that
lies inside $\cG^\theta$. Let $\cB$ be the image
of $G_{y,r_{i-1}^+}^i$ in $\cG$. Because $\cA$ is
a closed subgroup and $\cB$ is a compact subgroup
of the locally compact abelian group $\cG$,
the subgroup $\cA\cB$ is closed.

The restriction $\dot\phi_i^\circ\,|\, G_{y,r_{i-1}^+}^i$
factors to a character $\dot\varphi_i$ of $\cB$.
Suppose that $k\in G_{y,r_{i-1}^+}^i$ is such that
$k\in g\theta(g)G_{{\rm der}}^i$ for some $g\in G$,
that is, the image of $k$ in $\cG$ lies in $\cA$.
Then $\theta(k)\in \theta(g)gG_{{\rm der}}^i=kG_{{\rm der}}^i$.
Hence $\dot\phi^\circ_i(\theta(k))=\dot\phi^\circ_i(k)$.
However, as observed above, $\dot\phi^\circ_i(\theta(k))
=\dot\phi^\circ_i(k)^{-1}$.
Since $k\in G^i_{0^+}$ (and $p$ is odd),
we have $\dot\phi^\circ_i(k)=1$. Thus
$\dot\varphi_i\,|\, \cB\cap \cA$ is trivial,
and we may extend $\dot\varphi_i$ to
a character of $\cA\cB$ by setting
$\dot\varphi_i(ab)=\dot\phi_i(b)$ for all $a\in \cA$
and $b\in \cB$.

Next we use the fact that any quasicharacter of a closed
subgroup of a locally compact abelian group  extends
to the full group. 
Let $\dot\phi_i$ be the quasicharacter of
$G^i$ corresponding to an extension of $\dot\varphi_i$
to $\cG$. The fact that $\dot\varphi_i$ is trivial
on $\cA$ is equivalent to $\theta$-symmetry of
$\dot\phi_i$. The fact that 
$\dot\phi^\circ_i\,|\, G_{y,r_{i-1}^+}^i=\phi_i\chi_{i+1}
\,|\, G_{y,r_{i-1}^+}^i$
factors to $\dot\varphi_i$ on $\cB$ implies that
$\dot\phi_i\,|\, G_{y,r_{i-1}^+}^i=\phi_i\chi_{i+1}
\,|\, G_{y,r_{i-1}^+}^i$. Hence
$\dot\phi_i$ satisfies Condition~\textbf{F1}.
\end{proof}

\begin{remark}\label{extendibility}
 The arguments used in the proof of Lemma~5.13
show that if $\bG^\prime$ is a connected reductive $F$-group,
$\theta$ is an involution of $G^\prime$, $t>0$,
$x\in \cB(\bG^\prime,F)$,
and $\phi$ is a quasicharacter of $G^\prime$ such
that $\phi\,|\, G_{x,t}^{\prime\theta}=1$, then
there exists a $\theta$-symmetric quasicharacter $\dot\phi$ of
$G^\prime$ such that $\dot\phi\,|\, G_{x,t}^\prime=
\phi\,|\, G_{x,t}^\prime$.
\end{remark}

\begin{lemma}\label{secondhalf}
Fix a $(G,K)$-datum $\Psi$ and an integer $i\in \{\, 0,\dots ,d-1\,\}$, where $d$ is 
the degree of $\Psi$.  Suppose that $\theta_{i+1}$ is an involution of $G$ such 
that $G^{i+1},\dots , G^d$ are $\theta_{i+1}$-stable and 
$\vartheta (\Psi)\, |\, J^{i+1,\theta_{i+1}}_+=1$.  Assume  a quasicharacter 
$\dot\phi_j$ of $G^j$ has already been defined, satisfies Conditions {\bf F0} and {\bf F1} 
and is $\theta_{i+1}$-symmetric  for all $j\in \{\, i+1,\dots , d\,\}$.   
 Let $\dot\phi_{i}^\bullet$ be the character of $G^{i}_{y,r_{i-1}^+}$ given by 
$\dot\phi^\bullet_i (g) = \phi_i (g)\ \chi_{i+1}(g)$, $g\in G^i_{y,r_{i-1}^+}$,
where $\chi_{i+1}=\prod_{\ell=i+1}^d\phi_\ell \dot\phi_\ell^{-1}\,|\, G^{i+1}$.
 Assume that $\chi_{i+1}\,|\,J_+^{i+1}$ is realized by an element of $\z^{i+1,*}_{-r_i}$.
  Then there exists a $G^{i+1}$-generic element $\dot\Gamma_{i}\in \z^{i\,*}_{-r_{i}}$ of 
depth $-r_i$ that realizes $\dot\phi^\bullet_{i}\,|\, G^{i}_{y,r_{i}}$ and  an involution 
$\theta_{i}$ in the $J^{i+1}$-orbit of $\theta_{i+1}$ such that 
$\theta_{i} (\dot\Gamma_i) = -\dot\Gamma_i$ and $\theta_{i}(G^i) = G^i$.
\end{lemma}

In order to prove Lemma \ref{secondhalf}, we need the following two lemmas.  

\begin{lemma}\label{auxlemone}
Let $(\bG^\prime,\bG)$ be a tamely ramified twisted Levi sequence.  
Suppose $t\in \R$ and let 
$$\g^{\prime,*}_{t^+}= \bigcup_{x\in \cB(\bG^\prime,F)} \g^{\prime ,*}_{x,t^+}.$$  
Suppose $g\in G$ and $\Gamma$ is a $G$-generic element of depth $t$ in 
$\z^{\prime,*}_{t}$, where $\z^\prime$ is the center of the Lie algebra of $G^\prime$.  If
$$\Ad^* (g) (\Gamma + \g^{\prime,*}_{t^+})\cap (\Gamma + 
\g^{\prime,*}_{t^+})\ne \emptyset$$ then $g\in G^\prime$.
\end{lemma}

\begin{proof}
Our proof is based on the proof of Lemma 2.2.4 in \cite{KMu} and, in fact, our result is a direct analogue of the latter result for the Lie algebra duals.  

Assume, as usual, that $E$ is a tamely ramified finite Galois extension of $F$ which splits 
$(\bG^\prime,\bG)$. Suppose our assertion holds with $F$ replaced by $E$, that is, with 
$(G^\prime,G)$ being replaced by $(\bG^\prime(E),\bG(E))$.  Then it follows that our claim also holds in general.  Indeed, suppose $g\in G$ and $\Gamma$ is a $G$-generic element of depth $t$ in $\z^{\prime,*}_{t}$ and 
$$\Ad^* (g) (\Gamma + \g^{\prime,*}_{t^+})\cap (\Gamma + \g^{\prime,*}_{t^+})\ne \emptyset.$$
Then
$$\Ad^* (g) (\Gamma + \bfr{g}^\prime(E)^{*}_{t^+})\cap (\Gamma + \bfr{g}^\prime(E)^{*}_{t^+})\ne \emptyset$$  and hence $g\in G^\prime = \bG^\prime(E)\cap G$.  

We now assume $E=F$.
The intersection in the statement of the lemma is an intersection of open sets in $\g^{\prime,*}$.  We assume that it is nonempty and consequently it contains a (semisimple) regular element $\Gamma + Y$ with $Y\in \g^{\prime,*}_{t^+}$.  Let $X\in \g^{\prime,*}_{t^+}$ be defined by $\Ad^*(g) (\Gamma + X ) =\Gamma+Y$.   There exist chambers $C_1$ and $C_2$ in $\cB (\bG^\prime,F)$ such that $X\in \g^*_{x,t^+}$, for all $x\in C_1$, and $Y\in \g^*_{x,t^+}$, for all $x\in C_2$.  Choosing $x\in C_1$ and $m\in G^\prime$ such that $m\cdot C_2 = C_1$, we have $X,\Ad^*(m)Y\in \g^{\prime,*}_{x,t^+}$.  Since $\Ad^*(mg)(\Gamma +X) = \Gamma + \Ad^*(m) Y$, we have
$$\Ad^*(mg) (\Gamma + \g^{\prime,*}_{x,t^+})\cap (\Gamma + \g^{\prime,*}_{x,t^+})\cap 
\g^{\prime,*}_{\rm reg}\ne \emptyset,$$ where $\g^{\prime,*}_{\rm reg}$ is the set of regular elements in $\g^{\prime,*}$.  According to Lemma 8.3 of \cite{Y}, the latter condition implies $mg\in G^\prime$ and hence $g\in G^\prime$.  This proves our assertion.
\end{proof}

\begin{lemma}\label{auxlemtwo}
Let $(\bG^\prime,\bG)$ be a tamely ramified twisted Levi sequence.  
Suppose $t\in \R$ and let 
$\g^{\prime,*}_{t_+}$ be defined as in Lemma \ref{auxlemone}.  
Suppose $\alpha$ is an involution of $G$ and $\Gamma$ is a $G$-generic element of depth $t$ in 
$\z^{\prime,*}_{t}$, where $\z^\prime$ is the center of the Lie algebra of $G^\prime$.  If
$$
\alpha (\Gamma + \g^{\prime,*}_{t^+})\cap (-\Gamma + \g^{\prime,*}_{t^+})\ne \emptyset
$$ then $\Gamma +\alpha(\Gamma)\in \z^{\prime,*}_{t^+}$ and $\alpha (G^\prime) = G^\prime$.
\end{lemma}

\begin{proof} If $X\in \g^*$, we denote the isotropy group of $X$ in $G$ by $Z_G(X)$.
Assume the intersection in the statement of the lemma is non\-empty.  Since this is a nonempty intersection of open sets in $\g^{\prime,*}$, it must contain a regular element $\alpha (\Gamma+X) = -\Gamma + Y$, where $X,Y\in \g^{\prime,*}_{t^+}$.  
  Regularity implies that $Z_G (-\Gamma +Y)$ is a maximal torus $S$ in $G$.  
Lemma~\ref{auxlemone} now implies $Z_G(-\Gamma + Y) = Z_{G^\prime}(-\Gamma +Y) = 
Z_{G^\prime}(Y)$ or, in other words, $S= Z_{G^\prime}(Y)$.  In particular, 
$S\subset G^\prime$. We show that $\alpha (S)\subset G^\prime$ using a similar argument:
 $\alpha (S) = Z_G(\alpha(-\Gamma+Y)) = Z_G (\Gamma +X) = Z_{G^\prime}(\Gamma + X) 
= Z_{G^\prime}(X)$.  

Since $S$ and $\alpha (S)$ are maximal tori in $G^\prime$, the center $Z^\prime$ 
of $G^\prime$ must be contained in $S\cap \alpha (S)$.  So $\Gamma \in \z^{\prime,*}\subset \s^* \cap \alpha (\s^*)$.   Thus  $\Gamma + \alpha (\Gamma) \in \s^*_{t^+}$.
Applying Lemma \ref{auxlemone} implies $\alpha (G^\prime) = Z_G (\alpha (\Gamma)) = Z_G (-\Gamma + (\Gamma +\alpha (\Gamma)))= Z_{G^\prime}(-\Gamma +(\Gamma +\alpha (\Gamma))) = 
Z_{G^\prime}(\alpha (\Gamma))$.  In particular, $\alpha (G^\prime)\subset G^\prime$.
  Applying $\alpha$, this becomes $G^\prime\subset \alpha (G^\prime)$.  
Hence, $\alpha (G^\prime) = G^\prime$.   It follows that 
$\Gamma + \alpha (\Gamma ) \in  \z^{\prime,*}_{t^+}$.
\end{proof}

\begin{proof}[Proof of Lemma \ref{secondhalf}]
Recall that $J_+^{i+1}=(G^i,G^{i+1})_{y,(r_i,s_i^+)}$. The
lattice $\gJ^{i+1}=(\g^i,\g^{i+1})_{y,(r_i,s_i^+)}$ in
$\g^{i+1}$ has the property that $\gJ_+^{i+1}/\g_{y,r_i^+}^{i+1}$
is isomorphic to $J_+^{i+1}/G_{y,r_i^+}^{i+1}$ via restriction
of the isomorphism between $\g^{i+1}_{y,s_i^+:r_i^+}$ and
$G_{y, s_i^+:r_i^+}^{i+1}$. 

Since $\chi_{i+1}\,|\, J_+^{i+1}$ is realized by an element of $\z^{i+1,*}_{-r_i}$,
the restriction $\chi_{i+1}\,|\, G_{y,r_i}^i$ is realized by
the same element.
Because $\phi_i$ is $G^{i+1}$-generic of depth $r_i$ with respect to $y$, 
Lemma~\ref{generictrivia} implies that $\dot\phi^\bullet_i$ is $G^{i+1}$-generic 
of depth $r_i$ with respect to $y$.
That is, there exists a $G^{i+1}$-generic element $\dot\Gamma^\prime_i\in \z^{i,*}_{-r_i}$ 
that realizes $\dot\phi^\bullet_i\, |\, G^i_{y,r_i}$.
 
The element $\dot\Gamma^\prime_i$ also realizes a character $\hat{\dot\phi}^\bullet_i$ of 
$J^{i+1}_+$.  This character is associated by duality with the coset $\dot\Gamma^\prime_{i} +
\gJ^{i+1\bullet}_+$,
where 
$$\gJ_+^{i+1\bullet}=
\{\, X\in \g^{i+1, *} \ | \ X(\gJ_+^{i+1})
\subset \gP_F\,\}.
$$
Throughout this proof, if ${\frak s}\subset \g^{i+1}$, the
notation ${\frak s}^\bullet$ will be used for the set of
elements $X$ in $\g^{i+1, *}$ such that $X({\frak s})
\subset \gP_F$.

 Lemma 8.6 of \cite{Y}  implies that 
$$\dot\Gamma^\prime_{i} +
\gJ^{i+1 \bullet}_+=  \Ad^*
(G^{i+1}_{y,s_{i}})(\dot\Gamma^\prime_{i} + \g^{i,*}_{y,(-r_{i})^+}).$$ 
However, since $G^{i+1}_{y,s_{i}}= J^{i+1} G^{i}_{y,s_{i}}$ and
$\Ad^* (G^{i}_{y,s_{i}})(\dot\Gamma^\prime_{i} + \g^{i,*}_{y,(-r_{i})^+})=
\dot\Gamma^\prime_{i} + \g^{i,*}_{y,(-r_{i})^+}$, we have
$$\dot\Gamma^\prime_{i} +
\gJ^{i+1 \bullet}_+ = \Ad^*
(J^{i+1})(\dot\Gamma^\prime_{i} + \g^{i,*}_{y,(-r_{i})^+}).$$ 

We now show that $\hat{\dot\phi}^\bullet_{i}\, |\,J^{i+1,\theta_{i+1}}_+=1$.  
First, we note that, as a consequence of our assumption regarding $\chi_{i+1}\,|\, J_+^{i+1}$,
we have $\hat{\dot\phi}_i^\bullet (g) = \chi_{i+1} (g)\hat\phi_i(g)$, 
for all $g\in J^{i+1}_+$.   Using Lemma \ref{hatphikplus}, this becomes 
$$\hat{\dot\phi}_i^\bullet (g) = \vartheta (g) \prod_{\ell = i+1}^d \dot\phi_\ell (g)^{-1},$$ 
for all $g\in J^{i+1}_+$, where $\vartheta = \vartheta (\Psi)$. Since we are assuming that $\vartheta\, |\, J^{i+1,\theta_{i+1}}_+=1$, it follows that $\hat{\dot\phi}^\bullet_i = \prod_{\ell = i+1}^d \dot\phi_\ell^{-1}$ on $J^{i+1,\theta_{i+1}}_+$.
Therefore, to show that $\hat{\dot\phi}^\bullet_{i} \,|\, J^{i+1,\theta_{i+1}}_+=1$, it suffices to show that $\dot\phi_\ell (g)=1$ for all $g\in G^{i+1,\theta_{i+1}}_{y,s_i^+}$ and all $\ell \in \{\, i+1,\dots ,d\,\}$.

By assumption, $\theta_{i+1}(G^{i+1})=G^{i+1}$ and the restriction $\dot\phi_\ell\,|\, G^{i+1}$
is $\theta_{i+1}$-symmetric. Hence if $g\in G^{i+1,\theta_{i+1}}_{y,s_i^+}$,
we have $\dot\phi_\ell(g)=\dot\phi_\ell(\theta_{i+1}(g))=\dot\phi_\ell(g)^{-1}$.
Hence $\dot\phi_\ell(g)=\pm 1$. However, since $G_{y,s_i^+}^{i+1}$
is a pro-$p$-group and $p$ is odd, the character $\dot\phi_\ell\,|\, G_{y,s_i^+}^{i+1}$
does not assume the value $-1$. As indicated above, this forces
$\hat{\dot\phi}^\bullet_{i}\,|\, J^{i+1,\theta_{i+1}}_+=1$.

Using properties of (the restriction 
to $\gJ^{i+1}_+/\g^{i+1}_{y,r_i^+}$) of
 the canonical isomorphism $e=e_{y,r_i}:\g^{i+1}_{y,s_i^+:r_i^+}
\rightarrow G^{i+1}_{y,s_i^+:r_i^+}$,
the fact that $\hat{\dot\phi}^\bullet_{i}\,|\,
J^{i+1,\theta_{i+1}}_+=1$ translates into properties of
the coset that realizes $\hat{\dot\phi}^\bullet_i$.
Indeed, from Lemma~\ref{expequivariance},
if $X\in \gJ^{i+1}_+\cap\theta_{i+1}(\gJ^{i+1}_+)$, then
there exists $k\in J^{i+1}_+\cap\theta_{i+1}(J^{i+1}_+)$ such
that $e(X+\g^{i+1}_{y,r_i^+})=k\,G^{i+1}_{y,r_i^+}$
and $e(\theta_{i+1}(X)+\g^{i+1}_{y,r_i^+})=\theta_{i+1}(k)G^{i+1}_{y,r_i^+}$.
Similarly, if $k\in J_+^{i+1}\cap \theta_{i+1}(J_+^{i+1})$, 
there exists $X\in \gJ_+^{i+1}\cap\theta_{i+1}(\gJ_+^{i+1})$ such that
the above relations hold. 

It now follows that
$\hat{\dot\phi}^\bullet_i(e(X+\g_{y,r_i^+}^{i+1}))=\psi(\dot\Gamma_i^\prime(X))=1$
for all $X\in \gJ_+^{i+1,\theta_{i+1}}$. This
implies that $\dot\Gamma^\prime{i}$ lies in
$$
(\gJ_+^{i+1}\cap \g^{i+1,\theta_{i+1}})^\bullet = \gJ_+^{i+1 \bullet} + 
(\g^{i+1,\theta_{i+1}})^\bullet,
$$ 
with the above equality following from Lemma~19.1 of \cite{HC}.
Note that $(\g^{i+1,\theta_{i+1}})^\bullet=\{\, X\in \g^{i+1, *}\ | \
\theta_{i+1}(X)=-X\,\}$. Here, we transfer the action of (the differential
of) $\theta_{i+1}$ on $\g^{i+1}$ to $\g^{i+1, *}$ in the obvious way.
It follows that there exists $Y\in \gJ^{i+1 \bullet}_+$ such that 
$\theta_{i+1} (\dot\Gamma^\prime_{i}  +Y)= -\dot\Gamma^\prime_{i} -Y$.  As we saw earlier
in the proof,
 $\dot\Gamma^\prime_{i} + \gJ^{i+1 \bullet}_+=
\Ad^* (J^{i+1})(\dot\Gamma^\prime_{i} + \g^{i,*}_{y,(-r_{i})^+})$.  Thus we can choose 
$k\in J^{i+1}$ and $Z\in \g^{i,*}_{y,(-r_{i})^+}$ such that $\dot\Gamma^\prime_{i} 
+Y = \Ad^* (k) (\dot\Gamma^\prime_{i} + Z)$.  Hence, there exists $\theta_{i}$ in the 
$J^{i+1}$-orbit of $\theta_{i+1}$ such that $\theta_{i} (\dot\Gamma^\prime_{i}+Z) =
 -\dot\Gamma_{i}^\prime- Z$. 
Our assertion now follows from Lemma \ref{auxlemtwo}.
\end{proof}

\begin{lemma}\label{eqlevi}
Let $(\bG^0,\bG^\nat,\bG)$ and $(\bG^0,\bG^\flat,\bG)$ be
tamely ramified twisted Levi sequences in $\bG$.
Let $\phi^\nat$ and $\phi^\flat$ be quasicharacters
of $G^\nat$ and $G^\flat$, respectively. Suppose
that there exist 
$x^\nat$ and $x^\flat\in \cB(\bG^0,F)$ 
and a real number $r>0$
such that $\phi^\nat$
and $\phi^\flat$
are $G$-generic of depth $r$,
relative to $x^\nat$ and to $x^\flat$,
respectively. 
If $\phi^\nat$ and
$\phi^\flat$ agree on $G_{x^\nat,r}^0\cap  G_{x^\flat,r}^0$,
then $G^\nat=G^\flat$.
\end{lemma}

\begin{proof}
Choose $G$-generic elements $\Gamma^\nat\in \z_{-r}^{\nat,*}$
and $\Gamma^\flat\in \z_{-r}^{\flat,*}$
such that $\Gamma^\nat +\g_{x^\nat,(-r)^+}^{\nat,*}$
and $\Gamma^\flat + \g_{x^\flat,(-r)^+}^{\flat,*}$
realize $\phi^\nat\,|\, G^\nat_{x^\nat,r}$
and $\phi^\flat\,|\, G^{\flat}_{x^\flat,r}$, respectively.

According to Lemma~\ref{depth}, the restrictions
of $\phi^\nat$ and $\phi^\flat$ to $G^0$ have depth $r$.
The elements $\Gamma^\nat$ and $\Gamma^{\flat}$ belong to
$\z^{0,\ast}$ and the corresponding cosets in
$\g^{0,*}_{x^\nat,r:r^+}$ and $\g^{0,*}_{x^\flat,r:r^+}$
realize the restrictions $\phi^\nat\,|\, G_{x^\nat,r}^0$
and $\phi^\flat\,|\, G_{x^\flat,r}^0$. Since
$\phi^\nat$ and $\phi^\flat$ agree on the
intersection of $G_{x^\nat,r}^0$
with $G_{x^\flat,r}^0$, Lemma~\ref{depthlemma}
shows that $\phi^\nat$ and $\phi^\flat$
agree on $G_{r}^0$. 
Applying Lemma~\ref{centralrep}, we conclude
that $\Gamma^\nat-\Gamma^\flat\in \z_{(-r)^+}^{0,*}$.
Since $\z_{(-r)^+}^{0,*}\subset \g^{\nat,*}_{(-r)^+}$,
we have $\Gamma^\flat\in \Gamma^\nat + \g_{(-r)^+}^{\nat,*}$.
It now follows from Lemma~\ref{auxlemone} that
if $g\in G$ and $\Ad^*(g)(\Gamma^\flat)=\Gamma^\flat$ 
then $g\in G^\nat$. Because $\Gamma^\flat$ is
a $G$-generic element of $\z^{\flat,*}_{-r}$,
$G^\flat$ is equal to the set of $g\in G$
such that $\Ad^*(g)(\Gamma^\flat)=\Gamma^\flat$.
Hence we have $G^\flat\subset G^\nat$.
Reversing the roles of $\Gamma^\flat$ and
$\Gamma^\nat$, we obtain $G^\nat\subset G^\flat$.
Thus $G^\flat=G^\nat$.
\end{proof}

\begin{lemma}\label{extra} Let $\Psi=(\vec\bG,y,\rho,\vec\phi)$
be a $(G,K)$-datum.
Suppose that $\theta$ is an involution of $G$ such that
$\theta(\bG^0)=\bG^0$ and 
$\vartheta(\Psi)\,|\, K_+^0(\Psi)^\theta
=1$. 
Then $\theta(\vec\bG)=\vec\bG$ 
and
there exists a weakly $\theta$-symmetric
refactorization of $\Psi$.
\end{lemma}

\begin{proof}
Suppose that $\phi_d$ is nontrivial. Because
$\vartheta(\Psi)\,|\, G_{y,r_{d-1}^+}^0=\phi_d\,|\,
G_{y,r_{d-1}^+}^0$, we have $\phi_d\,|\, G_{y,r_{d-1}^+}^{0,
\theta}=1$. Applying Proposition~\ref{twodivprop},
we have that $\phi_d\circ\theta$
and $\phi_d^{-1}$ agree on $G_{y,r_{d-1}^+}^0
\cap G_{\theta(y),r_{d-1}^+}^0$. Applying
Lemma~\ref{depthlemma}, we see that
$\phi_d\circ\theta$ and $\phi_d^{-1}$ agree
on $G_{r_{d-1}^+}^0$. This means that
the depth of the quasicharacter
$(\phi_d\circ\theta\,|\, G^0)\phi_d\,|\, G^0$
of $G^0$
is at most $r_{d-1}$. Applying 
Lemma~\ref{depth} to conclude that
the depth of $(\phi_d\circ\theta)\phi_d$
is at most $r_{d-1}$, we see that
$\phi_d\,|\, G_{y,r_{d-1}^+}^\theta
=1$. According to Remark~\ref{extendibility},
there exists
a $\theta$-symmetric quasicharacter $\dot\phi_d$
of $G$ such that $\dot\phi_d$ and $\phi_d$
agree on $G_{y, r_{d-1}^+}$.

If $\phi_d$ is trivial, let $\dot\phi_d=\phi_d$.

If $d=0$, setting $\dot\rho=\rho(\phi_0\dot\phi_0^{-1})$
and $\dot\Psi=(\dot\rho,(\dot\phi_0))$, we obtain
the required weakly $\theta$-symmetric refactorization
of $\Psi$.

Suppose that $d>0$. 
Let $\xi=\phi_{d-1}(\phi_d\dot\phi_d^{-1})\,|\,
G^{d-1}$. 
Noting that $\phi_d\dot\phi_d^{-1}=1$ or the depth of 
$\phi_d\dot\phi_d^{-1}$ is at most $r_{d-1}$,
we apply Lemma~\ref{generictrivia} (together with
the definition of $G$-generic quasicharacter of $G^{d-1}$)
and conclude that $\xi$ is $G$-generic (of depth $r_{d-1}$)
relative to $y$. It is immediate from this that
$\xi\circ\theta$ is a $G$-generic quasicharacter of
$\theta(G^{d-1})$ (of depth $r_{d-1}$) relative to
$\theta(y)$. 
From the equality
$$
\vartheta(\Psi)\,|\, G_{y,r_{d-2}^+}^0=\phi_{d-1}\phi_d\,|\,
G_{y,r_{d-2}^+}^0=\xi\dot\phi_d\,|\, G_{y,r_{d-2}^+}^0,
$$
together with the fact that $\dot\phi_d$ is $\theta$-symmetric,
and the assumption 
$\vartheta(\Psi)\,|\, K_+(\Psi)^{0,\theta}=1$,
it follows that $\xi\,|\, G_{y,r_{d-2}^+}^{0,\theta}
=1$. Applying Proposition~\ref{twodivprop}, we
see that $\xi\circ\theta$ and $\xi^{-1}$ 
agree on $G^0_{y,r_{d-2}^+}$.
Hence we may apply Lemma~\ref{eqlevi},
with $\bG^\prime=\theta(\bG^{d-1})$, $y^\prime=\theta(y)$,
$\phi^\prime=\xi\circ\theta$, $\bG^\flat=\bG^{d-1}$,
$y^\flat=y$, and $\phi^\flat=\xi^{-1}$, to show
that $\theta(\bG^{d-1})=\bG^{d-1}$.

 Applying Lemmas~\ref{depthlemma} and \ref{depth}
to $\xi$ (in the same manner as for
the case $\phi_d$ nontrivial), we can deduce 
from $\xi\,|\, G_{y,r_{d-2}^+}^{0,\theta}=1$
that $\xi\,|\, G_{y,r_{d-2}^+}^{d-1,\theta}
=1$. Then, according to Remark~\ref{extendibility},
there exists a $\theta$-symmetric quasicharacter
$\dot\phi_{d-1}$ of $G^{d-1}$ that agrees with
$\xi=\phi_{d-1}(\phi_d\phi_d^{-1}\,|\, G^{d-1})$
on $G_{y,r_{d-2}^+}^{d-1}$.

Continuing in this manner, we find that
$\theta(G^i)=G^i$ for all $i$, and construct
$\dot\phi_d,\dots,\dot\phi_0$ in sequence.
Finally, setting $\dot\rho=\rho\otimes(\prod_{i=0}^d
\phi_i\dot\phi_i^{-1}\,|\, K^0)$, we obtain
a weakly $\theta$-symmetric refactorization
of $\Psi$.
\end{proof}

\subsection*{The remaining proofs}

\begin{proof}[Proof of Proposition  \ref{TFAErefactor}]
We start by proving (1).  Assume that $\Theta^\prime$ and $\xi$ are weakly compatible and 
fix $\theta\in \Theta^\prime$ and   $\Psi\in \xi$.
The case $d=0$ follows immediately from Lemma \ref{firsthalf}.  

Therefore, we assume that $d>0$.  We want to show there exists $k\in K$ and a  
refactorization $\dot\Psi = (\vec\bG,y,\dot\rho,\vec{\dot\phi})$ of $\Psi$ such that 
${}^k\dot\Psi$ is weakly $\theta$-symmetric. We do this by constructing 
$\dot\phi_d,\dots \dot\phi_0,\dot\rho$ successively.

Choose a quasicharacter $\dot\phi_d$ of $G$, by applying Lemma \ref{firsthalf} with $i=d$, 
so that  $\dot\phi_d$  is 
$\theta$-symmetric and satisfies Conditions \textbf{F0} and \textbf{F1}.  
Next, noting that Hypothesis~C($\bG$), together with the fact that
$\phi_d\dot\phi_d^{-1}\,|\, G_{y,r_{d-1}^+}=1$ and $G_{y,s_{d-1}^+}
\supset J^d_+$, shows that the
assumptions of Lemma~\ref{secondhalf} are satisfied, with $i=d-1$,
$\chi_{d+1}=\phi_d\dot\phi_d^{-1}$, and $\theta_d = \theta$.
Hence, letting $\dot\phi_{d-1}^\bullet$ be as in the statement of
Lemma~\ref{secondhalf}, 
there exists a $G$-generic element $\dot\Gamma_{d-1}$ that realizes 
$\dot\phi^\bullet_{d-1}\, |\, G^{d-1}_{y,r_{d-1}}$, and an involution $\theta_{d-1}$ 
in the $J^d$-orbit of $\theta$ such that $\theta_{d-1}(\dot\Gamma_{d-1}) = 
-\dot\Gamma_{d-1}$ and $\theta_{d-1}(G^{d-1}) = G^{d-1}$.

Now we apply Lemma \ref{firsthalf} with $i=d-1$ and $\theta$ replaced by $\theta_{d-1}$.
Note that $\dot\phi_d$ is $\theta_{d-1}$-symmetric, since it is $\theta$-symmetric 
and $\theta_{d-1}$ is in the $J^d$-orbit of $\theta$.  
We obtain a quasicharacter $\dot\phi_{d-1}$ that is $\theta_{d-1}$-symmetric and 
satisfies Condition \textbf{F1}.  (Condition \textbf{F0} is vacuous.)  
If $d=1$ we define $\dot\rho$ by Condition \textbf{F2} and we are done.  
Otherwise, we repeatedly apply Lemmas \ref{secondhalf} and \ref{firsthalf} until we 
have defined $\dot\phi_0$. Finally, we define $\dot\rho$ as in Condition~\textbf{F2}.

In this way we can construct an involution $\theta^\prime\in \Theta^\prime$ 
and a refactorization 
$\dot\Psi$ of $\Psi$ such that $\dot\Psi$ is weakly $\theta^\prime$-symmetric.  
Now choose $k\in K$ such that $\theta^\prime = k^{-1}\cdot\theta$.  
Then ${}^k\dot\Psi$ is a weakly $\theta$-symmetric element of $\xi$, and this proves (1).

The previous argument also yields the more difficult part of (2).  It only remains to show that if $\theta$ is an involution of $G$ and $\Psi$ is a $\theta$-symmetric $(G,K)$-datum then the character $\vartheta = \vartheta (\Psi)$ has trivial restriction to $K_+^\theta$.  It suffices to show that for all $i\in \{\, 0,\dots ,d\,\}$ the character $\hat\phi_i$ has trivial restriction to $K_+^\theta$.  According to Proposition \ref{KJfactor}, we have
$$K^\theta_+ = K^{0,\theta}_+ J^{1,\theta}_+\cdots J^{d,\theta}_+$$
and so Lemma \ref{hatphikplus} implies $$\hat\phi_i\, |\,K^\theta_+ = \inf\nolimits_{K^{i+1,\theta}_+}^{K^\theta_+}(\hat\phi_i\, |\, K^{i+1,\theta}_+).$$  It therefore suffices to show that $\hat\phi_i$ is trivial on $K^{i+1,\theta}_+ = K^{i,\theta}_+J^{i+1,\theta}_+$.
If $k\in K^{i,\theta}_+$, we have $\hat\phi_i (k) = \phi_i (k)= \phi_i (\theta(k)) = \phi_i(k)^{-1} = \hat\phi_i(k)^{-1}$.  Now using the fact that a character of a pro-$p$-group cannot assume the value $-1$, we deduce that $\hat\phi_i$ is trivial on $K^i_+$.  A similar argument may be used to show that $\hat\phi_i$ is trivial on $J^{i+1,\theta}_+$ since, according to Lemma \ref{stableJ}, we have 
$\hat\phi_i\circ\theta\,|\, J_+^{i+1} = \hat\phi_i^{-1}\, |\, J_+^{i+1}$.
\end{proof}

\begin{proof}[Proof of Proposition  \ref{oldLemmaA}]
Assume $\Theta^\prime$ and $\xi$ are moderately compatible and $\Psi\in \xi$.  
Then $\theta[y]=[y]$ and $\vartheta(\xi)\,|\, K_+^0(\Psi)^\theta=1$
for all $\theta\in \Theta^\prime$.
Now fix $\theta^\prime\in \Theta$.  According to Proposition \ref{TFAErefactor}(2) 
and Lemma \ref{GKeqlemone}, we may choose 
a refactorization $\dot\Psi$ of $\Psi$ and $k\in K$ such that 
${}^k\dot\Psi$ is $\theta^\prime$-symmetric.  
Let $\theta = k^{-1}\cdot\theta^\prime$.  Then (see Lemma~\ref{varthetaorbitinv})
the conditions of (3) are all 
satisfied with $\Psi$ replaced by $\dot\Psi$.
But, as remarked after the statement of Proposition \ref{oldLemmaA}, 
these conditions are invariant under refactorizations.  
Hence (1) implies (3).

Next,  assume the conditions in (2) are satisfied.  
Then Lemma \ref{extra} implies that $\theta(\vec\bG)=\vec\bG$
and there exists a weakly $\theta$-symmetric refactorization 
$\dot\Psi$ of $\Psi$. Suppose $\theta^\prime$ is an arbitrary element of 
$\Theta^\prime$.  Choose $k\in K$ so that $\theta^\prime = k\cdot \theta$.  
Then ${}^k\dot\Psi$ is a weakly $\theta^\prime$-symmetric element of $\xi$. 
Therefore Proposition \ref{TFAErefactor} implies $\Theta'$ and $\xi$ 
are moderately compatible, and (1) holds.

Since (3) clearly implies (2), our assertion now follows.
\end{proof}

\begin{proof}[Proof of Proposition  \ref{oldLemmasBandC}]
Assume  $\Theta^\prime,\Theta^{\prime\prime}\in \Theta^K$.  
In the statement of Proposition \ref{oldLemmasBandC}, we fix  $\theta\in\Theta^\prime$
 and $\Psi = (\vec\bG,y,\rho,\vec\phi)\in \xi\in \Xi^K$ such that some refactorization of $\Psi$ is weakly $\theta$-symmetric.  
It is easy to see that there is no loss in generality in assuming that $\Psi$ is itself weakly $\theta$-symmetric and we will do this.

Assume $\Theta^\prime$ and $\Theta^{\prime\prime}$ are  weakly compatible with $\xi$. 
 Choose $\theta^\prime\in \Theta^{\prime\prime}$.
  Proposition \ref{TFAErefactor}(1) allows us to choose a weakly 
$\theta^\prime$-symmetric element  $\Psi^\prime\in \xi$.  Lemma \ref{GKeqlemone}
implies that there exists $k\in K$ and a refactorization $\dot\Psi$ of $\Psi$ such that 
$\Psi^\prime = {}^k\dot\Psi$.  Let $\theta_{d-1} = k^{-1}\cdot \theta^\prime$. 
 Then $\dot\Psi$ is $\theta_{d-1}$-symmetric.

Choose $j_{d+1}\in G$ so that  $\theta_{d-1} = j_{d+1}\cdot \theta$.
We will show that we may choose $j_d\in J^d,\dots, j_1\in J^1$ such that if $h_{i+1} = j_{i+1}\cdots j_{d+1}$ and $g_i = h_{i+1}\theta (h_{i+1})^{-1}$ then $g_i\in G^i$.

Fix $i\in \{\, 0,\dots , d-1\,\}$ and assume $j_d,\dots, j_{i+2}$ have been defined.  Let $\beta_i$ be the coset in $\g^{i,*}_{-r_i:(-r_i)^+}$ corresponding to $\phi_i\, |\, G^i_{y,r_i}$.  
As can be seen upon examining the proof of Proposition \ref{TFAErefactor},  the 
condition $\phi_i\circ\theta = \phi_i^{-1}$ guarantees the existence of a $G^{i+1}$-generic 
element $\Gamma_i\in \z^{i,*}_{-r_i}\cap\beta_i$ such that $\theta(\Gamma_i)=-\Gamma_i$ or, 
in other words, $\Gamma_i\in (\g^{i+1,\theta})^\bullet$.  

The coset that realizes
the character $\hat\phi_i\, |\,J^{i+1}_+$ is the coset $\hat\beta_i = \Gamma_i + 
\gJ^{i+1\bullet}_+$, where $\gJ^{i+1}$ and $\gJ^{i+1\bullet}$ are as
in the proof of Lemma~\ref{secondhalf}.
As shown in the proof of Lemma \ref{secondhalf}, we have 
$\hat\beta_i = \Ad^*(J^{i+1})\beta_i$.  Let $h_{i+2} = j_{i+2}\cdots j_{d+1}$ and $\theta_i = h_{i+2} \cdot\theta$.
Let $\vartheta = \vartheta (\xi)$.  Weak compatibility implies $\vartheta\, |\, K^{\theta_i}_+ 
= 1$ and then Lemma \ref{hatphikplus}
implies $$\left( \hat\phi_i \mid J^{i+1,\theta_i}_+ \right)\prod_{\ell= i+1}^d\left( \phi_\ell \mid J^{i+1,\theta_i}_+\right) =1.$$
On the other hand, if $\ell\in \{\, i+1,\dots, d\,\}$ then it is easy to see that $\phi_\ell\, |\,
J^{i+1,\theta_i}_+=1$.  Indeed, if $\gamma \in J^{i+1,\theta_i}_+$ then $\gamma = h_{i+2}\alpha h_{i+2}^{-1}$ for some $\alpha\in G^\theta$ and we have $\phi_\ell (\gamma) = \phi_\ell (g_{i+1}\theta (\gamma)g_{i+1}^{-1}) = \phi_\ell (\theta (\gamma)) = \phi_\ell (\gamma)^{-1}$.  So $\phi_\ell\, |\,J^{i+1,\theta_i}_+$ must be trivial since it is a quadratic character of a pro-$p$-group.  
We deduce that $\hat\phi_i\, | \, J^{i+1,\theta_i}_+=1$
and thus, arguing as in the proof of Lemma~\ref{secondhalf},
 $\Gamma_i \in (\gJ^{i+1}_+\cap\g^{i+1,\theta_i} )^\bullet = 
\gJ^{i+1\bullet}_+  + (\g^{i+1,\theta_i})^\bullet$ or, equivalently, 
$\hat\beta_i \cap (\g^{i+1,\theta_i})^\bullet\ne \emptyset$.

We may now choose $j_{i+1}\in J^{i+1}$ and $Z_i\in \g^{i,*}_{y,(-r_i)^+}$ 
such that 
$$
\theta_i (\Ad^*(j_{i+1})(\Gamma_i+ Z_i)) = - \Ad^* (j_{i+1})
(\Gamma_i + Z_i).
$$  
Hence, $-\theta_{i-1}(\Gamma_i + Z_i) = \Gamma_i+ Z_i$, where 
$\theta_{i-1} = j_{i+1}\cdot\theta_i$.  Therefore, 
$-\theta_{i-1}(\beta_i)\cap\beta_i \ne\emptyset$.  This implies 
$-\Ad^* (g_i)\theta (\beta_i)\cap\beta_i \ne \emptyset$, where $h_{i+1} = j_{i+1}h_{i+2}$ and $g_i = h_{i+1}\theta(h_{i+1})^{-1}$.  
Consequently,
$$\Ad^*(g_i)^{-1} (\Gamma_i + \g^{i,*}_{(-r_i)^+})\cap (\Gamma_i + \g^{i,*}_{(-r_i)^+})\ne \emptyset.$$  Lemma \ref{auxlemone} implies $g_i\in G^i$.

This completes the construction of the sequence $j_{d+1},\dots ,j_1$. Taking $g = h_1 = j_1\cdots j_{d+1}$, we have $g\cdot \theta = \theta_{-1} \in \Theta^{\prime\prime}$ 
and $g\theta(g)^{-1} = g_0 \in G^0$ which proves Part (1) of Proposition \ref{oldLemmasBandC}.

In the previous discussion, if $\Theta^\prime$ and and $\Theta^{\prime\prime}$ 
are both moderately compatible
 with $\xi$ then $g\theta(g)^{-1}[y] = (g\theta(g)^{-1})\cdot\theta [y] = 
(g\cdot\theta)[y] =[y]$.  Therefore, $g\theta(g)^{-1} \in G^0_{[y]} = K^0$.
  This proves one half of Part (2).

Finally, we assume $\Theta^\prime$ and $\xi$ are moderately 
compatible and $g\in G$ satisfies $g\theta (g)^{-1}\in K^0$ and $g\cdot\theta\in 
\Theta^{\prime\prime}$. 
Then $(g\cdot \theta)\vec\bG = \Int(g\theta(g)^{-1})(\theta (\vec\bG)) = \Int (g\theta(g)^{-1})
(\vec\bG) = \vec\bG$.  Similarly, one may verify that 
$(g\cdot\theta)\vec\phi = \vec\phi^{-1}$ and we also have $(g\cdot\theta)[y]=[y]$.  
Therefore, $\Psi$ is $(g\cdot\theta)$-symmetric.  Assume $\theta^\prime\in 
\Theta^{\prime\prime}$.  Choose $k\in K$ so that $\theta^\prime = kg\cdot \theta$.  
Then ${}^k\Psi\in \xi$ is $\theta'$-symmetric.
Therefore, Proposition \ref{TFAErefactor}(2) implies that $\Theta^{\prime\prime}$
 and $\xi$ are moderately compatible.  This completes the proof of Part (2).
\end{proof}

\section{Strong compatibility}
\label{sec:compatorbs}

Let $\Theta$ be a $G$-orbit of involutions of $G$ and
fix a $K$-equivalence class $\xi\in \Xi^K$ of $(G,K)$-data.  
{As in the previous section, we assume that
Hypothesis~\rm{C}($\vec\bG$) holds for some (hence all) tamely 
ramified twisted Levi 
sequences
$\vec\bG$ that occur in the $(G,K)$-data in $\xi$.}
Recall that $\langle \Theta,\xi\rangle_G$ denotes
the dimension of $\Hom_{G^\theta}(\pi (\Psi),1)$, where $\theta$ and $\Psi$ are arbitrary elements of $\Theta$ and $\xi$.  Equation \ref{basicmulteq} expresses this in terms of the constants $\langle \Theta^\prime,\xi\rangle_K$, where $\Theta^\prime\in \Theta^K$ is a $K$-orbit in $\Theta$.  When $\langle \Theta^\prime,\xi\rangle_K$ is nonzero, we say that $\Theta^\prime$ and $\xi$ are strongly compatible.  The following result provides the key information about strong compatibility that 
ultimately allows us to obtain a formula for  $\langle \Theta^\prime,\xi\rangle_K$ in Theorem \ref{maindimformula}.
 
\begin{proposition}\label{compatorbsum} 
Suppose that $\Theta^\prime\in \Theta^K$ and $\xi\in \Xi^K$
are strongly compatible.
Then $\Theta^\prime$
and $\xi$ are moderately compatible.
\end{proposition}

\begin{proof} Assume $\Theta^\prime$ and $\xi$ are strongly compatible.  
Then they must also be weakly compatible.  Fix $\theta\in \Theta^\prime$.  
According to Proposition \ref{TFAErefactor}(1), we may also fix a weakly 
$\theta$-symmetric element $\Psi = (\vec\bG,y,\rho,\vec\phi)\in \xi$.  
It remains to show that $\theta [y] =[y]$.  
Therefore, we now suppose $\theta [y]\ne [y]$ and we proceed to arrive 
at a contradiction of the fact that $\langle \Theta^\prime,\xi\rangle_K$ is nonzero.

According to Corollary \ref{quaddistnecessary} there exists a character $\chi$ of $K^{0,\theta}$ such that $\chi^2=1$ and $\Hom_{K^{0,\theta}}(\rho,\chi)\ne 0$.

There exists an apartment $A$ in $\cB (\bG^0,F)$ that contains $y$ and $\theta(y)$.  There also exists $x\in \cB (\bG^0,F)$ such that $[y]$ lies on the boundary of the facet of 
$[x]$ in $A$ and $G^0_{x,0}/G^0_{y,0^+}$ is a proper parabolic subgroup of $G^0_{y,0:0^+}$ with unipotent radical $G^0_{x,0^+}/G^0_{y,0^+}$ and
$$G^0_{x,0^+}= (G^0_{y,0}\cap G^0_{\theta(y),0^+})G^0_{y,0^+}.$$ Here, we are viewing $G^0_{y,0:0^+}$ as the rational points of a reductive group over the residue field of $F$.

We claim that $G^0_{x,0^+}= G^{0,\theta}_{x,0^+}G^0_{y,0^+}$.  To see this, fix a coset in $G^0_{x,0^+}/G^0_{y,0^+}$.  We may choose a representative $g$ for this coset which lies in $G^0_{y,0}\cap G^0_{\theta(y),0^+}$.  Then the commutator $h= g^{-1}\theta (g)^{-1} g\theta (g)$ lies in $Z^1_\theta (G^0_{y,0^+}\cap G^0_{\theta(y),0^+})$, in the notation of
 Section~\ref{sec:stablesubgroups}.
  Using Proposition \ref{twodivprop}, we choose 
$\alpha\in G^0_{y,0^+}\cap G^0_{\theta(y),0^+}$ such that 
$h=\alpha\theta (\alpha)^{-1}$.  Now if $g'= g\theta (g)\theta (\alpha)$ 
then $g'\in G^{0,\theta}_{x,0^+}$ and $g'G^0_{y,0^+}= gG^0_{y,0^+}$. 
 This establishes that $G^0_{x,0^+}= (G^0_{x,0^+})^\theta G^0_{y,0^+}$.

The restriction of $\rho$ to $G_{y,0}^0$ factors to a cuspidal representation
$\bar{\rho}$ of $G^0_{y,0:0^+}$.
Let $H= G^0_{x,0^+}/G^0_{y,0^+}$.  Since $H$ is unipotent and $\chi^2=1$ and $p$ is odd, 
it must be the case that $\chi$ restricts to the trivial character of $H$.  Since 
$G^0_{x,0^+}= G^{0,\theta}_{x,0^+}G^0_{y,0^+}$, it follows that 
$\Hom_H (\bar{\rho},1)\ne 0$.  
But since this contradicts the cuspidality of $\bar{\rho}$, 
it follows that $\langle \Theta^\prime ,\xi\rangle_K =0$.
\end{proof}

\section{Finiteness results}\label{ledimestfin}

We show in this section that the constants $\langle \Theta,\xi\rangle_G$ are finite.    Recall that the expression for $\langle \Theta,\xi\rangle_G$ in equation \ref{basicmulteq} came directly from the Mackey theory formula
\begin{equation}\label{basicmulteqb}
\Hom_{G^\theta}(\pi (\Psi),1)\cong \bigoplus_{KgG^\theta\in K\bs G/G^\theta} \Hom_{K\cap gG^\theta g^{-1}} (\kappa (\Psi) ,1),
\end{equation}
 with $\theta\in \Theta$ and $\Psi\in \xi$.
The double coset space $K\bs G/G^\theta$ is discrete, since $K$ is open, but it is infinite when $G/G^\theta Z$ is noncompact.  Establishing  the finiteness of $\langle \Theta ,\xi\rangle_G$ is essentially equivalent to showing that only finitely many summands in equation \ref{basicmulteqb} can be nonzero.  

Consider the mapping from $\bG$ into itself defined by $g\mapsto g\theta (g)^{-1}$ and let $\bS_\theta$ denote the image of this map viewed as an affine variety on which $\bG$ acts by $g\cdot h= gh\,\theta(g)^{-1}$.   Note that the set of $F$-rational points in $\bS_\theta$ is $$\bS_\theta (F)= \{\, g\theta (g)^{-1} \ | \ g\in \bG\,\}\cap G$$ and this is not necessarily the same as the set
$$\cS_\theta = \{\, g\theta (g)^{-1}\ | \ g\in G\,\}.$$
Clearly, $\cS_\theta \subseteq \bS_\theta (F)$.

The following  lemma is standard, though its proof is not readily available in the literature:

\begin{lemma}\label{standardfact}
The map $g\mapsto g\theta(g)^{-1}$ gives a homeomorphism from $G/G^\theta$ to $\cS_\theta$, where $G$, $G^\theta$ and $\cS_\theta$ carry their natural $p$-adic topologies and $G/G^\theta$ has the quotient topology.
\end{lemma}

\begin{proof}
It is shown in \cite{R} that $\bS_\theta$ is a closed subvariety of $\bG$ and $g\mapsto g\theta(g)^{-1}$ gives an isomorphism $\bG/\bG^\theta \to \bS_\theta$ of affine $\bG$-varieties.  Since all of the objects involved are defined over $F$, this restricts to a Zariski  isomorphism $(\bG/\bG^\theta)(F)\to\bS_\theta (F)$ over $F$.  Since the latter map is an isomorphism of $F$-varieties, it must also be a homeomorphism with respect to the $p$-adic topology.  

The $p$-adic homeomorphism between $(\bG/\bG^\theta)(F)$ and $\bS_\theta (F)$ restricts to a $p$-adic homeomorphism $G/G^\theta \to \cS_\theta$, where $G/G^\theta$ has the subspace topology inherited from $(\bG/\bG^\theta)(F)$. 
We now  verify that the $p$-adic quotient topology on $G/G^\theta$ is in fact identical to the subspace topology obtained from $(\bG/\bG^\theta)(F)$.  We begin by observing that the open subsets of $\bS_\theta (F)$ have the form $\bS_\theta (F)\cap U$, as $U$ varies over the open subsets of $G$.    It follows that the open sets in $(\bG/\bG^\theta)(F)$ are the sets of the form
$$\{\, g\bG^\theta\ | \ g\in \bG,\ g\theta (g)^{-1}\in U\,\}.$$  Thus the open sets in $G/G^\theta$ in the subspace topology inherited from $(\bG/\bG^\theta)(F)$ are the sets
$$X_U = \{\, gG^\theta\ | \ g\in G, g\theta (g)^{-1}\in U\,\},$$ where $U$ is an open set in $\cS_\theta$.  Now each open set $U$ in $\cS_\theta$ pulls back via $g\mapsto g\theta (g)^{-1}$ to an open set $\widetilde U$ in $G$ such that $\widetilde U = \widetilde U G^\theta$.  We have
$$X_U = \{\, gG^\theta \ | \ g\in \widetilde U\,\}.$$ But the sets of the latter form are precisely the open sets in $G/G^\theta$ with respect to the quotient topology.  Thus we have established that we have a $p$-adic homeomorphism $G/G^\theta \to \cS_\theta$, where $G/G^\theta$ has its usual ($p$-adic) quotient topology.
\end{proof}

\begin{proposition}\label{finrelorbs}
Assume that $\Theta$ is a $G$-orbit of involutions of $G$ and $\xi\in \Xi^K$.
\begin{enumerate}
\item The set of $K$-orbits $\Theta^\prime\in \Theta^K$ that are 
moderately compatible with $\xi$ is finite.  
\item $\langle \Theta,\xi \rangle_G$ is finite.
\item The number of nonzero summands in equation \ref{basicmulteqb} is finite.
\end{enumerate}
\end{proposition}

\begin{proof}
It is easy to see from the discussion in Section~\ref{sec:Mackey} that all of our assertions follow from (1).
We may as well assume that there exists at least one element $\Theta^\prime\in \Theta^K$ that is moderately compatible with $\xi$.  Fix an involution $\theta\in \Theta^\prime$.
  Then, according to Proposition \ref{oldLemmasBandC}(2), the set of elements of $\Theta^K$ that are moderately compatible with $\xi$ is identical to the set
$$\{\, (Kg)\cdot \theta\ | \ g\in G^0,\ g\theta (g)^{-1}\in K^0\,\}.$$
The cardinality of the latter set is less than or equal to the cardinality of
$$S_1 =\{\, (Kg)\cdot \theta\ | \ g\in G,\ g\theta (g)^{-1}\in K\,\}.$$
Therefore it suffices to show that $S_1$ is finite.

We now use the theory from Section~\ref{sec:Mackey} to give another reformulation of the problem.  
From Lemma \ref{standardfact}, we see that $g\mapsto g\theta (g)^{-1}$ determines a homeomorphism of $G/G^\theta$ with $\cS_\theta$.  This gives a homeomorphism of the discrete space $K\bs G/G^\theta$ with the space $\cS^K_\theta$ of $K$-orbits in $\cS_\theta$.
Let $S_2$ be the set of $K$-orbits in $\cS_\theta$ with a representative in $\cS_\theta\cap K$.  In other words, $S_2$ is the preimage of $S_1$ in $\cS^K_\theta$ under the map $\cS^K_\theta\to \Theta^K$ of Section~\ref{sec:Mackey}.  So it suffices to show that $S_2$ is finite or, equivalently, that it is compact.

Instead of studying $K$-orbits, it is more convenient to study $Z$-orbits.  Note that  the set  of $Z$-orbits in $\cS_\theta$ is  identical to $\cS_\theta / B^1_\Theta$, where $B^1_\Theta$ is defined as in Section~\ref{sec:Mackey} and it acts on $\cS_\theta$ by translations.  The set of elements of $\cS_\theta / B^1_\Theta$ with a representative in $\cS_\theta\cap K$ is just $(\cS_\theta\cap K)/B^1_\Theta$.  It suffices to show that $(\cS_\theta\cap K)/B^1_\Theta$ is compact, since the quotient map $\cS_\theta /B^1_\Theta \to \cS^K_\theta$ from $Z$-orbits to $K$-orbits maps compact sets to compact sets and it maps $(\cS_\theta\cap K)/B^1_\Theta$ to $S_2$.

Next, we observe that $\cS_\theta$ is closed in $G$ and thus $\cS_\theta\cap K$ has compact image in $G/Z$.  This may be seen as follows.  Since $G$ is Hausdorff, the diagonal $\Delta G$ of $G\times G$ is closed.  It follows that $\{ (g,\theta (g))\ |\ g\in G\}$ is also closed, since $(g_1,g_2)\mapsto (g_1,\theta (g_2))$ gives a homeomorphism $G\times G\to G\times G$.  Now we use the fact that the map $(G\times G)/\Delta G\to G$ given by $(g,1)\Delta G\mapsto g$ is a homeomorphism to deduce that $\cS_\theta$ is closed in $G$.

Let
 $Z^1_\Theta$   be defined as in Section~\ref{sec:Mackey}.  The quotient map $K/Z^1_\Theta\to K/Z$ is injective on $(\cS_\theta\cap K)/Z^1_\Theta$.  Since $\cS_\theta\cap K$ has compact image in $K/Z$, it follows that $(\cS_\theta \cap K)/Z^1_\Theta$ is compact.
 
Showing that $(\cS_\theta \cap K)/B^1_\Theta$ is compact now reduces to showing that the quotient map $q_\theta : (\cS_\theta \cap K)/B^1_\Theta \to (\cS_\theta \cap K)/Z^1_\Theta$ is proper, that is, the preimage of a compact set is compact.

Lemma \ref{mThetabound} implies that
the quotient group $H^1_\Theta = Z^1_\Theta /B^1_\Theta$ is a finite abelian group.  Let $\{ z_1,\dots, z_m\}$ be a set of coset representatives.  To say that $q_\theta$ is proper is equivalent to saying that it is a closed map and the preimage of every point is compact.  Since the preimage of every point is finite, we only need to show that $q_\theta$ is closed.  But the image of a closed set $C$ under $q_\theta$ is the same as the image of $q_\theta^{-1} (q_\theta (C))$.  The latter set is closed since it is a finite union of the translates $z_iC$ of $C$.  But since $q_\theta$ is a quotient map, if
$S$ is a subset of the codomain and $q_\theta^{-1} (S)$ is closed then $S$ must itself be closed.  It follows that $q_\theta (C)$ is closed and hence $q_\theta$ is a closed map.
This completes the proof.
\end{proof}

\section{Application of the Heisenberg theory}
\label{sec:appliedHeis}

We have seen that the computation of the constants $\langle \Theta , \xi\rangle_G$ reduces 
to the computation of the constants $\langle \Theta^\prime,\xi\rangle_K$, 
with $\Theta^\prime\in \Theta^K$,
and we have shown that when $\langle \Theta^\prime,\xi\rangle_K$ is nonzero then we can 
choose $\theta\in \Theta^\prime$ and $\Psi\in \xi$ such that $\Psi$ is $\theta$-symmetric. 
 Recall that
$$
\langle \Theta^\prime,\xi\rangle_K = \dim \Hom_{K^\theta} (\kappa (\Psi),1).
$$ 
Computing the right hand side is greatly simplified by the fact that $\Psi$ is 
$\theta$-symmetric.

Indeed, we will see in the next section that when $\Psi$ is $\theta$-symmetric the space $\Hom_{K^\theta}(\kappa (\Psi),1)$ has a tensor product decomposition with one factor attached to each representation $\kappa_i$, except $\kappa_d$, in the usual tensor product decomposition $\kappa (\Psi) = \kappa_{-1}\otd \kappa_d$.
In this section, we define a space of linear forms on the space  of $\kappa_i$ that serves as the factor of $\Hom_{K^\theta}(\kappa(\Psi),1)$ associated to $\kappa_i$, when 
$i\in \{\, 0,\dots,d-1\,\}$.  It will turn out that these factors have dimension one and hence they do not affect the dimension of $\Hom_{K^\theta}(\kappa (\Psi),1)$.  
 Thus, we will see that $\langle \Theta^\prime,\xi\rangle_K$ is the dimension of the 
factor associated to a twist of $\kappa_{-1}$.

Fix an involution $\theta$ of $G$ and a $\theta$-symmetric generic cuspidal $G$-datum $\Psi = (\vec\bG, y,\rho,\vec\phi)$.
Fix also $i\in \{\,0,\dots , d-1\,\}$.  We now adopt our standard notations from 
Section \ref{sec:genHeis} with subscripts added to reflect the dependence on $i$.
According to Lemma \ref{stableJ}, the subgroup $N_i$ must be $\theta$-stable and the automorphism  $\alpha_i$ of $\cH_i$ induced by $\theta$ has order two and is nontrivial on $\cZ_i$.  As in Proposition \ref{stJpol}, the automorphism $\alpha_i$ yields the polarization
\begin{equation*}
\begin{split}
\cH^+_i&=\{\, h\in \cH_i\ | \ \alpha_i(h)= h\,\}\\
\widehat{\cH}^-_i&=\{\, h\in \cH_i\ | \ \alpha_i (h)= h^{-1}\,\}
\end{split}
\end{equation*}
and Yu's special isomorphism gives a subgroup $\cH^-_i$ which splits the polarization.  Restricting $\hat\phi_i$ to $J^{i+1}_+$ gives a character 
$\zeta_i$ of $\cZ_i$ and we let $(\tau_i,V_i)$ denote a Heisenberg representation of $\cH_i$ with central character $\zeta_i$.

Let $W^+_i$ and $W^-_i$ denote the images of $\cH^+_i$ and $\cH^-_i$ in $W_i$ and let 
$$
\cM_i = \{\, s\in \cS_i \ | \ s\cdot W^+_i\subset W^+_i\hbox{ and }s\cdot W^-_i\subset W^-_i\,\}.
$$  
Let $\chi^{\cM_i}$ be the unique character of $\cM_i$ of order two.
 Conjugation gives a homomorphism $f'_i : K^i\to \cS_i$ such that the 
image of $K^{i,\theta}$ is contained in $\cM_i$.
The character $\hat\eta_i(k)=  \phi_i(k) \chi^{\cM_i} (f'_i(k))$ of $K^{i,\theta}$ is quadratic.  Let $\eta_i$ denote the restriction of $\hat\eta_i$ to $K^{0,\theta}$.  Proposition \ref{KJfactor} says that $K^{i,\theta} = K^{0,\theta} J^{1,\theta}\cdots J^{i,\theta}$.  Since the groups $J^{1,\theta},\dots,J^{i,\theta}$ are pro-$p$-groups, it must be the case that
$\hat\eta_i$ is trivial on $J^{1,\theta}\cdots J^{i,\theta}$ and
$\hat\eta_i = \inf_{K^{0,\theta}}^{K^{i,\theta}}(\eta_i)$.

According to Theorem \ref{Heisthm}, the space $\Hom_{\cH_i^+}(\tau_i,1)$ has dimension one.  Fix a nonzero element $\lambda_i$ in this space.
  It follows from Theorem \ref{Heisthm} that
$$\lambda_i (\hat\tau^\sharp_i(f'_i(k))\varphi) = \chi^{\cM_i}(f'_i(k))\ \lambda_i (\varphi),$$ for all $k\in K^{i,\theta}$ and $\varphi\in V_i$.
Therefore, $\lambda_i$ lies in $\Hom_{K^{i+1,\theta}}(\phi'_i , \inf\nolimits_{K^{0,\theta}}^{K^{i+1,\theta}}(\eta_i))$ and hence $$\Hom_{K^{i+1,\theta}}(\phi'_i , \inf\nolimits_{K^{0,\theta}}^{K^{i+1,\theta}}(\eta_i))
= \Hom_{\cH^+_i}(\tau_i,1)= \C \lambda_i .$$
Inflating from $K^{i+1}$ to $K$ yields the identity
$$\Hom_{K^\theta}(\kappa_i, \inf\nolimits_{K^{0,\theta}}^{K^\theta}(\eta_i)) = 
 \Hom_{K^{i+1,\theta}}(\phi'_i , \inf\nolimits_{K^{0,\theta}}^{K^{i+1,\theta}}(\eta_i)).$$
The above discussion is summarized in the following result:
 
\begin{proposition}\label{multonekappai}
Suppose $\theta$ is an involution of $G$ and
$\Psi =(\vec\bG,
y,\rho,\vec\phi)$ is a $\theta$-symmetric generic  cuspidal $G$-datum.  If 
$i\in \{\,0,\dots ,d-1\,\}$  then
\begin{equation*}
\begin{split}
\Hom_{K^\theta}(\kappa_i,\inf\nolimits_{K^{0,\theta}}^{K^\theta} (\eta_i))&= \Hom_{K^{i+1,\theta}}(\phi'_i, \inf\nolimits_{K^{0,\theta}}^{K^{i+1,\theta}}(\eta_i))\cr
&=\Hom_{\cH_i^+} (\tau_i,1)
\end{split}
\end{equation*}
and the latter space is one-dimensional.
\end{proposition}

\section{Factorization of invariant linear forms}
\label{sec:factorization}

As in the previous section, we fix an  involution $\theta$ of $G$ and a $\theta$-symmetric generic cuspidal $G$-datum $\Psi = (\vec\bG, y,\rho , \vec\phi)$.  When $i\in \{\, 0,\dots,d-1\,\}$, we have defined a quadratic character $\eta_i$ of $K^{0,\theta}$ and, in Proposition \ref{multonekappai}, we  described the space $\Hom_{K^\theta} (\kappa_i, \inf_{K^{0,\theta}}^{K^\theta} (\eta_i))$ and
 showed that this space has dimension one.

We also attach a space of linear forms to 
 $\kappa_{-1}= \inf_{K^0}^K (\rho)$, namely, the space $$\Hom_{K^\theta}(\kappa_{-1}, \inf\nolimits_{K^{0,\theta}}^{K^\theta}(\eta ))= \Hom_{K^{0,\theta}}(\rho,\eta),$$
 where $\eta$ is the character of $K^{0,\theta}$ defined by
 $\eta (k) = \prod_{i=0}^d \eta_i(k)$, with $\eta_d = \phi_d\, |\, K^{0,\theta}$.
We show in this section $\Hom_{K^\theta}(\kappa,1)$ is canonically isomorphic to a tensor product of the spaces of linear forms we have associated to $\kappa_{-1},\dots,\kappa_{d-1}$.  All of these spaces of linear forms have dimension one, except for the factor attached to $\kappa_{-1}$.
Thus we obtain an isomorphism $$\Hom_{K^\theta}(\kappa,1)\cong \Hom_{K^{0,\theta}}(\rho,\eta).$$
More generally,  we show there exists  an isomorphism $$ \Hom_{K^\theta}( \kappa , \inf\nolimits^{K^\theta}_{K^{0,\theta}}(\eta\mu) ) \cong \Hom_{K^{0,\theta}}( \rho,\mu),$$  for each character $\mu$ of $K^{0,\theta}$ that is trivial 
on $K^{0,\theta}\cap J^1= G^{0,\theta}_{y,r_0}$. 

Our arguments use the inductive structure discussed in Section~\ref{sec:inductive}.   In particular, we use induction on the degree $d$ of the datum $\Psi$.   Recall our notations 
\begin{equation*}
\begin{split}
\partial (\bG^0,\dots,\bG^d)&= (\bG^0,\dots , \bG^{d-1})\\
\partial (\phi^0,\dots , \phi^d)&= (\phi^0,\dots , \phi^{d-1})\\
\partial (\vec\bG,y,\rho, \vec\phi)&= (\partial \vec\bG, y,\rho, \partial \vec\phi)\\
\partial (\vec\bG,\pi_0 , \vec\phi)&= (\partial \vec\bG, \pi_0 , \partial \vec\phi),
\end{split}
\end{equation*} which apply when $d>0$.
Recall also that Yu's construction actually associates to an extended datum $\Psi$ of degree $d$ a sequence $\vec\pi = (\pi_0,\dots, \pi_d)$ of tame supercuspidal representations of $G^0,\dots , G^d$, respectively, where $\pi_{d-i} = \partial^i\pi$ is associated to the datum $\partial^i \Psi$.  

The main tool developed in this section is the following:

\begin{lemma}\label{factorlem}
Fix an extended generic cuspidal $G$-datum.  For all $i\in \{\, 0,\dots ,d-1\,\}$, 
suppose we are 
given a subgroup $J^{i+1,\flat}$ of $J^{i+1}$ such that the space $\Hom_{J^{i+1,\flat}} (\tau_i ,1)$ has dimension one and suppose we have fixed some nonzero element $\lambda_i$ in\break 
 $\Hom_{J^{i+1,\flat}} (\tau_i ,1)$.  Assume that whenever $0\le i_1< i_2\le d-1$ the  group $J^{i_1+1,\flat}$ normalizes $J^{i_2+1,\flat}$.  When $i\in \{\,0,\dots,d-1\,\}$, assume 
$\phi_i\, |\, J^{1,\flat}\cdots J^{i,\flat} =1$ and 
 $$\Hom_{J^{1,\flat}\cdots J^{i+1,\flat}} (\phi'_i, 1) = \Hom_{J^{i+1,\flat}}(\tau_i,1) 
= \C \lambda_i.$$
  If $\lambda\in \Hom_{J^{1,\flat}\cdots J^{d,\flat}} (\kappa ,1 )$ there must exist a linear form $\lambda_{-1}\in \Hom (V_{-1},\C)$ such that $$\lambda (v_{-1}\otd v_{d-1}) = \lambda_{-1}(v_{-1})\cdots \lambda_{d-1}(v_{d-1}),$$ for all $v_{-1}\in V_{-1},\dots ,v_{d-1}\in V_{d-1}$.  The map $\lambda\mapsto\lambda_{-1}$ defines a linear isomorphism, $$\Hom_{J^{1,\flat}\cdots J^{d,\flat}}
 (\kappa ,1) \cong\Hom (V_{-1},\C).$$
\end{lemma}

\begin{proof}
It is easy to verify that if $\lambda_{-1}\in \Hom (V_{-1},\C )$ and $\lambda\in \Hom (V,\C)$ is defined on elementary tensors by
$$\lambda (v_{-1}\otd v_{d-1}) = \lambda_{-1}(v_{-1})\cdots \lambda_{d-1}(v_{d-1}),$$ for all $v_{-1}\in V_{-1},\dots ,v_{d-1}\in V_{d-1}$, then $\lambda$ must lie in $\Hom_{J^{1,\flat}\cdots J^{d,\flat}} (\kappa , 1 )$.

Now fix $\lambda\in \Hom_{J^{1,\flat}\cdots J^{d,\flat}}(\kappa, 1)$.  It only remains to show that $\lambda$ factors as indicated in the statement of the lemma.
We may as well assume that $\lambda$ and $d$ are nonzero, since otherwise our claim is trivial.  Choose $v_{-1}\in V_{-1},\dots , v_{d-2}\in V_{d-2}$ and define $\Lambda\in \Hom (V_{d-1},\C)$ by $$\Lambda (v_{d-1}) = \lambda( v_{-1}\otd v_{d-1}).$$  We claim that $\Lambda$ must lie in $\Hom_{J^{d,\flat}} (\tau_{d-1} ,1)$.  Indeed, using the fact that  $\kappa_i \,|\,J^d=1$ when $i<d-1$ (since $\kappa_i = \inf\nolimits_{K^{i+1}}^K(\phi'_i)$), we have for all $h\in J^{d,\flat}$
\begin{equation*}
\begin{split}
\Lambda(v_{d-1})&=\lambda (v_{-1}\otd v_{d-1})\\
&= \lambda (\kappa(h)(v_{-1}\otd v_{d-1})) \\
&= \lambda( \kappa_{-1}(h)v_{-1}\otd \kappa_{d-1}(h) v_{d-1}) \phi_d(h)\cr
&= \lambda (v_{-1}\otd v_{d-2}\otimes \kappa_{d-1}(h) v_{d-1})\\
&= \Lambda (\tau_{d-1}(h)v_{d-1}).
\end{split}
\end{equation*}
It follows that there exists a complex number $\partial\lambda (v_{-1}\otd v_{d-2})$ such that $$\lambda( v_{-1}\otd v_{d-1}) = \partial \lambda (v_{-1}\otd v_{d-2}) \  \lambda_{d-1}(v_{d-1}),$$ for all $v_{d-1}\in V_{d-1}$.  Clearly, $\partial \lambda$ defines a nonzero linear form on $$\partial V = V_{-1}\otd V_{d-2}.$$  If $d=1$, we are done.  So we assume $d>1$.  Our claim will now follow by induction on $d$ once we establish that $\partial\lambda$ lies in $\Hom_{J^{1,\flat}\cdots J^{d-1,\flat}} (\partial\kappa,1)$, where $$\partial\kappa = \partial\kappa_{-1}\otd \partial\kappa_{d-1},$$ with  $\partial \kappa_{d-1} = \phi_{d-1}\,|\,K^{d-1}$ and, otherwise, $\partial\kappa_i = \inf\nolimits_{K^{i+1}}^{K^{d-1}}(\phi'_i)= \kappa_i\, |\,K^{d-1}$.  
   
Fix $h\in J^{1,\flat}\cdots J^{d-1,\flat}$.
We have 
\begin{equation*}
\begin{split}
\partial\lambda (\partial\kappa(h)(v_{-1}&\otd v_{d-2})) \\
 &= 
{{\phi_{d-1}(h)\ \lambda(\kappa_{-1}(h)v_{-1}\otd \kappa_{d-2}(h) v_{d-2}
\otimes v_{d-1})
}\over{
\lambda_{d-1}(v_{d-1})}},
\end{split}
\end{equation*}
for all $v_{-1}\in V_{-1}, \dots, v_{d-1}\in V_{d-1}$ such that $\lambda_{d-1}(v_{d-1})\ne 0$.  Note that we are using the fact that $\partial\kappa_i (h) = \kappa_i (h)$ when 
$i\in \{\,-1,\dots,d-2\,\}$.  Since the previous identity holds for all $v_{-1},\dots , v_{d-1}$ with $\lambda_{d-1}(v_{d-1})\ne 0$ and since $\lambda_{d-1}(\kappa_{d-1}(h) v_{d-1}) = \lambda_{d-1}(\phi'_{d-1}(h)v_{d-1})=  \lambda_{d-1}(v_{d-1})$, we may replace $v_{d-1}$ by $\kappa_{d-1}(h) v_{d-1}$ to obtain
\begin{equation*}
\begin{split}
\partial\lambda (\partial\kappa(h)(v_{-1}\otd v_{d-2}))    &= 
{{\phi_{d-1}(h)\ \lambda(\kappa_{-1}(h)v_{-1}\otd \kappa_{d-1}(h) v_{d-1})
}\over{
\lambda_{d-1}(\kappa_{d-1}(h) v_{d-1})}}\\
&= 
{{\phi_{d-1}(h)\ \lambda(v_{-1}\otd v_{d-1})
}\over{\phi_d(h)\ \lambda_{d-1}(v_{d-1})}}\\
&= \partial\lambda (v_{-1}\otd v_{d-2}).
\end{split}
\end{equation*}
Since we have shown that $\partial\lambda$ lies in $\Hom_{J^{1,\flat}\cdots J^{d-1,\flat}} (\partial\kappa, 1)$, the proof is complete.
\end{proof}

We now apply the previous lemma taking $J^{i+1,\flat}$ to be the subgroup $J^{i+1,\theta}$.  For all $i\in \{\, 0,\dots , d-1\,\}$, Proposition \ref{multonekappai} says that if $\lambda_i$ is a nonzero element of $\Hom_{\cH_i^+}(\tau_i,1)$ then
$$\Hom_{K^{i+1,\theta}}(\phi'_i,\inf\nolimits_{K^{0,\theta}}^{K^{i+1, \theta}} (\eta_i)) =\Hom_{\cH_i^+} (\tau_i,1)= \C
\lambda_i.$$
Note that the fact that $H_\theta^1(N_i)$ is trivial (by Proposition~\ref{twodivprop})
implies that $\cH_i^+=J^{i+1,\theta}N_i/N_i$. It follows that
$$
\Hom_{J^{1,\theta}\cdots J^{i+1,\theta}} (\phi'_i,1) =\Hom_{\cH_i^+} (\tau_i,1).$$  
Lemma \ref{factorlem} now yields the following:

\begin{proposition}\label{appliedfactor}
Assume $\theta$ is an involution of $G$ and  $\Psi = (\vec\bG,y,\rho,\vec\phi)$ is 
a $\theta$-symmetric generic cuspidal $G$-datum. Fix a nonzero linear form $\lambda_i$ in 
the 1-dimensional space $\Hom_{\cH_i^+}(\tau_i,1)$, for all $i\in \{\, 0,\dots, d-1\,\}$.  
If $\lambda\in \Hom_{J^{1,\theta}\cdots J^{d,\theta}} (\kappa ,1)$ there must exist a 
linear form $\lambda_{-1}\in \Hom (V_{-1},\C)$ such that $$\lambda (v_{-1}\otd v_{d-1}) =
 \lambda_{-1}(v_{-1})\cdots \lambda_{d-1}(v_{d-1}),$$ for all 
$v_{-1}\in V_{-1},\dots ,v_{d-1}\in V_{d-1}$.  The map $\lambda\mapsto \lambda_{-1}$ 
defines a  linear isomorphism
$$
\Hom_{J^{1,\theta}\cdots J^{d,\theta}} (\kappa ,1) \cong
\Hom (V_{-1},\C).
$$ 
 If $\mu$ is a character of $K^{0,\theta}$ that is trivial on $G^0_{y,r_0}$ then the
 latter isomorphism restricts to an isomorphism
$$
\Hom_{K^\theta}(\kappa,  \inf\nolimits_{K^{0,\theta}}^{K^\theta}(\eta\mu))\cong 
\Hom_{K^{0,\theta}}(\rho, \mu ).
$$
\end{proposition}

\begin{proof}
The existence of $\lambda_{-1}$ and the first isomorphism of Hom-spaces follow immediately from the discussion above.  It remains to establish the second isomorphism.
From Proposition \ref{KJfactor}, we have $K^\theta = K^{0,\theta} J^{1,\theta}\cdots J^{d,\theta}$.
Hence, $\Hom_{K^\theta}(\kappa ,\inf_{K^{0,\theta}}^{K^\theta}(\eta\mu))$ is just the space of linear forms $\lambda$ in $\Hom_{J^{1,\theta}\cdots J^{d,\theta}} (\kappa ,1 )$ 
such that
$$\lambda (\kappa (h) v) = \eta (h)\mu(h)\lambda(v),$$ for all $v\in V$ and $h\in K^{0,\theta}$.  
  Suppose $\lambda\in \Hom_{J^{1,\theta}\cdots J^{d,\theta}} (\kappa ,1 )$ and $h\in K^{0,\theta}$.  
Then
$$\lambda (\kappa(h)(v_{-1}\otd v_{d-1})) = \lambda_{-1}(\rho(h)v_{-1}) \lambda_0(v_0)\cdots \lambda_{d-1}(v_{d-1})\ \eta (h),$$ for all $v_{-1}\in V_{-1},\dots, v_{d-1}\in V_{d-1}$.  Therefore, $\lambda$ lies in  
$\Hom_{K^\theta}(\kappa,\inf_{K^{0,\theta}}^{K^\theta}(\eta\mu))$ exactly when 
$$
\lambda (v_{-1}\otd v_{d-1})\,\eta(h)\,\mu(h)
 = \lambda_{-1}(\rho(h)v_{-1}) \lambda_0(v_0)\cdots \lambda_{d-1}(v_{d-1})\,
\eta(h),
$$ 
for all $h\in K^{0,\theta}$, $v_{-1}\in V_{-1},\dots, v_{d-1}\in V_{d-1}$.  
Equivalently, $\lambda_{-1}$ must lie in $\Hom_{K^{0,\theta}} (\rho,\mu)$.
\end{proof}

\section{Dimension formulas}
\label{sec:dimformulas}

In this section, we state our main results concerning
distinguishedness of tame supercuspidal representations.
Throughout this chapter, we have fixed our group $G$ 
and considered tame supercuspidal representations arising from Yu's construction via induction from some fixed group $K$.  These representations are attached to $(G,K)$-data 
$\Psi = (\vec\bG,y,\rho,\vec\phi)$, where a $(G,K)$-datum is simply a generic cuspidal $G$-datum for which the associated inducing subgroup $K(\Psi)$ is $K$.   (The group $K(\Psi)$ is defined 
in Section~\ref{sec:notations} and Yu's construction is discussed in 
Section \ref{sec:construction}.)

A notion of refactorization of cuspidal $G$-data
was introduced in Section~\ref{sec:refactorization}. In Section~\ref{sec:symmetrizing}, 
we defined an 
equivalence relation on the set of $(G,K)$-data, in terms
of refactorization, $K$-conjugation and elementary transformations. 
As indicated in Section~\ref{sec:symmetrizing}, the equivalence class of the inducing representation
$\kappa(\Psi)$ only depends on the $K$-equivalence class $\xi$ of the $(G,K)$-datum $\Psi$. Hence the equivalence class of the tame supercuspidal representation $\pi(\Psi)$ also only depends on $\xi$.

Fix a $G$-orbit $\Theta$ of involutions of $G$ and a $K$-equivalence class $\xi$ of 
$(G,K)$-data.  As in Section~\ref{sec:symmetrizing}, we assume that 
Hypothesis~C($\vec\bG$) is satisfied for all tamely ramified twisted Levi
 sequences
$\vec\bG$ that occur in any $(G,K)$-datum in $\xi$.
Recall that if  $\theta\in \Theta$ and $\Psi\in \xi$ then  the dimension of  
$\Hom_{G^\theta}
(\pi(\Psi),1)$ only depends on $\Theta$ and $\xi$ and we denote this number by
$\langle\Theta,\xi\rangle_G$. Similarly, if $\theta\in \Theta$ and $\Theta^\prime$ is the $K$-orbit of $\theta$ 
then the dimension of ${\Hom}_{K\cap G^\theta}(\kappa(\Psi),1)$ only depends on $\Theta^\prime$ and $\xi$ and it is denoted by $\langle \Theta^\prime,\xi\rangle_K$. 

Now fix $\Theta^\prime\in \Theta^K$, where $\Theta^K$ denotes the set of $K$-orbits of involutions of $G$ that are contained in $\Theta$. 
Assume that $\langle \Theta^\prime,\xi\rangle_K$ is nonzero.
Then, according to Lemma~\ref{GKeqlemone}, Proposition \ref{compatorbsum} 
and Proposition \ref{TFAErefactor}(2) 
there exists a $k\in K$ and a refactorization $\dot\Psi$ of $\Psi$ such
that ${}^k\dot\Psi$ is $\theta$-symmetric in the sense of
Definition~\ref{defweaksymdatum}. Equivalently, $\dot\Psi$ is $(k\cdot \theta)$-symmetric.

For simplicity, we replace $\theta$ by $(k^{-1}\cdot\theta)$ or, in other words, we assume that $\dot\Psi$ is $\theta$-symmetric.  The fact that $\dot\Psi$ is $\theta$-symmetric allows us to obtain a formula for the dimension of $\Hom_{K^\theta}(\kappa(\dot\Psi),1)$ using the methods of the previous two sections.  We obtain:
$$
\langle\Theta^\prime ,\xi\rangle_K=\dim \Hom_{K^0(\dot\Psi)^\theta}
(\rho(\dot\Psi),\eta_\theta(\dot\Psi)),
$$
where the character $\eta_\theta (\dot\Psi)$ is the character $\eta$ defined in Section~\ref{sec:factorization}.

In practice, it is desirable not to actually have to find a $\theta$-symmetric refactorization of $\Psi$.  Define a representation of $K^0 (\Psi)= K^0(\dot\Psi)$ by $$\rho' (\Psi) = \rho (\Psi)\otimes \prod_{i=0}^d (\phi_i\, |\,K^0(\Psi))$$ and define a quadratic character of $K^0(\Psi)^\theta = K^0(\dot\Psi)^\theta$ by
$$\eta'_\theta (k)
=\prod_{i=0}^{d-1}\chi^{\cM_i}(f_i^\prime(k)),
$$
in the notations of Section~\ref{sec:appliedHeis}. 
Then $\rho'(\dot\Psi) = \rho'(\Psi)$ and, similarly, $\eta'_\theta$ is invariant under refactorizations.

We have
\begin{equation*}
\begin{split}
\Hom_{K^0(\dot\Psi)^\theta}
(\rho(\dot\Psi),\eta_\theta(\dot\Psi))&=
\Hom_{K^0(\dot\Psi)^\theta}
(\rho'(\dot\Psi),\eta'_\theta(\dot\Psi))\cr
&=
\Hom_{K^0(\Psi)^\theta}
(\rho'(\Psi),\eta'_\theta(\Psi)).
\end{split}
\end{equation*}
Therefore
$$\langle \Theta',\xi\rangle_K = \dim 
\Hom_{K^0(\Psi)^\theta}
(\rho'(\Psi),\eta'_\theta(\Psi)),$$
which expresses $\langle \Theta',\xi\rangle_K$ in terms of 
$\Psi$ and eliminates
 the need to use a refactorization.

\begin{theorem}\label{maindimformula}
Let $\Psi$, $\xi$, $\Theta$, $K$, etc., be as above.
\begin{enumerate}
\item $\langle\Theta,\xi\rangle_G=m_K(\Theta)\,\sum_{\Theta^\prime
\in \Theta^K}\, \langle\Theta^\prime,\xi\rangle_K$, where $m_K(\Theta)$
is as in Section~\ref{sec:Mackey}.
\item The constants $\langle \Theta,\xi\rangle_G$, $\langle \Theta',\xi\rangle_K$ and $m_K(\Theta)$ in  (1) are all finite.  The number of nonzero summands in (1) is also finite.
\item If $\theta\in\Theta'\in \Theta^K$ and if  there exists
a $\theta$-symmetric refactorization   of $\Psi$ then
$$
\langle\Theta' ,\xi\rangle_K=
\dim\Hom_{K^0(\Psi)^\theta}(\rho' (\Psi),
\eta'_\theta (\Psi)).
$$
\item If $\theta\in\Theta'\in \Theta^K$ and $\langle
\Theta',\xi\rangle_K \not=0$ then there exists $k\in K$ 
and a refactorization $\dot\Psi$ 
of $\Psi$ such that ${}^k\dot\Psi$ is $\theta$-symmetric element of $\xi$.  In this case, $\dot\Psi$ is $\theta'$-symmetric where $\theta' = k^{-1}\cdot \theta\in \Theta'$.
\item Suppose that $\langle\Theta,\xi\rangle_G\not=0$.
Choose $\theta\in \Theta$ such that some refactorization
of $\Psi$ is $\theta$-symmetric. Let $g_1,\dots,g_m$ be a 
maximal (finite) sequence in $G$ such that $g_j\theta(g_j^{-1})
\in K^0=K^0(\Psi)$ and the $K$-orbits of the elements
$\theta_j=g_j\cdot\theta$ are distinct. Then
$$
{\rm Hom}_{G^\theta}(\pi(\Psi),1)\cong m_K(\Theta)\bigoplus_{j=1}^m
{\rm Hom}_{K^{0,\theta_j}}(\rho^\prime(\Psi),\eta'_{\theta_j}(\Psi)).
$$
Hence, letting $\Theta_j^\prime$ be the $K$-orbit of $\theta_j$, 
 $j\in \{\, 1,\dots,m\,\}$, 
$$
\langle\Theta,\xi\rangle_G=m_K(\Theta)\,\sum_{j=1}^m\langle\Theta_j^\prime,
\xi\rangle_K.
$$
\end{enumerate}
\end{theorem}

\begin{proof}
The first statement is identical to Equation \ref{basicmulteq}.  
The second statement follows from Lemma \ref{mThetabound} and 
Proposition \ref{finrelorbs}.  Statements (3) and (4) follow from the 
discussion at the beginning of this section.  

Suppose that $\dot\Psi$ is a refactorization of $\Psi$ that
is $\theta$-symmetric. It is easy to check (see 
Definition~\ref{defweaksymdatum})
that the fact that $\dot\Psi$ is $\theta$-symmetric
and $g_j\theta(g_j^{-1})\in K^0$ implies that 
$\dot\Psi$ is $\theta_j$-symmetric. 
Statement (5) follows from the results used to prove
 (1), (2) and (3), and from
 Proposition \ref{oldLemmasBandC}.
\end{proof}

\begin{remark}\label{remamainthm}
In this \paperbook, apart from those cases related to equivalence of tame
supercuspidal representations (see the next chapter), 
we do
not address the issue of computing the terms
$\dim\Hom_{K^0(\Psi)^\theta}(\rho' (\Psi),
\eta'_\theta (\Psi))$.
The representation $\rho' (\Psi)$ is a twist
of $\rho(\Psi)$ by a quasicharacter of $K^0(\Psi)$,
and $\rho(\Psi)\,|\, G_{y,0}^0$ is the inflation
of a cuspidal representation
of the finite group $G_{y,0:0^+}^0$.
Also, in the $\theta$-symmetric case,
$\theta$ factors to an involution of
this finite group.
Because of this, together with Proposition~\ref{quaddistprop},
it is natural to expect that computing terms 
of the above form  might be connected
to quadratic distinguishedness properties of cuspidal 
representations
of finite groups of Lie type. This was the
case for the distinguished tame supercuspidal
representations of general linear groups
that were studied in \cite{HM2} and \cite{HM3}.
\end{remark}

\begin{remark}\label{differentcontributions}
As we will see in Section~\ref{sec:examples},
there are examples where the terms
on the right hand side of (5) give different
contributions to ${\rm Hom}_{G^\theta}(\pi(\Psi),1)$.
That is, distict $K$-orbits that are moderately
compatible with $\xi$ do not necessarily contribute
in the same way to $\langle\Theta,\xi\rangle_G$.
\end{remark}

\begin{remark}\label{symrem}
Suppose that $\langle\Theta,\xi\rangle_G\not=0$.
Then there exist $\theta\in\Theta$ and $\Psi\in \xi$
such that $\Psi$ is $\theta$-symmetric
and $\Hom_{K^0(\Psi)^\theta}(\rho^\prime(\Psi),
\eta_\theta^\prime(\Psi))\not=0$. 
The results of Section~\ref{sec:contragredients}
show that $\kappa_i(\Psi)\circ\theta\simeq \widetilde
{\kappa_i(\Psi)}$
for all $i\in \{\, 0,\dots,d\,\}$. If we also
had $\rho(\Psi)\circ\theta\simeq\widetilde{\rho(\Psi)}$,
then we would have $\kappa(\Psi)\circ\theta\simeq
\widetilde{\kappa(\Psi)}$, hence
$\pi(\Psi)\circ\theta\simeq\widetilde{\pi(\Psi)}$.
However, it is not clear whether nonvanishing
of $\Hom_{K^0(\Psi)^\theta}(\rho^\prime(\Psi),
\eta_\theta^\prime(\Psi))$ implies a relation
between $\rho(\Psi)\circ\theta$
and $\widetilde{\rho(\Psi)}$.
In certain cases, (for example, see
\cite{HM2} and \cite{HM3}) the condition
$\pi(\Psi)\circ\theta\simeq\widetilde{\pi(\Psi)}$
is a necessary condition for a
tame supercuspidal representation $\pi(\Psi)$ 
of a general linear group to be distinguished by
an involution $\theta$.
\end{remark}

We now specialize to the case in which $\bG^0$ is a torus.    Then since $G^0/Z$ is compact the reduced building $\cB_{\rm red} (\bG^0,F)$ reduces to a point and thus $K^0 = G^0$. 
When $\theta(\bG^0)=\bG^0$, the condition $\theta [y] = [y]$ is automatic.  Therefore, the notions of weak and moderate compatibility coincide, as do the notions of weak $\theta$-symmetry and $\theta$-symmetry.   Note that $\rho$ is a quasicharacter of $G^0$.

\begin{definition}\label{toraldef}
A $(G,K)$-datum $\Psi = (\vec\bG, y,\rho,\vec\phi)$ is \textit{toral} if $\bG^0$ is a torus.  A $K$-equivalence class $\xi\in \Xi^K$ is \textit{toral} if one (hence all) of its elements are toral.
\end{definition}

Suppose $\Psi = (\vec\bG,y,\rho,\vec\phi)$ is a toral $(G,K)$-datum.  Let $\xi$ be its $K$-equivalence class.   In order for there to exist a $K$-orbit $\Theta'$ that is moderately compatible with $\xi$, there must be some involution $\theta\in \Theta$ such that $\theta (\vec\bG) = \vec\bG$ and $\prod_{i=0}^d \phi_i(k)=1$ for all $k\in G^{0,\theta}_{0^+}$.
If no moderately compatible orbits exist then $\langle \Theta,\xi\rangle_G =0$.  Otherwise, we have:

\begin{proposition}\label{invdimtorus} 
Assume that $\xi\in \Xi^K$ is toral and $\Theta'\in \Theta^K$ is moderately 
compatible with $\xi$.
Choose $\theta\in \Theta'$ and $\Psi\in \xi$ so that $\Psi$  is $\theta$-symmetric.  Then:
\begin{enumerate}
\item $\langle \Theta',\xi\rangle_K =\begin{cases}
1,&\text{if $\rho'(\Psi)\eta^{\prime}_\theta (\Psi)^{-1}\, |\, G^{0,\theta}= 1$,}\\
0,&\text{otherwise.}
\end{cases}
$
\item $\langle \Theta'',\xi\rangle_K = \langle \Theta',\xi\rangle_K$ whenever $\Theta''\in \Theta^K$ is moderately compatible with $\xi$.
\item
$\langle \Theta ,\xi \rangle_G= m_K(\Theta)\langle \Theta',\xi\rangle_K\cdot \#  S$, where $S$ is the set of elements of $\Theta^K$ that are moderately compatible with $\xi$.
\item If $\langle\Theta , \xi\rangle_G$ is nonzero then it is equal to the number of double cosets in $K\bs G/G^\theta$ that contain a representative $g$ such that $g\theta(g)^{-1}\in G^0$.
\end{enumerate}
\end{proposition}

\begin{proof}
The first assertion follows from the discussion at the beginning of this section together with the fact that $\rho'(\Psi)$ is 1-dimensional.
Now suppose $\Theta''$ is moderately compatible with $\xi$.
Proposition \ref{oldLemmasBandC}(2) says that there exists $g\in G$ with $g\theta (g)^{-1}\in G^0$ such that $g\cdot \theta\in \Theta''$.
Because $g\theta(g)^{-1}\in G^0$ and $G^0$ is abelian, it is elementary to verify that $\Psi$ is $(g\cdot\theta)$-symmetric and $G^{0,g\cdot\theta} = G^{0,\theta}$.  
To deduce (2) from (1), it now suffices to show that $\eta'_\theta (\Psi)$ and $\eta'_{g\cdot\theta}(\Psi)$ agree on $G^{0,\theta}$.  But  Lemma \ref{realtrace} implies
$$
\eta'_\theta (\Psi)(k) = \prod_{i=0}^{d-1} {\rm sign} (\hat\tau_i^\sharp (f'_i (k))) 
= \eta'_{g\cdot\theta}(\Psi)(k),$$ for all $k\in G^{0,\theta}$.
(Here, ${\rm sign} (\hat\tau_i^\sharp(f^\prime_i(k)))$ is defined to be
the sign of the trace of $\hat\tau_i^\sharp(f^\prime_i(k))$, for
$k\in G^{0,\theta}$.)
Therefore, (2) follows.  Statement (3) follows from Theorem \ref{maindimformula}.
Statement (4) follows from (3) and the discussion in Section \ref{sec:Mackey}.
\end{proof}

\section{Multiplicity}
\label{sec:multiplicity}

Let $G$ be a totally disconnected group and let $H$ be a closed subgroup.  
We are especially interested in 
pairs $(G,H) = (G,G^\theta)$ for which $G$ is the group $\bG(F)$ of $F$-rational points of a connected reductive $F$-group $\bG$ and $H$ is the group of fixed points of an involution $\theta$ of $G$.

In applications of representation theory to number theory, it can be quite useful when $\Hom_{H}(\pi,1)$ has dimension at most one for a given representation $\pi$ of $G$ or for  some specified collection of representations.  This property plays a role in the harmonic analysis on $H\backslash G$ that is technically similar to the role played by the uniqueness of Whittaker models in representation theory.  It is frequently referred to as ``the multiplicity one property.'' \label{multoneproplabel}

When $\pi$ fails to have the multiplicity one property with respect to $H$, it could be important to detect the source of failure.  Doing so might allow one to enlarge $H$ to some subgroup $H'$ such that $(G,H')$ has the multiplicity one property.  In practice, replacing $H$ by $H'$ would only be useful if $(G,H')$ were to retain the essential features of $(G,H)$ for a given application.

We will identify several potential sources of failure of the multiplicity one property for tame supercuspidal representations.  In some cases, the failures may be analyzed in more than one way.  

\subsection*{The group $N_G(H)/H$}
Suppose, as usual, that $\pi$ is induced via Yu's construction from an irreducible smooth representation $\kappa$ of a compact-mod-center subgroup $K$ of $G$.  We observe  that if
$n$ lies in the normalizer $N_G(H)$ of $H$ in $G$
then $$\Hom_{K\cap gHg^{-1}}(\kappa,1)=
\Hom_{K\cap gnHn^{-1}g^{-1}}(\kappa,1).$$ 
Therefore, if $m_g$ is the number of double
cosets of $K\bs G/H$ lying in the double coset
$KgN_G(H)$ then:
$$\Hom_H(\pi,1)\cong \bigoplus_{KgN_G(H)\in
K\bs G/N_G(H)}m_g\cdot \Hom_{K\cap
gHg^{-1}}(\kappa,1).$$ The phenomenon of repeated
terms is directly related to the existence of
nontrivial orbits with respect to the action of
the group $N_G(H)/H$ on $\Hom_H(\pi,1)$ by
$$(nH\cdot \Lambda)(f) =
\Lambda(\pi(n)^{-1}f).$$
  Example \ref{exGLO} (when $-1$ is a square in $F^\times$)
 and Example~\ref{exampleA} provide examples in which the 
multiplicity one property fails to hold for a tame supercuspidal 
representation because of the phenomenon just 
described.
In such cases, it is natural to ask whether all irreducible tame supercuspidal
representations have the multiplicity one property relative to $N_G(H)$.
  
\subsection*{The constant $m_K(\Theta)$}
Assume again that $\pi$ is induced from an irreducible smooth representation of a 
compact-mod-center subgroup $K$ of $G$.  
Recall from Section  \ref{sec:Mackey} that we have associated to a triple $(G,K,\Theta)$ 
a constant $m_K(\Theta)$.  In particular, if $\theta$ is an involution of $G$ whose 
$G$-orbit is $\Theta$, then $K$ acts by $\theta$-twisted conjugation on 
$\{\, g\theta (g)^{-1} \ | \ g\in G\,\}$.  The constant $m_K(\Theta)$ is identical to the number of 
$K$-orbits in this set that contain an element of the center $Z$ of $G$.  
For $\pi$ as above,
the dimension of $\Hom_{G^\theta}(\pi,1)$ must be a multiple of $m_K(\Theta)$.
 
In many cases, we have $m_K(\Theta)=1$. For example, this occurs if $Z$ is trivial.  
More generally, in cases where $K\supset Z$, Lemma \ref{mThetabound} gives a cohomological 
condition  that, when satisfied, implies $m_K(\Theta)=1$.
(Note that in  our applications the condition $K\supset Z$ is automatic.)  
Example \ref{exGLO} (in the case where $-1$ is a square in $F^\times$)
 and Example~\ref{exampleA} provide examples for 
which $m_K(\Theta)>1$.

\subsection*{The toral case}

Let $\Theta$ be a $G$-orbit of involutions of $G$.
Assume $\Psi$ is a toral $(G,K)$-datum and denote its equivalence 
class by $\xi$.  Let $n_\xi (\Theta)$ be the number of distinct
 $K$-orbits $\Theta'\in \Theta^K$ that are moderately compatible with $\xi$.
Proposition \ref{invdimtorus} implies that if $\langle \Theta,\xi\rangle_G$ is nonzero then it equals $m_K(\Theta) n_\xi (\Theta)$.
Moreover, if $\langle \Theta,\xi\rangle_G $ is nonzero, then Lemma~\ref{mThetabound}  and 
Remark~\ref{remoldLemBandC} provide cohomological conditions that are sufficient for the triviality of $m_K(\Theta)$ and $n_\xi(\Theta)$, respectively.

\subsection*{The nontoral case}

For nontoral $G$-data, the situation is more complicated.  Higher multiplicities may 
result from the sources mentioned above, as well as the fact that the quantities $\langle \Theta' , \xi\rangle_K$ may themselves be larger than one (since $\rho$ is generally not
one-dimensional in the nontoral case). 

Another complicating factor in the nontoral case is that $\langle \Theta',\xi\rangle_K$ 
may be nonconstant as $\Theta'$ varies over the elements in $\Theta^K$ that are moderately compatible with $\xi$.  The latter phenomenon is evident in Example \ref{exampleB}.  Indeed, Example \ref{exampleB} shows that one can have 
$\langle \Theta,\xi \rangle_G =1$, even when $n_\xi (\Theta)>1$.

\subsection*{Gelfand pairs}
The theory of Gelfand pairs, in the context of representations of totally disconnected groups, was developed by Gelfand-Kazhdan \cite{GK} and Bernstein \cite{B}.  It is surveyed in \cite{Gr}, where the following definition appears:

\begin{definition}\label{Gelfandpair}
Assume $G$ is a totally disconnected group and $H$ is a closed subgroup such that $H\backslash G$ carries a $G$-invariant measure.  Then $(G,H)$ is a \textit{Gelfand pair} if  $$\dim\Hom_H (\pi,1)\cdot \dim\Hom_H (\tilde\pi ,1)\le 1$$ for every irreducible, admissible (smooth) representation $\pi$ of $G$.
\end{definition}

In the context of the theory of $G^\theta$-distinguished tame supercuspidal representations of a group $G$, we observe that if  it is known that $(G,G^\theta)$ is a Gelfand pair, then the problem of determining the dimension of $\Hom_{G^\theta}(\pi,1)$ simplifies.  For example, there is no need to compute the constants $m_K(\Theta)$.  If distinguished tame supercuspidal representations exist then the constants $m_K(\Theta)$ must be 1 (and if distinguished representations do not exist these constants are irrelevant).

The standard tool for establishing that $(G,G^\theta)$ is a Gelfand pair is the following lemma due to Gelfand-Kazhdan  (known colloquially as ``Gelfand's Lemma''):

\begin{lemma}[\cite{GK},\cite{Gr}]\label{GelfandLemma}
Assume $G$ is a totally disconnected group and $H$ is a closed subgroup such that $H\backslash G$ carries a $G$-invariant measure.  Suppose $\sigma$ is an anti-automorphism of $G$ of order two that stabilizes $H$ and acts trivially on every bi-$H$-invariant distribution on $G$.  Then $(G,H)$ is a Gelfand pair.
\end{lemma}

The hypotheses of the previous lemma suggest that one must fully understand the family of bi-$H$-invariant distributions on $G$ in order to apply the lemma.  However, it is sometimes the case that these hypotheses are satisfied because one has an identity $H\sigma (g)H = HgH$, for all $g\in G$.
For supercuspidal representations, one may weaken the latter double coset identity as follows.
Assume $(G,H)= (G,G^\theta)$,  where $G=\bG(F)$ for some connected reductive $F$-group $\bG$ and some involution $\theta$ of $G$.

\begin{lemma}[\cite{H}]\label{superGelfandLemma}
Suppose $\theta'$ is an automorphism of $G$ of order two such that $\theta' (H) = H$ and the set of all $g\in G$ such that $ZH \theta'(g)^{-1} H = ZH gH$ has full measure in $G$.  If $\pi$ is an irreducible, supercuspidal representation of $G$ such that $\Hom_{H}(\pi,1)\ne 0$ then $\tilde\pi \simeq \pi^{\theta'}$ and the spaces $\Hom_{H}(\pi,1)$ and $\Hom_{H}(\tilde\pi,1)$ have dimension one.
\end{lemma}

We observe that it follows from Remark \ref{contrmult} that for tame supercuspidal representations the spaces $\Hom_{H}(\pi,1)$ and $\Hom_{H}(\tilde\pi,1)$ have the same dimension.

\section{Examples involving representations of general linear groups}
\label{sec:glexamples}
In the papers \cite{HM2} and \cite{HM3}, we
studied distinguishedness of tame supercuspidal
representations of general linear groups relative 
to three kinds of involutions. In this section,
we discuss relations between the methodology
and results of \cite{HM2} and \cite{HM3} and those
of this paper. We also indicate how 
Theorem~\ref{maindimformula} can be applied to 
strengthen some results of \cite{HM3}.

At the end of the section, we indicate how to
correct an error in the definitions of certain
special isomorphisms that were defined in 
these papers. We remark that this adjustment
to the definitions leaves the results of
the papers intact.

To be consistent with the aforementioned papers, we denote our ground field by $F'$, 
rather than $F$.  In this section, we assume that $F^\prime$ has characteristic zero
(as this was an assumption in \cite{HM2} and \cite{HM3}).
We consider the following cases:
\begin{itemize}
\item[Case 1.] $\bG=R_{F/F^\prime}\bGL_n$, where $F$ is a quadratic extension of $F'$.  Fix $\eta\in G$ that is hermitian with respect to $F/F'$ and let $\theta$ be the involution of $G = \bG (F') = {\bf GL}_n (F)$ given by $\theta (g) = \eta\ {}^t\bar g^{-1}\  \eta^{-1}$, where $\bar g= (\bar g_{ij})$ and $\alpha \mapsto \bar\alpha$ is the nontrivial Galois automorphism of $F/F'$.  So $G^\theta$ is a unitary group in $n$ variables. (See \cite{HM2}.)

\item[Case 2.] $\bG=R_{F/F'}\bGL_n$, where $F$ is a quadratic extension of $F'$.  Take $\theta$ to be the involution
of $G$ given by $\theta (g) = \bar g$, with notations as in Case~1.  So $G = {\bf GL}_n (F)$ and $G^\theta = {\bf GL}_n (F')$. 
(See \cite{HM3}.)

\item[Case 3.] $\bG=\bGL_n$, with $n$ even. The involution $\theta$ is given by
conjugation by the diagonal matrix whose first $n$ diagonal
entries are $1$, and whose remaining diagonal entries are $-1$. In
this case, $G^\theta\cong \bGL_{n/2}(F')\times \bGL_{n/2}(F')$.  To treat all three cases with a uniform approach, we take $F = F'$ in this case.
  (See \cite{HM3}.)
\end{itemize}

Note that in all three cases our group $\bG$ is defined over $F'$ and we have defined an extension $F$ of $F'$ such that  $G = \bG(F')= {\bf GL}_n(F)$.   (Note that  this contrasts  with our conventions in the rest of the \paperbook\  to use $F$ as the ground field.)

Let $\pi$ be an irreducible tame supercuspidal
representation of $G$.
Let $\varphi$ be an
$F$-admissible quasicharacter of the multiplicative group
of a tamely ramified extension $E$ of $F$ of degree $n$ 
such that $\pi$ belongs to the equivalence
class of supercuspidal representations that
is associated to the quasicharacter
$\varphi$ via Howe's construction.  
The results of \cite{HM2} and \cite{HM3} express
conditions for distinguishedness of $\pi$ in
terms of properties of $\varphi$. 
In all three cases, it is relatively easy to show that $\pi$ cannot be
$\theta$-distinguished unless $\pi\circ\theta$
is equivalent to the contragredient $\tilde\pi$ of
$\pi$. So we assume from now on that $\pi\circ\theta\simeq\tilde\pi$.

The representation
$\tilde\pi$ belongs to the equivalence
class associated to the $F$-admissible
quasicharacter $\varphi^{-1}$ of $E^\times$.
As discussed in Section~2 of \cite{HM2}
and Section~3 of \cite{HM3},
there exists an embedding of $E^\times$
into $G$ having the property that $E^\times$
is $\theta$-stable and $\varphi\circ\theta
=\varphi^{-1}$. Furthermore, there exists
a Howe factorization of $\varphi$
that has good symmetry properties relative
to $\theta\,|\, E^\times$. 
These symmetry properties and the discussion in Section \ref{sec:Howeconstruction} yield an extended generic
cuspidal $G$-datum $\Psi=(\vec\bG,y,\rho,\vec\phi)$
that is $\theta$-symmetric and satisfies $\rho\circ\theta\simeq \tilde\rho$ and $\pi (\Psi)\simeq \pi$.

There exists a quadratic character
$\eta_\theta$ of $K^0(\Psi)^\theta$
such that
 $$\Hom_{K(\Psi)^\theta}(\kappa(\Psi),1)\cong \Hom_{K^0(\Psi)^\theta}(\rho,\eta_\theta).$$
The methods applied in \cite{HM2} and \cite{HM3} to obtain 
this result are essentially
special cases of some of the methods of Sections~5.4 and 5.5
of this paper. Recall that the equivalence class of
$\kappa(\Psi)$ is not affected by varying any of
the relevant special isomorphisms that are used 
in Yu's construction.   (See Lemma \ref{rsiequiv}. The arguments of \cite{HM2} and
\cite{HM3} could be simplified somewhat
by applying this result.)

Note that the toral case is the case in which $E=E_0$, where $E_0$ is defined as in Section \ref{sec:Howeconstruction}.
In this case, $\rho$ is the trivial character of $G^0 = E^\times$ and the space  $\Hom_{K(\Psi)^\theta}(\kappa(\Psi),1)$ is nonzero (and, in fact, 1-dimensional)
if and only if $\eta_\theta$ is the trivial
character of $E^{\times,\theta}$.

In the nontoral case, the quotient $G^0_{y,0 :0^+}$ is isomorphic
to a finite general linear group and the cuspidal
representations of such finite groups are well understood
(see Section~\ref{sec:Howeconstruction}), and
so the space $\Hom_{K^0(\Psi)^\theta}(\rho,\eta_\theta)$
can be analyzed further.   (This is a special
feature of general linear groups.)
Recall that Theorem~\ref{maindimformula} does not address the question
of analyzing $\Hom_{K^0(\Psi)^\theta}(\rho,\eta_\theta)$
when $\Psi$ is not toral.
The analysis carried out in Sections~12 and 6 of \cite{HM2} 
and \cite{HM3}, respectively, shows that, as in the toral case,
in order for $\Hom_{K^0(\Psi)^\theta}(\rho,\eta_\theta)$ to
be nonzero, a certain quadratic character of $E^{\times,\theta}$
must be trivial.

In Case~1, $\varphi\,|\, E^{\times,\theta}$
is trivial and, as shown in \cite{HM2},
$\Hom_{K(\Psi)^\theta}(\kappa(\Psi),1)$ is
1-dimensional.
Furthermore, 
$\Hom_{K(\Psi)^{g\cdot\theta}}(\kappa(\Psi),1)=0$
 whenever $g\notin K(\Psi)\,G^\theta$.
Hence $\pi\circ\theta\simeq\tilde\pi$
(that is, $\pi$ is $\gal (F/F')$-invariant)  if and only
if $\pi$ is $\theta$-distinguished, in which
case, ${\Hom}_{G^\theta}(\pi,1)$ has dimension one.

Cases~2 and 3 were studied in \cite{HM3}.
In each of these cases, it was already known
that $\dim\Hom_{G^\theta}(\pi,1)\le 1$.
Furthermore, as discussed in \cite{HM3}, there exists
a unique $\sigma\in \text{Aut}(E/F^\prime)$
of order $2$, such that $\sigma$
and $\theta$ agree on $E^\times$.
The condition $\varphi\circ(\theta\,|\, E^\times)
=\varphi^{-1}$ is equivalent to $\varphi\,|\,
N_{E/E^\sigma}(E^\times)=1$,
where $E^\sigma$ denotes the fixed field
of $\sigma$. Hence,
either $\varphi\,|\, E^{\times,\theta}$ is
trivial or it is the character
${\rm sgn}_{E/E^\sigma}$ of $E^{\times,\theta}=E^{\sigma\,\times}$
associated  to
the quadratic extension $E/E^\sigma$ by class field theory.
One of the main results of \cite{HM3} gives
necessary and sufficient conditions
on $\varphi\,|\, E^{\times,\theta}$
for nonvanishing of $\Hom_{K(\Psi)^\theta}
(\kappa(\Psi),1)$. For precise statements
of these conditions,
the reader may refer to \cite{HM3}.

In Case~2, when $n$ is odd, results of
\cite{HM3} show that  $\varphi\,|\, E^{\times,
\theta}$ is trivial if and only if $\Hom_{K(\Psi)^\theta}
(\kappa(\Psi),1)\not=0$. Furthermore,
when $\varphi\,|\, E^{\times,\theta}$ is nontrivial,
the central character of $\pi$ is nontrivial
on $F^{\prime\times}$. Because  any linear functional
on the space of $\pi$ tranforms under the action of
$H\cap Z=F^{\prime\times}$ by the restriction
of the inverse of the central character of 
$\pi$ to $H\cap Z$, 
nontriviality
of this restriction implies that $\pi$ is not
$\theta$-distinguished.
Thus, for Case~2 with
$n$ odd, $\pi$ is $\theta$-distinguished if and only 
if $\varphi\,|\, E^{\times,\theta}$ is trivial. This coincides with 
Kable's results (\cite{Ka}) showing that when $n$ is odd, $\pi$ is 
$\theta$-distinguished if and only if the central character of $\pi$ 
is trivial.

In \cite{HM3},
we did not study $\Hom_{K(\Psi)^{g\cdot\theta}}(
\kappa(\Psi),1)$ for $g\notin K(\Psi)G^\theta$.
Therefore, except for Case~2 when $n$ is odd,
we did not rule out the possibility
of $\Hom_{G^\theta}(\pi,1)$ being nonzero
when $\Hom_{K(\Psi)^\theta}(\kappa(\Psi),1)=0$.
In Case~3 and in Case~2 when $n$ is even,
Theorem~\ref{maindimformula}(3) and (5) can be used
to sharpen the results of \cite{HM3}. In
particular, the necessary and sufficient
conditions for nonvanishing of
$\Hom_{K(\Psi)^\theta}(\kappa(\Psi),1)$
can be shown to be equivalent to
nonvanishing of $\Hom_{G^\theta}(\pi,1)$.

In the toral case, this is particularly easy
to see because $K^0(\Psi)=E^\times$.
Let $g\in G$ be such that $g\theta(g)^{-1}
\in E^\times$. Then $\sigma(g\theta(g)^{-1})
=(g\theta(g)^{-1})^{-1}$, so there exists
$\alpha\in E^\times$ such that $\alpha\sigma(\alpha)^{-1}
=\alpha\theta(\alpha)^{-1}=g\theta(g)^{-1}$.
It follows that $g\in E^\times G^\theta\subset K(\Psi)G^\theta$.
Applying Theorem~\ref{maindimformula}(3) and (5), together
with the fact that $\dim\Hom_{G^\theta}(\pi,1)\le 1$, yields
$$\Hom_{G^\theta}(\pi,1)\cong 
\Hom_{K(\Psi)^\theta}(\kappa(\Psi),1).$$

In the nontoral case, $K^0(\Psi)\simeq E_0^\times
\bGL_{m}(\gO_{E_0})$ with $m\ge 2$,  and it is not immediately 
obvious that Theorem~\ref{maindimformula}(3) 
and (5)
imply that $$\Hom_{G^\theta}(\pi,1)
\cong \Hom_{K(\Psi)^\theta}(\kappa(\Psi),1).$$
To obtain the latter isomorphism, it is necessary 
to show that
if $g\in G$ is such that $g\theta(g)^{-1}
\in K^0(\Psi)$ and $g\notin K(\Psi)G^\theta$,
then $\Hom_{K^{0}(\Psi)^{g\cdot\theta}}(\rho,\eta_{g\cdot\theta})
=0$. We do not provide the details here.

In Case~2, when $n$ is even, there is an alternate
way to prove that vanishing of $\Hom_{K(\Psi)^\theta}(\kappa(\Psi),1)$
 implies $\Hom_{G^\theta}(\pi,1)=0$. As discussed
in \cite{HM3}, we may define a character $\chi$
of $F^\times$ that agrees with the character
${\rm sgn}_{F/F^\prime}$ of $F^{\prime\times}$ associated to $F/F'$ by class field theory  and is such
that $\phi(\chi\circ N_{E/F})$ is $F$-admissible
and $\pi\otimes (\chi\circ\det)$
is equivalent to $\ind_{K(\Psi)}^G(\kappa(\Psi)\otimes
(\chi\circ\det))$. Also $(\chi\circ N_{E/F})\,|\,
E^{\sigma\times}={\rm sgn}_{E/E^\sigma}$.
Kable (\cite{Ka}) proved that exactly one of the two representations $\pi$ and $\pi\otimes (\chi\circ\det)$ is $\theta$-distinguished.
Suppose that $\Hom_{K(\Psi)^\theta}
(\kappa(\Psi),1)=0$. Then $$\Hom_{K(\Psi)^\theta}
(\kappa(\Psi)\otimes (\chi\circ\det),1)\not=0.$$
Hence $\Hom_{G^\theta}(\pi\otimes(\chi\circ\det),1)\not=0$.
Kable's result tells us that $\pi$ is not $\theta$-distinguished. Thus $\Hom_{G^\theta}(\pi,1)=0$
 when $\Hom_{K(\Psi)^\theta}(\kappa(\Psi),1)=0$.

In addition to Kable's results for Case 2 that are mentioned above, D.~Prasad \cite{Pr}
has also obtained results in
Cases~1 and 2 when the extension $F/F'$ is
unramified.  

There is an error in the definition of the special isomorphisms $\nu_i^\natural$ 
and $\natural$ of \cite{HM2} and \cite{HM3}, respectively. 
The groups $\cH^+_i$ and $\cH^-_i$ defined there
both contain the center of the Heisenberg group $\cH_i$. Hence these groups 
do not form a split polarization of $\cH_i$, as is required in the
 definition of the special isomorphism. However if we replace $\cH^+_i$ 
by the image $\dot\cH_i^+$ of $J_i^\theta$ in $\cH_i$, then
$\dot\cH_i^+$ and
$\cH_i^-$ do form a polarization of $\cH_i$, and if one takes
the appropriate splitting of $(\dot\cH_i^+,\cH_i^-)$ and defines the
special isomorphism relative to the splitting (as in Proposition~\ref{stJpol}), 
then one obtains a relevant special isomorphism. 
The reader may refer to Sections~\ref{sec:Heis}
and \ref{sec:genHeis} of this paper for information on split polarizations,
the associated special isomorphisms, and the choices of split polarizations
that give rise to relevant special isomorphisms.
With this adjusted definition of the special isomorphisms,
the main results of \cite{HM2} and \cite{HM3} remain intact, since
the adjusted special isomorphisms have all the required properties.

\section{More examples}
\label{sec:examples}

The first two examples presented here each involve an equivalence class
$\xi$ of toral $(G,K)$-data, together with
a $G$-orbit $\Theta$ of involutions, for which $\langle \Theta,\xi\rangle_G=2$.
In the first example, we have $m_K(\Theta)=1$ or $2$, depending on 
whether or not $-1$ is a square in $F^\times$, while in the second example,
$m_K(\Theta)=2$. The third example exhibits equivalence classes
of nontoral $(G,K)$-data $\xi$, together with a $G$-orbit $\Theta$ of involutions,
having the property that there are two $K$-orbits in $\Theta$
that are moderately compatible with $\xi$. Furthermore, exactly
one of these $K$-orbits is strongly compatible with $\xi$.

Note that in all of our examples, when describing
the various filtration groups $G_{y,r}$, $y\in \cB(\bG,F)$, $r\ge 0$,
we assume that our valuation $v_F$ on $F$ is normalized
so that $v_F(F^\times)=\mathbb Z$.

\begin{example}\label{exGLO} 
Let $\bG = {\bf GL}_2$ and let $\theta$ be the involution of $\bG$ 
given by $\theta(g) = 
{}^tg^{-1}$ (where ${}^t g$ denotes the transpose of $g$). 
Let $E$ be an unramified quadratic extension of $F$.
Fix $a$, $b\in \gO_F^\times$
such that $a^2+b^2$ is not a square in $\gO_F^\times$. Set
$\varepsilon=a^2+b^2$. 
Then $E=F(\sqrt\varepsilon)$. Let $\eta=\left(\smallmatrix a & b\\
                                                            b & -a\endsmallmatrix\right)$.
Then $\eta^2=\varepsilon \, I$ (where $I$ is the $2\times 2$ identity matrix),
and the group $\{\, c\,I+d\,\eta\ | \ c,\, d\in F\,\}\cap G$ 
is isomorphic to
$E^\times$. Hence this group is equal to the $F$-rational
points $T=\bT(F)$ of a maximal $F$-torus $\bT$ in $\bG$ such
that $\bT=R_{E/F}\bGL_1$. Let $Z$ be the center of $G$.
Then $T/Z\simeq E^\times/F^\times$ and the building $\cB(\bT,F)$
embeds in $\cB(\bG,F)$ as a line. Fix a point $y\in \cB(\bT,F)$.
If $\ell$ is a nonnegative integer and $\ell<r\le \ell+1$,
then
$$
G_{y,r}=I + M_{2\times 2}(\gP_F^{\ell+1}).
$$

The Lie algebra $\t$ of $T$ consists of matrices of the form
$c\, I + d\, \eta$, $c$, $d\in F$. Given $c^\prime$, $d^\prime\in F$,
define $X^*_{(c^\prime,d^\prime)}\in \t^*$ by
$X^*_{(c^\prime,d^\prime)}(c\, I + d\, \eta)=c^\prime c + d^\prime d\varepsilon$,
$c$, $d\in F$. Fix a positive integer $j$. Fix $c^\prime$, $d^\prime
\in F^\times$ such that $v_F(c^\prime)\ge 2j+1=v_F(d^\prime)$.
Then the linear functional $X^*_{(c^\prime,d^\prime)}$ is
$G$-generic of depth $-r_0=-2j-1$. If $c\, I + d\, \eta
\in T_{2j+1}$, then $c-1$, $d\in \gP_F^{2j+1}$
and the map $(c\, I + d\,\eta)T_{2j+2}
\mapsto (c-1)I+d\,\eta + \t_{2j+2}$ defines an
isomorphism between $T_{2j+1:2j+2}\simeq (1+\gP_E^{2j+1})
/ (1+\gP_E^{2j+2})$ and $\t_{2j+1:2j+2}\simeq \gP_E^{2j+1}
/\gP_E^{2j+2}$. If $\psi$ is a character of $F$ that
is nontrivial on $\gO_F$ and trivial on $\gP_F$,
then the map
$$
(c\, I + d\, \eta)T_{2j+2}\mapsto \psi(X^*_{(c^\prime,d^\prime)}((c-1)I+d\,\eta))
=\psi(c^\prime(c-1)+d^\prime d\varepsilon)
$$
defines a $G$-generic character of $T_{r_0:r_0^+}=T_{2j+1:2j+2}$.

Let $\phi_0$ be a quasicharacter of $T$ that is trivial on
$T_{2j+2}$ and whose restriction to $T_{2j+1}$ factors
to the above character of $T_{2j+1:2j+2}$. The pair
$\vec\bG=(\bT,\bG)$ is a tamely ramified twisted Levi
sequence in $\bG$. Let $\rho$ and $\phi_1$ be the trivial characters of
$T$ and of $G$, respectively. Set $\vec\phi=(\phi_0,\phi_1)$
and $\Psi=(\vec\bG, y,\rho,\vec\phi)$. Then $\Psi$
is an extended generic cuspidal $G$-datum, and it is toral.
(We remark that $\phi_0$ is an $F$-admissible quasicharacter
of $E^\times$ and, as can be seen from the discussion
in Section~\ref{sec:Howeconstruction}, $\varphi_0$ gives 
rise (via Howe's construction)
 to the same equivalence
class of tame supercuspidal representations of $G$
as the datum $\Psi$ (via Yu's construction).)

Note that $G_{y,j+1/2}=G_{y,(j+1/2)^+}=G_{y,j+1}$. Hence the
irreducible representation $\kappa_0$ of $K=K(\Psi)
=T G_{y,j+1}$ corresponding to $\phi_0$ is one-dimensional.
Since both $\rho$ and $\phi_1$ are trivial, we
have $\kappa=\kappa(\Psi)=\kappa_0$ and $\pi=\pi(\Psi)
=\ind_K^G\kappa_0$.

Note that $\theta(t)=t^{-1}$ for all $t\in T$.
It follows that $\theta([y])=[y]$, $\theta(K)=K$,
and $\phi_0\circ\theta=\phi_0^{-1}$. According
to Proposition~\ref{invdimtorus}(1),
$$
\Hom_{K^\theta}(\kappa,1)\simeq \Hom_{T^\theta}(\phi_0,1),
$$
since $\rho$ and $\eta_\theta^\prime(\Psi)$ are trivial.
Observe that $T^\theta=\{\,\pm I\,\}$. Thus $\Hom_{K^\theta}(\kappa,1)$
is nonzero if and only if $\phi_0(-I)=1$. Note that there exist
quasicharacters $\phi_0$ as above that do satisfy $\phi_0(-I)=1$.
Indeed, if $\phi_0(-I)=-1$, we may obtain the desired quasicharacter
by squaring $\phi_0$.

Now we turn to computing $m_K(\Theta)$ for $\Theta$ equal
to the $G$-orbit of $\theta$. Suppose that $g\in G$
is such that $g\theta(g)^{-1}=g{}^t g\in Z\simeq F^\times$.
Then there exist $c$, $d\in F$ such that 
$g\theta(g)^{-1}=(c^2+d^2)I$
and either $g=\left(\smallmatrix c & d\\
                         -d & c\endsmallmatrix\right)$
or $g=  \left(\smallmatrix c & d\\
                         d & -c\endsmallmatrix\right)$.
Note that if
$g\theta(g)^{-1}= (c^2+d^2)I$ and $c^2+d^2=\alpha^2$ 
for some $\alpha\in E^\times$,
then $\alpha^{-1}g\in G^\theta$. That is, $g\in T\, G^\theta
\subset K\, G^\theta$.

Observe that if $g\in KG^\theta$ has the form $g=tkh$,
with $t\in T$, $k\in G_{y,j+1}$, and $h\in G^\theta$,
then $g\theta(g)^{-1}=t^2(t^{-1}k\theta(k)^{-1}t)
\in t^2 G_{y,j+1}$. Now if we also have $g\theta(g)^{-1}
=u\,I$, $u\in F^\times$, we can easily see (using $j+1>0$) that $u\in 
t^2(1+\gP_E^{j+1})\subset (E^\times)^2$. 

In view of the above, we conclude that if
$g\theta(g)^{-1}=(c^2+d^2)I$, then
$g\in K\, G^\theta$ if and only if $c^2+d^2$
is a square in $E^\times$.

Let $\varpi$ be a prime element in $F$. The set 
$F^\times$ is the disjoint union of $(F^\times)^2$,
 $\varepsilon (F^\times)^2$,
$\varpi(F^\times)^2$ and $\varepsilon\varpi (F^\times)^2$.
As $\varepsilon$ is a square in $E^\times$, and $\varpi$ is
not a square in $E^\times$,
the set of elements in $F^\times$ that are nonsquares in $E^\times$
is equal to $\varpi((E^\times)^2\cap F^\times)$. 
Suppose that $c$, $d\in F$ and $c^2+d^2\not=0$.
Observe that if $-1$ is not in $(F^\times)^2$, 
then $c^2+d^2$ cannot have odd valuation, and hence
cannot be a nonsquare in $E^\times$. Therefore
$m_K(\Theta)=1$ whenever $-1\notin (F^\times)^2$.

Now suppose that $-1\in (F^\times)^2$. Let $F(\sqrt\varpi)$
be the extension of $F$ generated by a square root
of $\varpi$, and let $N_{F(\sqrt\varpi)/F}$ be
the norm map from $F(\sqrt\varpi)^\times$
to $F^\times$. Then 
it is easy to see that the fact that
the image of $N_{F(\sqrt\varpi)/F}$ contains
squares in $F^\times$ implies existence
of $c$ and $d$ in $F$ such that $c^2+d^2
\in \varpi(F^\times)^2$.
Choose such elements $c$ and $d$, and let $g$
be an element of $G$ such that $g\theta(g)^{-1}
=(c^2+d^2)I$. Then we have seen above that
the fact that $c^2+d^2\notin (E^\times)^2$
guarantees that $g\notin K\, G^\theta$. 
Thus the two (distinct) double cosets
$K\, G^\theta$ and $KgG^\theta$ both give
rise to the same $K$-orbit of involutions,
namely the $K$-orbit of $\theta$.
Furthermore, even though we can also find
$g^\prime\in G$ such that $g^\prime\theta(g^\prime)^{-1}=
u\,I$ with 
$u\in \varepsilon\varpi(F^\times)^2$, the
fact that $\varepsilon\in (E^\times)^2$
implies that $g^\prime\theta(g^\prime)^{-1}
\in g\theta(g)^{-1}(E^\times)^2$, hence
$g^\prime\in TgG^\theta\subset KgG^\theta$.
Thus, when $-1\in (F^\times)^2$, we have
$m_K(\Theta)=2$. Note that the property
$g\theta(g)^{-1}\in Z$ implies that $g$
normalizes $G^\theta$.

Summarizing, if the quasicharacter
$\phi_0$ is chosen so that $\phi_0(-I)=1$
then
the supercuspidal representation $\pi(\Psi)$
satisfies $\Hom_{G^\theta}(\pi(\Psi),1)\not=0$.
If, in addition, $-1$ is a square in $F^\times$,
then $m_K(\Theta)=2$. We remark that
this shows that there exist distinguished tame 
supercuspidal
representations of general linear groups for
which multiplicity one fails. 
\end{example}

\begin{example}\label{exampleA}

Let $\varepsilon\in \gO_F^\times$ be a nonsquare, $E=F
(\sqrt\varepsilon)$, and $J=\left(\smallmatrix 0 & 1\\
                        1 & 0\endsmallmatrix\right)$.
Let $\bG$ be the corresponding $2\times 2$ unitary group, with
$F$-rational points
$$
G=\{\, g\in \bGL_2(E)\ | \ {}^t\bar g J g=J\, \},
$$
where, if $g=(g_{ij})_{1\le i,j\le 2}$,
then $\bar g=(\bar g_{ij})_{1\le i,j\le 2}$,
with $\bar\alpha$ denoting the image of $\alpha\in E$
under the nontrivial element of $\gal(E/F)$.

Let $\bT$ be the maximal $F$-torus in $\bG$ having
$F$-rational points 
$$
T=\left\{ \ t_{\alpha,\beta}=\left(\begin{matrix} \alpha & \beta\\
                                    \beta & \alpha\end{matrix}\right)
\ \mid\ \alpha,\, \beta\in E,\ \alpha \bar\beta=-\bar\alpha\beta,
\ \alpha\bar\alpha + \beta\bar\beta=1\ \right\}.
$$
The map $t_{\alpha,\beta}\mapsto (\alpha+\beta,\alpha-\beta)$
is an isomorphism of $T$ with $E^1\times E^1$, where $E^1$
is the kernel of the norm map $N_{E/F}:E^\times\rightarrow F^\times$.
Because $T$ is compact, the building $\cB(\bT,F)$ embeds as a point
$\{y\}$ in $\cB(\bG,F)$. If $\ell$ is a nonnegative
integer and $\ell < r\le \ell +1$, then
$$
G_{y,r}=\{\, g\in G \ | \ g-I\in M_{2\times 2}(\gP_E^{\ell +1})\, \}.
$$
The Lie algebra $\t$ of $T$ consists of matrices of the form
$$
X_{(a,b)} =\left( \begin{matrix} a\sqrt\varepsilon & b\sqrt\varepsilon\\
                 b\sqrt\varepsilon & a\sqrt\varepsilon\end{matrix}\right),
\qquad a,\, b\in F.
$$
Given $a$, $b\in F$, define $X_{(a,b)}^*\in \t^*$ by
$X_{(a,b)}^*(X_{(c,d)})=ac+bd$, $c$, $d\in F$.
Fix a positive integer $j$. Then,
if $a$, $b\in F^\times$ satisfy $v_F(a)\ge {2j+1}=v_F(b)$,
the linear functional $X_{(a,b)}^*\in \t^*$ is $G$-generic
of depth $-r_0=-2j-1$. Fix such an $a$ and $b$. 
If $t_{\alpha,\beta}\in T_{2j+1}$, then $(I-t_{\alpha,\beta})
(I+t_{\alpha,\beta})^{-1}\in \t_{2j+1}$
and $t_{\alpha,\beta}T_{2j+2}
\mapsto (I-t_{\alpha,\beta})(I+ t_{\alpha,\beta})^{-1}
+ \t_{2j+2}$ defines an isomorphism
between $T_{2j+1:2j+2}$ and $\t_{2j+1:2j+2}$.
If $\psi$ is a character of $F$ that is nontrivial
on $\gO_F$ and trivial on $\gP_F$, then
the map
\begin{equation*}
\begin{split}
t_{\alpha,\beta}T_{2j+2}
\mapsto &
\psi(X_{(a,b)}^*((1-t_{\alpha,\beta})(1+t_{\alpha,\beta})^{-1}))\\
&=\psi((a(1-\alpha^2+\beta^2)-2b\beta)((1+\alpha)^2-\beta^2)^{-1})
\end{split}
\end{equation*}
defines a $G$-generic character of $T_{r_0:r_0^+}=T_{2j+1:2j+2}$.

Let $\phi_0$ be a character of $T$ that is trivial on
$T_{2j+2}$ and whose restriction
to $T_{2j+1}$ factors to the above character of $T_{2j+1}/T_{2j+2}$.
Let $\vec\bG=(\bT,\bG)$. If $\rho$ and $\phi_1$ are the trivial
characters of $T$ and of $G$, respectively, and
$\vec\phi=(\phi_0,\phi_1)$, then
$\Psi=(\vec\bG,y, \rho,\vec\phi)$ is an extended generic cuspidal
 $G$-datum. Note that $G_{y,j+1/2}=G_{y,(j+1/2)^+}=G_{y,j+1}$.
Hence the irreducible representation $\kappa_0$ of $K=K(\Psi)=TG_{y,j+1}$
corresponding to $\phi_0$ is one-dimensional. Since both $\rho$ and
$\eta_\theta^\prime(\Psi)$ are trivial, $\kappa=\kappa(\Psi)$ is equal to $\kappa_0$,
and $\pi=\pi(\Psi)={\rm ind}_K^G \kappa_0$.

Let $\theta$ be the involution of $G$ defined by
$\theta(g)=\bar g$. Then $G^\theta$ is the $2\times 2$ orthogonal
group defined by $J$. Note that
$\theta(t_{\alpha,\beta})=t_{\alpha,\beta}^{-1}$ for all
$t_{\alpha,\beta}\in T$. It follows that $\theta(y)=y$,
$\theta(K)=K$, and $\phi_0\circ\theta=\phi_0^{-1}$. 
According to Proposition~\ref{invdimtorus}(1), 
$$
\Hom_{K^\theta}(\kappa_0,1)\cong \Hom_{T^\theta}(\phi_0,1),
$$
since $\rho$ and $\eta^\prime_\theta(\Psi)$ are trivial.
Observe that $T^\theta=\{\, \pm I,\pm J\,\}$.
Thus $\Hom_{K^\theta}(\kappa_0,1)$ is nonzero
if and only if $\phi_0\,|\, \langle -I,\,J\rangle
=1$, and in that case, it is one-dimensional.
According to Proposition~\ref{invdimtorus}, if $\phi_0\,|\, 
\langle - I,\,J\rangle=1$, then
$\dim\Hom_{G^\theta}(\pi(\Psi),1)$ is
equal to the number of distinct $K$-$G^\theta$
double cosets that contain an element $g$
with $g\theta(g)^{-1}\in T$.

Because the set of squares in $T$ coincides with
the set $\{\, t\theta(t)^{-1}\ | \ t\in T\,\}$,
if $g\in G$ is such that $g\theta(g)^{-1}$ is a square
in $T$, then $g\in TG^\theta\subset KG^\theta$.
Suppose that $t_{\alpha,\beta}\in T$ is equal to
$g\theta(g)^{-1}$ for some $g\in G$. It is easy
to see that $\det t_{\alpha,\beta}=\alpha^2-\beta^2$ 
must be a square in $E^1$.
If both $\alpha+\beta$ and $\alpha-\beta$ are squares
in $E^1$, then $t_{\alpha,\beta}$ is a square in
$T$, which implies $g\in KG^\theta$.
It follows that if $g\notin KG^\theta$, then
neither $\alpha-\beta$ nor $\alpha+\beta$
is a square in $E^1$. Assume that this is the case.
Let $\omega\in 
{\gO}_E^\times$ be such that $\omega\bar\omega^{-1}$
is a nonsquare in $E^1$. Then
$\alpha-\beta=\omega\bar\omega^{-1}\gamma^2$ and
$\alpha+\beta=\omega\bar\omega^{-1}\gamma^2\delta^2$
for some $\gamma$, $\delta\in E^1$. This implies that
$t_{\alpha,\beta}=t_{\omega\bar\omega^{-1},0}t^2$,
where $t\in T$ is the element corresponding to
$(\gamma^2\delta^2,\gamma^2)\in E^1\times E^1$.
That is, $t_{\alpha,\beta}=g^\prime\theta(g^\prime)^{-1}$
where $g^\prime=t\,\left(\smallmatrix \omega & 0\\
                           0 & \bar\omega^{-1}\endsmallmatrix\right)$.
This implies that $g\in T\left(\smallmatrix \omega & 0\\
                 0 & \bar\omega^{-1}\endsmallmatrix\right)G^\theta
\subset K \left(\smallmatrix\omega & 0\\
                     0 & \bar\omega^{-1}\endsmallmatrix\right)G^\theta$.
This shows that there are exactly two $K$-$G^\theta$ double cosets
in $G$ containing elements $g$ with $g\theta(g)^{-1}\in T$.
This allows us to conclude that, if $\pi=\pi(\Psi)$, then
$$
\dim\Hom_{G^\theta}(\pi,1)=\begin{cases} 2, & {\rm if}\ \phi_0\,|\, \langle -I,J\rangle
= 1,\\
0, & {\rm otherwise.}
\end{cases}
$$
It is a simple matter to check  that there exist 
characters $\phi_0$ as above that 
satisfy $\phi_0(-I)=\phi_0(J)=1$. Suppose that $\phi_0$
is chosen to have this property.
We remark that for $g=\left(\smallmatrix\omega & 0\\
             0 & \bar\omega^{-1}\endsmallmatrix\right)$,
we have $g\theta(g)^{-1}\in Z$. Thus $g\cdot\theta=\theta$.
As observed above, $g\notin KG^\theta$. Hence, letting
$\Theta$ denote the $G$-orbit of $\theta$, we have
$m_K(\Theta)=2$ for this example. Note that $g\theta(g)^{-1}
\in Z$ implies that $g$ normalizes $G^\theta$.
\end{example}

\begin{example}\label{exampleB}
In this example, our base field will be denoted $F^\prime$ (rather than the usual notation $F$),  
while $F$ will denote a quadratic totally ramified 
extension of $F^\prime$, where $n$ is an even positive integer.
Since we will be using results from \cite{HM2}, we
assume that $F^\prime$ has characteristic zero.
We take $G=(R_{F/F^\prime}\bGL_{n})(F^\prime)\cong \bGL_{n}(F)$.
Let $\theta$ be an involution of $G$ whose
group of fixed points is an $n\times n$ unitary group relative to
the quadratic extension $F/F^\prime$. If $\alpha\in F$,
let $\bar\alpha$ be the image of $\alpha$ under the nontrivial
element of $\gal(F/F^\prime)$, and if 
$g=(g_{ij})_{1\le i,j\le n}\in G$, let
$\bar g=(\bar g_{ij})_{1\le i,j\le n}$. 
The example that we consider here involves
irreducible tame supercuspidal representations
$\pi$ of $G$ that are Galois invariant in the sense
that the representation $g\mapsto \pi(\bar g)$ is equivalent 
to $\pi$. As we saw in \cite{HM2}, all such
representations satisfy $\dim\Hom_{G^\theta}(\pi,1)=1$.

Suppose that $E$ is an extension of $F$ of degree
$n$ having the property that there exists $\sigma
\in {\rm Aut}(E/F^\prime)$ of order $2$,
with $E$ totally ramified over $E^\sigma$
and $\sigma\,|\, F$ nontrivial. Assume that
$\varphi$ is an $F$-admissible quasicharacter
of $E^\times$ (see Section~\ref{sec:Howeconstruction})
having the following properties:
\begin{itemize}
\item $\varphi\circ\sigma=\varphi$
\item There exists a subfield $E_0$ of
$E$ such that $F\subset E_0$, $E$ is unramified
over $E_0$ of even degree $n_0$, and $\varphi
\,|\, 1+\gP_{E}$ factors through
the norm map $N_{E/E_0}$. In addition,
if $E_0\not=F$, then $\varphi\,|\, 1+\gP_E$
does not factor through $N_{E/L}$ for any
field $L$ with $F\subset L\subsetneq E_0$.
\end{itemize}

As discussed in Section~\ref{sec:Howeconstruction},
 we may associate a 
generic cuspidal $G$-datum $\Psi$ to each Howe factorization
of $\varphi$, in such a way that the equivalence
class of $\pi(\Psi)$ coincides with the equivalence
class arising from $\varphi$ via Howe's construction.
As shown in \cite{HM2}, we may choose a Howe factorization
of $\varphi$ that has good symmetry properties relative
to the involution $\theta$. We may also choose an
embedding of $E^\times$ in $G$ that has the property
$\theta(\alpha)=\sigma(\alpha)^{-1}$, $\alpha\in E^\times$.
This implies that $\Psi$ may be chosen to have the following
properties: There exist $\sigma$-stable fields $F=E_d\subsetneq\cdots\subsetneq
E_0$ such that 
\begin{equation*}
\begin{split}
& G^i\cong (R_{E_i/F^\prime}\bGL_{n_i})(F^\prime)=\bGL_{n_i}(E_i), \ 
n_i=n[E_i:F]^{-1}, \ \forall\ i\in \{\, 0,\dots,d\,\}, \\
& y\in {\cB}(R_{E/F^\prime}\bGL_1,F^\prime), \ \ \ \ \ 
G_{y,0}^0=\bGL_{n_0}(\gO_{E_0}).
\end{split}
\end{equation*}
Furthermore, 
$\phi_i\circ\theta =\phi_i^{-1}$ for all $i\in \{\, 0,\dots,d\,\}$, and
the involution of
$G^0_{y,0: 0^+}$ induced by $\theta\,|\,G_{y,0}^0$
has fixed points equal to an orthogonal group. 
Note that $\theta(G^i)=G^i$ is implied by the fact
that $\theta$ stabilizes the center $E_i^\times$ of
$G^i$.
The representation $\rho$ is an irreducible representation
of $G_{[y]}^0=E_0^\times \bGL_{n_0}(\gO_{E_0})$
whose equivalence class corresponds to the $\gal(E/E_0)$-orbit
of an $E_0$-admissible quasicharacter $\varphi_{-1}$ of $E^\times$ 
of depth zero, with $\varphi_{-1}\circ\sigma=\varphi_{-1}$. 
The case $d=0$ and $\phi_0$ trivial was studied in
\cite{HMa2} and the more general cases were
studied in \cite{HM2}.

Let $\xi$ be the equivalence class of $(G,K)$-data
that contains $\Psi$, where $K=K(\Psi)$, and
let $\Theta$ be the $G$-orbit of $\theta$.
Note that we have chosen $\Psi$ in such a way
that $K\cdot\theta$ and $\xi$ are moderately compatible.
We will see that there are two distinct $K$-orbits in
$\Theta$ that are moderately compatible with $\xi$,
exactly one of which, namely $K\cdot\theta$, is strongly 
compatible with $\xi$. 

According to results of \cite{HM2},
expressed in the notation of this paper, $\eta^\prime_\theta(\Psi)$
extends to a character of $K^0$, and 
$\rho^\prime(\Psi)\otimes\eta_\theta^\prime(\Psi)\,|\, G_{y,0^+}^0$ 
factors to
an irreducible cuspidal representation
of $G_{y,0:0^+}^0$ whose central
character is trivial on $-1$. Then the fact
that the image of $K^{0,\theta}\cap G_{y,0}^0$ in $G_{y,0:0^+}^0$
is an orthogonal group allows an application
of results of \cite{HMa2} to conclude that
$\dim\Hom_{K^{0,\theta}}(\rho^\prime(\Psi),\eta^\prime_\theta(\Psi))=1$.
It follows that
$
\langle K\cdot\theta,\xi\rangle_K=1$.

Now we turn to studying other orbits in
$\Theta^K$ that are moderately compatible
with $\xi$. Suppose that $\Theta^\prime\in \Theta^K$ is moderately
compatible with $\xi$ and $\theta\notin\Theta^\prime$.
 Choose $g\in G$ such that
$g\cdot\theta\in \Theta^\prime$. Then
$\vartheta(\Psi)\,|\, K_+^{g\cdot\theta}=1$
(Lemma~\ref{varthetaorbitinv}) and the arguments 
used in the proof of Lemma~11.5
of \cite{HM2} show that 
$g\in KG^0G^\theta$. (We remark
that the groups $H$ and $G_{r-1}$ and fields
$E_{r-1}$ and $E^\prime$ of \cite{HM2} are our
$G^\theta$, $G^0$, $E_0$ and $E^\sigma$, respectively.)
According to Proposition~\ref{oldLemmasBandC}(2),
there
exists $g^\prime\in KgG^\theta\subset KG^0G^\theta$
such that $g^\prime\theta(g^\prime)^{-1}\in G_{[y]}^0$.
According to Hilbert's Theorem~90, an element
$z$ of $Z$ such that $\theta(z)=z^{-1}$ is of the
form $z^\prime\theta(z^\prime)^{-1}$ for some
$z^\prime\in Z$.
Putting these facts together, we find that there exists
$\dot g\in G^0$ such that $\dot g\in KgG^\theta$
and $\dot g\theta(\dot g)^{-1}\in G_{[y]}^0$.
In \cite{HM2} (see page~234 and Lemma~12.6), we
 give a parametrization (based on results of
\cite{HMa2})
of the
$G_{y,0}^0$-$G^{0,\theta}$ double cosets in $G^0$
that contain elements $h$ with $h\,\theta(h)^{-1}
\in G_{[y]}^0$ and do not lie inside $G_{[y]}^0G^{0,\theta}$.
There are infinitely many such double cosets, one
for each odd integer.
Furthermore, using $G_{[y]}^0=E_0^\times G_{y,0}^0$,
it is easy to see that these $G_{y,0}^0$-$G^{0,\theta}$
double cosets belong to the same $G_{[y]}^0$-$G^{0,\theta}$
double coset. The element $\dot g$ belongs to this double coset
and Lemma~12.6 of \cite{HM2} tells us that if $\dot\theta=\dot g
\cdot\theta$, then the restriction $\dot\theta\,|\, G_{y,0}^0$
factors to an involution of $G_{y,0:0^+}^0$ whose
fixed points are a finite symplectic group.

Hence there are two distinct $K$-orbits $K\cdot\theta$ and
$K\cdot \dot\theta$ that are moderately
compatible with $\xi$.
By Theorem~\ref{maindimformula}, and remarks above, we have 
$$
\langle\Theta,\xi\rangle_G=\langle K\cdot\theta,\xi\rangle_K +
\langle K\cdot\dot\theta,\xi\rangle_K =
1 + \langle K\cdot\dot\theta,\xi\rangle_K.
$$
To finish, we indicate why $\langle K\cdot\dot\theta,\xi\rangle_K=0$.
By Theorem~\ref{maindimformula}, $\langle K\cdot\dot\theta,\xi\rangle_K
=\dim \Hom_{K^{0,\dot\theta}}(\rho^\prime(\Psi),\eta_{\dot\theta}^\prime(\Psi))$.
Note that $\dot\theta$-symmetry of each $\phi_i$ implies that
$\phi_i^2\,|\, K^{0,\dot\theta}=1$. By definition,
$\eta_{\dot\theta}^\prime(\Psi)^2=1$. In view of the above information
concerning $\dot\theta\,|\, G_{y,0}^0$, we see that the
restrictions of $\phi_i$ and $\eta_{\dot\theta}^\prime(\Psi)$
 to $G_{y,0}^{0,\dot\theta}$ must
be trivial, because a finite symplectic group has no nontrivial
characters. It follows that if 
$\Hom_{K^{0,\dot\theta}}(\rho^\prime(\Psi),\eta_{\dot\theta}^\prime(\Psi))$
is nonzero, then $\Hom_{G_{y,0}^{0,\dot\theta}}(\rho,1)$
is nonzero. But, as explained in the proof of Proposition~12.2(3)
of \cite{HM2}, results of Heumos and Rallis (\cite{HR}) can be used
to prove that an irreducible cuspidal representation of
a finite general linear group cannot be distinguished by
a symplectic group. Therefore $\Hom_{G_{y,0}^{0,\dot\theta}}
(\rho,1)=0$, which implies that $\langle K\cdot\dot\theta,\xi\rangle_K=0$.

\end{example}

\chapter{Equivalence of tame supercuspidal representations}

\section{Statement of results}
\label{sec:equivalence}

In this chapter, we apply results from the previous chapter
to obtain necessary and sufficient conditions for equivalence
of a pair of tame supercuspidal representations $\pi(\Psi)$
and $\pi(\dot\Psi)$ associated to generic cuspidal
$G$-data $\Psi$ and $\dot\Psi$.
Theorem~\ref{partialeqprob} is our first result
of this nature. It asserts that a simple equivalence
relation, called $G$-equivalence, defined on the
set of generic cuspidal $G$-data, coincides 
with the equivalence relation given by
$\pi(\Psi)\simeq \pi(\dot\Psi)$.
The notion of $G$-equivalence involves refactorization,
along with a couple of simple manipulations of $G$-data.

In Theorem~\ref{equivtheorem}, we reformulate the
conditions in Theorem~\ref{partialeqprob} in a way
that avoids reference to the notion of refactorization.
Suppose that $\Psi=(\vec\bG,\pi_{-1},\vec\phi)$ is
a reduced generic cuspidal $G$-datum. Let $\phi=\phi(\Psi)$
be the quasicharacter of $G^0$ given by $\phi=\prod_{i=0}^d
\phi_i\,|\, G^0$. Then Theorem~\ref{equivtheorem}
essentially says that the $G$-conjugacy class of
the equivalence class of the representation 
$\pi_{-1}\otimes\phi$ of $G^0$ determines
the equivalence class of $\pi(\Psi)$.
This can also be phrased in terms of $G$-conjugacy
of the inducing data for the
representation $\pi_{-1}\otimes\phi$.

In the previous chapter, we fixed $G$ and $K$ and used the terminology ``$(G,K)$-datum'' to 
refer to an extended generic cuspidal $G$-datum $\Psi$ such that $K(\Psi) = K$.  
We defined an equivalence relation on the set of $(G,K)$-data, called ``$K$-equivalence,'' 
via refactorization, $K$-conjugation and elementary transformations.
If we also allow conjugation by arbitrary elements of $G$,  we obtain an equivalence 
relation on the set of all extended generic cuspidal $G$-data.

\begin{definition}\label{Geqdef}
Two extended generic cuspidal $G$-data $\Psi$ and $\dot\Psi$ are said to be 
{\it  $G$-equivalent} if $\dot\Psi$ can be obtained from $\Psi$ by a finite 
sequence of refactorizations, $G$-conjugations and elementary transformations.
\end{definition}

Note that if $\Psi$ and $\dot\Psi$ are $G$-equivalent we may have
$K(\Psi)\not\not=K(\dot\Psi)$ (though $K(\Psi)$ and $K(\dot\Psi)$ will be conjugate).
Theorem \ref{partialeqprob} shows that $G$-equivalence 
of extended cuspidal
$G$-data corresponds to equivalence of the corresponding tame supercuspidal representations. 
Recall that if $\Psi=(\vec\bG,y,\rho,\vec\phi)$ is an extended generic cuspidal $G$-datum,
then, setting $\pi_{-1}={\rm ind}_{K^0(\Psi)}^{G^0}\rho$, we obtain
a reduced cuspidal $G$-datum $(\vec\bG,\pi_{-1},\vec\phi)$.
Of course, if we start with a reduced generic cuspidal $G$-datum, as discussed
in Section \ref{sec:notations}, we can produce
various extended generic cuspidal $G$-data $(\vec\bG,y,\rho,\vec\phi)$
 via various choices of inducing data $(y,\rho)$ for the
depth zero supercuspidal representation $\pi_{-1}$ of $G^0$. As shown by
Yu, the extended $G$-data arising in this way give rise to equivalent 
supercuspidal
representations of $G$.
If we start with a reduced generic cuspidal $G$-datum $\Psi^\prime=(\vec\bG,\pi_{-1},\vec\phi)$,
the various extended $G$-data $\Psi$ to which $\Psi^\prime$ is associated 
will have different inducing subgroups $K(\Psi)$. Hence there is not a natural choice
of open compact-mod-center
subgroup of $G$ that can be attached to the reduced datum $\Psi^\prime$.
Thus the notion of $K$-equivalence for $(G,K)$-data that was introduced
in Section \ref{sec:symmetrizing} does not have an obvious analogue
for reduced data.
By contrast, the above notion of $G$-equivalence for extended 
$G$-data translates into an analogous notion of $G$-equivalence for reduced 
$G$-data.

\begin{definition}\label{Geltran}
If $\Psi = (\vec\bG ,\pi_{-1},\vec\phi)$ is a reduced generic cuspidal 
$G$-datum and $\pi_{-1}$ is replaced by an equivalent representation of $G^0$,
then we say $\Psi$ has undergone an {\it elementary transformation}.
\end{definition}

\begin{definition}\label{redGeqdef}
Two reduced generic cuspidal $G$-data $\Psi$ and $\dot\Psi$ are said to be {\it  $G$-equivalent} if $\dot\Psi$ can be obtained from $\Psi$ by a finite sequence of refactorizations, $G$-conjugations and elementary transformations.\end{definition}

We have an analogue of Lemma \ref{GKeqlemone}:

\begin{lemma}\label{Geqlemone}
Let $\Psi$ and $\dot\Psi$
be (reduced or extended) generic cuspidal $G$-data. Then the following are equivalent:
\begin{enumerate}
\item $\Psi$ and $\dot\Psi$ are $G$-equivalent.
\item $\dot\Psi$ is an elementary transformation of a $G$-conjugate of a refactorization of $\Psi$.
\item The previous statement remains valid when the terms ``elementary transformation,'' ``$G$-con\-ju\-gate'' and ``refactorization''  are permuted arbitrarily.
\end{enumerate}
\end{lemma}

We omit the proof of Lemma \ref{Geqlemone} since it is entirely analogous to the proof of Lemma \ref{GKeqlemone}.  

\begin{lemma}\label{GeqKeq}
Suppose $\Psi$ and $\dot\Psi$ are extended generic cuspidal $G$-data.   Let $\Psi'$ and $\dot\Psi'$ be the associated reduced data.  Then $\Psi$ and $\dot\Psi$ are $G$-equivalent if and only if  $\Psi'$ and $\dot\Psi'$ are $G$-equivalent.
\end{lemma}

\begin{proof}
Let 
$
\Psi=(\vec\bG,y,\rho,\vec\phi)$ and
$\dot\Psi=(\vec{\dot\bG},\dot y,\dot\rho,\vec{\dot\phi})
$
be extended generic cuspidal $G$-data with associated reduced data
$\Psi'=(\vec\bG,\pi_{-1},\vec\phi)$ and
$\dot\Psi' =(\vec{\dot\bG},\dot\pi_{-1},\vec{\dot\phi})$.
It is routine to show that if $\Psi$ and $\dot\Psi$ are $G$-equivalent then so are $\Psi'$ and $\dot\Psi'$.

The converse is also straightforward once we show that if $\Psi'$ and $\dot\Psi'$ are related by an elementary transformation then there must exist an element $g\in G^0 = G^0(\Psi)= G^0(\dot\Psi)$ such that ${}^g\Psi$ and $\dot\Psi$ are related by an elementary transformation.  But the assumption that $\Psi'$ and $\dot\Psi'$ are related by an elementary transformation implies that 
$\rho\,|\,G^0_{y,0}$ and $\dot\rho\,|\,G^0_{\dot y,0}$ are associate in the sense of Moy and Prasad.  (See \cite{MP1}, \cite{MP2}.)  This means there must exist $g\in G^0$ such that $G^0_{gy,0}\cap G^0_{\dot y,0}$ maps surjectively onto both $G^0_{gy,0:0^+}$ and $G^0_{\dot y,0:0^+}$.  It must be the case that $g[y] = [\dot y]$, since otherwise $G^0_{gy,0}\cap G^0_{\dot y,0}$ would lie in a nonmaximal parahoric subgroup and thus it could not surject onto
$G^0_{gy,0:0^+}$ and $G^0_{\dot y,0:0^+}$.  It follows that ${}^g \Psi$ and $\dot\Psi$ are related by an elementary transformation.
\end{proof}

If $\Psi$ is an extended generic cuspidal $G$-datum then the equivalence class of 
$\pi (\Psi)$ only depends on the $G$-equivalence class of $\Psi$.  This 
easily follows from the corresponding fact for $(G,K)$-data.  
Next we state the two main results of this chapter.

\begin{theorem}\label{partialeqprob} 
Suppose $\Psi$ and $\dot\Psi$ are extended generic cuspidal $G$-data. 
Assume that Hypotheses~${\rm C}(\vec\bG)$ and ${\rm C}(\vec{\dot\bG})$ hold.
Then:
\begin{enumerate}
\item $\pi (\Psi)\simeq \pi (\dot\Psi)$ if and only if $\Psi$ and 
$\dot\Psi$ are $G$-equivalent.  
\item  If $\Psi$ and $\dot\Psi$ are $(G,K)$-data 
with $K=K(\Psi)=K(\dot\Psi)$, then $\kappa (\Psi)
\simeq \kappa (\dot\Psi)$ if and only if $\Psi$ and $\dot\Psi$ are 
$K$-equivalent.  
\end{enumerate}
\end{theorem}

If $H$ is a subgroup of $G$ and $g\in G$, let ${}^gH=\Int(g)H$.
Recall from Section~\ref{sec:Mackey} that if $\tau$ is a representation 
of a subgroup $H$ of $G$ and $g\in G$,
we define a representation ${}^g\tau$ of ${}^gH$ by setting
${}^g\tau(g^\prime)=\tau(\Int g^{-1}(g^\prime))$, $g^\prime\in  {}^gH$.

\begin{theorem}\label{equivtheorem} 
Suppose $\Psi=(\vec\bG,\rho,y,\vec\phi)$ and 
$\dot\Psi=(\vec{\dot\bG}, \dot\rho,\dot y,\vec{\dot\phi})$ are extended generic cuspidal 
$G$-data. Assume that Hypotheses~{\rm{C}}($\vec\bG$) and {\rm{C}}($\vec{\dot\bG}$) hold.
Set
$\phi=\prod_{i=0}^d \phi_i\,|\, G^0$, $\dot\phi=\prod_{i=0}^{\dot d}
\dot\phi_i\,|\, \dot G^0$,  $\pi_{-1}={\rm ind}_{K^0}^{G^0}\rho$, 
$\dot\pi_{-1}={\rm ind}_{\dot K^0}^{\dot G^0}\dot\rho$,
$\rho^\prime(\Psi)=\rho\otimes\phi$, and $\rho^\prime(\dot\Psi)
=\dot\rho\otimes\dot\phi$. Then the following are equivalent:
\begin{enumerate}
\item $\pi(\Psi) \simeq \pi (\dot\Psi)$. 
\item There exists $g\in G$ such that $K^0={}^g\dot K^0$
and $\rho^\prime(\Psi)\simeq {}^g \rho^\prime(\dot\Psi)$.
\item There exists $g\in G$ such that $K^0={}^g\dot K^0$,
 $\vec \bG={}^g\vec{\dot \bG}$
and $\rho^\prime(\Psi)\simeq {}^g\rho^\prime(\dot\Psi)$.
\item There exists $g\in G$ such that $\bG^0={}^g{\dot \bG}^0$
and $\pi_{-1}\otimes \phi \simeq {}^g(\dot\pi_{-1}\otimes\dot\phi)$.
\item There exists $g\in G$ such that $\vec \bG={}^g\vec{\dot \bG}$
and $\pi_{-1}\otimes \phi \simeq {}^g(\dot\pi_{-1}\otimes\dot\phi)$.
\end{enumerate}
\end{theorem}

\begin{remark} Note that it follows from the above theorems
that if $\Psi$ and $\dot\Psi$ are generic cuspidal $G$-data
and
$\pi(\Psi)\simeq \pi(\dot\Psi)$, not only must the
the twisted Levi sequences for $\Psi$ and $\dot\Psi$
be conjugate, the sequences $\vec r$ and $\vec{\dot r}$
of depths must be identical.
\end{remark}

\begin{remark} Let $\pi=\pi(\Psi)$ be as in Theorem~\ref{equivtheorem}.
It follows from the theorem that the supercuspidal
representation $\partial\pi$ defined in Section~\ref{sec:inductive}
depends only on the representation $\pi$, not on the datum
$\Psi$.
\end{remark}

Recall that we say a $G$-datum is toral if 
the first group $\bG^0$ in the associated
twisted Levi sequence is a torus. 
In this case, $\pi_{-1}=\rho$ is a quasicharacter
of $G^0$ that is trivial on $G^0_{0^+}$,
and $\rho^\prime(\Psi)=\rho\phi$.
It is clear from the results above that
if $\Psi$ and $\dot\Psi$ are generic cuspidal $G$-data
such that $\pi(\Psi)$ and $\pi(\dot\Psi)$ are equivalent,
then $\Psi$ is toral if and only if $\dot\Psi$
is toral.
In the toral case, Theorem~\ref{equivtheorem}
translates into the following corollary.

\begin{corollary}\label{toralequiv} Let $\Psi$
and $\dot\Psi$ be toral generic
cuspidal $G$-data. Let $\phi$ and $\dot\phi$ be as above.
Then $\pi(\Psi)\simeq \pi(\dot\Psi)$ if and only if
there exists $g\in G$ such that $G^0={}^g\dot G^0$
and $\rho\phi={}^g(\dot\rho\dot\phi)$ (and in that case,
$\vec\bG={}^g\vec{\dot\bG}$).
\end{corollary}

Let $n$ be an integer such that $n\ge 2$,
and let $\varphi$
be an $F$-admissible quasicharacter of $E^\times$,
where $E$ is a tamely ramified
 extension of $F$ of degree $n$.
Recall that in Section~\ref{sec:Howeconstruction},
we discussed how to attach an extended generic cuspidal 
$G$-datum $\Psi$ to $\varphi$, in such a way
that the supercuspidal representations  of $\bGL_n(F)$
attached
via the constructions of Howe and Yu to
$\varphi$ and $\Psi$, respectively, are equivalent.
Let $\dot\varphi$ be an $F$-admissible quasicharacter
of $\dot E^\times$, where $\dot E$ is  a tamely ramified
degree $n$ extension of $F$, and let $\dot\Psi$
be the associated generic cuspidal $G$-datum.
If $\Psi$ and $\dot\Psi$ are toral, then the condition
for equivalence of $\pi(\Psi)$ and $\pi(\dot\Psi)$ given
in Corollary~\ref{toralequiv} coincides with
Howe's notion of $F$-conjugacy of $\varphi$
and $\dot\varphi$. (See Section~\ref{sec:Howeconstruction}
for the definition.)
If $\Psi$ and $\dot\Psi$ are not toral, then a straightforward
argument, involving facts about equivalence of twists
of depth zero supercuspidal representations of
general linear groups, shows that Condition~(2) of
Theorem~\ref{equivtheorem} is equivalent to
$F$-conjugacy of $\varphi$ and $\dot\varphi$.
Thus, because Howe proved that $F$-conjugacy of
$\varphi$ and $\dot\varphi$ corresponds to
equivalence of the associated supercuspidal
representations, the results of this section,
when applied to tame supercuspidal
representations of $\bGL_n(F)$, are equivalent to Howe's
criterion for equivalence of the representations.

\section{Heisenberg $p$-groups in the group case}
\label{sec:groupHeis}

If $\cG$ is a group there is a corresponding  group $\v{\cG}= \cG\times\cG$ and an automorphism $\alpha (x,y) = (y,x)$ of $\v{\cG}$.  The diagonal of $\v{\cG}$ is the group $\v{\cG}^\alpha$ of points fixed by $\alpha$ and we have a bijection $\v{\cG}^\alpha (x,y)\mapsto y^{-1}x$ from $\v{\cG}^\alpha \bs \v{\cG}$ to $\cG$.
In the this way, each group $\cG$ may be canonically identified with a quotient space which is analogous to a symmetric space.  In the literature which surrounds harmonic analysis on symmetric spaces, the ``group case'' of symmetric space theory involves the theory of those symmetric spaces which are manufactured from groups in the manner we have described.

In many cases, translating harmonic analysis on a group $\cG$ into harmonic analysis on $\v{\cG}^\alpha\bs \v{\cG}$ merely yields new proofs of known results.  In other cases, this translation gives new insights which lead to new results about $\cG$ or points the way to results which apply to general symmetric spaces.  For example, trace formulas which involve averages over the diagonal of a kernel function $K(x,y)$ on $\cG\times \cG$ have been powerful tools in studying harmonic analysis on a group $\cG$.  
Our development in this section of the Heisenberg theory in the group case will be used in the next section in the proof of Theorem \ref{partialeqprob}.

The theory in this section is a  special case of the theory discussed in Section \ref{sec:Heis}.  Fix a Heisenberg $p$-group $\cH$ with center $\cZ$.  Let $\alpha$ be the automorphism of $\cH\times \cH$ given by $\alpha (h_1,h_2) = (h_2,h_1)$ and let
$$\v{\cH} = (\cH\times \cH)/(\cZ\times \cZ)^\alpha,$$ where $$(\cZ\times \cZ)^\alpha = \{\, (z,z)
\ | \ z\in \cZ\,\}.$$  
The automorphism $\alpha$ gives rise to an automorphism of $\v{\cH}$ which we also denote by $\alpha$.  It is easy to see that both the commutator subgroup and the center of $\v{\cH}$ are equal to
$$\v{\cZ} = (\cZ\times \cZ)/(\cZ\times \cZ)^\alpha.$$  It follows that $\v{\cH}$ is a Heisenberg $p$-group with center $\v{\cZ}$.  
Let $\v{W} = \v{\cH}/\v{\cZ} = W\times W$.

As in Section \ref{sec:Heis}, we define abelian groups
\begin{equation*}
\begin{split}
\v{\cH}_\alpha^+ &= \{\,  h\in \v{\cH} \ | \ \alpha (h) = h\,\}\cr
\widehat{\v{\cH}}{}_\alpha^-&= \{\,  h\in \v{\cH}\ | \ \alpha (h) = h^{-1}\,\}.
\end{split}
\end{equation*}
More concrete descriptions of these groups may be obtained by examining their preimages with respect to the natural projection
$$\cH\times \cH \to \v{\cH}.$$  Indeed, it is easy to show that an element $(h_1, h_2)\in \cH\times \cH$ projects to an element of $\v{\cH}_\alpha^+$ exactly when $h_1= h_2$.  On the other hand, $(h_1,h_2)$ projects to an element of $\v{\widehat{\cH}}{}_\alpha^-$ exactly when $h_1h_2\in \cZ$.  Moreover, each element of $\v{\widehat{\cH}}{}_\alpha^-$ has a unique preimage of the form $(h,h^{-1})$, for some $h\in \cH$. 

\begin{lemma}\label{groupsi}
Suppose $\nu : \cH \to W^\sharp$ is a special isomorphism.  Let $\v{\nu}: \v{\cH}\to \v{W}^\sharp$ be the special isomorphism obtained from the map $\nu \times \nu: \cH\times \cH \to W^\sharp \times W^\sharp$ and the obvious identification of $\v{W}^\sharp = \v{W}\boxtimes \v{\cZ}$ with $(W^\sharp \times W^\sharp)/(\cZ\times \cZ)^\alpha$.  As in Lemma \ref{polsplitting}, we may use $\v{\nu}$ to obtain a splitting map $\v{W}_\alpha^- \to \widehat{\v{\cH}}{}_\alpha^-$.  The image $\v{\cH}_\alpha^-$ of this splitting coincides with the image of the set
$$\{\, (h,h^{-1}) \ | \ h \in \nu^{-1}(W\times 1)\,\}$$ in $\widehat{\v{\cH}}{}_\alpha^-$.   Conversely, the special isomorphism associated by Lemma \ref{sppolarspeciso} to this split polarization  is $\v{\nu}$.
\end{lemma}

\begin{proof}
We have 
\begin{equation*}
\begin{split}
\widehat{\v{\cH}}{}_\alpha^-&= \{\, (h,h^{-1}) (\cZ\times\cZ)^\alpha\ | \ h\in \cH\,\}\cr
\v{\cH}{}_\alpha^-&= \widehat{\v{\cH}}{}_\alpha^-\cap \v{\nu}^{-1}(\v{W}\times 1)\cr
\v{W}_\alpha^-&= \{\, (w,w^{-1})\ | \ w\in W\,\}\cr
\v{\nu}^{-1}(\v{W}\times 1)&= \nu^{-1}(W\times 1)\times \nu^{-1}(W\times 1)\quad (\hbox{mod }(\cZ\times \cZ)^\alpha).
\end{split}
\end{equation*}
We see that $\v{\cH}_\alpha^-$ is a coset space which has coset representatives of the form $(hz,h^{-1}z)$, where $hz,h^{-1}z\in \nu^{-1}(W\times 1)$.  The fact that both $hz$ and $h^{-1}z$ lie in $\nu^{-1}(W\times 1)$ implies that $hz= hz^{-1}$.  But then $z^2=1$ and thus $z=1$.  Therefore $\v{\cH}_\alpha^-$ has the form asserted in the statement of the lemma. 

If $h\in \cH$ and $\nu (h) = (w,1)$ let $s(w,w^{-1}) = (h,h^{-1})(\cZ\times \cZ)^\alpha$.  This is the splitting map associated to our split polarization.  One easily verifies that $\v{\nu}$ is the same as the special isomorphism associated to $s$ by checking that the two special isomorphisms on $\v{\cH}$ agree on $\v{\cH}_\alpha^+$, $\v{\cH}_\alpha^-$ and $\cZ$.
\end{proof}

Now fix $\nu$ and use the notations of Lemma \ref{groupsi}. Let $\v{\zeta}$ be a nontrivial character of $\v{\cZ}$ and let $\v{\tau} $ be a Heisenberg representation of $\v{\cH}$ with central character $\v{\zeta}$.
Theorem \ref{Heisthm} implies  that $\Hom_{\v{\cH}_\alpha^+}(\v{\tau} ,1)$ has dimension one. 

We now recall a standard result:

\begin{lemma}\label{lemflath}
Let $\cG$ be a totally disconnected group and let $\v{\cG}= \cG\times \cG$.  Assume $\v{\pi}$ is an irreducible admissible representation of $\v{\cG}$ such that $\Hom_{\v{\cG}^\alpha}(\v{\pi},1)\ne 0$, where $\v{\cG}^\alpha$ is the diagonal of $\cG\times \cG$.  Then there exists an irreducible admissible representation $\pi$ of $\cG$ such that $\v{\pi}$ is equivalent to $\pi\times \tilde\pi$.
\end{lemma}

\begin{proof}
According to Theorem 1 in \cite{F}, $\v{\pi}$ must factor as a product $\pi\times \pi'$ of two irreducible admissible representations of $\cG$.  But if $\v{\lambda}$ is a nonzero element of $\Hom_{\v{\cG}^\alpha}(\v{\pi},1)$ then $\v{\lambda}$ corresponds to a $\cG$-invariant bilinear pairing between $\pi$ and $\pi'$.  Hence, $\pi$ and $\pi'$ must be contragredients of each other.
\end{proof}

We will use this fact for various groups $\cG$.
In particular, we may apply it to our representation $\v{\tau}$, viewed as a representation of $\cH\times \cH$.  We deduce that there must exist an irreducible representation $\tau$ of $\cH$ such that $\v{\tau}\simeq \tau\times \tilde\tau$.  The central character of $\tau$ must be the character $\zeta (z) = \v{\zeta}(z,1)$, where we are viewing $\v{\zeta}$ as a character of $\cZ\times \cZ$.  Since $\zeta$ is nontrivial, $\tau$ must be a Heisenberg representation of $\cH$ with central character~$\zeta$.  

The Heisenberg representation $\v{\tau}^\sharp=\v{\tau}\circ\v{\nu}^{-1}$ of $\v{W}^\sharp$ has a unique extension to a Heisenberg-Weil representation $\hat{\v{\tau}}^\sharp$ of $\v{\cS}\ltimes \v{W}^\sharp$, where  $\v{\cS}= {\rm Sp}(\v{W})$.  Pulling back via $\v{\nu}$ gives a representation $\hat{\v{\tau}}$ of $\v{\cS}\ltimes_{\v{\nu}} \v{\cH}$ which extends $\v{\tau}$.  
Similarly, $\tau^\sharp= \tau\circ \nu^{-1}$ extends to a representation $\hat\tau^\sharp$ of $\cS\ltimes W^\sharp$, where $\cS = {\rm Sp}(W)$.  It should be emphasized that $\cS\times\cS$ is properly contained in $\v{\cS}$.

\begin{lemma}\label{tauhatsharp}
The pullback of $\v{\hat\tau}^\sharp$ to $(\cS\ltimes W^\sharp)\times (\cS\ltimes W^\sharp)$ is isomorphic to $\hat\tau^\sharp\times \tilde{\hat\tau}{}^\sharp$.  The pullback of $\v{\hat\tau}$ to $(\cS\ltimes_\nu\cH)\times (\cS\ltimes_\nu \cH)$ is isomorphic to $\hat\tau\times  \tilde{\hat\tau}$.
\end{lemma}

\begin{proof}
The representation $\hat\tau^\sharp\times \tilde{\hat\tau}{}^\sharp$ is the unique representation of 
$(\cS\ltimes W^\sharp)\times (\cS\ltimes W^\sharp)$ which extends  $\tau^\sharp \times \tilde\tau^\sharp$.  To prove the first assertion, it therefore suffices to show that the pullback of $\v{\hat\tau}^\sharp$ extends a representation equivalent to $\tau^\sharp \times \tilde\tau^\sharp$.  But $\v{\hat{\tau}}^\sharp$ extends $\v{\tau}^\sharp = \v{\tau}\circ \v{\nu}^{-1}$ and $\v{\tau}$, viewed as a representation of $\cH\times \cH$ is equivalent to $\tau\times \tilde\tau$.  This gives our first assertion.  The second assertion also follows directly.
\end{proof}

Let us now return to the setup of Section \ref{sec:genHeis}.  
There, we had a pair $(\bG^\prime,\bG)$, a torus $\bT$ and a point $y\in A(\bG,\bT,E)$.  
We now replace $\bG$ by $\v{\bG} = \bG\times \bG$ and consider the involution $\theta (a,b)= (b,a)$ on $\v{\bG}$.  Let  
\begin{equation*}
\begin{split}
\v{\bG^\prime}&= \bG^\prime\times\bG^\prime\cr
\v{\bT}&=\bT \times \bT\cr
 \v{\phi}&= \phi\times \phi^{-1}\cr
 \v{J} &= J\times J\cr 
 \v{J}_+&= J_+\times J_+\cr
 \v{\zeta}&= \zeta\times \zeta^{-1}\cr
\v{N} &= \ker( \v{\zeta}) = (\v{J}_+)^\theta(N\times N)\cr
\v{\cH} &= \v{J}/\v{N}\cr
\v{\cZ}&= \v{J}_+/ \v{N}.
\end{split}
\end{equation*}  
We also let $\v{K}' = G^\prime_{[y]}\times G^\prime_{[y]}$, and let $\v{f}: \v{K}'\to {\rm Sp}(\cH)$ and $\v{f}': \v{K'}\to \v{\cS}$ be the maps coming from conjugation.  Fix a relevant special isomorphism $\nu :\cH \to W^\sharp$.  Then the associated special isomorphism $\v{\nu} : \v{\cH}\to \v{W}^\sharp$ must also be relevant.  (In fact, the mapping 
$K'\to {\rm Sp}({\v{W}}^\sharp) : k\mapsto \v{\nu}\circ \v{f}(k)
\circ\v{\nu}^{-1}$ has image in $\cS\times\cS$.)
In particular, letting $\nu^\bullet : \cH\to W^\sharp$ be Yu's canonical special isomorphism, it must be the case that the associated special isomorphism $\v{\nu}^\bullet$ on $\v{\cH}$ is relevant.  We actually have:

\begin{lemma}\label{grYusi}
$\v{\nu}^\bullet : \v{\cH}\to \v{W}^\sharp$ is Yu's special isomorphism on $\v{\cH}$.
\end{lemma}

\begin{proof}
Recall that Yu's special isomorphism is defined by first working over a splitting field and then restricting.  It is therefore evident that it suffices to prove our claim in the split case.  Now observe that $\Phi (\v{\bG},\v{\bT},F) = \Phi (\bG\times 1,\bT\times 1, F)\sqcup \Phi (1\times \bG,1\times \bT,F)$.   Choose an ordering of $\Phi (\bG,\bT,F)$ and use this to order $\Phi (\bG\times 1,\bT\times 1,F)$ and $\Phi (1\times\bG,1\times \bT,F)$.  The resulting set of positive roots in $\Phi (\v{\bG},\v{\bT},F)$ determines an ordering on $\Phi (\v{\bG},\v{\bT},F)$.  Given this ordering, we get subgroups $\v{J}(+)$ and $\v{J}(-)$, as in the definition of $\nu_E^\bullet$ in Section \ref{sec:genHeis}.  By construction, $\v{J}(+) = J(+)\times J(+)$ and $\v{J} (-) = J(-)\times J(-)$, with the obvious notations.  The rest of the proof is now routine.
\end{proof}

Having fixed $\nu$, we define a representation $\v{\phi}'$ of $\v{K}'\v{J}$, 
as in Section \ref{sec:genHeis}, by
$$\v{\phi}'(kj) = \v{\phi}(k)\; \v{\hat\tau}^\sharp (\v{f}' (k),\v{\nu}(j)),$$
with $k=(k_1,k_2)\in \v{K}'$ and $j= (j_1,j_2)\in \v{J}$.  
Lemma \ref{tauhatsharp} implies that $\v{\phi}' = \phi'\times \tilde\phi'$, where $\phi'$
is defined on $G'_{[y]}J$ with the usual procedure.

Let $K' = G^\prime_{[y]}$.  The group $\v{\cM}$ consists of the elements of 
$\v{\cS}$ which preserve our polarization.  It contains the automorphisms 
$\v{f}' (\v{k})$, where $\v{k}\in \v{K}'\cap \v{G}^\theta$.  
Let $\chi^{\v{\cM}}$ be the unique character of $\v{\cM}$ of order two.

\begin{lemma}\label{grtrivchar}
The characters 
$k\mapsto\v{\phi}(\v{k})$
 and 
 $k\mapsto \chi^{\v{\cM}}(\v{f}'(\v{k}))$ of $\v{K}'\cap \v{G}^\theta$ are trivial.   Consequently, 
$\Hom_{\v{K}'\cap \v{G}^\theta}(\v{\phi}',1) = \Hom_{\v{\cH}^+_\alpha}(\v{\tau},1)$.
\end{lemma}
\begin{proof}
It is elementary to verify that $\v{\phi}$ is trivial on $\v{K}'\cap \v{G}^\theta$.
Now suppose $j\in J$ and let $j^+ = (j,j)$ and $j^-= (j,j^{-1})$.  Then $j^+$ and $j^-$  modulo $\v{N}$ are elements of $\v{\cH}^+_\alpha$ and $\v{\cH}^-_\alpha$, respectively.  Let $\v{k} = (k,k)$, with $k\in K'$.  Then, modulo $\v{N}$, we have $\v{f}'(\v{k})j^\pm = (kjk^{-1}, kj^{\pm 1}k^{-1})\in \v{\cH}^\pm_\alpha$.  Now $\v{f}' (\v{k})$ restricts to give an invertible linear transformation of the $\F_p$-vector space $\v{W}^+_\alpha$, where the latter space is the image of $\v{\cH}^+_\alpha$ in $\v{W}$.   The determinant of this transformation is $\chi^{\v{\cM}}(\v{f}' (\v{k}))$.   In fact, the operator $\v{f}' (\v{k})$ on $\v{W}^+_\alpha$ is naturally identified with  the operator $f' (k)$ on $W$.  Since the latter operator lies in $\cS$, it must have determinant one.  Thus, in general, $\chi^{\v{\cM}}(\v{f}' (\v{k}) )=1$.   It follows from 
Proposition~\ref{multonekappai} that $\Hom_{\v{K}'\cap \v{G}^\theta}
(\v{\phi}',1) = \Hom_{\v{\cH}^+_\alpha}(\v{\tau},1)$.
\end{proof}

Finally, we switch to the setup of Section \ref{sec:construction}, replacing 
$\bG$ by $\v{\bG} = \bG\times \bG$ and using the involution $\theta (g_1,g_2)= (g_2,g_1)$,
$g_1$, $g_2\in \bG$.
Suppose that we have a generic cuspidal
 $G$-datum $\Psi= (\vec\bG , y,  \rho, \vec\phi)$.
Let $\dot y\in \cB(\bG^0,F)$ be such that $[\dot y]=[y]$.
Set $\vec\phi^{-1}=(\phi_0^{-1}, \dots,\phi_d^{-1})$, and
let $\dot\rho$ be a representation of $K^0$ such that
$\dot\Psi=(\vec\bG,\dot y,\dot\rho,\vec\phi^{-1})$ is a generic
 cuspidal $G$-datum (for this, it suffices that $\dot\rho$ satisfies
Condition \textbf{D4} in the definition of cuspidal $G$-datum).
Next, let
 $\v{\Psi} = (\v{\vec\bG},\v{y},\v{\rho},\v{\vec{\phi}})$, where 
\begin{equation*}
 \begin{split}
 \v{\bG}^i &= \bG^i\times \bG^i\cr
 \v{y} &= (y,\dot y),\cr
 \v{\rho} &= \rho\times\dot\rho\cr
 \v{\phi}_i &= \phi_i\times\phi_i^{-1}.
 \end{split}
\end{equation*}
In other words, $\v{\Psi} = \Psi \times \dot{\Psi}$ is the product
 $\v{G}$-datum attached to $\Psi$ and $\dot{\Psi}$, as defined in
Section~\ref{sec:Yuproducts}. Note that $\v{\Psi}$ is
a $\theta$-symmetric generic cuspidal $\v{G}$-datum.

Let $i\in \{\,0,\ldots,d-1\,\}$. Set $\v{K}^{i+1}=K^{i+1}(\v{\Psi})$.
Then $\v{K}^{i+1}=K^{i+1}\times K^{i+1}$, where $K^{i+1}=K^{i+1}(\Psi)=K^{i+1}(\dot\Psi)$.
 The above discussion can be used to construct a representation $\v{\phi}'_i$ of 
$\v{K}^{i+1}$.  According to Lemma \ref{rsiequiv}, the
 choice of relevant special isomorphism used in the construction does not matter 
(up to isomorphism).  The representation $\v{\phi}'_i$ inflates to a representation 
$\v{\kappa}_i$ of $\v{K}= K\times K$.  Since $\v{\phi}'_i = \phi'_i\times\tilde\phi'_i$, 
we have $\v{\kappa}_i =\kappa_i\times \tilde\kappa_i$.  
Proposition \ref{multonekappai} and Lemma \ref{grtrivchar} imply that,
 $$
\Hom_{\v{K}^\theta}(\v{\kappa}_i,1) = \Hom_{\v{K}^{i+1,\theta}}(\v{\phi}'_i,1) = 
\Hom_{\v{\cH}_i^+} (\v{\tau}_i ,1)
$$
 and Proposition \ref{appliedfactor} implies
 $$
\Hom_{\v{K}^\theta}(\kappa(\v{\Psi}), 1)\cong
\Hom_{\v{K}^{0,\theta}}(\rho(\v{\Psi}),1).
$$
By definition, $\rho(\v{\Psi})=\rho\times\dot\rho$, and
Lemma~\ref{productgenericity} shows that 
$\kappa(\v{\Psi})\simeq \kappa\times\dot\kappa$, where
$\kappa=\kappa(\Psi)$ and $\dot\kappa=\kappa(\dot\Psi)$.
Hence the above isomorphism can be rewritten in the
form
$$
\Hom_{\v{K}^\theta}(\kappa\times\dot\kappa,1)\cong
\Hom_{\v{K}^{0,\theta}}(\rho\times\dot\rho,1).
$$
These results will be applied to $\theta$-symmetric
refactorizations of $\v{G}$-data in the next section.


\section{Proofs of Theorems 6.6 and 6.7}
\label{sec:proofofthm}

In the proofs of Theorems~\ref{partialeqprob} and \ref{equivtheorem},
we apply results from the previous chapter to the group
$\v{\bG}=\bG\times\bG$ and to the $\v{G}$-orbit of the involution
$\theta$ of $\v{G}$ defined by $\theta(g_1,g_2)=(g_2,g_1)$,
$g_1$, $g_2\in G$. 

For the rest of this section, let 
$\Psi=(\vec\bG,y,\rho,\vec\phi)$ and $\dot\Psi
=(\vec{\dot\bG},\dot y,\dot\rho,\vec{\dot\phi})$ be
extended generic cuspidal $G$-data.
Assume that Hypotheses~C($\vec\bG$)
and C($\vec{\dot\bG}$) hold.

Recall that the $G$-datum $\widetilde{\dot\Psi}
=(\vec{\dot\bG},\dot y,\tilde{\dot\rho},\vec{\dot\phi}^{-1})$,
as defined in Section~\ref{sec:contragredients},
has the property that $K(\widetilde{\dot\Psi})
=K(\dot\Psi)$ and $\kappa(\widetilde{\dot\Psi})$
is contragredient to $\kappa(\dot\Psi)$, as
shown in Theorem~\ref{contrathm}.
Hence $\pi(\widetilde{\dot\Psi})$ is contragredient
to $\pi(\dot\Psi)$.

Let $\v{\Psi}=\Psi\times\widetilde{\dot\Psi}$
be the product $\v{G}$-datum associated to
$\Psi$ and $\widetilde{\dot\Psi}$.
As shown in Lemma~\ref{productgenericity},
$\v{\Psi}$ is a generic cuspidal $\v{G}$-datum,
and $\kappa(\v{\Psi})\simeq \kappa(\Psi)
\times \kappa(\widetilde{\dot\Psi})$.

\begin{lemma}\label{symprop} 
Let $\theta$ be the involution of $\v{G}$ defined
above. Suppose that there exists a $\theta$-symmetric
refactorization of $\v{\Psi}$.
Then 
\item{(1)} $\vec\bG=\vec{\dot\bG}$.
\item{(2)} The depths of $\phi_i$ and $\dot\phi_i$ are equal,
for all $i\in \{\,0,\dots,d\,\}$, where $d$ is the degree of 
$\Psi$ (and of $\dot\Psi$).
\item{(3)} $[y]=[\dot y]$.
\item{(4)} $K(\Psi)=K(\dot\Psi)$ and $K^0(\Psi)=K^0(\dot\Psi)$.
\item{(5)} $\kappa(\Psi)\simeq\kappa(\dot\Psi)$ if and only
if $\rho^\prime(\Psi)\simeq \rho^\prime(\dot\Psi)$.
\end{lemma}

\begin{proof} Let $\v{\check\Psi}$ be a $\theta$-symmetric
refactorization of $\v{\Psi}$. As the twisted Levi
sequence for $\v{\check\Psi}$ is the same as that
for $\v{\Psi}$, the fact that each subgroup occurring
in the twisted Levi sequence is $\theta$-stable translates
into statement~(1). Part~(3) also follows immediately
from the definition
of $\theta$-symmetric refactorization.
It is also clear from that definition, and
from the definition of $\v{\Psi}$ that if
the depths of $\phi_d$ and $\dot\phi_d$
are distinct, there cannot exist a
$\theta$-symmetric refactorization
of $\v{\Psi}$. 
For the rest of (2), in view of (1), the fact that $\v{\Psi}$ is a 
generic cuspidal
$\v{G}$-datum requires that for each $i\in \{\,0,\dots,d-1\,\}$
the quasicharacter of
${G}^i\times G^i$ that occurs in $\v{\Psi}$
is $G^{i+1}\times G^{i+1}$-generic.
Looking at the definition of $\v{\Psi}$,
the only way that this can happen
is if $\phi_i$ and $\dot\phi_i$
are of equal depth.

Part~(4) is an immediate consequence of Parts~(1)--(3).

Let $K=K(\Psi)=K(\dot\Psi)$, and $K^0=K^0(\Psi)=K^0(\dot\Psi)$.
Let $\tilde\rho^\prime(\dot\Psi)$ and $\tilde\kappa(\dot\Psi)$
be the contragredients of $\rho^\prime(\dot\Psi)$ and
$\kappa(\dot\Psi)$, respectively.
Referring back to the discussion at the end of
Section~\ref{sec:groupHeis}, we see that, because
$\v{\check\Psi}$ is a $\theta$-symmetric refactorization
of $\v{\Psi}$, the character
$\eta'_\theta (\v{\check\Psi})$ is trivial,
and thus we have
$$
\Hom_{(K\times K)^\theta}(\kappa(\v{\check\Psi}),1)
\simeq \Hom_{(K^0\times K^0)^\theta}(\rho(\v{\check\Psi}),1)
=
\Hom_{(K^0\times K^0)^\theta}(\rho^\prime(\v{\check\Psi}),1).
$$
Combining this with 
$$\kappa(\v{\check\Psi})\simeq
\kappa(\v{\Psi})\simeq \kappa(\Psi)\times \tilde{\kappa}(\dot\Psi)
$$
and
$$
\rho^\prime(\v{\check\Psi}) \simeq \rho^\prime(\v{\Psi})\simeq
\rho^\prime(\Psi)\times \tilde{\rho}^\prime(\dot\Psi)
$$
we obtain
$$
\Hom_{(K\times K)^\theta}(\kappa(\Psi)\times 
\tilde{\kappa}(\dot\Psi),1)\simeq \Hom_{(K^0\times K^0)^\theta}
(\rho^\prime(\Psi)\times \tilde{\rho}^\prime(\dot\Psi),1).
$$
This relation is equivalent to the statement in Part~(5).
\end{proof}

\begin{proof} [Proof of Theorem~\ref{partialeqprob}]
Suppose that $\Psi$ and $\dot\Psi$ are $G$-equivalent. It follows from
properties of $G$-equivalence and Proposition \ref{refactorequiv}
that there
exists $g\in G$ such that $K(\dot\Psi)={}^gK(\Psi)$ and $\kappa(\dot\Psi)
\simeq {}^g \kappa(\Psi)$. This implies that $\pi(\dot\Psi)\simeq
\pi(\Psi)$. Similarly, if $K(\Psi)=K(\dot\Psi)$ and $\Psi$ and $\dot\Psi$
are $K=K(\Psi)$-equivalent, then Proposition \ref{refactorequiv}
yields $\kappa(\dot\Psi)\simeq \kappa(\Psi)$.

Suppose that $\pi(\Psi)\simeq \pi(\dot\Psi)$.  Let $\tilde\pi(\dot\Psi)$
be the contragredient of $\pi(\dot\Psi)$. Because
$\pi(\v{\Psi})$ is equivalent to $\pi(\Psi)\times\tilde{\pi}(\dot\Psi)$,
we have 
$\Hom_{\v{G}^\theta}(\pi(\v{\Psi}),1)\not=0$.
Hence $\langle \Theta ,\xi\rangle_{\v{G}}\ne 0$, where $\Theta$ is the 
$\v{G}$-orbit of $\theta$ 
and $\xi$ is the $K(\v{\Psi})$-equivalence class of $\v{\Psi}$. 
Hence $\langle \Theta',\xi\rangle_{K(\v{\Psi})}\ne 0$ for some 
$K(\v{\Psi})$-orbit $\Theta'$ contained in $\Theta$.  After replacing $\v{\Psi}$ by a 
$\v{G}$-conjugate, we may assume that $\Theta'$ contains the involution 
$\theta$.
We note that replacing $\v{\Psi}$ by a $\v{G}$-conjugate is equivalent to 
replacing $\Psi$ and $\dot\Psi$ by $G$-conjugates. 
Hence to complete the proof of~(1), it suffices
to show that if $\Theta^\prime$ is the $K(\v{\Psi})$-orbit
of $\theta$, then $\langle\Theta^\prime,\xi\rangle_{K(\v{\Psi})}\ne
0$ implies that $K(\Psi)=K(\dot\Psi)$ and
$\Psi$ and $\dot\Psi$ are $K(\Psi)$-equivalent.
According to Theorem \ref{maindimformula}(4), there exists a 
$\theta$-symmetric 
element $\v{\check\Psi}\in \xi$. As observed at the end of
Section~\ref{sec:groupHeis},
$\eta^\prime_\theta(\v{\check\Psi})$ is trivial and
$$
0\not=\langle\Theta^\prime,\xi\rangle_{K(\v{\Psi})}=
\dim {\Hom}_{K^0(\v{\check\Psi})^\theta}
(\rho^\prime(\v{\check\Psi}),1).
$$

There exist generic cuspidal $G$-data $\Psi^\prime$
and $\dot\Psi^\prime$ such that
such that
\begin{itemize}
\item $K(\Psi^\prime)=K(\Psi)$ and $\Psi^\prime$ is 
$K(\Psi)$-equivalent to $\Psi$,
\item $K(\dot\Psi^\prime)=K(\dot\Psi)$ and $\dot\Psi^{\prime}$ is 
$K(\dot\Psi)$-equivalent to $\dot\Psi$,
\item $\v{\check\Psi} = \Psi^{\prime}\times \widetilde{\dot\Psi}^\prime$.
\end{itemize}

Applying Lemma~\ref{symprop} to the $\theta$-symmetric
$\v{G}$-datum $\v{\check\Psi}$, and using 
the first two items above, we find that $[y]=[\dot y]$,
$\vec\bG=\vec{\dot\bG}$, and 
$K(\Psi^\prime)=K(\dot\Psi^\prime)$. 

To finish the proof of Part~(1) of the theorem, it suffices
to show that $\Psi^\prime$ and $\dot\Psi^\prime$ are 
$K(\Psi^\prime)$-equivalent.
In view of the above conditions, 
to do that, it remains to show that $\rho^\prime(\Psi^\prime)$ and 
$\rho^\prime(\dot\Psi^\prime)$ are equivalent.
This last equivalence is immediate from
$\Hom_{K^0(\v{\check\Psi})^\theta}(\rho^\prime(\v{\check\Psi}),1)
\not=0$ together with the equivalence of 
$\rho^\prime(\v{\check\Psi})$ and $\rho^\prime(\Psi^\prime)
\times \tilde{\rho}^\prime(\dot\Psi^\prime)$. (Here,
$\tilde{\rho}^\prime(\dot\Psi^\prime)$ denotes the
contragredient of $\rho^\prime(\dot\Psi^\prime)$.)

To complete the proof of the remaining direction of Part~(2) of 
Theorem \ref{partialeqprob},
we assume that $K(\Psi)=K(\dot\Psi)$ and $\kappa(\Psi)\simeq 
\kappa(\dot\Psi)$. Then from the equivalence of $\kappa(\v{\Psi})$
and $\kappa(\Psi)\times \kappa(\widetilde{\dot\Psi})$,
we have 
$\Hom_{K(\v{\Psi})^\theta}(\kappa(\v{\Psi}),1)\not=0$.
As shown above, the latter inequality
implies that $\Psi$ and $\dot\Psi$ are $K(\Psi)$-equivalent.
\end{proof}

\begin{proof} [Proof of Theorem~\ref{equivtheorem}]
Suppose that $K^0=\dot K^0$. It is easy to see that
if two $F$-subgroups
of $\bG$ have the property that their $F$-rational
points share a neighbourhood of the identity that is open
in both subgroups, then the subgroups are equal.
Hence $K^0=\dot K^0$ implies that $\bG^0=\dot{\bG}^0$.
Since $[y]$ and $[\dot y]$ are vertices in the
reduced building of $G^0$, if $[y]\not=[\dot y]$,
there exists an element of $G^0$ which moves
$[y]$ and fixes $[\dot y]$. 
Combining this with the fact that $K^0$ and $\dot K^0$ are the
stabilizers in $G^0$
of $[y]$ and $[\dot y]$, respectively,
we have that 
$K^0\not=\dot K^0$ whenever $[y]\not=[\dot y]$.
Thus, as we have assumed that $K^0=\dot K^0$,
we must have $[y]=[\dot y]$.

Suppose that $\rho^\prime(\Psi)\simeq \rho^\prime(\dot\Psi)$.
The above equivalence implies that $\phi$ and
$\dot\phi$ agree on $G^0_{y,0^+}$.
Note that $\phi(\v{\Psi})=\phi\times\dot\phi^{-1}$
and $K_+^0(\v{\Psi})=G_{y,0^+}^0\times G_{y,0^+}^0$.
Hence $\phi(\v{\Psi}) \,|\, K_+^0(\v{\Psi})^\theta
\equiv 1$. Applying Lemma~\ref{extra}, we
conclude that there exists a weakly $\theta$-symmetric
refactorization of $\v{\Psi}$. Since
the assumption $[y]=[\dot y]$ guarantees that
any weakly $\theta$-symmetric refactorization
of $\v{\Psi}$ is $\theta$-symmetric,
we may apply Lemma~\ref{symprop} to conclude
that $K(\Psi)=K(\dot\Psi)$ and $\kappa(\Psi)
\simeq\kappa(\dot\Psi)$.
Adding in conjugation by an element $g\in G$,
the above argument shows that (2) implies (3)
and (3) implies (1).

Assume (1). Then $\Psi$ and $\dot\Psi$ are $G$-equivalent,
by Theorem~\ref{partialeqprob}. Hence there
exists $g\in G$ such that $K(\Psi)=K({}^g\dot\Psi)={}^gK(\dot\Psi)$,
and $\Psi$ and ${}^g\dot\Psi$ are $K=K(\Psi)$-equivalent.
It then follows from the definition of $K$-equivalent
$(G,K)$-data that there exists $k\in K$
with $K^0(\Psi)={}^{kg}K^0(\dot\Psi)$
and $\rho^\prime(\Psi)\simeq \rho^\prime({}^{kg}\dot\Psi)
={}^{kg}\rho^\prime(\Psi)$. Hence (1) implies (2).

Assume (3). That is, assume that there exists
$g\in G$ with $K^0={}^g\dot K^0$, $\vec\bG={}^g\vec{\dot\bG}$,
and
$$
\rho^\prime(\Psi)\simeq {}^g\rho^\prime(\dot\Psi)
=\rho^\prime({}^g\dot\Psi).
$$
Recall that $\rho^\prime(\Psi)=\rho(\Psi)\otimes\phi\,|\, K^0$,
and ${}^g\rho^\prime(\dot\Psi)={}^g(\rho(\dot\Psi)\otimes\dot\phi\,|\, \dot K^0)$. Inducing up to $G^0$, we obtain equivalence
of $\pi_{-1}\otimes\phi$ and ${}^g(\dot\pi_{-1}\otimes\dot\phi)$.
Hence (3) implies (5), (which in turn implies (4)).

Finally, assume (4). Let $\Psi^0=((\bG^0),\rho,y,\phi)$
and $\dot\Psi^0=(({\bG}^0),{}^g\dot\rho,g\cdot\dot y,{}^g\dot\phi)$.
Then $\Psi^0$ and $\dot\Psi^0$ are extended generic cuspidal
$G^0$-data. By assumption, the corresponding supercuspidal
 representations
$\pi(\Psi^0)$ and $\pi(\dot\Psi^0)$ of $G^0$ are equivalent.
From the implication (1) implies (2) for these representations,
we see that there
exists $h\in G^0$ such that $K(\Psi^0)={}^h K(\dot\Psi^0)$
and $\rho^\prime(\Psi^0)\simeq {}^h \rho^\prime(\dot\Psi^0)$.
Noting that $\rho^\prime(\Psi^0)=\rho^\prime(\Psi)$
and $\rho^\prime(\dot\Psi^0)={}^g\rho^\prime(\dot\Psi)$,
we conclude that (2) holds. Thus (4) implies (2).
\end{proof}

\backmatter

\bibliographystyle{amsalpha}

\Printindex{HM-index}{Index of Terminology}
\Printindex{HM-notations}{Index of Notations}
\end{document}